%% file: master.tex
\documentclass[english]{article}
\usepackage[T1]{fontenc}
\usepackage[latin9]{inputenc}
\usepackage{geometry}
\usepackage{mathrsfs}
\usepackage{mathtools}
\usepackage{amsmath}
\usepackage{amsthm}
\usepackage{amssymb}
\usepackage[numbers]{natbib}
\usepackage[unicode=true,pdfusetitle,
 bookmarks=true,bookmarksnumbered=false,bookmarksopen=false,
 breaklinks=false,pdfborder={0 0 1},backref=false,colorlinks=true]
 {hyperref}

\makeatletter
\theoremstyle{plain}
\newtheorem{thm}{\protect\theoremname}
  \theoremstyle{remark}
  \newtheorem{rem}[thm]{\protect\remarkname}
  \theoremstyle{plain}
  \newtheorem{cor}[thm]{\protect\corollaryname}
  \theoremstyle{plain}
  \newtheorem{lem}[thm]{\protect\lemmaname}
  \theoremstyle{plain}
  \newtheorem{prop}[thm]{\protect\propositionname}

\input{preamble}

\makeatother

\usepackage{babel}
  \providecommand{\corollaryname}{Corollary}
  \providecommand{\lemmaname}{Lemma}
  \providecommand{\propositionname}{Proposition}
  \providecommand{\remarkname}{Remark}
\providecommand{\theoremname}{Theorem}

\begin{document}

\title{Sampling normalizing constants in high dimensions using inhomogeneous
diffusions}

\author{Christophe Andrieu\thanks{University of Bristol, School of Mathematics},
James Ridgway and Nick Whiteley\thanks{University of Bristol, School of Mathematics}}
\maketitle
\begin{abstract}
Motivated by the task of computing normalizing constants and importance
sampling in high dimensions, we study dimension dependence of fluctuations
for additive functionals of time-inhomogeneous overdamped Langevin
type diffusions on $\mathbb{R}^{d}$. The main results are non-asymptotic
variance and bias bounds, and a central limit theorem in the $d\to\infty$
regime. We demonstrate that a temporal discretization inherits the
fluctuation properties of the underlying diffusion, which are controlled
at a computational cost growing at most polynomially with $d$. The
key steps include establishing Poincaré inequalities for time-marginal
distributions of the diffusion and nonasymptotic bounds on deviation
from Gaussianity in a martingale central limit theorem.
\end{abstract}
\tableofcontents{}

\newpage{}

\include{child-intro}

\include{child-poincare}

\include{child-CLT-inhomogenous}

\include{child-intro-proofs}

\include{child-poincare-proofs}

\include{child-CLT-inhomogenous-proofs}

\include{child-poisson}

\include{child-discretization}

\include{child-auxiliary-results}

\bibliographystyle{plain}
\bibliography{thermo_integration}

\end{document}

%% file: preamble.tex
\usepackage{stmaryrd}
\usepackage[textsize=tiny, disable]{todonotes}
\usepackage{mathabx}

\newcounter{hypA}
\newenvironment{hyp}{\refstepcounter{hypA}\begin{itemize}
\item[({\bf A\arabic{hypA}})]}{\end{itemize}}

%% file: child-intro.tex
\section{Introduction\label{sec:Introduction}}

Consider $(X_{t}^{\epsilon})_{t\in[0,1]}$ the time-inhomogeneous
diffusion on $\mathbb{R}^{d}$ which solves 
\begin{equation}
X_{t}^{\epsilon}=X_{0}^{\epsilon}-\epsilon^{-1}\int_{0}^{t}\nabla U_{s}(X_{s}^{\epsilon})\mathrm{d}s+\sqrt{2\epsilon^{-1}}\int_{0}^{t}\mathrm{d}B_{s},\label{eq:SDE_intro}
\end{equation}
where $B_{t}$ is $d$-dimensional Brownian motion, $\epsilon>0$
is a parameter and $(U_{t})_{t\in[0,1]}$ is a family of $\mathbb{R}$-valued
potentials such that, with Lebesgue measure and the Borel $\sigma$-algebra
denoted by $\mathrm{d}x$ and $\mathcal{B}(\mathbb{R}^{d})$, $(\pi_{t})_{t\in[0,1]}$
given by
\begin{equation}
Z_{t}\coloneqq\int_{\mathbb{R}^{d}}\exp\{-U_{t}(x)\}\mathrm{d}x,\quad\quad\pi_{t}(A)\coloneqq Z_{t}^{-1}\int_{A}\exp\{-U_{t}(x)\}\mathrm{d}x,\quad A\in\mathcal{B}(\mathbb{R}^{d}),\label{eq:Zs_and_pis_front}
\end{equation}
are well-defined as probability measures. 

This work concerns dependence on the dimension, $d$, of fluctuations
associated with
\begin{equation}
S_{\epsilon}\coloneqq\int_{0}^{1}f_{t}(X_{t}^{\epsilon})\mathrm{d}t,\quad\quad S_{\epsilon,h}\coloneqq h\sum_{k=0}^{\left\lfloor 1/h\right\rfloor -1}f_{kh}(X_{kh}^{\epsilon}),\quad\quad\tilde{S}_{\epsilon,h}\coloneqq h\sum_{k=0}^{\left\lfloor 1/h\right\rfloor -1}f_{kh}(\tilde{X}_{kh}^{\epsilon,h}),\label{eq:S_intro}
\end{equation}
where $(f_{t})_{t\in[0,1]}$ is a family of $\mathbb{R}$-valued functions
such that each $f_{t}$ is centred with respect to $\pi_{t}$, and
$(\tilde{X}_{t}^{\epsilon,h})_{t\in[0,1]}$ is an approximation to
$(X_{t}^{\epsilon})_{t\in[0,1]}$ such that the skeleton variables
$\tilde{X}_{kh}^{\epsilon,h}$ can be simulated by a time-discretization
method, and $h\in(0,1]$ is a step-size parameter such that the cost
of the discretization scheme is proportional to $h^{-1}$. 

Amongst our key assumptions, which we state precisely later, will
be strong convexity in $x$ of $U_{t}(x)$, or equivalently strong
log-concavity of $\pi_{t}$. As accounted in \cite{bakry2013analysis},
thorough investigations have been made of the connections between
concentration of measure phenomena, Poincaré and other functional
inequalities for log-concave measures and the ergodic properties of
time-homogeneous Markov processes, such as the diffusion in (\ref{eq:SDE_intro})
in the case that $U_{t}$ does not depend on $t$. These connections
have been exploited to study the computational cost of approximate
sampling from log-concave measures using \emph{Markov chain Monte
Carlo} (MCMC) algorithms, via bounds on distance to equilbrium and
error estimates for ergodic averages which elicit dependence on dimension,
e.g. \cite{frieze1994sampling,frieze1999log,joulin2010curvature,dalalyan2016theoretical}. 

Our primary motivation for studying the time-inhomogeneous case is
connected with another Monte Carlo technique: \emph{importance sampling,}
which along with MCMC is one of the most popular simulation-based
methods for numerical integration, and is applied across scientific
disciplines such as statistical physics, signal processing and machine
learning. Although as we shall illustrate next, importance sampling
in its most basic form can perform exponentially badly in high dimensions,
one of the main insights which can be drawn from our results is that
a more sophisticated type of importance sampling technique using an
inhomogeneous Markov process can be practically reliable, in a sense
which we shall make precise, at a cost polynomial in $d$. 

\subsection{Motivation: importance sampling and thermodynamic integration}

As an elementary example, consider the task of numerically approximating
the ratio of normalizing constants $Z_{1}/Z_{0}$ and the expectation
$\pi_{1}(f):=\int_{\mathbb{R}^{d}}\varphi(x)\pi_{1}(\mathrm{d}x)$
for some test function $\varphi$, assuming that one is able to simulate
$(\zeta{}_{1},\ldots,\zeta_{m})\stackrel{\mathrm{i.i.d.}}{\sim}\pi_{0}$
and evaluate $U_{0}$, $U_{1}$ and $\varphi$ pointwise. With $W_{i}\coloneqq\exp[-\{U_{1}(\zeta_{i})-U_{0}(\zeta_{i})\}]$,
so
\[
\frac{Z_{1}}{Z_{0}}=\mathbb{E}[W_{i}],\quad\quad\pi_{1}(\varphi)=\frac{\mathbb{E}[\varphi(\zeta_{i})W_{i}]}{\mathbb{E}[W_{i}]},
\]
the basic importance sampling method reports the approximations:
\begin{equation}
\frac{Z_{1}}{Z_{0}}\approx\frac{1}{m}\sum_{i=1}^{m}W_{i},\quad\quad\pi_{1}(\varphi)\approx\frac{\sum_{i=1}^{m}\varphi(\zeta_{i})W_{i}}{\sum_{i=1}^{m}W_{i}}.\label{eq:IS_basic}
\end{equation}

If for sake of illustration the potentials are of the form:
\begin{equation}
U_{t}(x)=\sum_{j=1}^{d}u_{t}(x^{j}),\quad x=(x^{1},\ldots,x^{d}),\label{eq:potentials_iid}
\end{equation}
we have for any $i$,
\begin{equation}
\frac{\mathrm{var}[W_{i}]}{\mathbb{E}[W_{i}]^{2}}=c^{d}-1,\label{eq:var_W}
\end{equation}
where $c\coloneqq\mathbb{E}[\exp-2\{u_{1}(\zeta_{1}^{1})-u_{0}(\zeta_{1}^{1})\}]/\mathbb{E}[\exp-\{u_{1}(\zeta_{1}^{1})-u_{0}(\zeta_{1}^{1})\}]^{2}$
does not depend on $d$, and $\zeta_{1}^{1}$ is the first of the
$d$ co\textendash ordinates of $\zeta_{1}$. By Jensen's inequality
$c\geq1$ with equality if and only if $\pi_{1}=\pi_{0}$, so putting
aside that trivial case, (\ref{eq:var_W}) indicates that the cost
of the simulation, governed by $m$, must be increased exponentially
in $d$ in order to prevent growth of the relative errors associated
with (\ref{eq:IS_basic}). Also when $c>1$, the total variation distance
between $\pi_{0}$ and $\pi_{1}$ is monotonically increasing in $d$,
and indeed as $d$ reaches infinity, $\pi_{0}$ and $\pi_{1}$ become
singular in the sense of Kakutani's theorem on infinite product measures.
Intuitively the ``one-step'' importance sampling correction from
$\pi_{0}$ to $\pi_{1}$ in (\ref{eq:IS_basic}) is defeated by this
phenomenon.

An alternative approach is based around the representation formulae:
\begin{align}
\frac{Z_{1}}{Z_{0}} & =\exp\left\{ -\int_{0}^{1}\pi_{t}(\partial_{t}U_{t})\mathrm{d}t\right\} =\mathbb{E}\left[\exp\left\{ -\int_{0}^{1}\partial_{t}U_{t}(X_{t}^{\epsilon})\mathrm{d}t\right\} \right],\label{eq:Jarz}\\
\pi_{1}(\varphi) & =\frac{\mathbb{E}\left[\varphi(X_{1}^{\epsilon})\exp\left\{ -\int_{0}^{1}\partial_{t}U_{t}(X_{t}^{\epsilon})\mathrm{d}t\right\} \right]}{\mathbb{E}\left[\exp\left\{ -\int_{0}^{1}\partial_{t}U_{t}(X_{t}^{\epsilon})\mathrm{d}t\right\} \right]},\label{eq:FK}
\end{align}
where $(X_{t}^{\epsilon})_{t\in[0,1]}$ as in (\ref{eq:SDE_intro})
with any $\epsilon>0$ and $X_{0}^{\epsilon}\sim\pi_{0}$, and $\partial_{t}U_{t}$
is the partial derivative of $U_{t}$ w.r.t. $t$, and $\pi_{t}(\partial_{t}U_{t})$
is the integral with respect to $\pi_{t}$ (we shall later discuss
conditions under which validity of (\ref{eq:Jarz})\textendash (\ref{eq:FK})
can be rigorously established). The equalities in (\ref{eq:Jarz})
have roots in the statistical physics literature, the first being
known as the thermodynamic integration or path sampling identity,
see \cite{gelman1998simulating} for an account of its history, the
second as Jarzynki's equality \cite{jarzynski1997nonequilibrium,jarzynski1997equilibrium}.
The expectations in (\ref{eq:Jarz})\textendash (\ref{eq:FK}) have
an importance sampling interpretation: $\exp\left\{ -\int_{0}^{1}\partial_{t}U_{t}(X_{t}^{\epsilon})\mathrm{d}t\right\} \frac{Z_{0}}{Z_{1}}$
can be derived as the Radon-Nikodym derivative with respect to the
path measure of $(X_{t}^{\epsilon})_{t\in[0,1]}$ as per (\ref{eq:SDE_intro})
with $X_{0}^{\epsilon}\sim\pi_{0}$, of the law the process with drift
transformed such that distribution of $X_{1}^{\epsilon}$ is $\pi_{1}$,
see \cite[Section 3.2, p.62]{rousset_thesis} for a time-reversal
perspective and \cite[Ch. VIII, Sec. 3]{RevuzYor} for background
on this type of transformation. The discrete-time counterpart of (\ref{eq:FK})
is the basis for the Annealed Importance Sampling method of \cite{neal2001annealed}. 

In light of (\ref{eq:Jarz})\textendash (\ref{eq:FK}), an alternative
to the basic importance sampling method described above is obtained
by replacing each pair $W_{i},\varphi(\zeta_{i})$ in (\ref{eq:IS_basic})
with an independent copy of the pair $\exp\left\{ -\int_{0}^{1}\partial_{t}U_{t}(X_{t}^{\epsilon})dt\right\} ,\varphi(X_{1}^{\epsilon})$,
or in practice some approximation thereof involving time-discretization.
If in (\ref{eq:S_intro}) one takes $f_{t}(x)=\partial_{t}U_{t}(x)-\pi_{t}(\partial_{t}U_{t})$,
then from (\ref{eq:Jarz}),
\[
S_{\epsilon}=\int_{0}^{1}\partial_{t}U_{t}(X_{t}^{\epsilon})-\pi_{t}(\partial_{t}U_{t})\mathrm{d}t=\int_{0}^{1}\partial_{t}U_{t}(X_{t}^{\epsilon})\mathrm{d}t-\log\frac{Z_{0}}{Z_{1}},
\]
hence our interest in the dimension dependence of the fluctuations
associated with (\ref{eq:S_intro}). 

To see why there is hope that this scheme can perform well in high
dimensions, note that in the setting (\ref{eq:potentials_iid}) with
$X_{0}^{\epsilon}\sim\pi_{0}$, the co-ordinates $(X_{t}^{\epsilon,1},\ldots,X_{t}^{\epsilon,d})$
of $X_{t}^{\epsilon}$ are i.i.d., as are the summands in: 
\[
S_{\epsilon}=\sum_{j=1}^{d}\int_{0}^{1}\partial_{t}u_{t}(X_{t}^{\epsilon,j})-\pi_{t}(\partial_{t}u_{t})\mathrm{d}t,
\]
where $\pi_{t}(\partial_{t}u_{t})$ is the integral of $\partial_{t}u_{t}$
w.r.t. any of the $1$-dimensional marginals of $\pi_{t}$. So, if
the variance and mean of $\int_{0}^{1}\partial_{t}u_{t}(X_{t}^{\epsilon,j})-\pi_{t}(\partial_{t}u_{t})$
are order $O(\epsilon)$ as $\epsilon\to0$, and $\epsilon$ is chosen
to be $d^{-1},$ then using the independence, $\mathbb{E}[S_{\epsilon}^{2}]$
is of order $O(1)$ as $d\to\infty$. If also $\sum_{j=1}^{d}\int_{0}^{1}\partial_{t}u_{t}(X_{t}^{\epsilon,j})$
can be well-approximated by discretization at a cost proportional
to $h^{-1}$ and polynomial in $\epsilon^{-1}$, then overall one
obtains a method to approximate (\ref{eq:Jarz})\textendash (\ref{eq:FK})
which does not suffer from exponentially bad behaviour in high dimensions. 

Of course in situations of practical interest, each $\pi_{t}$ is
usually not a product measure, i.e. $U_{t}$ is not of the form in
(\ref{eq:potentials_iid}), and the dependence on $d$ of the fluctuations
of $S_{\epsilon}$ in such situations is a less simple matter. Discussion
of our approach and related literature is given after introducing
notation and assumptions.

\subsection{Notation\label{sec:Notation-definitions}}

Inner-product and Euclidean norm on $\mathbb{R}^{d}$ are denoted
by respectively $\left\langle \cdot,\cdot\right\rangle $ and $\left\Vert \cdot\right\Vert $.
The $d\times d$ zero and identity matrices are written $0_{d}$ and
$I_{d}$, and $e_{i}$ denotes the vector in $\mathbb{R}^{d}$ whose
$i$'th entry is $1$ and whose other entries are zeros. For a $q$-dimensional
array $A$ with real entries $A[i_{1},\cdots,i_{q}]=a_{i_{1},\cdots,i_{q}}$,
$(i_{1},\ldots,i_{q})\in\{1,\dots,d\}^{q}$, the Hilbert-Schmidt norm
is denoted $\|A\|_{\mathrm{H.S.}}:=\left(\sum_{(i_{1},\ldots,i_{q})\in\{1,\dots,d\}^{q}}a_{i_{1},\cdots,i_{q}}^{2}\right)^{1/2}$.
When such an array depends on an argument $x\in\mathbb{R}^{d}$, we
define for $p\geq1$,
\begin{equation}
\|A\|_{p}\coloneqq\sup_{x\in\mathbb{R}^{d}}\frac{\|A(x)\|_{\mathrm{H.S.}}}{1+\|x\|^{2p}}.\label{eq:function_norm_A}
\end{equation}

For a function $f:\mathbb{R}^{d}\to\mathbb{R}$, we write $\nabla^{(q)}f$
for the $q$-dimensional array of $q$-th order partial derivatives
of $f$, with entries $\nabla^{(q)}f[i_{1},\ldots,i_{q}]=\frac{\partial^{q}f}{\partial x_{i_{1}}\cdots\partial x_{i_{q}}}$,
where $(i_{1},\ldots,i_{q})\in\{1,\dots,d\}^{q}$. In particular the
usual gradient is $\nabla^{(1)}\equiv\nabla$ and by convention we
take $\nabla^{(0)}f\equiv f$. The Laplacian operator is denoted $\Delta$.
As instances of (\ref{eq:function_norm_A}) we have for example,

\begin{equation}
\|f\|_{p}=\sup_{x\in\mathbb{R}^{d}}\frac{|f(x)|}{1+\|x\|^{2p}},\qquad\|\nabla^{(q)}f\|_{p}=\sup_{x\in\mathbb{R}^{d}}\frac{\|\nabla^{(q)}f(x)\|_{\mathrm{H.S.}}}{1+\|x\|^{2p}}.\label{eq:function_norm_f}
\end{equation}

We follow the convention of terminology that a $0$-times continuously
differentiable function is continuous. For $q\geq0$ and $p\geq1$,
let $C_{q}^{p}(\mathbb{R}^{d})$ be the set of functions $f:\mathbb{R}^{d}\to\mathbb{R}$
which are $q$-times continuously differentiable and such that $\|\nabla^{(r)}f\|_{p}<+\infty$,
for $0\leq r\leq q$.

We shall frequently encounter $\mathbb{R}$-valued functions with
domain $[0,1]\times\mathbb{R}^{d}$ or some subset thereof. For such
a function, say $f:(t,x)\in[0,1]\times\mathbb{R}^{d}\mapsto f(t,x)\in\mathbb{R}$,
we shall write interchangeably $f_{t}(x)\equiv f(t,x)$ . With $t$
fixed, we write $\nabla^{(q)}f_{t}$ for the array of $q$th-order
derivatives of the function $f(t,\cdot):\mathbb{R}^{d}\to\mathbb{R}$,
and with $x$ fixed, we write $\partial_{t}^{q}f_{t}(x)$ for the
$q$-th partial derivative of $f(\cdot,x):[0,1]\mapsto\mathbb{R}$,
with $\partial_{t}^{1}\equiv\partial_{t}$. Then $\|\nabla^{(q)}f_{t}\|_{p}$
(resp. $\|\partial_{t}^{q}f_{t}\|$) is as in (\ref{eq:function_norm_f})
with $\nabla^{(q)}f$ there replaced by $\nabla^{(q)}f_{t}$ (resp.
$\partial_{t}^{q}f_{t}$).

For nonnegative integers $q_{t},q_{x}$, let $C_{q_{t},q_{x}}^{p}([0,1]\times\mathbb{\mathbb{R}}^{d})$
be the set of functions $f:[0,1]\times\mathbb{\mathbb{R}}^{d}\to\mathbb{R}$
such that $f(t,x)$ is $q_{t}$-times continuously differentiable
in $t$, $q_{x}$-times continously differentiable in $x$, 
\[
\sup_{t\in[0,1]}\|\partial_{t}^{r}f_{t}\|_{p}<+\infty,\quad0\leq r\leq q_{t},\quad\mathrm{and}\quad\sup_{t\in[0,1]}\|\nabla^{(r)}f_{t}\|_{p}<+\infty,\quad0\leq r\leq q_{x}.
\]
Define 
\[
V(x)\coloneqq\|x\|^{2},\quad\quad\bar{V}(x)\coloneqq1+V(x),\quad\quad\bar{V}^{(p)}(x)\coloneqq1+V^{p}(x),\quad p>0.
\]
Below we shall identify for each $t\in[0,1]$ a distinguished point
$x_{t}^{\star},$ then write $V_{t}(x)\coloneqq\|x-x_{t}^{\star}\|^{2}$,
$\bar{V}_{t}(x)\coloneqq1+V_{t}(x)$, $\bar{V}_{t}^{(p)}(x)\coloneqq1+V_{t}^{p}(x)$.

The total variation distance between two probability measures $\nu,\nu^{\prime}$
on a $\sigma$-algebra $\mathcal{G}$ is written $\|\nu-\nu^{\prime}\|_{\mathrm{tv}}=\sup_{A\in\mathcal{G}}|\nu(A)-\nu^{\prime}(A)|$.
The integral of a function $f$ w.r.t. a measure $\nu$ is written
$\nu f$ or $\nu(f)$. The Borel $\sigma$-algebra and Lebesgue measure
on $\mathbb{R}^{d}$ are denoted respectively $\mathcal{\mathcal{B}}(\mathbb{R}^{d})$
and $\mathrm{d}x$. The set of probability measures $\nu$ on $\mathcal{\mathcal{B}}(\mathbb{R}^{d})$
such that $\nu(V^{p})<+\infty$ is denoted $\mathcal{P}^{p}(\mathbb{R}^{d})$. 

Throughout the paper $(\Omega,\mathcal{F},(\mathcal{F}_{t})_{t\in\mathbb{R}_{+}},\mathbb{P})$
is a filtered probability space satisfying the usual conditions, on
which all the random variables we encounter are defined, and $(B_{t})_{t\in\mathbb{R}_{+}}$
is a $d$-dimensional $(\mathcal{F}_{t})_{t\in\mathbb{R}_{+}}$-Brownian
motion. Expectation with respect to $\mathbb{P}$ is denoted $\mathbb{E}$.

With $U_{t}$ and $Z_{t}$ as in (\ref{eq:Zs_and_pis_front}), we
denote:
\begin{equation}
\phi_{t}(x):=-\partial_{t}U_{t}(x)-\partial_{t}\log Z_{t}.\label{eq:phi_defn}
\end{equation}

\subsection{Assumptions}

Fix a function $U:(t,x)\in[0,1]\times\mathbb{R}^{d}\mapsto U(t,x)\in\mathbb{R}^{+}$.

\begin{hyp}\label{hyp:basic_hyp_on_U}

For some $p_{0}\geq1,$ $U\in C_{1,2}^{p_{0}}([0,1]\times\mathbb{\mathbb{R}}^{d})$.

\end{hyp}

\begin{hyp}(time-uniform Lipschitz gradient)\label{hyp:U_grad_lipschitz_x}
$\exists L<+\infty$ s.t.

\[
\sup_{t\in[0,1]}\|\nabla U_{t}(x)-\nabla U_{t}(y)\|\leq L\|x-y\|,\quad\forall x,y.
\]

\end{hyp}

\begin{hyp}(regularity in time)\label{hyp:U_time_reg}
\begin{equation}
\sup_{t\in[0,1]}\|\nabla U_{t}(x)\|\leq L(1+\|x\|),\quad\forall x,\label{eq:U_grad_at_zero_uniform bound}
\end{equation}
where $L$ is as in (A\ref{hyp:U_grad_lipschitz_x})

\end{hyp}

\begin{hyp}(time-uniform strong convexity)\label{hyp:U_strong_convex}
$\exists K>0$ s.t. $\forall v\in\mathbb{R}^{d}$

\[
\inf_{(t,x)\in[0,1]\times\mathbb{R}^{d}}\sum_{i,j}v_{i}\frac{\partial^{2}U_{t}(x)}{\partial x_{i}\partial x_{j}}v_{j}\geq K\|v\|^{2}.
\]
We shall write $x_{t}^{\star}$ for the unique minimizer of $U_{t}$
and without loss of generality we assume that $x_{0}^{\star}=0$.

\end{hyp}

\begin{hyp}(continuity in time)\label{hyp:U_time_cont}$\exists M<\infty$
such that
\[
\|\nabla U_{t}(x)-\nabla U_{s}(x)\|\leq M|t-s|\sqrt{1+\|x-x_{t\wedge s}^{\star}\|^{2}},\quad\forall x,t,s.
\]

\end{hyp}

\begin{hyp}(bounded 3rd derivatives)\label{hyp:third_order_deriv}
The third order derivatives respect to $x$ of $U_{t}(x)$ exist,
are continuous, and bounded uniformly in $t$ and $x$. 

\end{hyp}

\subsection{Discussion of the literature and our approach}

For a review of methods for sampling from a log-concave distribution
see \cite[Sec. 7]{dalalyan2016theoretical}. Notable recent contributions
include \cite{durmus2015non}, which gives bounds on the distance
to the target distribution in total variation for an Unadjusted Langevin
Algorithm (an Euler-type discretization of a Langevin diffusion),
under a variety of assumptions on discretization step size and the
target density, including bounded perturbation of a log-concave density
and strong log-concavity outside a ball. Under the latter assumption,
convergence rates for Wasserstein distances and mean square error
bounds for empirical averages of Lipschitz functions for the diffusion
are given in \cite{eberle2015reflection}. Under conditions which
allow for strong log concavity of the target distribution, exponential
deviation inequalities of empirical averages of Lipschitz test functions
are obtained in \cite{joulin2010curvature}, and in the strongly log-concave
case, bounds on total-variation and Wasserstein distances, bounds
on mean square error and exponential deviation inequalities for a
discretized diffusion, again for Lipschitz tests functions, are obtained
in the recent pre-print \cite{durmus2016non}. 

Compared to the assumptions in the aforementioned works, which consider
processes with a fixed invariant distribution, the time-uniform strong
log-concavity assumption (A\ref{hyp:U_strong_convex}) provides a
natural starting point from which to analyze the time-inhomogeneous
process $(X_{t}^{\epsilon})_{t\in[0,1]}$. It seems likely that some
of the techniques in the aforementioned works may be useful in helping
relax this condition, but investigating this matter would lead to
an even more lengthy and technical exposition. On the other hand,
it should be noted that one of our key intermediate results, namely
the commutation relation Lemma \ref{lem:com_relation}, cannot hold
under anything weaker than (A\ref{hyp:U_strong_convex}), see Remark
\ref{rem:convexity_necessary}, so one cannot expect results of precisely
the same form as ours to hold more generally.

Lemma \ref{lem:com_relation} allows us to establish Poincaré inequalities
for the time-inhomogeneous process in section \ref{sec:Ergodicity-and-Poincare},
which are among our main technical tools. A key reference for functional
inequalities for inhomogeneous processes is \cite{collet2008logarithmic},
and some of our developments are informed by their approach. However
we are not able to use their results directly since they do not accommodate
our assumptions. In particular we explicitly work with possibly unbounded
test functions $f_{t}(x)$ which may grow polynomially fast as $\|x\|\to\infty$,
and this requires us to rigorously derive the results in section \ref{sec:Ergodicity-and-Poincare}
from scratch. 

\todo{CA: I wonder how wise it is to cite the latter as this reduces significantly the pool of "good" reviewers?
In the former they do not consider resampling either
--if OK then mention that a CLT is established in the former to link to ours? NW: } In \cite{beskos2014stability}, the stability of a sequential Monte
Carlo algorithm in discrete time was studied in the high-dimensional
regime, by establishing a functional central limit theorem implying
convergence in distribution of the effective sample size as $d\to\infty$,
under the assumption that the target distributions are of product
form as in (\ref{eq:potentials_iid}), and that the Markov transition
kernels in their algorithm factorize across dimensions in the same
manner. One of our main motivations is to relax that kind of independence
assumption because it is unrealistic, although of course our setup
is somewhat different to that of \cite{beskos2014stability}, since
we start from a continuous time perspective. It should also be noted
that we do not consider any resampling operations, where as \cite{beskos2014stability}
consider algorithms with and without resampling. \todo{To address that paper:}
In \cite{brosse2018normalizing} the authors consider a classical
product identity closely related to a discretization of (\ref{eq:Jarz}),
for a specific family $\big(U_{t}\big)_{t\in[0,1]}$, and propose
to estimate each term in the product independently, using a collection
of time-homogeneous and discretized Langevin diffusions. This allows
them to avoid the study of the time inhomogeneous processes and associated
averages of the form considered here and they exploit their earlier
results \cite{durmus2016non} concerned with time-homogeneous Langevin
diffusions to deduce quantitative bounds on mean square error and
establish polynomial complexity for their estimator. They also do
not consider a central limit theorem.

The arXiv preprint \cite{narayanan2013efficient} studies an algorithm
for sampling from time-varying log-concave distributions. The process
they work with is a discrete time Markov chain and conductance techniques
are used in the analysis. Among their key assumptions are that the
target distributions are supported on a compact convex subset of $\mathbb{R}^{d}$
and that one can compute an associated self-concordant barrier. 

\subsection{Statement of main results\label{subsec:Statement-of-main}}

Throughout section \ref{subsec:Statement-of-main} and unless stated
otherwise, $\epsilon$ is fixed to an arbitrary positive value, $(X_{t}^{\epsilon})_{t\in[0,1]}$
is as in (\ref{eq:SDE_intro}) with $X_{0}^{\epsilon}$ an $\mathcal{F}_{0}$-measurable
random variable with distribution $\mu_{0}$, and for $t\in(0,1]$,
$\mu_{t}^{\epsilon}$ is the distribution of $X_{t}^{\epsilon}$.

\subsubsection{Non-asymptotic variance and bias bounds\label{subsec:Non-asymptotic-variance-and}}
\begin{thm}
\label{thm:intro_var_and_bias_bounds} Fix $p\geq1$, assume $\mu_{0}\in\mathcal{P}^{2p}(\mathbb{R}^{d})$
and that there exists a constant $K_{0}>0$ such that 
\begin{equation}
\mathrm{var}_{\mu_{0}}[f]\leq\frac{1}{K_{0}}\mu_{0}(\|\nabla f\|^{2}),\quad\forall f\in C_{2}^{p}(\mathbb{R}^{d}).\label{eq:mu_0_poincare_hyp}
\end{equation}
1) For each $t\in[0,1]$, the distribution $\mu_{t}^{\epsilon}$ satisfies
a Poincaré inequality: 
\[
\mathrm{var}_{\mu_{t}^{\epsilon}}[f]\leq\left[(1-e^{-Kt/\epsilon})\frac{1}{K}+e^{-Kt/\epsilon}\frac{1}{K_{0}}\right]\mu_{t}^{\epsilon}(\|\nabla f\|^{2}),\quad\forall f\in C_{2}^{p}(\mathbb{R}^{d}).
\]

\noindent 2) For any $f\in C_{0,2}^{p}([0,1]\times\mathbb{R}^{d})$
such that $\pi_{t}f_{t}=0$ for all $t\in[0,1]$, and any $h\in(0,1]$,
define
\begin{equation}
S_{\epsilon}\coloneqq\int_{0}^{1}f_{t}(X_{t}^{\epsilon})\mathrm{d}t,\qquad S_{\epsilon,h}\coloneqq h\sum_{k=0}^{\left\lfloor 1/h\right\rfloor -1}f_{kh}(X_{kh}^{\epsilon}).\label{eq:MSE_thm_S_defn}
\end{equation}
 Then
\begin{align*}
\mathrm{var}[S_{\epsilon}] & \leq\frac{2\epsilon}{K_{0}\wedge K}\sup_{t\in[0,1]}\mathrm{var}_{\mu_{t}^{\epsilon}}[f_{t}],\\
|\mathbb{E}[S_{\epsilon}]| & \leq\frac{\epsilon}{K}\sup_{t\in[0,1]}\mathrm{var}_{\pi_{t}}[\phi_{t}]^{1/2}\sup_{t\in[0,1]}\mathrm{var}_{\pi_{t}}[f_{t}]^{1/2}+\alpha_{p}W^{(p)}(\mu_{0},\pi_{0})\frac{\epsilon}{K}\sup_{t\in[0,1]}\|\nabla f_{t}\|_{p},
\end{align*}
\begin{align*}
\mathrm{var}[S_{\epsilon,h}] & \leq h\left(1+\frac{2}{1-e^{-(K_{0}\wedge K)h/\epsilon}}\right)\sup_{t\in[0,1]}\mathrm{var}_{\mu_{t}^{\epsilon}}[f_{t}],\\
|\mathbb{E}[S_{\epsilon,h}]| & \leq\frac{\epsilon}{K}\sup_{t\in[0,1]}\mathrm{var}_{\pi_{t}}[\phi_{t}]^{1/2}\sup_{t\in[0,1]}\mathrm{var}_{\pi_{t}}[f_{t}]^{1/2}+\frac{\alpha_{p}h}{1-e^{-Kh/\epsilon}}W^{(p)}(\mu_{0},\pi_{0})\sup_{t\in[0,1]}\|\nabla f_{t}\|_{p},
\end{align*}
where $\alpha_{p}$, given in Lemma \ref{lem:drift}, is a constant
depending only on $\epsilon$, $p$, $K$, $d$, $\sup_{t\in(0,1)}\|\partial_{t}x_{t}^{\star}\|$,
$\sup_{t\in[0,1]}\|x_{t}^{\star}\|$, and 
\[
W^{(p)}(\mu_{0},\pi_{0})\coloneqq\inf_{\gamma\in\Gamma(\mu_{0},\pi_{0})}\int_{\mathbb{R}^{2d}}\left(1+\|x\|^{2p}\vee\|y\|^{2p}\right)\|x-y\|\gamma(\mathrm{d}x,\mathrm{d}y),
\]
where $\Gamma(\mu_{0},\pi_{0})$ is the set of all couplings of $\mu_{0}$
and $\pi_{0}$. 
\end{thm}

\begin{proof}
See section \ref{sec:Proofs-for-intro}.
\end{proof}
\begin{rem}
See section \ref{subsec:Drift,-regularity-and} for discussion of
the assumption in Theorem \ref{thm:intro_var_and_bias_bounds} that
$f$ is twice continuously differentiable w.r.t. $x$. 
\end{rem}

So far in section \ref{subsec:Statement-of-main}, the dimension $d$
has been regarded as a constant. Our next task is to explicitly quantify
the dependence on $d$ of the variance and bias bounds in Theorem
\ref{thm:intro_var_and_bias_bounds}. We are particularly interested
in growth which is at most polynomial in $d$. Pursuant to this, in
the remainder of section \ref{subsec:Non-asymptotic-variance-and}
we adopt the perspective that $d$ is an independent parameter on
which various quantities may possibly depend, including $h$, $\epsilon$
and the quantities in hypothesis (A\ref{hyp:dimension_dependence_mse-1})
below, which we shall verify for an example in section \ref{subsec:Example:-Logistic-Regression}.
The phrasing of this hypothesis in terms of asymptotic behaviour as
$d\to\infty$ is chosen for convenience, to achieve a balance between
precision and ease of presentation in Corollary \ref{cor:dimension-dependence-MSE-1}
of Theorem \ref{thm:intro_var_and_bias_bounds} below, its proof and
application.

\begin{hyp}(Polynomial dependence on dimension)\label{hyp:dimension_dependence_mse-1}
For a given $p\geq1$, and for each $d\in\mathbb{N}$ a given $\mu_{0}\in\mathcal{P}^{2p}(\mathbb{R}^{d})$,
$K_{0}$ satisfying (\ref{eq:mu_0_poincare_hyp}), and $f\in C_{0,2}^{p}([0,1],\times\mathbb{R}^{d})$,
there exists a constant $q\geq0$ independent of $d$ such that, as
$d\to\infty$,
\[
W^{(p)}(\mu_{0},\pi_{0})\vee\sup_{t\in[0,1]}\|\nabla f_{t}\|_{p}\vee K^{-1}\vee K_{0}^{-1}\vee L^{4}\vee\sup_{t\in[0,1]}\|x_{t}^{\star}\|^{2}=O(d^{q}),
\]
and 
\[
\mu_{0}(V^{2p})=O(d^{q+1}).
\]

\end{hyp}
\begin{cor}
\label{cor:dimension-dependence-MSE-1}Assume that the $p$, $\mu_{0}$,
$K_{0}$ and $f$ in Theorem \ref{thm:intro_var_and_bias_bounds}
satisfy (A\ref{hyp:dimension_dependence_mse-1}), and let $q$ be
as in the latter. If
\[
\frac{\epsilon}{K}\sup_{t\in(0,1)}\|\partial_{t}x_{t}^{\star}\|=O(1),
\]
 as $d\to\infty$, then

\begin{align*}
 & \mathrm{var}[S_{\epsilon}]=O\left(\frac{\epsilon}{K\wedge K_{0}}r_{1}(d)\right),\qquad\qquad\qquad\qquad\;\;\;|\mathbb{E}[S_{\epsilon}]|=O\left(\frac{\epsilon}{K}r_{2}(d)+\frac{\epsilon}{K}r_{3}(d)\right),\\
 & \mathrm{var}[S_{\epsilon,h}]=O\left(h\left[1+\frac{2}{1-e^{-(K_{0}\wedge K)h/\epsilon}}\right]r_{1}(d)\right),\qquad|\mathbb{E}[S_{\epsilon,h}]|=O\left(\frac{\epsilon}{K}r_{2}(d)+\frac{h}{1-e^{-Kh/\epsilon}}r_{3}(d)\right),
\end{align*}
where

\[
r_{1}(d)\coloneqq d^{4q+2p(q+1)+1},\quad r_{2}(d)\coloneqq d^{7q/4+3pq+3p/2+1/2},\quad r_{3}(d)\coloneqq d^{2q+pq+p}.
\]
\end{cor}

\begin{proof}
See section \ref{sec:Proofs-for-intro}.
\end{proof}

\subsubsection{A central limit theorem in the high-dimensional regime\label{subsec:CLT_front}}

The expressions in Corollary \ref{cor:dimension-dependence-MSE-1}
suggest that the behaviour of $\mathrm{var}[S_{\epsilon,h}]$ and
$|\mathbb{E}[S_{\epsilon,h}]|$ as $\epsilon\to0$ depends on the
scaling relationship between $\epsilon$ and $h$. We now introduce
a parameter $\ell\geq0$ to delineate two cases.

\begin{hyp}($\ell$-dependent scaling of $h$ with $\epsilon$)\label{hyp:eps_and_h_scaling_intro}

1. In the case $\ell=0$, we assume $h(\epsilon)=O(\epsilon^{c})$
for an arbitrary $c>1$. 

2. In the case $\ell>0$, we set $h(\epsilon)\coloneqq\ell\epsilon$

\end{hyp}Throughout the remainder of section \ref{subsec:CLT_front},
the value of $\ell\geq0$ should be regarded as being chosen independently,
and (A\ref{hyp:eps_and_h_scaling_intro}) is assumed to hold.

To state our next main result we need to introduce some further notation.
For each $s\in[0,1]$ and $\epsilon>0$, let $(Y_{t}^{s,\epsilon})_{t\in\mathbb{R}^{+}}$
be the solution of: 
\[
Y_{t}^{s,\epsilon}=Y_{0}^{s,\epsilon}-\epsilon^{-1}\int_{0}^{t}\nabla U_{s}(Y_{u}^{s,\epsilon})\mathrm{d}u+\sqrt{2\epsilon^{-1}}\int_{0}^{t}\mathrm{d}B_{u},
\]
where $Y_{0}^{s,\epsilon}$ is an $\mathcal{F}_{0}$-measurable random
variable with distribution $\pi_{s}$. Then writing $L_{2}(\pi_{s})$
for the collection of all real-valued functions that are square-integrable
with respect to $\pi_{s}$, standard results for stationary reversible
Markov processes and Markov chains ensure that for any $s\in[0,1]$
and $f_{s}\in L_{2}(\pi_{s})$, the following limits exist:
\begin{align*}
 & \varsigma_{0}(s)\coloneqq\lim_{\epsilon\to0}\mathrm{var}\left[\epsilon^{-1/2}\int_{0}^{1}f_{s}(Y_{t}^{s,\epsilon})\mathrm{d}t\right],\\
 & \varsigma_{\ell}(s)\coloneqq\lim_{\epsilon\to0}\mathrm{var}\left[\epsilon^{-1/2}h(\epsilon)\sum_{k=0}^{\left\lfloor 1/h\right\rfloor -1}f_{s}(Y_{kh}^{s,\epsilon})\right],\quad\ell>0.
\end{align*}

With $Q_{t}^{s}(f)(y)\coloneqq\mathbb{E}[f(Y_{t}^{s,1})|Y_{0}^{s,1}=y]$
and $\mathcal{L}_{s}f\coloneqq-\left\langle \nabla U_{s},\nabla f\right\rangle +\Delta f$,
it is well known that the following bounds, in terms of $L_{2}(\pi_{s})$
spectral gaps and constant $K$ from (A\ref{hyp:U_strong_convex}),
hold:

\[
\varsigma_{\ell}(s)\leq2\mathrm{var}_{\pi_{S}}[f_{s}]\cdot\begin{cases}
\mathrm{Gap}(\mathcal{L}_{s})^{-1},\quad & \ell=0,\\
\ell\mathrm{Gap}(Q_{\ell}^{s})^{-1},\quad & \ell>0,
\end{cases}
\]
and
\[
\mathrm{Gap}(\mathcal{L}_{s})\geq K,\quad\quad\mathrm{Gap}(Q_{\ell}^{s})^{-1}\geq\frac{1-\exp(-K\ell)}{\ell}.
\]
 Indeed $\mathrm{Gap}(\mathcal{L}_{s})\geq K$ is a direct consequence
of the standard Poincaré inequality for the strongly log-concave distribution
$\pi_{s}$. These bounds suggest that under hypotheses such as (A\ref{hyp:dimension_dependence_mse-1}),
for each $s\in[0,1]$, fluctuations of the additive functionals $\int_{0}^{1}f_{s}(Y_{t}^{s,\epsilon})\mathrm{d}t$
and $h(\epsilon)\sum_{k=0}^{\left\lfloor 1/h(\epsilon)\right\rfloor -1}f_{s}(Y_{kh}^{s,\epsilon})$
associated with the time-homogeneous process $(Y_{t}^{s,\epsilon})_{t\in\mathbb{R}^{+}}$
could possibly be controlled by choosing $\epsilon^{-1}$ to be polynomial
in $d$. Our next main result, Theorem \ref{thm:intro_CLT}, establishes
that a similar phenomenon holds for additive functionals associated
with time-inhomogeneous process $(X_{t}^{\epsilon})_{t\in[0,1]}$.

Under our assumptions, for any $\ell\geq0$ , $s\mapsto\varsigma_{\ell}(s)$
can be shown to be integrable (see the proof of Lemma \ref{lem:riemannconvergence}),
and therefore 
\begin{equation}
\sigma_{\ell}^{2}\coloneqq\int_{0}^{1}\varsigma_{\ell}(s)\mathrm{d}s\label{eq:intro_sigma}
\end{equation}
 is well-defined. In the context of Theorem \ref{thm:intro_CLT} below,
it is important to note that $\varsigma_{\ell}$ and $\sigma_{\ell}^{2}$
depend on the dimension $d$, but this dependence is not shown in
the notation.
\begin{thm}
\label{thm:intro_CLT}Fix $p\geq1$ and for each $d\in\mathbb{N}$,
fix a function $f\in C_{1,2}^{p}([0,1]\times\mathbb{R}^{d})$ such
that for each $t\in[0,1]$ $\pi_{t}f_{t}=0$, and a probability measure
$\mu_{0}\in\mathcal{P}^{2p}(\mathbb{R}^{d})$ and a constant $K_{0}>0$
satisfying (\ref{eq:mu_0_poincare_hyp}). Assume that (A\ref{hyp:dimension_dependence_mse-1})
holds and assume additionally that for each $\ell\geq0$, $\sup_{t}1/\varsigma_{\ell}(t)$
and $\sup_{t}\|\partial_{t}f_{t}\|_{p}$ grow at most polynomially
fast as $d\to\infty.$ Then for any $\ell\geq0$ there exists $a>0$
such that with $\epsilon(d)=O(d^{-a})$ and $d\mapsto h(d)$ such
that (A\ref{hyp:eps_and_h_scaling_intro}) holds, 
\[
\lim_{d\to\infty}\left|\mathrm{var}\left[\epsilon(d)^{-1/2}S_{\epsilon(d),h(d)}-\sigma_{\ell}^{2}\right]\right|=0,
\]
and 
\[
\lim_{d\to\infty}\sup_{w\in\mathbb{R}}\left|\mathbb{P}\left[\epsilon(d)^{-1/2}S_{\epsilon(d),h(d)}/\sqrt{\sigma_{\ell}^{2}}\leq w\right]-\Phi(w)\right|=0,
\]
where $S_{\epsilon,h}$ is as in Theorem \ref{thm:intro_var_and_bias_bounds},
and $\Phi$ is the standard Gaussian c.d.f.
\end{thm}

\begin{proof}
See section \ref{sec:Quantitative-CLT-bound}.
\end{proof}
\begin{rem}
It is in principle possible to calculate quantitative bounds on the
rates of convergence in Theorem \ref{thm:intro_CLT}, by agreggation
of various bounds found in our proof. We do not pursue this here due
to a lack of space and the limited interest of such bounds in practice.
\end{rem}

\begin{rem}
Note that compared to Theorem \ref{thm:intro_var_and_bias_bounds},
Theorem \ref{thm:intro_CLT} requires additional assumptions that
$s\mapsto f_{s}(x)$ is continuously differentiable for any $x\in\mathbb{R}^{d}$.
This condition is required in order to obtain explicit control on
the error in Riemann sums involved in our calculations, and could
be relaxed easily to H\"older continuity, at the expense of additional
notation.
\end{rem}

\begin{rem}
As an aside, it is natural to investigate the impact of $\ell$ on
the asymptotic variance $\sigma_{\ell}^{2}$. Theorem \ref{thm:variance-spectral-decomposition+Geyer}
establishes that $\sigma_{\ell}^{2}$ is a non-decreasing function
of $\ell$. This result can be understood as being a generalisation
of \cite[Theorem 3.3]{geyer1992practical}, an important fact in the
area of discrete time Markov chain Monte Carlo methods, concerned
with ``thinning'' in the context of ergodic averages.

\end{rem}

\begin{rem}
By inspecting the proofs in section \ref{sec:Quantitative-CLT-bound},
one can check that similar statements hold in the fixed dimension
case, that is with $d\in\mathbb{N}$ held constant and $h(\epsilon)$
as in (A\ref{hyp:eps_and_h_scaling_intro}),
\[
\lim_{\epsilon\to0}\left|\mathrm{var}\left[\epsilon^{-1/2}S_{\epsilon,h(\epsilon)}-\sigma_{\ell}^{2}\right]\right|=0,\qquad\lim_{\epsilon\to\infty}\sup_{w\in\mathbb{R}}\left|\mathbb{P}\left[\epsilon^{-1/2}S_{\epsilon,h(\epsilon)}/\sqrt{\sigma_{\ell}^{2}}\leq w\right]-\Phi(w)\right|=0.
\]
\end{rem}

\subsubsection{Discretization of the process\label{subsec:intro_Discretization}}

One typically resorts to simulating some approximation to the diffusion
$(X_{t}^{\epsilon})_{t\in[0,1]}$ involving discretization in order
to obtain a practical approximation to $S_{\epsilon}$ or $S_{\epsilon,h}$.
There are many possible approaches to discretization of diffusions
and it is not our objective to investigate or discuss their relative
merits. Instead, we consider a simple Euler-Maruyama discretization
scheme, since it is a generally applicable method whose practical
computational cost is easy to assess and whose approximation properties
can be quite directly analyzed.

We present next a general purpose lemma which allows control of moments
of functions on the path space of one diffusion in terms of those
of another, which we shall subsequently apply to the Euler-Maruyama
discretization scheme. 

Let $E$ be the Polish space of continuous functions $z:t\in[0,1]\mapsto z_{t}\in\mathbb{R}^{d}$
endowed with the metric $\rho(z,\tilde{z})=\sup_{t\in[0,1]}\|z_{t}-\tilde{z}_{t}\|$,
and let $\mathcal{B}(E)$ be its Borel $\sigma$-algebra. 
\begin{lem}
\label{lem:gamma_coupling}For any $(E,\mathcal{B}(E))$-valued random
elements $X,\widetilde{X}$, any measurable function $\varphi:(E,\mathcal{B}(E))\to(\mathbb{R},\mathcal{B}(\mathbb{R}))$,
and any $p,q,r\in[1,+\infty)$ such that $1/q+1/r=1$,
\begin{align*}
 & \sup_{c\in\mathbb{R}}\left|\mathbb{P}[\varphi(\widetilde{X})\leq c]-\mathbb{P}[\varphi(X)\leq c]\right|\leq\|\mu-\widetilde{\mu}\|_{\mathrm{tv}},\\
 & \mathbb{E}[|\varphi(\widetilde{X})|^{p}]^{1/p}\leq\mathbb{E}[|\varphi(X)|^{p}]^{1/p}+\|\mu-\widetilde{\mu}\|_{\mathrm{tv}}^{1/pq}\left\{ \mathbb{E}[|\varphi(X)|^{pr}]^{1/pr}+\mathbb{E}[|\varphi(\widetilde{X})|^{pr}]^{1/pr}\right\} .
\end{align*}
where
\[
\mu(A)\coloneqq\mathbb{P}[X\in A],\qquad\widetilde{\mu}(A)\coloneqq\mathbb{P}[\widetilde{X}\in A],\quad A\in\mathcal{B}(E).
\]
\end{lem}

\begin{proof}
See section \ref{sec:Proofs-for-intro}.
\end{proof}
For $\epsilon>0$ and $h\in(0,1]$, let $\widetilde{X}^{\epsilon,h}=(\widetilde{X}_{t}^{\epsilon,h})_{t\in[0,1]}$
be the solution of 
\begin{equation}
\widetilde{X}_{t}^{\epsilon,h}=X_{0}^{\epsilon}-\epsilon^{-1}\int_{0}^{t}\widetilde{\nabla U}_{s}(\widetilde{X}_{s}^{\epsilon,h})\mathrm{d}s+\sqrt{2\epsilon^{-1}}\int_{0}^{t}\mathrm{d}B_{s},\label{eq:sde_disc-1}
\end{equation}
where $X_{0}^{\epsilon}$ is the same $\mathcal{F}_{0}$-measurable
random variable with distribution $\mu_{0}$ as in (\ref{eq:SDE_intro}),
and the following short-hand notation is used: 
\begin{equation}
\widetilde{\nabla U}_{t}(\widetilde{X}_{t}^{\epsilon,h})\coloneqq\sum_{k=0}^{\left\lfloor 1/h\right\rfloor -1}\nabla U_{kh}(\widetilde{X}_{kh}^{\epsilon,h})\mathbb{I}_{[kh,(k+1)h)}(t).\label{eq:tilde_grad_U}
\end{equation}

In practice, one does not simulate the entire trajectory $(\widetilde{X}_{t}^{\epsilon,h})_{t\in[0,1]}$
but rather the skeleton $(\widetilde{X}_{kh}^{\epsilon,h})_{k=0,\ldots,,\left\lfloor 1/h\right\rfloor -1}$.
The point of writing (\ref{eq:sde_disc-1})-(\ref{eq:tilde_grad_U})
is to highlight that the term $\sqrt{2\epsilon^{-1}}\int_{0}^{t}\mathrm{d}B_{s}$
is common to both (\ref{eq:sde_disc-1}) and (\ref{eq:SDE_intro})
so that the laws of $(X_{t}^{\epsilon})_{t\in[0,1]}$ and $(\widetilde{X}_{t}^{\epsilon,h})_{t\in[0,1]}$
are mutually absolutely continuous. Via Girsanov's theorem and Pinsker's
inequality, Dalalyan \cite{dalalyan2016theoretical} when studying
a time-homogeneous process used this fact to estimate the total variation
distance between the time-marginal distributions of a overdamped Langevin
diffusion and its discretization, analogous in the present context
to the distributions of say $X_{1}^{\epsilon}$ and $\widetilde{X}_{1}^{\epsilon,h}$.
However, this Girsanov/Pinsker technique allows one to estimate the
total variation distance not only between time-marginal distributions,
but also between the laws of $(X_{t}^{\epsilon})_{t\in[0,1]}$ and
$(\widetilde{X}_{t}^{\epsilon,h})_{t\in[0,1]}$, i.e. the probability
measures 
\[
\mu^{\epsilon}(A)\coloneqq\mathbb{P}[X^{\epsilon}\in A]\quad\quad\widetilde{\mu}^{\epsilon,h}(A)\coloneqq\mathbb{P}[\widetilde{X}^{\epsilon,h}\in A],\quad\quad A\in\mathcal{B}(E),
\]
and we shall exploit that fact in the application of Lemma \ref{lem:gamma_coupling}
in Section \ref{subsec:Example:-Logistic-Regression} to transfer
the distributional convergence in Theorem \ref{thm:intro_CLT} to
the discretized process. In particular, Proposition \ref{prop:tv_dim_dependence}
together with standard Foster-Lyapunov techniques will be applied
to control the terms in the bounds of Lemma \ref{lem:gamma_coupling}. 
\begin{prop}
\label{prop:tv_dim_dependence}For any $q\geq0$, if 
\begin{align}
 & M^{2}\vee L^{4}\vee K^{-1}\vee\sup_{t}\|\partial_{t}x_{t}^{\star}\|^{2}=O(d^{q}),\quad\mu_{0}(V)=O(d^{q+1}),\label{eq:tv_bound_constants_hyp}\\
 & h\vee\epsilon\vee\frac{h}{\epsilon}\frac{L^{2}}{K}=o(1),\quad\frac{h}{\epsilon}d=O(1),\nonumber 
\end{align}
as $d\to\infty$, then 
\[
\|\mu^{\epsilon}-\widetilde{\mu}^{\epsilon,h}\|_{\mathrm{tv}}=O\left(\sqrt{\frac{h}{\epsilon^{2}}d^{4q+1}}\right).
\]
\end{prop}

\begin{proof}
See section \ref{subsec:Intermediate-results-concerning}.
\end{proof}

\subsection{Example: Marginal likelihood computation for logistic regression\label{subsec:Example:-Logistic-Regression}}

\subsubsection{Model specification and verification of assumptions}

Consider observations $Y_{1},\ldots,Y_{m}$ each valued in $\{0,1\}$,
covariate vectors $c_{1},\ldots,c_{m}$ each valued in $\mathbb{R}^{d}$,
and an unknown parameter vector $x\in\mathbb{R}^{d}$. The observations
are modelled as conditionally independent given the covariates and
$x$, with the conditional probability of $\{Y_{i}=1\}$ being $\varrho_{i}(x)\coloneqq1/(1+e^{-\left\langle x,c_{i}\right\rangle })$.
In a Bayesian approach to statistical inference we place an isotropic
Gaussian prior distribution over the unknown parameter $x$, with
covariance matrix $I_{d}/\tilde{\sigma}^{2}$. The posterior density
over $x$ has density on $\mathbb{R}^{d}$ proportional to:
\[
\exp\left\{ y^{T}Cx-\sum_{i=1}^{m}\log(1+e^{\left\langle x,c_{i}\right\rangle })-\frac{\|x\|^{2}}{2\tilde{\sigma}^{2}}\right\} ,
\]
with the vector $y\coloneqq(y_{1},\ldots,y_{m})^{T}$ and matrix $C$
whose $i$th row is $c_{i}$. 

Let the functions $(U_{t})_{t\in[0,1]}$ be given by
\begin{equation}
U_{t}(x)=-ty^{T}Cx+t\sum_{i=1}^{m}\log(1+e^{\left\langle x,c_{i}\right\rangle })+\frac{\|x\|^{2}}{2\tilde{\sigma}^{2}},\label{eq:logit_U}
\end{equation}
Then the distributions $\pi_{0}$ and $\pi_{1}$ specified by $U_{0}$
and $U_{1}$ are respectively the prior and posterior. Evaluating
the ``marginal likelihood'' $Z_{1}=\int_{\mathbb{R}^{d}}\exp\{-U_{1}(x)\}\mathrm{d}x$
allows one to assess the quality of model fit. 

We shall now verify assumptions (A\ref{hyp:basic_hyp_on_U})-(A\ref{hyp:third_order_deriv}).
We have
\begin{equation}
\nabla U_{t}(x)=-ty^{T}C+t\sum_{i=1}^{m}c_{i}\varrho_{i}(x)+\frac{x}{\tilde{\sigma}^{2}},\qquad\nabla^{(2)}U_{t}(x)=t\sum_{i=1}^{m}\varrho_{i}(x)\{1-\varrho_{i}(x)\}c_{i}c_{i}^{T}+\frac{I_{d}}{\tilde{\sigma}^{2}}.\label{eq:logit_grad_hess}
\end{equation}
\begin{equation}
\frac{\partial^{3}U_{t}(x)}{\partial x_{j}\partial x_{k}\partial x_{\ell}}=t\sum_{i=1}^{m}c_{ij}c_{ik}c_{i\ell}\varrho_{i}(x)\{1-\varrho_{i}(x)\}\{1-2\varrho_{i}(x)\}\label{eq:logit_3rd_Deriv}
\end{equation}
where $c_{ij}$ is the $j$th element of $c_{i}$.

By inspection of (\ref{eq:logit_U})-(\ref{eq:logit_grad_hess}),
(A\ref{hyp:basic_hyp_on_U}) holds with $p_{0}=1$. By considering
the spectral norm of $\nabla^{(2)}U_{t}$, one obtains 
\[
\sup_{t\in[0,1]}\|\nabla U_{t}(x)-\nabla U_{t}(y)\|\leq(0.25m\lambda_{\mathrm{max}}+\tilde{\sigma}^{-2})\|x-y\|,
\]
where $\lambda_{\mathrm{max}}$ is the largest eigenvalue of $m^{-1}\sum_{i=1}^{m}c_{i}c_{i}^{T}$,
and with 
\begin{equation}
\xi\coloneqq\|y^{T}C\|+\sum_{i=1}^{m}\|c_{i}\|,\label{eq:logit_xi}
\end{equation}
we have

\[
\|\nabla U_{t}(x)\|\leq(\xi\vee\tilde{\sigma}^{-2})(1+\|x\|),\qquad\|\nabla U_{t}(x)-\nabla U_{s}(y)\|\leq\xi|t-s|.
\]
So for the constants appearing in (A\ref{hyp:U_grad_lipschitz_x})-(A\ref{hyp:U_time_cont})
one make take 
\begin{equation}
K=\frac{1}{\tilde{\sigma}^{2}},\quad L=\left(0.25m\lambda_{\mathrm{max}}+\frac{1}{\tilde{\sigma}^{2}}\right)\vee\left(\xi\vee\frac{1}{\tilde{\sigma}^{2}}\right),\quad M=\xi.\label{eq:logit_KLM}
\end{equation}
 (A\ref{hyp:third_order_deriv}) is satisfied by inspection of (\ref{eq:logit_3rd_Deriv}). 

\subsubsection{Dimension dependence of the error}

Let us now discuss application of Theorems \ref{thm:intro_var_and_bias_bounds}
and \ref{thm:intro_CLT}. Observe from (\ref{eq:logit_U}) that we
have
\begin{equation}
\partial_{t}U_{t}(x)=-y^{T}Cx+\sum_{i=1}^{m}\log(1+e^{\left\langle x,c_{i}\right\rangle }),\label{eq:d_t_U_t_logit}
\end{equation}
and define
\[
\Delta_{\epsilon,h}\coloneqq-h\sum_{k=0}^{\left\lfloor 1/h\right\rfloor -1}\left.\partial_{t}U_{t}(\widetilde{X}_{kh}^{\epsilon,h})\right|_{t=kh}-\log\frac{Z_{1}}{Z_{0}},
\]
where $(\widetilde{X}_{t}^{\epsilon,h})_{t\in[0,1]}$ is as in (\ref{eq:sde_disc-1}). 

Consider the following condition:

\begin{hyp}(Polynomial dependence on dimension for logistic regression)\label{hyp:logit_dim_depend}
There exists $q\geq0$ such that: 
\[
\tilde{\sigma}^{2}\vee\left(0.25m\lambda_{\mathrm{max}}+\frac{1}{\tilde{\sigma}^{2}}\right)\vee\xi=O(d^{q/4})
\]
as $d\to\infty$. 

\end{hyp}

In the proof of the following proposition, (A\ref{hyp:logit_dim_depend})
allows us to verify (A\ref{hyp:dimension_dependence_mse-1}), apply
Corollary \ref{cor:dimension-dependence-MSE-1} and Theorem \ref{thm:intro_CLT}
with
\begin{equation}
f_{t}=-\partial_{t}U_{t}+\pi_{t}(\partial_{t}U_{t}),\label{eq:logit_f_t}
\end{equation}
 and Proposition \ref{prop:tv_dim_dependence} and Lemma \ref{lem:gamma_coupling}. 

\begin{prop}
\label{prop:logit_order} Assume that $\mu_{0}=\pi_{0}$ and that
(A\ref{hyp:logit_dim_depend}) holds for some given $q$. 

\noindent 1) If
\begin{equation}
h\vee\epsilon=o(1),\quad\quad\frac{h}{\epsilon^{2}}d^{3q/2+1}\vee\epsilon d^{7q+3}=O(1)\label{eq:logit_eps_h_hyp}
\end{equation}
 as $d\to\infty$, then
\[
\mathbb{E}[|\Delta_{\epsilon,h}|]=O\left(\sqrt{\epsilon d^{7q+3}}+\left[\frac{h}{\epsilon^{2}}\right]^{1/4}d^{9(q+1)/4}+hd^{5q+2}\right).
\]

\noindent 2) If 
\begin{equation}
\left[\inf_{t\in[0,1]}t^{2}\sum_{j=1}^{d}\left\{ \int_{\mathbb{R}^{d}}l(y;x)\left[\sum_{i=1}^{m}(y_{i}-\varrho_{i}(x))c_{ij}-\frac{x_{j}}{\tilde{\sigma}^{2}}\right]\pi_{t}(\mathrm{d}x)\right\} ^{2}\right]^{-1}\label{eq:logit_clt_hyp}
\end{equation}
grows at most polynomially fast as $d\to\infty$, where $l(y;x)$
is the log-likelihood:
\[
l(y;x)\coloneqq-y^{T}Cx+\sum_{i=1}^{m}\log(1+e^{\left\langle x,c\right\rangle }),
\]
then for any $c>2$, there exists $a>0$ such that with $\epsilon=O(d^{-a})$
and $h=\epsilon^{c}$, 
\[
\lim_{d\to\infty}\sup_{w\in\mathbb{R}}\left|\mathbb{P}\left[\epsilon(d)^{-1/2}\Delta_{\epsilon(d),h(d)}/\sqrt{\sigma_{0}^{2}}\leq w\right]-\Phi(w)\right|=0,
\]
where $\sigma_{0}^{2}$ is as in (\ref{eq:intro_sigma}) with $f_{t}$
as in (\ref{eq:logit_f_t}).
\end{prop}

\begin{proof}
See section \ref{sec:Proofs-for-intro}.
\end{proof}

%% file: child-poincare.tex
\section{Poincaré inequalities, variance and bias decay for the inhomogeneous
Langevin diffusion\label{sec:Ergodicity-and-Poincare}}

Throughout section \ref{sec:Ergodicity-and-Poincare}, $\epsilon>0$
is a fixed constant. 

\subsection{Preliminaries about the process\label{subsec:Preliminaries-about-proces}}

\subsubsection{Existence and Lipschitz continuity with respect to initial conditions}

Let $(B_{t})_{t\in[0,1]}$ be $d$-dimensional Brownian motion. Under
(A\ref{hyp:U_grad_lipschitz_x}), (A\ref{hyp:U_time_reg}) and (A\ref{hyp:U_time_cont}),
for each $s\in[0,1]$ there exists a strong solution to:
\begin{equation}
X_{s,t}^{x}=x-\epsilon^{-1}\int_{s}^{t}\nabla U_{u}(X_{s,u}^{x})\mathrm{d}u+\sqrt{2\epsilon^{-1}}\int_{s}^{t}\mathrm{d}B_{u},\quad t\in[s,1].\label{eq:sde}
\end{equation}
pathwise uniqueness holds, see for example \cite[Thm. 2.9, p.190]{durrett1996stochastic},
\cite[Thm 3.4, p. 71]{khasminskii2011stochastic} or \cite[Thm. 4, p. 402]{gikhman1969introduction},
and the solution is non-explosive \cite[p. 75]{khasminskii2011stochastic}.
Moreover, as noted by \cite[Thm. 2.2, Ch. 2, p. 211]{kunita1984stochastic},
we can work with a version of $X_{s,t}^{x}$ which is continuous in
$s,t,x$ almost surely, and satisfies (\ref{eq:sde}) for all $s,t,x$,
almost surely. 

Throughout section \ref{sec:Ergodicity-and-Poincare}, we take:
\[
P_{s,t}f(x)\coloneqq\mathbb{E}[f(X_{s,t}^{x})],\qquad\mathcal{L}_{t}f\coloneqq-\epsilon^{-1}\left\langle \nabla U_{t},\nabla f\right\rangle +\epsilon^{-1}\Delta f,
\]
with the dependence on $\epsilon$ suppressed from the notation. 

We shall make extensive use of the following observation, noted in
the time-homogeneous case by \cite{cattiaux2014semi}. 
\begin{lem}
\label{lem:pathwise_sde_bound}Almost surely, the following holds
for all $x,y$ and $s\leq t$,
\[
\|X_{s,t}^{x}-X_{s,t}^{y}\|\leq e^{-K(t-s)/\epsilon}\|x-y\|.
\]
\end{lem}

\begin{proof}
Ito's lemma gives 
\begin{eqnarray*}
 &  & e^{2K(t-s)/\epsilon}\|X_{s,t}^{x}-X_{s,t}^{y}\|^{2}\\
 &  & =\|x-y\|^{2}\\
 &  & +\frac{2}{\epsilon}\int_{s}^{t}\left(K\|X_{s,u}^{x}-X_{s,u}^{y}\|^{2}-\left\langle \nabla U_{u}(X_{s,u}^{x})-\nabla U_{u}(X_{s,u}^{y})\,,\,X_{s,u}^{x}-X_{s,u}^{y}\right\rangle \right)e^{2K(u-s)/\epsilon}\mathrm{d}u,
\end{eqnarray*}
and by Lemma \ref{lem:strong_convex_equiv}, (A\ref{hyp:U_strong_convex})
is equivalent to 
\[
\left\langle \nabla U_{t}(x)-\nabla U_{t}(y),x-y\right\rangle \geq K\|x-y\|^{2},\quad\forall x,y.
\]
\end{proof}

\subsubsection{Drift, regularity and validity of forward and backward equations\label{subsec:Drift,-regularity-and}}
\begin{lem}
\label{lem:drift}For any $p\geq1$ and $\kappa\in(0,Kp)$ define:
\begin{eqnarray*}
\delta & \coloneqq & \epsilon^{-1}(Kp-\kappa),\\
r & \coloneqq & \frac{p}{\kappa}\epsilon\sup_{t\in(0,1)}\|\partial_{t}x_{t}^{\star}\|+\sqrt{\frac{p^{2}}{\kappa^{2}}\epsilon^{2}\sup_{t\in(0,1)}\|\partial_{t}x_{t}^{\star}\|^{2}+2\frac{p}{\kappa}[2(p-1)+d]}\\
b & \coloneqq & 2pr^{2p-1}\left[\sup_{t\in(0,1)}\|\partial_{t}x_{t}^{\star}\|+\frac{2(p-1)+d}{\epsilon r}\right],\\
\alpha_{p} & \coloneqq & 2^{4p-2}\vee\left[1+2^{2p-1}\left(\frac{b}{\delta}+(1+2^{2p-1})\sup_{t\in[0,1]}\|x_{t}^{\star}\|^{2p}\right)\right].
\end{eqnarray*}
Then the following hold:
\begin{eqnarray}
 &  & \partial_{t}V_{t}^{p}(x)+\mathcal{L}_{t}V_{t}^{p}(x)\leq-\delta V_{t}^{p}(x)+b\,\mathbb{I}\{\|x-x_{t}^{\star}\|\leq r\},\label{eq:drift_bound_gen}\\
 &  & \mathbb{E}\left[\int_{s}^{t}V_{u}^{p}(X_{s,u}^{x})\mathrm{d}u\right]=\int_{s}^{t}P_{s,u}V_{u}^{p}(x)\mathrm{d}u<+\infty,\label{eq:drift_bound_int}\\
 &  & P_{s,t}V_{t}^{p}(x)\leq e^{-\delta(t-s)}V_{s}^{p}(x)+\frac{b}{\delta}(1-e^{-\delta(t-s)}),\label{eq:drift_bound_semi}\\
 &  & \sup_{s\leq t}\mathbb{E}\left[1+\|X_{s,t}^{x}\|^{2p}\right]\leq\alpha_{p}(1+\|x\|^{2p}).\label{eq:drift_bound_uniform}
\end{eqnarray}
\end{lem}

\begin{proof}
See section \ref{subsec:Proofs-of-drift_lemmas}.
\end{proof}
Proposition \ref{prop:C_2^p_closed} establishes regularity properties
which are used in rigorously establishing the validity of the forward
and backward equations in Proposition \ref{prop:fwd_and_bck_eqs}
and various manipulations in section \ref{subsec:Variance-bounds-for}.
Although the topic of differentiability and other regularity properties
of $x\mapsto P_{s,t}f(x)$ as in (\ref{eq:C_closed_P}) is classical,
we were not able to find in the literature results which give us exactly
the conclusions we need under our assumptions, in particular allowing
for time-inhomogeneity of $P_{s,t}f(x)$, and for $f(x)$ and $\nabla U_{t}(x)$
to be unbounded in $x$. The proof of Proposition \ref{prop:C_2^p_closed}
which we provide in section \ref{subsec:Proofs_fwd_bwd_eqs} to make
the paper self-contained, does not exploit the elipticity of (\ref{eq:sde}),
which is why $f$ is taken to be $q$-times differentiable on the
left hand side of the implication in (\ref{eq:C_closed_P}). This
differentiability requirement propagates through our results, e.g.,
explaining why $f$ is assumed twice differentiable in $x$ in part
2) of Theorem \ref{thm:intro_var_and_bias_bounds}. This restriction
might be removed if existing results for elliptic diffusions, see
for instance \cite[Sec. 1.5, p.48]{cerrai2001second}, could be generalized
to our setup, but that seems to involve a large amount of extra work
which would further lengthen this paper.
\begin{prop}
\label{prop:C_2^p_closed}For any given $p\geq1$, 
\begin{eqnarray}
f\in C_{q}^{p}(\mathbb{R}^{d}) & \Rightarrow & x\mapsto P_{s,t}f(x)\;\in\;C_{q}^{p}(\mathbb{R}^{d}),\quad\forall s\leq t,\,q=1,2,\label{eq:C_closed_P}\\
f\in C_{1,2}^{p}([0,1]\times\mathbb{R}^{d}) & \Rightarrow & \begin{cases}
(t,x)\mapsto|\partial_{t}f_{t}(x)|+|\mathcal{L}_{t}f_{t}(x)|\;\in\;C_{0,0}^{p+1/2}([0,1]\times\mathbb{R}^{d}),\\
(s,x)\mapsto\mathcal{L}_{s}P_{s,t}f_{t}(x)\;\in\;C_{0,0}^{p+1/2}([0,1]\times\mathbb{R}^{d}),\quad\forall t.
\end{cases}\label{eq:C_transfer_L}
\end{eqnarray}
\end{prop}

\begin{proof}
See section \ref{subsec:Proof-and-supporting_C_closed}. 
\end{proof}
\begin{prop}
\label{prop:fwd_and_bck_eqs}For any $p\geq1$, $f\in C_{1,2}^{p}([0,1]\times\mathbb{R}^{d})$
and $\nu\in\mathcal{P}^{p+1/2}(\mathbb{R}^{d})$ , the following equalities
hold: 
\begin{eqnarray}
 &  & \partial_{t}\nu P_{s,t}f_{t}=\nu P_{s,t}\left(\partial_{t}f_{t}+\mathcal{L}_{t}f_{t}\right),\label{eq:fwd_equation}\\
 &  & \partial_{s}P_{s,t}f_{t}(x)=-\mathcal{L}_{s}P_{s,t}f_{t}(x),\quad\forall x,\label{eq:bwd_equation}
\end{eqnarray}
and for any fixed $t$, the map $(s,x)\mapsto P_{s,t}f_{t}(x)$ is
a member of $C_{1,2}^{p+1/2}([0,1]\times\mathbb{R}^{d})$. 
\end{prop}

\begin{proof}
See section \ref{subsec:Proofs_fwd_bwd_eqs}.
\end{proof}
Before closing section \ref{subsec:Preliminaries-about-proces}, it
is opportunte to discuss the derivation of the expectation formulae
in (\ref{eq:Jarz})-(\ref{eq:FK}) (see also Lemma \ref{lem:the_TI_identity}
for the thermodynamic integration identity). Define 
\[
T_{s,t}f(x)\coloneqq\mathbb{E}\left[f(X_{s,t}^{x})\exp\left\{ -\int_{s}^{t}\partial_{u}U_{u}(X_{s,u}^{x})\mathrm{d}u\right\} \right].
\]
To rigorously derive the path-integral representations of $Z_{1}/Z_{0}$
and $\pi_{1}(f)$ in (\ref{eq:Jarz})-(\ref{eq:FK}) (note that we
have already proved the first equality in (\ref{eq:Jarz}) by Lemma
\ref{lem:the_TI_identity} ), it is sufficient to verify the hypotheses
on $T_{s,t}f$ of Lemma \ref{lem:FK_derive} below. Although we have
not found an explicit verification of these hypotheses in the literature
under exactly our assumptions (A\ref{hyp:basic_hyp_on_U})-(A\ref{hyp:U_time_cont}),
we believe they are approachable using techniques similar to those
in the proofs of Propositions \ref{prop:C_2^p_closed} and \ref{prop:fwd_and_bck_eqs}.
For example, a direct application of \cite[Thm 2, p. 415]{gikhman1969introduction}
would require boundedness $|\partial_{t}U_{t}(\cdot)|$, but this
condition seems not to be essential for the proof technique used there
to work. A comprehensive account of the details would be very lengthy
but not particularly interesting, and since we have already proved
Lemma \ref{lem:the_TI_identity} and none of our main results actually
rely on (\ref{eq:FK_semigroup}), we do not pursue this matter further.
\begin{lem}
\label{lem:FK_derive}Suppose that for any $p\geq1$ and $f\in C_{2}^{p}(\mathbb{R}^{d})$
there exists $q\geq0$ such that for any t, $(s,x)\mapsto T_{s,t}f(x)$
is a member of $C_{1,2}^{p+q}([0,1]\times\mathbb{R}^{d})$, and 
\[
\partial_{s}T_{s,t}f(x)=-\mathcal{L}_{s}T_{s,t}f(x)+T_{s,t}f(x)\cdot\partial_{s}U_{s}(x),\quad\forall x.
\]
Then 
\begin{equation}
\frac{Z_{1}}{Z_{0}}=\pi_{0}T_{0,1}1,\quad\quad\pi_{1}(f)=\frac{\pi_{0}T_{0,1}f}{\pi_{0}T_{0,1}1}.\label{eq:FK_semigroup}
\end{equation}
\end{lem}

\begin{proof}
We shall prove
\[
\frac{\partial}{\partial s}Z_{s}\pi_{s}T_{s,t}f=0,
\]
which implies
\[
\pi_{s}T_{s,t}f=\frac{Z_{t}}{Z_{s}}\pi_{t}f,\quad\forall s\leq t,
\]
and in turn (\ref{eq:FK_semigroup}). 

We have

\begin{align*}
\partial_{s}Z_{s}\pi_{s}T_{s,t}f & =\partial_{s}\int_{\mathbb{R}^{d}}\exp[-U_{s}(x)]T_{s,t}f(x)\mathrm{d}x\\
 & =-\int_{\mathbb{R}^{d}}\partial_{s}U_{s}(x)\exp[-U_{s}(x)]T_{s,t}f(x)\mathrm{d}x\\
 & \quad-\int_{\mathbb{R}^{d}}\exp[-U_{s}(x)][\mathcal{L}_{s}T_{s,t}f(x)-T_{s,t}f(x)\partial_{s}U_{s}(x)]\mathrm{d}x\\
 & =0,
\end{align*}
where the interchange of differentiation and integration is justified
by arguments similar to those in the proof of Lemma \ref{lem:the_TI_identity},
using (A\ref{hyp:basic_hyp_on_U}), (A\ref{hyp:U_grad_lipschitz_x}),
(A\ref{hyp:U_strong_convex}), the assumption of the lemma and Lemma
\ref{lem:drift}; and the final equality holds since $\pi_{s}\mathcal{L}_{s}T_{s,t}f=0$.
\end{proof}

\subsection{Poincaré inequalities, variance and bias bounds\label{subsec:Variance-bounds-for}}

\subsubsection{The commutation relation}
\begin{lem}
\label{lem:com_relation}For any $p\geq1$, $f\in C_{2}^{p}(\mathbb{R}^{d})$,
and $s\leq t$,
\begin{equation}
\|\nabla P_{s,t}f\|\leq e^{-K(t-s)/\epsilon}P_{s,t}\|\nabla f\|.\label{eq:commutation_relation}
\end{equation}
\end{lem}

\begin{proof}
 By the mean value theorem,
\[
f(X_{s,t}^{x})-f(X_{s,t}^{y})=\left\langle \nabla f(Z_{s,t}^{x,y})\,,\,X_{s,t}^{x}-X_{s,t}^{y}\right\rangle ,
\]
 for some $Z_{s,t}^{x,y}$ on the line segment between $X_{s,t}^{x}$
and $X_{s,t}^{y}$. By Cauchy-Schwarz and Lemma \ref{lem:pathwise_sde_bound},
\[
|f(X_{s,t}^{x})-f(X_{s,t}^{y})|\leq\|\nabla f(Z_{s,t}^{x,y})\|\|X_{s,t}^{x}-X_{s,t}^{y}\|\leq\|\nabla f(Z_{s,t}^{x,y})\|e^{-K(t-s)/\epsilon}\|x-y\|,
\]
hence
\begin{equation}
|P_{s,t}f(x)-P_{s,t}f(y)|\leq\mathbb{E}\left[|f(X_{s,t}^{x})-f(X_{s,t}^{y})|\right]\leq\mathbb{E}\left[\|\nabla f(Z_{s,t}^{x,y})\|\right]e^{-K(t-s)/\epsilon}\left\Vert x-y\right\Vert .\label{eq:grad_P_fd}
\end{equation}

Now pick any $v\in\mathbb{R}^{d}$ such that $\|v\|=1$ and set $y(n)\coloneqq x+\frac{1}{n}v$.
Our next step is to use dominated convergence to show:
\begin{equation}
\lim_{n\to\infty}\mathbb{E}\left[\|\nabla f(Z_{t}^{x,y(n)})\|\right]=\mathbb{E}\left[\|\nabla f(X_{s,t}^{x})\|\right].\label{eq:exp_grad}
\end{equation}
Using Lemma \ref{lem:pathwise_sde_bound}, $Z_{s,t}^{x,y(n)}\to X_{s,t}^{x}$
a.s., hence $\|\nabla f(Z_{s,t}^{x,y(n)})\|\to\|\nabla f(X_{s,t}^{x})\|$,
a.s. By the assumption $f\in C_{1}^{p}(\mathbb{R}^{d})$, there exists
a constant $c<\infty$ such that 
\[
\|\nabla f(Z_{s,t}^{x,y})\|\leq c(1+\|Z_{s,t}^{x,y}\|^{2p}),
\]
and using the convexity of $a\mapsto a^{2p}$, 
\begin{eqnarray*}
\|\nabla f(Z_{s,t}^{x,y(n)})\| & \leq & c\left[1+2^{2p-1}\left(\|Z_{s,t}^{x,y(n)}-X_{s,t}^{x}\|^{2p}+\|X_{s,t}^{x}\|^{2p}\right)\right]\\
 & \leq & c\left[1+2^{2p-1}\|X_{s,t}^{y(n)}-X_{s,t}^{x}\|^{2p}+2^{2p-1}\|X_{s,t}^{x}\|^{2p}\right]\\
 & \leq & c\left[1+2^{2p-1}\|x-y(n)\|^{2p}e^{-2pK(t-s)/\epsilon}+2^{2p-1}\|X_{s,t}^{x}\|^{2p}\right]\\
 & \leq & c\left[1+2^{2p-1}e^{-2pK(t-s)/\epsilon}+2^{2p-1}\|X_{s,t}^{x}\|^{2p}\right].
\end{eqnarray*}
Therefore 
\[
\mathbb{E}\left[\sup_{n}\|\nabla f(Z_{s,t}^{x,y(n)})\|\right]\leq c\left[1+2^{2p-1}e^{-2pK(t-s)/\epsilon}+2^{2p-1}\mathbb{E}\left[\|X_{s,t}^{x}\|^{2p}\right]\right]<+\infty,
\]
using Lemma \ref{lem:drift} for the final inequality. Thus we have
proved that indeed (\ref{eq:exp_grad}) holds. 

As $f\in C_{1}^{p}(\mathbb{R}^{d})$, (\ref{eq:C_closed_P}) implies
$\nabla P_{s,t}f(x)$ exists and is continuous in $x$. Since $y(n)-x=v/n$,
we have for some $z(n)$ between $y(n)$ and $x$, 
\[
P_{s,t}f(y(n))-P_{s,t}f(x)=\frac{1}{n}\left\langle \nabla P_{s,t}f(z(n)),v\right\rangle ,
\]
so by the continuity  of $\nabla P_{s,t}f$ we then obtain from (\ref{eq:grad_P_fd})
and (\ref{eq:exp_grad}):

\[
\left|\left\langle \nabla P_{s,t}f(x),v\right\rangle \right|=\lim_{n}\frac{|P_{s,t}f(x)-P_{s,t}f(y(n))|}{\|x-y(n)\|}\leq e^{-K(t-s)/\epsilon}P_{s,t}(\|\nabla f\|)(x).
\]
Taking $v=\nabla P_{s,t}f(x)/\|\nabla P_{s,t}f(x)\|$ completes the
proof.
\end{proof}
\begin{rem}
\label{rem:convexity_necessary}It can be shown that in fact the strong
log-concavity assumption (A\ref{hyp:U_strong_convex}) is necessary
for the statement of Lemma \ref{lem:com_relation} to hold. Indeed,
when that statement does hold, the same line of argument as \cite[Lem. 1.2 or 1.3]{ledoux2000geometry}
shows that the Bakry-Émery criterion holds for $U_{t}$ with constant
$K$, uniformly in $t$, i.e. for all $f\in C_{2}^{p}(\mathbb{R}^{d})$,
\[
\inf_{t\in[0,1]}\left\langle \nabla^{(2)}U_{t}\cdot\nabla f,\nabla f\right\rangle +\|\nabla^{(2)}f\|_{\mathrm{H.S.}}^{2}\geq K\|\nabla f\|^{2}.
\]
So for an arbitrary $v=(v_{1},\ldots,v_{d})\in\mathbb{R}^{d}$, choosing
$f(x)=\sum_{i=1}^{d}v_{i}x_{i}$ gives $\nabla f=v$ and $\|\nabla^{(2)}f\|_{\mathrm{H.S.}}^{2}=0$,
hence 
\[
\inf_{t\in[0,1]}\left\langle \nabla^{(2)}U_{t}\cdot v,v\right\rangle \geq K\|v\|^{2},
\]
which is exactly (A\ref{hyp:U_strong_convex}).
\end{rem}

\subsubsection{Poincaré inequalities }
\begin{lem}
\label{lem:local_poincare}For any $s\leq t$ and $f\in C_{2}^{p}(\mathbb{R}^{d})$,
\begin{equation}
P_{s,t}(f^{2})-(P_{s,t}f)^{2}\leq\frac{1}{K}(1-e^{-2K(t-s)/\epsilon})P_{s,t}(\|\nabla f\|^{2}).\label{eq:local_poincare}
\end{equation}
\end{lem}

\begin{proof}
Consider $t$ fixed and write $g(u,x)=(P_{u,t}f(x))^{2}$. By Proposition
\ref{prop:fwd_and_bck_eqs}, $(u,x)\mapsto P_{u,t}f_{t}(x)$ is a
member of $C_{1,2}^{p+1/2}([0,1]\times\mathbb{R}^{d})$, so $g\in C_{1,2}^{2p+1}([0,1]\times\mathbb{R}^{d})$.
We then may apply (\ref{eq:fwd_equation}) with $\nu=\delta_{x}$
to obtain:
\begin{eqnarray*}
\partial_{u}P_{s,u}\left[(P_{u,t}f)^{2}\right] & = & \partial_{u}P_{s,u}g_{u}\\
 & = & P_{s,u}\left[\frac{\partial g}{\partial u}+\mathcal{L}_{u}g_{u}\right]\\
 & = & -2P_{s,u}\left[(P_{u,t}f)(\mathcal{L}_{u}P_{u,t}f)\right]+P_{s,u}\left[\mathcal{L}_{u}(P_{u,t}f)^{2}\right]\\
 & = & 2\epsilon^{-1}P_{s,u}(\|\nabla P_{u,t}f\|^{2})\\
 & \leq & 2\epsilon^{-1}e^{-2K(t-u)/\epsilon}P_{s,t}(\|\nabla f\|^{2}),
\end{eqnarray*}
where the penultimate equality is an application of (\ref{eq:bwd_equation}),
the final equality holds due to the well known Carré du champ identity:
$\mathcal{L}_{u}(P_{u,t}f)^{2}-2(P_{u,t}f)(\mathcal{L}_{u}P_{u,t}f)=2\epsilon^{-1}\|\nabla P_{u,t}f\|^{2}$,
and the inequality is due to Lemma \ref{lem:com_relation} and Jensen's
inequality. Integrating w.r.t. to $u$ from $s$ to $t$ gives (\ref{eq:local_poincare}).
\end{proof}
\begin{rem}
\label{rem:Poincare_pi} It is well known that under (A\ref{hyp:U_strong_convex}),
each $\pi_{t}$ satisfies a Poincaré inequality with constant $K$,
that is 
\begin{equation}
\mathrm{var}_{\pi_{t}}[f]\leq\frac{1}{K}\pi_{t}(\|\nabla f\|^{2}),\label{eq:poincare_pi}
\end{equation}
for $f$ in some class of suitably smooth functions. We have particular
interest in the case $f\in C_{2}^{p}(\mathbb{R}^{d}),$ and one can
verify that indeed (\ref{eq:poincare_pi}) holds for that class of
functions using Lemma \ref{lem:local_poincare}; for example considering
$\pi_{0}$, assume that $U_{t}=U_{0}$ for all $t\in(0,1]$, so that
$P_{s,t}$ becomes time-homogeneous and $\pi_{0}P_{0,t}=\pi_{0}$.
Then with $s=0$, $t=1$, integrating (\ref{eq:local_poincare}) w.r.t.
$\pi_{0}$ gives
\[
\mathrm{var}_{\pi_{0}}[f]\leq\mathrm{var}_{\pi_{0}}[P_{0,1}f]+\frac{1}{K}(1-e^{-2K/\epsilon})\pi_{0}(\|\nabla f\|^{2}),
\]
and $\mathrm{var}_{\pi_{0}}[P_{0,1}f]\to0$ as $\epsilon\to0$ by
standard results for the time-homogeneous Langevin diffusion (a particular
rate of convergence for $\mathrm{var}_{\pi_{0}}[P_{0,1}f]\to0$ is
not need for this computation).
\end{rem}

\begin{lem}
\label{lem:poincare_transfer}Fix $p\geq1$. If for some given $\nu\in\mathcal{P}^{2p}(\mathbb{R}^{d})$
and constant $K_{\nu}>0$,
\begin{equation}
\mathrm{var}_{\nu}[f]\leq\frac{1}{K_{\nu}}\nu(\|\nabla f\|^{2}),\quad\forall f\in C_{2}^{p}(\mathbb{R}^{d}),\label{eq:poincare_assume}
\end{equation}
then for all $s\leq t$, 
\[
\mathrm{var}_{\nu P_{s,t}}[f]\leq\left[(1-e^{-2K(t-s)/\epsilon})\frac{1}{K}+e^{-2K(t-s)/\epsilon}\frac{1}{K_{\nu}}\right]\nu P_{s,t}(\|\nabla f\|^{2}),\quad\forall f\in C_{2}^{p}(\mathbb{R}^{d}).
\]
\end{lem}

\begin{proof}
Since $\nu\in\mathcal{P}^{2p}(\mathbb{R}^{d})$ we are guaranteed
$\nu(\|\nabla f\|^{2})<+\infty$, and using Lemma \ref{lem:drift},
$\nu P_{s,t}(\|\nabla f\|^{2})<+\infty$. Integrating (\ref{eq:local_poincare})
w.r.t. $\nu$ gives
\[
\nu P_{s,t}(f^{2})-\nu[(P_{s,t}f)^{2}]\leq\frac{1}{K}(1-e^{-2K(t-s)/\epsilon})\nu P_{s,t}(\|\nabla f\|^{2}).
\]
By Proposition \ref{prop:C_2^p_closed}, if $f\in C_{2}^{p}(\mathbb{R}^{d})$
then $P_{s,t}f\in C_{2}^{p}(\mathbb{R}^{d})$, so under the hypotheses
of the lemma, the inequality (\ref{eq:poincare_assume}) holds with
$f$ replaced by $P_{s,t}f$. This observation, together with Lemma
\ref{lem:com_relation} and Jensen's inequality give: 
\begin{eqnarray*}
\mathrm{var}_{\nu P_{s,t}}[f] & \leq & \mathrm{var}_{\nu}[P_{s,t}f]+\frac{1}{K}(1-e^{-2K(t-s)/\epsilon})\nu P_{s,t}(\|\nabla f\|^{2})\\
 & \leq & \frac{1}{K_{\nu}}\nu(\|\nabla P_{s,t}f\|^{2})+\frac{1}{K}(1-e^{-2K(t-s)/\epsilon})\nu P_{s,t}(\|\nabla f\|^{2})\\
 & \leq & \frac{1}{K_{\nu}}\nu P_{s,t}(\|\nabla f\|^{2})e^{-2K(t-s)/\epsilon}+\frac{1}{K}(1-e^{-2K(t-s)/\epsilon})\nu P_{s,t}(\|\nabla f\|^{2}).
\end{eqnarray*}
\end{proof}

\subsubsection{Variance bounds}
\begin{lem}
\label{lem:L_2_convergence}Fix $p\geq1$ and $s\leq t$. If for some
given $\nu\in\mathcal{P}^{2p}(\mathbb{R}^{d})$ and a strictly positive,
continuous function $\kappa_{\nu}:u\in[s,t]\mapsto\kappa_{\nu}(u)\in\mathbb{R}^{+}$,
\[
\mathrm{var}_{\nu P_{s,u}}[f]\leq\frac{1}{\kappa_{\nu}(u)}\nu P_{s,u}(\left\Vert \nabla f\right\Vert ^{2}),\quad\forall f\in C_{2}^{p}(\mathbb{R}^{d}),\;u\in[s,t],
\]
then 
\[
\mathrm{var}_{\nu P_{s,u}}[P_{u,t}f]\leq\exp\left[-\frac{2}{\epsilon}\int_{u}^{t}\kappa_{\nu}(\tau)\mathrm{d}\tau\right]\mathrm{var}_{\nu P_{s,t}}[f],\quad\forall f\in C_{2}^{p}(\mathbb{R}^{d}),\;u\in[s,t].
\]
\end{lem}

\begin{proof}
Arguing similarly to the proof of Lemma \ref{lem:local_poincare},
the map $(u,x)\mapsto(P_{u,t}f(x))^{2}$ is a member of $C_{1,2}^{2p+1}([0,1]\times\mathbb{R}^{d})$
and $P_{u,t}f\in C_{2}^{p}(\mathbb{R}^{d})$. Applying (\ref{eq:fwd_equation})
and (\ref{eq:bwd_equation}), 
\begin{eqnarray*}
\partial_{u}\mathrm{var}_{\nu P_{s,u}}[P_{u,t}f] & = & \partial_{u}\nu P_{s,u}[(P_{u,t}f)^{2}]\\
 & = & \nu P_{s,u}\mathcal{L}_{u}[(P_{u,t}f)^{2}]-2\nu P_{s,u}[(P_{u,t}f)(\mathcal{L}_{u}P_{u,t}f)]\\
 & = & \frac{2}{\epsilon}\nu P_{s,u}(\|\nabla P_{u,t}f\|^{2})\\
 & \geq & \frac{2}{\epsilon}\kappa_{\nu}(u)\mathrm{var}_{\nu P_{s,u}}[P_{u,t}f],
\end{eqnarray*}
where the inequality holds by the hypothesis of the lemma. With $\beta(u)\coloneqq\mathrm{var}_{\nu P_{s,u}}[P_{u,t}f]$
we have shown
\[
\beta^{\prime}(u)\geq\frac{2}{\epsilon}\kappa_{\nu}(u)\beta(u),
\]
so
\[
u\mapsto\beta(u)\exp\left[-\frac{2}{\epsilon}\int_{s}^{u}\kappa_{\nu}(\tau)\mathrm{d}\tau\right]
\]
is a non-decreasing function on $[s,t]$, which implies
\[
\beta(u)\leq\beta(t)\exp\left[-\frac{2}{\epsilon}\int_{u}^{t}\kappa_{\nu}(\tau)\mathrm{d}\tau\right],
\]
as required.
\end{proof}

\subsubsection{Bias bounds}

Introduce
\[
W^{(p)}(\nu,\bar{\nu})\coloneqq\inf_{\gamma\in\Gamma(\nu,\bar{\nu})}\int_{\mathbb{R}^{2d}}\left(1+\|x\|^{2p}\vee\|y\|^{2p}\right)\|x-y\|\gamma(\mathrm{d}x,\mathrm{d}y),
\]
where $\Gamma(\nu,\bar{\nu})$ is the set of all couplings of two
probability measures $\nu,\bar{\nu}$ on $\mathcal{B}(\mathbb{R}^{d})$.
\begin{lem}
\label{lem:bias_nu_bar_nu}For any $p\geq1$, $f\in C_{2}^{p}(\mathbb{R}^{d})$
and $\nu,\bar{\nu}\in\mathcal{P}^{p}(\mathbb{R}^{d})$,

\[
|\nu P_{s,t}f-\bar{\nu}P_{s,t}f|\leq\alpha_{p}\|\nabla f\|_{p}e^{-K(t-s)/\epsilon}W^{(p)}(\nu,\bar{\nu})
\]
where $\alpha_{p}$ is the constant from Lemma \ref{lem:drift},
which depends on $\epsilon,K,p,d$, $\sup_{t}\|\partial_{t}x_{t}^{\star}\|$
and $\sup_{t}\|x_{t}^{\star}\|$. 
\end{lem}

\begin{proof}
Pick any $x,y\in\mathbb{R}^{d}$ and $s\leq t$. Then by the mean
value theorem there exists a point $z$ on the line segment between
$x$ and $y$ such that, 
\begin{align*}
\left|P_{s,t}f(x)-P_{s,t}f(y)\right| & =\left|\left\langle \nabla P_{s,t}f(z),x-y\right\rangle \right|\\
 & \leq\|\nabla P_{s,t}f(z)\|\|x-y\|\\
 & \leq e^{-K(t-s)/\epsilon}P_{s,t}(\|\nabla f\|)(z)\|x-y\|\\
 & \leq\|\nabla f\|_{p}e^{-K(t-s)/\epsilon}(1+\mathbb{E}[\|X_{s,t}^{z}\|^{2p}])\|x-y\|\\
 & \leq\alpha_{p}\|\nabla f\|_{p}e^{-K(t-s)/\epsilon}\left[1+\|x\|^{2p}\vee\|y\|^{2p}\right]\|x-y\|,
\end{align*}
where the second inequality is due to Lemma \ref{lem:com_relation},
and the fourth inequality uses Lemma \ref{lem:drift} and the fact
$\|z\|\leq\|x\|\vee\|y\|$. The proof is completed by noting:
\[
|\nu P_{s,t}f-\bar{\nu}P_{s,t}f|\leq\int|P_{s,t}f(x)-P_{s,t}f(y)|\gamma(\mathrm{d}x,\mathrm{d}y),\quad\forall\gamma\in\Gamma(\nu,\bar{\nu}).
\]
\end{proof}
\begin{lem}
\label{lem:deriv_int_interchange}For any $p\geq1$,
\begin{equation}
\sup_{t}\int_{\mathbb{R}^{d}}\|x\|^{2p}\pi_{t}(\mathrm{d}x)<+\infty,\label{eq:pi_t_finite_polynomial_moments}
\end{equation}
 and for any $f\in C_{2}^{p}(\mathbb{R}^{d})$,
\begin{equation}
\int_{\mathbb{R}^{d}}\sup_{s}\left|\partial_{s}\left\{ \frac{\exp[-U_{s}(x)]}{Z_{s}}P_{s,t}f(x)\right\} \right|\mathrm{d}x<+\infty.\label{eq:integrability_sup_ds}
\end{equation}
\end{lem}

\begin{proof}
We have
\begin{align*}
 & \left|\partial_{s}\left\{ \frac{\exp[-U_{s}(x)]}{Z_{s}}P_{s,t}f(x)\right\} \right|\\
 & =\frac{\exp[-U_{s}(x)]}{Z_{s}}\left|\phi_{s}(x)P_{s,t}f(x)-\mathcal{L}_{s}P_{s,t}f(x)\right|\\
 & \leq\frac{\exp[-U_{s}(x)]}{Z_{s}}\left[|\phi_{s}(x)||P_{s,t}f(x)|+|\mathcal{L}_{s}P_{s,t}f(x)|\right]
\end{align*}
Under (A\ref{hyp:U_grad_lipschitz_x}) and (A\ref{hyp:U_strong_convex}),
for all $s\in[0,1]$ and $x\in\mathbb{R}^{d}$,
\begin{equation}
\inf_{t}U_{t}(x_{t}^{\star})+\left(\|x\|-\inf_{t}\|x_{t}^{\star}\|\right)^{2}\frac{K}{2}\leq U_{s}(x)\leq\frac{L}{2}\left(\|x\|+\sup_{t}\|x_{t}^{\star}\|\right)^{2}+\sup_{t}U_{t}(x_{t}^{\star}),\label{eq:U_upper_and_lower_uniform_bounds}
\end{equation}
where the infima and suprema are finite, since by Lemma \ref{lem:minimizer},
$t\mapsto\|x_{t}^{\star}\|$ is continuous on $[0,1]$, and $U_{t}(x)$
is continous in $(t,x)$ by (A\ref{hyp:basic_hyp_on_U}). It follows
from (\ref{eq:U_upper_and_lower_uniform_bounds}) that $\inf_{t}Z_{t}>0$
and $\sup_{s}\exp[-U_{s}(x)]\leq\exp[-c_{1}\|x\|^{2}+c_{2}]$ for
some finite constants $c_{1},c_{2}>0$, which implies (\ref{eq:pi_t_finite_polynomial_moments}).
Also, since $U\in C_{1,2}^{p_{0}}([0,1]\times\mathbb{R}^{d})$ under
(A\ref{hyp:basic_hyp_on_U}), it follows from (\ref{eq:U_upper_and_lower_uniform_bounds})
and Lemma \ref{lem:the_TI_identity} that $(t,x)\mapsto\phi_{t}(x)$
is a member of $C_{0,2}^{p_{0}}([0,1]\times\mathbb{R}^{d})$. Since
$f\in C_{2}^{p}(\mathbb{R}^{d})$, it follows from Proposition \ref{prop:fwd_and_bck_eqs}
that $(s,x)\mapsto P_{s,t}f(x)$ is a member of $C_{1,2}^{p+1/2}([0,1]\times\mathbb{R}^{d})$
and from Proposition \ref{prop:C_2^p_closed} that $(s,x)\mapsto\mathcal{L}_{s}P_{s,t}f(x)$
is a member of $C_{0,0}^{p+1/2}([0,1]\times\mathbb{R}^{d})$. These
observations together imply (\ref{eq:integrability_sup_ds}).
\end{proof}
\begin{lem}
\label{lem:bias_pi_0_pi_t}For any $p\geq1$ and $f\in C_{2}^{p}(\mathbb{R}^{d})$,
\[
\left|\pi_{0}P_{0,t}f-\pi_{t}f\right|\leq\sup_{s\in[0,t]}\mathrm{var}_{\pi_{s}}[\phi_{s}]^{1/2}\mathrm{var}_{\pi_{t}}[f]^{1/2}\frac{\epsilon}{K}(1-e^{-Kt/\epsilon}).
\]
\end{lem}

\begin{proof}
Write
\begin{equation}
\pi_{t}f-\pi_{0}P_{0,t}f=\int_{0}^{t}\partial_{s}\pi_{s}P_{s,t}f\mathrm{d}s,\label{eq:bias_integral}
\end{equation}
and
\begin{align*}
\partial_{s}\pi_{s}P_{s,t}f & =\int_{\mathbb{R}^{d}}\partial_{s}\left[\dfrac{\exp[-U_{s}(x)]}{Z_{s}}P_{s,t}f(x)\right]\mathrm{d}x\\
 & =-\pi_{s}[\phi_{s}P_{s,t}f]-\pi_{s}\mathcal{L}_{s}P_{s,t}f\\
 & =-\pi_{s}[(\phi_{s}-\pi_{s}\phi_{s})(P_{s,t}f-\pi_{s}P_{s,t}f)],
\end{align*}
where the first equality is validated by Lemma \ref{lem:deriv_int_interchange};
the second equality holds by the definition of $\phi_{s}$, see (\ref{eq:phi_defn}),
and Proposition \ref{prop:fwd_and_bck_eqs}; and the third equality
holds because by Lemma \ref{lem:the_TI_identity} $\pi_{s}\phi_{s}=0$,
and $\mathcal{L}_{s}$ is the generator of a Langevin diffusion with
invariant distribution $\pi_{s}$. Therefore
\begin{align*}
\left|\partial_{s}\pi_{s}P_{s,t}f\right|^{2} & \leq\mathrm{var}_{\pi_{s}}[\phi_{s}]\mathrm{var}_{\pi_{s}}[P_{s,t}f]\\
 & \leq\mathrm{var}_{\pi_{s}}[\phi_{s}]\mathrm{var}_{\pi_{t}}[f]e^{-2K(t-s)/\epsilon}
\end{align*}
where Cauchy-Schwartz and Lemmas \ref{lem:poincare_transfer} and
\ref{lem:L_2_convergence} with $\nu=\pi_{s}$ have been applied,
noting Remark \ref{rem:Poincare_pi}. Plugging this bound into (\ref{eq:bias_integral})
and integrating completes the proof.
\end{proof}

%% file: child-CLT-inhomogenous.tex
\section{Quantitative CLT bound for the diffusion skeleton\label{sec:Quantitative-CLT-bound}}

\subsection{Set-up and main results}

As before we assume throughout section \ref{sec:Quantitative-CLT-bound}
that for $s\in[0,1]$ $\pi_{s}f_{s}=0$ and for $\epsilon>0$ we let$\bar{f}_{s,\epsilon}:=f_{s}-\mu_{s}^{\epsilon}f_{s}$.
Let$(B_{t})_{t\in\mathbb{R}_{+}}$ be a $d-$dimensional Brownian
motion. As earlier, for any $\epsilon>0$ we define $\big(X_{t}^{\epsilon}\big)_{t\in[0,1]}$
as the continuous solution for $t\in[0,1]$ of 
\begin{equation}
X_{t}^{\epsilon}=X_{0}^{\epsilon}-\epsilon^{-1}\int_{0}^{t}\nabla U_{u}(X_{u}^{\epsilon}){\rm d}u+\sqrt{2\epsilon^{-1}}\int_{0}^{t}{\rm d}B_{u},\label{eq:defXinCLT}
\end{equation}
with $X_{0}^{\epsilon}=:X_{0}$ being $\mathcal{F}_{0}-$measurable
and of distribution $\mu_{0}$. One may be interested in the distributional
limiting behaviour as $\epsilon\rightarrow0$ of
\[
\epsilon^{-1/2}S_{\epsilon}=\epsilon^{-1/2}\int_{0}^{1}f_{t}(X_{t}^{\epsilon}){\rm d}t,
\]
and it is expected that a central limit theorem (CLT) may hold. We
do not focus on this here, but rather investigate the following related
problem. Define, for any $h\in(0,1)$, quantities resulting from a
Riemann sum approximation of the integral above, 
\[
\epsilon^{-1/2}S_{\epsilon,h}:=\epsilon^{-1/2}h\sum_{i=0}^{n-1}f_{ih}(X_{ih}^{\epsilon}).
\]
where $n:=\lfloor1/h\rfloor$ (note that $n\geq1$ by assumption).
The aims of this section are to characterize $\lim_{\epsilon\rightarrow0}{\rm var}\left[\epsilon^{-1/2}S_{\epsilon,h(\epsilon)}\right]$
and the limiting distributional behaviour of $\epsilon^{-1/2}S_{\epsilon,h(\epsilon)}$
as $\epsilon\rightarrow0$, for various choices of $h(\cdot):\mathbb{R}_{+}\rightarrow(0,1)$.
Note that in order to alleviate notation below we may use $h$ for
$h(\epsilon)$ when no confusion is possible. 

In order to present the main result of this section we introduce quantities
related to the following family of time homogeneous and stationary
processes $\big(Y_{t}^{s,\epsilon}\big)_{(s,t)\in[0,1]\times\mathbb{R}_{+},\epsilon>0}$.
Let for any $s\in[0,1]$, $\epsilon>0$, $t\in\mathbb{R}_{+}$ ,
\[
Y_{t}^{s,\epsilon}=Y_{0}^{s,\epsilon}-\epsilon^{-1}\int_{0}^{t}\nabla U_{s}(Y_{u}^{s,\epsilon}){\rm d}u+\sqrt{2\epsilon^{-1}}\int_{0}^{t}{\rm d}B_{u}
\]
with $Y_{0}^{s,\epsilon}=:Y_{0}^{s}$ $\mathcal{F}_{0}-$measurable
of distribution $\pi_{s}$. We naturally use $\mathbb{P}\big[\cdot\big]$
and $\mathbb{E}\big[\cdot\big]$ for the laws and expectations of
both $\big(X_{t}^{\epsilon}\big)_{t\in[0,1],\epsilon>0}$ and $\big(Y_{t}^{s,\epsilon}\big)_{(s,t)\in[0,1]\times\mathbb{R}_{+},\epsilon>0}$.
For $s\in[0,1]$ we let $L_{2}(\pi_{s})$ be the set of real valued
and $\pi_{s}-$square integrable functions on $\mathbb{R}^{d}$. For
any $(s,t)\in[0,1]\times\mathbb{R}_{+}$, $f\in L^{2}(\pi_{s})$,
$\epsilon>0$ and $x\in\mathbb{R}^{d}$ we let $Q_{t}^{s,\epsilon}f(x):=\mathbb{E}\big[f(Y_{t}^{s,\epsilon})\mid Y_{0}^{s}=x\big]$,
$Q_{t}^{s}f(x):=Q_{t}^{s,1}f(x)$ and $P_{s,t}^{\epsilon}f(x):=\mathbb{E}\big[f(X_{t}^{\epsilon})\mid X_{s}=x\big]$.
Standard results on stationary reversible Markov processes and Markov
chains, together with our geometric ergodicity assumptions ensure
that the following limits exist and are finite for $f_{s}\in L^{2}(\pi_{s})$,
\[
\varsigma_{0}(s):=\lim_{\epsilon\rightarrow0}{\rm var}\left[\epsilon^{-1/2}\int_{0}^{1}f_{s}(Y_{t}^{s,\epsilon}){\rm d}t\right]\;\text{and}\;\varsigma_{\ell}(s):=\lim_{\epsilon\rightarrow0}{\rm var}\left[\epsilon^{-1/2}h(\epsilon)\sum_{i=0}^{n-1}f_{s}(Y_{ih(\epsilon)}^{s,\epsilon})\right]\;\text{whenever}\;\ell=h(\epsilon)\epsilon^{-1}>0,
\]
where ${\rm var}\big[\cdot\big]$ is the variance operator associated
with $\mathbb{E}\big[\cdot\big]$. Note the broad use we make throughout
of $\ell$ to refer to scenarios and not just a numerical value. It
is well known that the following upper bounds, in terms of either
spectral gap or $K$ in (A\ref{hyp:U_strong_convex}), hold
\[
\varsigma_{\ell}(s)\leq2{\rm var}_{\pi_{s}}\big(f_{s}\big)\cdot\begin{cases}
\ell{\rm Gap}_{R}\big(Q_{\ell}^{s}\big)^{-1}\leq\big[(1-\exp(-K\ell))/\ell\big]^{-1} & \text{for\;}\ell>0\\
{\rm Gap}\big(\mathcal{L}_{s}\big)^{-1}\leq K^{-1} & \text{for}\;\ell=0
\end{cases}.
\]
The last inequality follows from the fact that from Poincaré's inequality
${\rm var}_{\pi_{s}}\big[f_{s}\big]\leq K^{-1}\mathcal{E}_{\mathcal{L}_{s}}\big[f_{s}\big]$
(with $\mathcal{E}_{\mathcal{L}_{s}}\big[f_{s}\big]:=-\int f_{s}\mathcal{L}_{s}f_{s}){\rm d}\pi_{s}$)
and the variational representation of the spectral gap. These spectral
gap bounds are classic, and can, for example, be deduced from the
spectral representations in Theorem \ref{thm:variance-spectral-decomposition+Geyer}.
Under our assumptions, for any $\ell\geq0$, $s\mapsto\varsigma_{\ell}(\cdot),{\rm var}_{\pi_{s}}\big(f_{s}\big)$
can be shown to be continuous functions (see the proof of Lemma \ref{lem:riemannconvergence},
which exploits the results of Lemma \ref{lem:poisson_continuity_etc}
and the representation (\ref{eq:homogeneousAICwithPoisson}) of $\varsigma_{\ell}(\cdot)$),
and
\begin{equation}
\sigma_{\ell}^{2}:=\int_{0}^{1}\varsigma_{\ell}(s){\rm d}s\label{eq:nonhomogeneousvariance}
\end{equation}
is therefore well defined. The results of this section rely on the
following assumptions. We consider a sequence of processes as above,
indexed by the dimension of the problem $d$, for which we assume
the following.

\begin{hyp}(Polynomial dependence on dimension)\label{hyp:polynomialdependence}We
assume that (A\ref{hyp:dimension_dependence_mse-1}) holds and that
in addition we have
\begin{enumerate}
\item $\frac{\epsilon}{K}\sup_{t\in(0,1)}\|\partial_{t}x_{t}^{\star}\|=O(1),$
\item $\sup_{s\in[0,1]}\|\partial_{t}f_{s}\|_{p}$ and $\sup_{s\in[0,1]}1/\varsigma^{(d)}(s)$
grow at most polynomially in $d$ as $d\rightarrow\infty$.
\end{enumerate}
\end{hyp}

\noindent  We impose the following dependence of $h$ on $\epsilon$.

\begin{hyp}(Dependence between $\epsilon$ and $h$)\label{hyp:dependence-epsilon-h-on-d}
\begin{enumerate}
\item for any $\ell>0$ we set $h(\epsilon):=\ell\epsilon$,
\item for $\ell=0$ we set $h(\epsilon)=O\big(\epsilon^{c}\big)$ for some
$c>1$.
\end{enumerate}
\end{hyp}

We can now formulate our first result. Throughout $C$ is a constant,
not dependent on the quantities in assumptions (A\ref{hyp:basic_hyp_on_U}-\ref{hyp:U_time_cont}),
and whose value may change upon each appearance.
\begin{thm}
\label{thm:convergence-of-variance-with-d}Let $p\geq1$ and for any
$d\in\mathbb{N}$, let $(X_{t}^{\epsilon}(d))_{t\in[0,1]}$ be as
defined in (\ref{eq:defXinCLT}) and $f^{(d)}\in C_{1,2}^{p}([0,1]\times\mathbb{R}^{d})$.
Assume that for any $d\in\mathbb{N}$ (A\ref{hyp:basic_hyp_on_U}-\ref{hyp:U_time_cont})
and (A\ref{hyp:polynomialdependence}) hold. Then for any $\ell\geq0$
there exists $a>0$ such that with $\epsilon(d)=O(d^{-a})$ and $d\mapsto h(d)$
satisfying (A\ref{hyp:dependence-epsilon-h-on-d}), then
\[
\lim_{d\rightarrow\infty}\Bigl|{\rm var}\left[\epsilon(d)^{-1/2}S_{\epsilon(d),h(d)}\right]-\sigma_{\ell}^{2}(d)\Bigr|=0.
\]
\end{thm}

This result is a consequence of Theorem \ref{thm:CLT-variance-convergence}.
As an aside, it is natural to investigate the impact of $\ell$ on
this asymptotic variance $\sigma_{\ell}^{2}$. The following result
confirms our intuition that the smaller $\ell$, the better; the result
below can be understood as being a generalisation of \cite[Theorem 3.3]{geyer1992practical},
an important fact in the area of discrete time Markov chain Monte
Carlo methods, concerned with thinning in the context of ergodic averages.
The proof can be found in Section \ref{subsec:Proof-of-Theorem-geyer}.
\begin{thm}
\label{thm:variance-spectral-decomposition+Geyer}For $s\in[0,1]$
and any $f_{s}\in L^{2}(\pi_{s})$ there exists a non-negative measure
$\nu_{s}$ on $\big([0,\infty),\mathcal{B}([0,\infty))\big)$ such
that for $\ell>0$
\[
\varsigma_{\ell}(s)=\ell\int_{0}^{\infty}\frac{1+\exp(-\ell\lambda)}{1-\exp(-\ell\lambda)}\nu_{s}({\rm d}\lambda),
\]
and
\[
\varsigma_{0}(s)=2\int_{0}^{\infty}\lambda^{-1}\nu_{s}({\rm d}\lambda).
\]
Further, for any $s\in[0,1]$, $\ell\mapsto\varsigma_{\ell}(s)$ is
a non-decreasing function on $[0,\infty)$.
\end{thm}

Let $\Phi(\cdot)$ be the cumulative distribution function of the
standardized normal distribution. The main result of this section
is 
\begin{thm}
\label{thm:CLT}Let $p\geq1$ and for any $d\in\mathbb{N}$, let $(X_{t}^{\epsilon}(d))_{t\in[0,1]}$
be as defined in (\ref{eq:defXinCLT}) and $f^{(d)}\in C_{1,2}^{p}([0,1]\times\mathbb{R}^{d})$.
Assume that for any $d\in\mathbb{N}$ (A\ref{hyp:basic_hyp_on_U}-\ref{hyp:U_time_cont})
and (A\ref{hyp:polynomialdependence}) hold. Then for any $\ell\geq0$
there exists $a>0$ such that with $\epsilon(d)=O(d^{-a})$ and $d\mapsto h(d)$
satisfying (A\ref{hyp:dependence-epsilon-h-on-d}), then
\[
\lim_{d\rightarrow\infty}\sup_{w\in\mathbb{R}}\big|\mathbb{P}\Bigl[\epsilon(d)^{-1/2}S_{\epsilon(d),h(d)}/\sqrt{\sigma_{\ell}^{2}(d)}\leq w\Bigr]-\Phi(w)\big|=0.
\]
\end{thm}

As seen in Proposition (\ref{prop:tv_dim_dependence}), the scenario
we are particularly interested in corresponds to the choice $h(d)=o\big(\epsilon(d)^{2}/d\big)$
or $h=h(\epsilon)=O\big(\epsilon(d)^{2}/d\big)$ as $d\rightarrow\infty$
(or even fixed $d$ and $\epsilon\rightarrow0$), in which case the
CLT is inherited by the discretized Langevin process, see Section
\ref{sec:Controlling-the-discretization}. The proof of the theorem
above relies on a martingale approximation and a quantitative bound
for the CLT for martingales.
\begin{proof}
First we consider the upper bound suggested by Proposition \ref{prop:decompositionCLT}.
Then we choose $\varepsilon_{1}(d)=Cd^{-c}$ with $c\in(0,1/2)$ as
in Lemma \ref{lem:controlbiasforCLT} and Lemma \ref{lem:controlremainderforCLT},
$\varepsilon_{2}(d)$ as in Corollary \ref{cor:CLT-approximation-replace-variance-with-asymp-variance}
with, say $r_{2}>1/2$, implying that $\lim_{d\rightarrow\infty}\varepsilon_{1}(d)\varepsilon_{2}^{-1}(d)=\infty$.
The result then follows from Theorem \ref{thm:CLTforMartingale}.
\end{proof}

\subsection{Quantitative Martingale approximation for the CLT\label{subsec:Quantitative-Martingale-approxim}}

The main result of this section is Proposition \ref{prop:decompositionCLT}
which establishes a bound on $\sup_{w\in\mathbb{R}}\big|\mathbb{P}\bigl[S_{\epsilon,h}/\sqrt{\epsilon\sigma_{\ell}^{2}}\leq w\bigr]-\Phi(w)\big|$
in terms of the sum of $\sup_{w\in\mathbb{R}}\big|\mathbb{P}\bigl[M_{\epsilon}\leq w\bigr]-\Phi(w)\big|$,
where $M_{\epsilon}$ is the last term of a Martingale sequence, and
additional negligible terms for which we derive quantitative bounds.
We find a quantitative upper bound on $\sup_{w\in\mathbb{R}}\big|\mathbb{P}\bigl[M_{\epsilon}\leq w\bigr]-\Phi(w)\big|$
in section \ref{subsec:Quantitative-bound-CLT-variance}. There are
essentially two routes to constructing such an approximation. An approach
consists of using solutions to the set of time homogeneous Poisson
equations $g_{s}-Q_{h\epsilon^{-1}}^{s}g_{s}=f_{s}$, but we here
follow an approach inspired by \cite{sethuraman2005martingale}, which
consists of treating bias and variance separately by centering $f_{t}$
around $\mu_{t}^{\epsilon}f_{t}$, and not $\pi_{t}f_{t}$. Note that
we have also avoided the use of the solutions of the Poisson equation
for the continous time processes involved (that is either $\mathcal{L}_{s}g_{s}=-f_{s}$
or its time inhomogeneous counterpart) as this would have required
quantitative bounds on their gradients with respect to $x$ and on
their time derivatives. Such bounds are currently not available with
sufficient generality \cite{pardoux2001,pardoux2005,Veretennikov2011}
to cover our scenario. We introduce $B_{\epsilon,h}:=\mathbb{E}\big[S_{\epsilon,h}\big],$
and construct our martingale approximation of $S_{\epsilon,h}/\sqrt{\epsilon\sigma_{\ell}^{2}}$.
Following \cite{sethuraman2005martingale} we introduce for $k\in\{0,\ldots,n-1\}$
and $x\in\mathbb{R}^{d}$
\begin{align*}
\gamma_{k,\epsilon}(x): & =\sum_{i=k}^{n-1}P_{kh,ih}^{\epsilon}\bar{f}_{ih,\epsilon}(x).
\end{align*}
Remark that for $0\leq k\leq n-2$, $\gamma_{k,\epsilon}$ satisfies
\begin{equation}
\bar{f}_{kh,\epsilon}(x)=\gamma_{k,\epsilon}(x)-P_{kh,(k+1)h}^{\epsilon}\gamma_{k+1,\epsilon}(x)\label{eq:generalizedPoisson}
\end{equation}
for any $x\in\mathbb{R}^{d}$\textendash this can be thought of as
a generalization of Poisson's equation. In order to formulate our
explicit bounds concisely and in a unified manner we introduce some
notation and establish useful identities in Proposition \ref{lem:intermediateresultsonV}.
Define for $q>0$ $V^{(q)}(x):=\|x\|^{2q}$, $\bar{V}^{(q)}(x):=1+\|x\|^{2q}$,
$\bar{V}_{t}^{(q)}(x):=1+V_{t}^{(q)}(x):=1+\|x-x_{t}^{\star}\|^{2q}$
(with notational simplifications $\bar{V}_{t}:=\bar{V}_{t}^{(1)}$
and $V_{t}:=V_{t}^{(1)}$ etc.). In addition to what is proposed in
Section \ref{sec:Notation-definitions}, for $f:[0,1]\times\mathbb{R}^{d}\rightarrow\mathbb{R}$
we let $\|\partial_{t}f\|_{p}:=\sup_{t\in[0,1]}\|\partial_{t}f_{t}\|_{p}$
and $\|\nabla^{(r)}f\|_{p}:=\sup_{t\in[0,1]}\|\nabla^{(r)}f_{t}\|_{p}$.
We let $\vvvert f\vvvert_{p}:=\|f\|_{\bar{V}^{(p)}}\vee\|\nabla f\|_{\bar{V}^{(p)}}\vee\|\Delta f\|_{\bar{V}^{(p)}}$.
The proofs not present in this subsection can be found in subsection
\ref{subsec:Quantitative-CLT-constants}.
\begin{lem}
\label{lem:Poisson-and-martingale}Let $p\geq1$ and $f\in C_{0,2}^{p}([0,1]\times\mathbb{R}^{d})$.
\begin{enumerate}
\item For any $\epsilon,h>0$ and $k\in\{0,\ldots,n-1\}$, $\gamma_{k,\epsilon}\in C_{2}^{p}([0,1]\times\mathbb{R}^{d})$
and we have the quantitative bound 
\[
\max_{k\in\{0,\ldots,n-1\}}\big\{\big|P_{kh,(k+1)h}^{\epsilon}\gamma_{k+1,\epsilon}(x)\big|\vee\big|\gamma_{k,\epsilon}(x)\big|\big\}\leq\alpha_{p}\frac{\|\nabla f\|_{p}}{1-\exp\big(-K\epsilon^{-1}h\big)}W^{(p)}(\delta_{x},\mu_{0}).
\]
\item $\mathbb{P}-$a.s. we have
\[
S_{h,\epsilon}-\mathbb{E}\big[S_{h,\epsilon}\big]=\sum_{k=0}^{n-1}\bar{f}_{kh,\epsilon}(X_{kh}^{\epsilon})=\gamma_{0,\epsilon}(X_{0}^{\epsilon})+\sum_{k=1}^{n-1}\gamma_{k,\epsilon}(X_{kh}^{\epsilon})-P_{(k-1)h,kh}^{\epsilon}\gamma_{k,\epsilon}(X_{(k-1)h}^{\epsilon}),
\]
\item For $1\leq k\leq n-1$ define $\xi_{k,\epsilon}:=\left(\gamma_{k,\epsilon}(X_{kh}^{\epsilon})-P_{(k-1)h,kh}^{\epsilon}\gamma_{k,\epsilon}(X_{(k-1)h}^{\epsilon})\right)$,
$\xi_{0,\epsilon}:=0$,
\[
\upsilon(\epsilon):=\epsilon^{-1}h^{2}{\rm var}\left[\sum_{i=0}^{n-1}\xi_{i,\epsilon}\right],
\]
and for $0\leq k\leq n-1$ and $\epsilon>0$ such that $\upsilon(\epsilon)>0$
we let 
\begin{align*}
M_{k,\epsilon} & :=\epsilon^{-1/2}h\sum_{i=0}^{k}\xi_{i,\epsilon}/\sqrt{\upsilon(\epsilon)}.
\end{align*}
Then $\big(M_{i,\epsilon},\mathcal{F}_{ih}\big)_{i\in\{0,\ldots,n-1\}}$
is a martingale.
\end{enumerate}
\end{lem}

\begin{proof}
For notational simplicity we drop $\epsilon$ from $P_{s,t}^{\epsilon}$
here. For the first statement we first apply Proposition \ref{prop:C_2^p_closed}
and then use Lemma \ref{lem:bias_nu_bar_nu} in order to obtain the
quantitative bound : for any $x\in\mathbb{R}^{d}$ 
\[
|\delta_{x}P_{0,t}^{\epsilon}f_{t}-\mu_{0,t}^{\epsilon}f_{t}|\leq\alpha_{p}\|\nabla f_{t}\|_{p}W^{(p)}(\delta_{x},\pi_{0})\exp\big(-K\epsilon^{-1}t\big)
\]
and therefore for $k\in\{0,\ldots,n-1\}$, 
\begin{align*}
\big|P_{kh,(k+1)h}^{\epsilon}\gamma_{k+1,\epsilon}(x)\big|\vee\big|\gamma_{k,\epsilon}(x)\big| & \leq\alpha_{p}\frac{\|\nabla f\|_{p}}{1-\exp\big(-K\epsilon^{-1}h\big)}W^{(p)}(\delta_{x},\mu_{0}).
\end{align*}
The second statement: from (\ref{eq:generalizedPoisson}) we have
for $1\leq k\leq n-2$
\[
\bar{f}_{kh,\epsilon}(X_{kh}^{\epsilon})=\gamma_{k,\epsilon}(X_{kh}^{\epsilon})-P_{(k-1)h,kh}\gamma_{k,\epsilon}(X_{(k-1)h}^{\epsilon})+P_{(k-1)h,kh}\gamma_{k,\epsilon}(X_{(k-1)h}^{\epsilon})-P_{kh,(k+1)h}\gamma_{k+1,\epsilon}(X_{kh}^{\epsilon})
\]
and therefore
\[
\sum_{k=1}^{n-2}\bar{f}_{kh,\epsilon}(X_{kh}^{\epsilon})=P_{0,h}\gamma_{1,\epsilon}(X_{0}^{\epsilon})-P_{(n-2)h,(n-1)h}\gamma_{n-1,\epsilon}(X_{(n-2)h}^{\epsilon})+\sum_{k=1}^{n-2}\gamma_{k,\epsilon}(X_{kh}^{\epsilon})-P_{(k-1)h,kh}\gamma_{k,\epsilon}(X_{(k-1)h}^{\epsilon})
\]
Now, since $\bar{f}_{(n-1)h,\epsilon}(X_{(n-1)h}^{\epsilon})=\gamma_{n-1,\epsilon}(X_{(n-1)h}^{\epsilon})$
and $\bar{f}_{0,\epsilon}(X_{0}^{\epsilon})=\gamma_{0,\epsilon}(X_{0}^{\epsilon})-P_{0,h}\gamma_{1,\epsilon}(X_{0}^{\epsilon}),$
we conclude. The third statement follows from $\mathbb{E}\Bigl[\gamma_{k,\epsilon}(X_{kh}^{\epsilon})-P_{(k-1)h,kh}\gamma_{k,\epsilon}(X_{(k-1)h}^{\epsilon})\mid\mathcal{F}_{(k-1)h}\Bigr]=0$
for $k\in\{1,\ldots,n-1\}$ and the first statement combined with
Lemma \ref{lem:drift} (for the lemma's $p$ sufficiently large) and
the fact that $\sup_{t\in[0,1]}\|x_{t}^{\star}\|<\infty$ from Lemma
\ref{lem:minimizer}, which establishes that for any $i\in\{0,\ldots,n-1\}$,
$\mathbb{E}(|M_{i,\epsilon}|)<\infty$.
\end{proof}
In what follows we let $M_{\epsilon}:=M_{n-1,\epsilon}$ where the
latter is defined in Lemma \ref{lem:Poisson-and-martingale}. The
following proposition will be used to establish that one can obtain
the desired quantitative CLT bounds by focusing on the martingale
approximation (Section \ref{subsec:Quantitative-bound-CLT-variance})
and the appropriate control of vanishing terms (Lemma \ref{lem:controlbiasforCLT}
and Lemma \ref{lem:controlremainderforCLT}). 
\begin{prop}
\label{prop:decompositionCLT}For any $\varepsilon_{1},\varepsilon_{2}>0$
and $\epsilon>0$ such that $\upsilon(\epsilon)>0$,
\begin{multline*}
\sup_{w\in\mathbb{R}}\big|\mathbb{P}\bigl[S_{\epsilon,h}/\sqrt{\epsilon\upsilon(\epsilon)}\leq w\bigr]-\Phi(w)\big|\leq\sup_{w\in\mathbb{R}}\big|\mathbb{P}\bigl[M_{\epsilon}\leq w\bigr]-\Phi(w)\big|+\mathbb{P}\bigl[|B_{\epsilon,h}|/\sqrt{\epsilon\upsilon(\epsilon)}>\varepsilon_{1}/2\bigr]\\
+\mathbb{P}\bigl[h|\gamma_{0,\epsilon}(X_{0}^{\epsilon})|/\sqrt{\epsilon\upsilon(\epsilon)}>\varepsilon_{1}/2\bigr]+(2\pi)^{-1/2}\varepsilon_{1},
\end{multline*}
and 
\begin{multline*}
\sup_{w\in\mathbb{R}}\big|\mathbb{P}\bigl[S_{\epsilon,h}/\sqrt{\epsilon\sigma_{\ell}^{2}}\leq w\bigr]-\Phi(w)\big|\leq2\sup_{w\in\mathbb{R}}\big|\mathbb{P}\bigl[S_{\epsilon,h}/\sqrt{\epsilon\upsilon(\epsilon)}\leq w\bigr]-\Phi(w)\big|+1-\Phi(\varepsilon_{1}\varepsilon_{2}^{-1})\\
+\mathbb{P}\bigl[\big|\upsilon^{1/2}(\epsilon)/\sigma_{\ell}-1\big|>\varepsilon_{2}\bigr]+(2\pi)^{-1/2}\varepsilon_{1}.
\end{multline*}
\end{prop}

\begin{proof}
We have the general result that for $\varepsilon>0$ and two random
variables $Z_{1},Z_{2}$
\[
\mathbb{P}\bigl[Z_{1}\leq w-\varepsilon\bigr]-\mathbb{P}\big[|Z_{2}|>\varepsilon\big]\leq\mathbb{P}\big[Z_{1}+Z_{2}\leq w\big]\leq\mathbb{P}\big[Z_{1}\leq w+\varepsilon\big]+\mathbb{P}\big[|Z_{2}|>\varepsilon\big],
\]
and therefore
\begin{multline*}
\mathbb{P}\big[Z_{1}\leq w-\varepsilon\big]-\Phi(w-\varepsilon)+\Phi(w-\varepsilon)-\Phi(w)-\mathbb{P}\big[|Z_{2}|>\varepsilon\big]\leq\mathbb{P}\big[Z_{1}+Z_{2}\leq w\big]-\Phi(w)\\
\leq\mathbb{P}\big[Z_{1}\leq w+\varepsilon\big]-\Phi(w+\varepsilon)+\Phi(w+\varepsilon)-\Phi(w)+\mathbb{P}\big[|Z_{2}|>\varepsilon\big].
\end{multline*}
Now notice that $\max_{a\in\{\varepsilon,-\varepsilon\}}\big|\Phi(w+a)-\Phi(w)\big|\leq(2\pi)^{-1/2}\varepsilon$
and conclude that
\[
\sup_{w\in\mathbb{R}}\big|\mathbb{P}\big[Z_{1}+Z_{2}\leq w\big]-\Phi(w)\big|\leq\sup_{w'\in\mathbb{R}}\big|\mathbb{P}\big[Z_{1}\leq w'\big]-\Phi(w')\big|+\mathbb{P}\big[|Z_{2}|>\varepsilon\big]+(2\pi)^{-1/2}\varepsilon.
\]
We have
\[
S_{\epsilon,h}/\sqrt{\epsilon\upsilon(\epsilon)}=(h\gamma_{0,\epsilon}(X_{0}^{\epsilon})+B_{\epsilon,h})/\sqrt{\epsilon\upsilon(\epsilon)}+M_{\epsilon},
\]
and
\[
S_{\epsilon,h}/\sqrt{\epsilon\sigma_{\ell}^{2}}=S_{\epsilon,h}/\sqrt{\epsilon\upsilon(\epsilon)}+\epsilon^{-1/2}S_{\epsilon,h}\big(\sigma_{\ell}^{-1}-\upsilon^{-1/2}(\epsilon)).
\]
We can apply the above general inequality to these two identities
in turn. In the first case we also note the fact that $\mathbb{P}\big[|Z_{1}+Z_{2}|>\varepsilon\big]\leq\mathbb{P}\big[|Z_{1}|+|Z_{2}|>\varepsilon\big]\leq\mathbb{P}\big[|Z_{1}|>\varepsilon/2\big]+\mathbb{P}\big[|Z_{2}|>\varepsilon/2\big]$.
In the second case we have that, in general, for non-negative random
variables $Z_{1},Z_{2}$ and any $\varepsilon_{1},\varepsilon_{2}>0$
\[
\mathbb{P}\big[Z_{1}Z_{2}>\varepsilon_{1}\big]\leq\mathbb{P}\big[Z_{1}>\varepsilon_{1}\varepsilon_{2}^{-1}\big]+\mathbb{P}\big[Z_{2}>\varepsilon_{2}\big]
\]
and therefore 
\[
\mathbb{P}\big[\epsilon^{-1/2}\big|S_{\epsilon,h}\big|\big|\sigma_{\ell}^{-1}-\upsilon^{-1/2}(\epsilon)\big|>\varepsilon_{1}\big]\leq\mathbb{P}\big[\big|S_{\epsilon,h}\big|/\sqrt{\epsilon\upsilon(\epsilon)}>\varepsilon_{1}\varepsilon_{2}^{-1}\big]+\mathbb{P}\big[\big|\upsilon^{1/2}(\epsilon)/\sigma_{\ell}-1\big|>\varepsilon_{2}\big].
\]
Finally
\[
\mathbb{P}\big[\big|S_{\epsilon,h}\big|/\sqrt{\epsilon\upsilon(\epsilon)}>\varepsilon_{1}\varepsilon_{2}^{-1}\big]=1-\mathbb{P}\big[\big|S_{\epsilon,h}\big|/\sqrt{\epsilon\upsilon(\epsilon)}\leq\varepsilon_{1}\varepsilon_{2}^{-1}\big]+\Phi(\varepsilon_{1}\varepsilon_{2}^{-1})-\Phi(\varepsilon_{1}\varepsilon_{2}^{-1}),
\]
from which we conclude.
\end{proof}
The following lemmata, whose proofs can be found in Subsection \ref{subsec:proofs-Quantitative-Martingale-approxim},
establish quantitative bounds for some of the vanishing terms appearing
in one of the upper bounds in Proposition \ref{prop:decompositionCLT}.
A quantitative bound for $\mathbb{P}\big[\big|\upsilon^{1/2}(\epsilon)/\sigma_{\ell}-1\big|>\varepsilon_{2}\big]$
is established later in Corollary \ref{cor:CLT-approximation-replace-variance-with-asymp-variance}. 
\begin{lem}
\label{lem:controlbiasforCLT}Let $p\geq1$ and $f\in C_{0,2}^{p}([0,1]\times\mathbb{R}^{d})$,
and assume (A\ref{hyp:basic_hyp_on_U}-\ref{hyp:U_time_cont}) and
(A\ref{hyp:polynomialdependence}). Then 
\begin{enumerate}
\item for any $\varepsilon_{1}>0$, $\ell\geq0$, $\gimel>1$ and $\epsilon,h,K>0$
such that $\gimel^{-1}\leq1-Kh\epsilon^{-1}/2$, 
\[
\mathbb{P}\big[|B_{\epsilon,h}|/\sqrt{\epsilon\upsilon(\epsilon)}>\varepsilon_{1}/2\big]\leq\mathbb{I}\{F(d)>\upsilon(\epsilon)^{1/2}\epsilon^{-1/2}\varepsilon_{1}\},
\]
where, with the notation of Corollary \ref{cor:dimension-dependence-MSE-1},
\[
F(d):=C\frac{1}{K}\left[r_{2}(d)+\gimel r_{3}(d)\right],
\]
\item further assuming (A\ref{hyp:dependence-epsilon-h-on-d}), we deduce
that for any $c\in(0,1/2)$ and the choice $\varepsilon_{1}(d)=C\epsilon(d)^{c}$
there exists $a_{0}>0$ and $d_{0}\in\mathbb{N}$ such that with $\epsilon(d)=Cd^{-a}$,
for $a\geq a_{0}$ and $d\geq d_{0}$
\[
\mathbb{P}\big[|B_{\epsilon(d),h(d)}|/\sqrt{\epsilon(d)\upsilon_{d}\big(\epsilon(d)\big)}>\varepsilon_{1}(d)/2\big]=0.
\]
\end{enumerate}
\end{lem}

\begin{lem}
\label{lem:controlremainderforCLT}Assume (A\ref{hyp:basic_hyp_on_U}-\ref{hyp:U_time_cont})
and (A\ref{hyp:polynomialdependence}). Then
\begin{enumerate}
\item there exists $C>0$ such that for any $\epsilon,\varepsilon_{1},h>0$
such that $\upsilon(\epsilon)>0$ and for some $\gimel>1$ and $\gimel^{-1}\leq1-Kh\epsilon^{-1}/2$
\[
\mathbb{P}\big[h|\gamma_{0,\epsilon}(X_{0}^{\epsilon})|/\sqrt{\epsilon\upsilon(\epsilon)}>\varepsilon_{1}/2\big]\leq C\left(\frac{\alpha_{p}}{\epsilon^{-1/2}\varepsilon_{1}\sqrt{\upsilon(\epsilon)}}\frac{\gimel\|\nabla f\|_{p}}{K}\mu_{0}\bar{V}^{(p+1/2)}\mu_{0}\bar{V}^{(p+1/2)}\right),
\]
\item for any $c\in(0,1/2)$ and the choice $\varepsilon_{1}(d)=C\epsilon(d)^{c}$
there exists $a_{0}>0$ sufficiently large such that for any $a>a_{0}$
and $\epsilon(d)=Cd^{-a}$
\[
\lim_{d\rightarrow\infty}\mathbb{P}\big[h(d)|\gamma_{0,\epsilon(d)}(X_{0}^{\epsilon(d)})|/\sqrt{\epsilon(d)\upsilon_{d}\big(\epsilon(d)\big)}>\varepsilon_{1}(d)/2\big]=0
\]
\end{enumerate}
\end{lem}

\subsection{Quantitative bound in the CLT for the Martingale approximation\label{subsec:Quantitative-bound-CLT-variance}}

We now state an intermediate result which motivates subsequent developments
to prove the quantitative bounds in Theorem \ref{thm:CLT}. 
\begin{thm}
\label{thm:CLTforMartingale}Let $p\geq1$ and for any $d\in\mathbb{N}$,
let $(X_{t}^{\epsilon}(d))_{t\in[0,1]}$ be as defined in (\ref{eq:defXinCLT})
and $f^{(d)}\in C_{1,2}^{p}([0,1]\times\mathbb{R}^{d})$. Assume that
for any $d\in\mathbb{N}$ (A\ref{hyp:basic_hyp_on_U}-\ref{hyp:U_time_cont})
and (A\ref{hyp:polynomialdependence}) hold. Let $M_{\epsilon}:=M_{n-1,\epsilon}$
where the latter is defined in Lemma \ref{lem:Poisson-and-martingale}.
Then for any $\ell\geq0$ there exists $a>0$ such that with $\epsilon(d)=O(d^{-a})$
and $d\mapsto h(d)$ satisfying (A\ref{hyp:dependence-epsilon-h-on-d})
\[
\lim_{d\rightarrow\infty}\sup_{w\in\mathbb{R}}\big|\mathbb{P}\big[M_{\epsilon(d)}\leq w\big]-\Phi(w)\big|=0.
\]
\end{thm}

\begin{proof}
The proof relies on the upper bound established in Proposition \ref{prop:CLTforMartingale}
and bounds for $A_{\epsilon},B_{\epsilon}$ and $C_{\epsilon}$ which
can be deduced from Lemma \ref{prop:B_epsilon} and \ref{lem:C_epsilon-moment-D_epsilon},
and Theorem \ref{thm:CLT-variance-convergence}. More precisely, choose
$\kappa>c-1$, where $c$ is given in (A\ref{hyp:dependence-epsilon-h-on-d}).
For $A_{\epsilon}$: from (A\ref{hyp:polynomialdependence}) and Lemma
\ref{lem:alpha_dimension_dependence} one deduces that the bound on
$\mathbb{E}\big[\big|D_{\epsilon}\big|{}^{1+\kappa}\big]^{1/(1+\kappa)}$
in Lemma \ref{lem:C_epsilon-moment-D_epsilon} grows at most as a
polynomial of $d$, say of power $\delta$. (A\ref{hyp:polynomialdependence})
implies the existence of $r>0$ such that $\sigma_{\ell}^{2}(d)\geq Cd^{-r}$
and Theorem \ref{thm:CLT-variance-convergence} implies the existence
of $a_{0},d_{0}>0$ such that for any $a\geq a_{0}$ and $d\geq d_{0}$
\begin{equation}
\sigma_{\ell}^{2}(d)+\upsilon\big(\epsilon(d)\big)-\sigma_{\ell}^{2}(d)\geq\sigma_{\ell}^{2}(d)/2,\label{eq:lowerboundv_epsilon}
\end{equation}
providing us with an upper bound on $\upsilon^{-1}\big(\epsilon(d)\big)$.
Further, again from Theorem \ref{thm:CLT-variance-convergence} we
can choose $b$ sufficiently large (and hence $a$ sufficiently large)
such that the term
\[
\big|\upsilon\big(\epsilon(d)\big)-\sigma_{\ell}^{2}(d)\big|\mathbb{E}\big[\big|D_{\epsilon}\big|{}^{1+\kappa}\big]^{1/(1+\kappa)}\sigma_{\ell}^{-4}(d)\leq Cd^{-b}d^{\delta}d^{2r}
\]
vanishes. Therefore $\lim_{d\rightarrow0}A_{\epsilon(d)}=0$. For
$B_{\epsilon}$ we use Lemma \ref{prop:B_epsilon}, its Corollary,
the lower bound ((\ref{eq:lowerboundv_epsilon})) and Corollary \ref{cor:dimension-dependence-MSE-1}
of Theorem \ref{thm:intro_var_and_bias_bounds} to conclude that for
$a\geq a_{0}$ sufficiently large $\lim_{d\rightarrow0}B_{\epsilon(d)}=0$.
Finally $\lim_{d\rightarrow0}C_{\epsilon(d)}=0$ follows from Lemma
\ref{lem:C_epsilon-moment-D_epsilon} and its Corollary \ref{cor:C-epsilon-bound},
since we have assumed $\kappa>c-1$ in order to cover the scenario
$\ell=0$.
\end{proof}
Let 
\[
D_{\epsilon}:=\epsilon^{-1}h^{2}\sum_{k=0}^{n-1}\mathbb{E}\big[\xi_{k,\epsilon}^{2}|\mathcal{F}_{(k-1)h}\big],
\]
where $\xi_{k,\epsilon}$is as in Lemma (\ref{lem:Poisson-and-martingale}).
\begin{prop}
\label{prop:CLTforMartingale}For any $\kappa>0$ that there exists
a finite $\mathscr{C}{}_{\kappa}>0$, dependent on $\kappa$ only,
such that 
\[
\sup_{w\in\mathbb{R}}\big|\mathbb{P}\big[M_{\epsilon}\leq w\big]-\Phi(w)\big|\leq\mathscr{C}{}_{\kappa}\Bigl\{\Bigl(A_{\epsilon}+B_{\epsilon}\Bigr)^{1+\kappa}+C_{\epsilon}\Bigr\}^{1/(3+2\kappa)},
\]
where
\begin{align*}
A_{\epsilon} & :=\big|\upsilon(\epsilon)-\sigma_{\ell}^{2}\big|\Bigl[1+\mathbb{E}\big[\big|D_{\epsilon}\big|{}^{1+\kappa}\big]^{1/(1+\kappa)}/\upsilon(\epsilon)\Bigr]/\sigma_{\ell}^{2},\\
B_{\epsilon} & :=\mathbb{E}\Bigl[\big|D_{\epsilon}-\upsilon(\epsilon)\big|{}^{1+\kappa}\Bigl]^{1/(1+\kappa)}/\sigma_{\ell}^{2},\\
C_{\epsilon} & :=(\epsilon^{-1}h^{2}/\upsilon(\epsilon))^{(1+\kappa)}\sum_{i=0}^{n-1}\mathbb{E}\Bigl[\big|\xi_{i,\epsilon}\big|^{2(1+\kappa)}\Bigl].
\end{align*}
\end{prop}

\begin{proof}
Let $\Delta_{\epsilon}:=\sup_{w\in\mathbb{R}}\big|\mathbb{P}\big[M_{\epsilon}\leq w\big]-\Phi(w)\big|$.
From \cite[Theorem 1]{haeusler1988rate} we have
\[
\Delta_{\epsilon}\leq\mathscr{C}{}_{\kappa}\Bigl\{\mathbb{E}\Bigl[\big|D_{\epsilon}/\upsilon(\epsilon)-1\big|{}^{1+\kappa}\Bigl]+(\epsilon^{-1}h^{2}/\upsilon(\epsilon))^{(1+\kappa)}\sum_{i=0}^{n-1}\mathbb{E}\Bigl[\big|\xi_{i,\epsilon}\big|^{2(1+\kappa)}\Bigl]\Bigr\}{}^{1/(3+2\kappa)}.
\]
We upper bound the first term between braces using Minkowski's inequality
\begin{align*}
\mathbb{E}\Bigl[\big|D_{\epsilon}/\upsilon(\epsilon)-1\big|{}^{1+\kappa}\Bigl]^{1/(1+\kappa)} & \leq\mathbb{E}\Bigl[\big|D_{\epsilon}/\sigma_{\ell}^{2}-1\big|{}^{1+\kappa}\Bigl]^{1/(1+\kappa)}+\mathbb{E}\Bigl[\big|D_{\epsilon}\left(\upsilon^{-1}(\epsilon)-\sigma_{\ell}^{-2}\right)\big|{}^{1+\kappa}\Bigl]^{1/(1+\kappa)}\\
 & \leq\mathbb{E}\Bigl[\big|D_{\epsilon}-\sigma_{\ell}^{2}\big|{}^{1+\kappa}\Bigl]^{1/(1+\kappa)}/\sigma_{\ell}^{2}+\big|\upsilon^{-1}(\epsilon)-\sigma_{\ell}^{-2}\big|\mathbb{E}\Bigl[\big|D_{\epsilon}\big|{}^{1+\kappa}\Bigl]^{1/(1+\kappa)},
\end{align*}
and further
\[
\mathbb{E}\Bigl[\big|D_{\epsilon}-\sigma_{\ell}^{2}\big|{}^{1+\kappa}\Bigl]^{1/(1+\kappa)}\leq\mathbb{E}\Bigl[\big|D_{\epsilon}-\upsilon(\epsilon)\big|{}^{1+\kappa}\Bigl]^{1/(1+\kappa)}+\big|\upsilon(\epsilon)-\sigma_{\ell}^{2}\big|,
\]
from which we conclude.
\end{proof}
We need to find explicit upper bounds for the three terms above. In
the next two propositions we will make use of the following alternative
expression for $D_{\epsilon}$
\begin{align*}
D_{\epsilon} & =\epsilon^{-1}h^{2}\sum_{k=1}^{n-1}\mathbb{E}\bigl[\gamma_{k,\epsilon}^{2}(X_{kh}^{\epsilon})-\bigl(P_{(k-1)h,kh}^{\epsilon}\gamma_{k,\epsilon}(X_{(k-1)h}^{\epsilon})\bigr)^{2}\mid\mathcal{F}_{(k-1)h}\bigr]\\
 & =\epsilon^{-1}h^{2}\left\{ P_{(n-2)h,(n-1)h}^{\epsilon}\gamma_{n-1,\epsilon}^{2}(X_{(n-2)h}^{\epsilon})-\big[P_{0,h}^{\epsilon}\gamma_{1,\epsilon}(X_{0}^{\epsilon})\big]^{2}\right.\\
 & \hspace{2cm}+\left.\sum_{k=1}^{n-2}\mathbb{E}\bigl[\gamma_{k,\epsilon}^{2}(X_{kh}^{\epsilon})-\bigl(P_{kh,(k+1)h}^{\epsilon}\gamma_{k+1,\epsilon}(X_{kh}^{\epsilon})\bigr)^{2}\mid\mathcal{F}_{(k-1)h}\bigr]\right\} \\
 & =\epsilon^{-1}h^{2}\left\{ P_{(n-2)h,(n-1)h}^{\epsilon}\bar{f}_{n-1,\epsilon}^{2}(X_{(n-2)h}^{\epsilon})-\big[P_{0,h}^{\epsilon}\gamma_{1,\epsilon}(X_{0}^{\epsilon})\big]^{2}\right\} +\tilde{D}_{\epsilon}.
\end{align*}
where
\[
\tilde{D}_{\epsilon}:=\epsilon^{-1}h^{2}\sum_{k=1}^{n-2}\mathbb{E}\bigl[\bar{f}_{kh,\epsilon}(X_{kh}^{\epsilon})\big(\gamma_{k,\epsilon}(X_{kh}^{\epsilon})+P_{kh,(k+1)h}^{\epsilon}\gamma_{k+1,\epsilon}(X_{kh}^{\epsilon})\big)\mid\mathcal{F}_{(k-1)h}\bigr].
\]

The proof of the following two lemmata can be found in Subsection
\ref{subsec:Proofs-Quantitative-bound-CLT-variance}.
\begin{lem}
\label{prop:B_epsilon}For any $\kappa>1$, $r>(1+\kappa)/2$, $\gimel>1$
and $K,h$ and $\epsilon^{-1}$ such that $\gimel^{-1}<1-Kh\epsilon^{-1}/2$,
then with $m:=\big((1+\kappa)r-2\big)/(r-1)$ we have
\begin{align*}
\|D_{\epsilon}-\upsilon(\epsilon)\|_{L_{1+\kappa}} & \leq C\big(\|\tilde{D}_{\epsilon}-\mathbb{E}\big(\tilde{D}_{\epsilon}\big)\|_{L_{2}}\big)^{1/[(1+\kappa)r]}\left(\alpha_{2pm}^{1/m}\big(\mu_{0}\bar{V}^{(2pm)}\big)^{1/m}+\alpha_{2p}\mu_{0}\bar{V}^{(2p)}\right)^{m/(1+\kappa)}\\
 & \hspace{3cm}\times\left(\alpha_{p}\alpha_{2p+1/2}\frac{\gimel[1+\alpha_{p}\mu_{0}\bar{V}^{(p)}]\|f\|_{p}\|\nabla f\|_{p}}{K}\pi_{0}\bar{V}^{(p+1/2)}\right)^{m/(1+\kappa)}\\
 & \hspace{1cm}+C\epsilon\left(\alpha_{p}\frac{\gimel\|\nabla f\|_{p}}{K}\pi_{0}\bar{V}^{(p+1/2)}\right)^{2}\cdot\left(\alpha_{(1+\kappa)(2p+1)}\mu_{0}\bar{V}^{([1+\kappa][2p+1])}\right)^{1/(1+\kappa)}\\
 & \hspace{1cm}+C\epsilon^{-1}h^{2}\alpha_{2p}[1+\alpha_{p}\mu_{0}\bar{V}^{(p)}]^{2}\|f\|_{p}^{2}\left(\alpha_{2p(1+\kappa)}\mu_{0}\bar{V}^{(2p[1+\kappa])}\right)^{1/(1+\kappa)}.
\end{align*}
\end{lem}

\begin{cor}
From Theorem \ref{thm:intro_var_and_bias_bounds} we can conclude
that under (A\ref{hyp:polynomialdependence}) and (A\ref{hyp:dependence-epsilon-h-on-d}),
for any $\kappa>1$, there exist $r_{1},r_{2}>0$ such that $\|D_{\epsilon(d)}-\upsilon(\epsilon(d))\|_{L_{1+\kappa}}\leq Cd^{r_{1}}\epsilon^{r_{2}}(d)$.
\end{cor}

\begin{lem}
\label{lem:C_epsilon-moment-D_epsilon}For any $\kappa>0$ there exist
$C$ dependent on $\kappa$ only, such that for any $\gimel>1$ and
$K,\epsilon,h>0$ such that $\gimel^{-1}\leq1-Kh\epsilon^{-1}/2$
and $\ell\geq0$, then 
\[
C_{\epsilon}\leq C\upsilon(\epsilon){}^{-(1+\kappa)}(\epsilon h^{-1+\kappa/(1+\kappa)})^{1+\kappa}\Bigl\{\alpha_{p}\frac{\gimel\|\nabla f\|_{p}}{K}\mu_{0}\bar{V}^{(p+1/2)}\Bigr\}^{2(1+\kappa)}\cdot\alpha_{2(1+\kappa)(p+1/2)}\mu_{0}\bar{V}^{(2[1+\kappa][p+1/2])},
\]
and
\begin{align*}
\mathbb{E}\Bigl[\big|D_{\epsilon}\big|{}^{1+\kappa}\Bigl]^{1/(1+\kappa)} & \leq C\alpha_{p}\alpha_{2p+1/2}\alpha_{(1+\kappa)(2p+1/2)}\frac{\gimel\|f\|_{p}\|\nabla f\|_{p}[1+\alpha_{p}\mu_{0}\bar{V}^{(p)}]}{K}\mu_{0}\bar{V}^{(p+1/2)}\cdot\bigl\{\mu_{0}\bar{V}^{([1+\kappa][2p+1/2])}\bigr\}{}^{1/(1+\kappa)}.\\
 & \hspace{1cm}+C\epsilon\left(\alpha_{p}\frac{\gimel\|\nabla f\|_{p}}{K}\mu_{0}\bar{V}^{(p+1/2)}\right)^{2}\cdot\left(\alpha_{(1+\kappa)(2p+1)}\mu_{0}\bar{V}^{([1+\kappa][2p+1])}\right)^{1/(1+\kappa)}\\
 & \hspace{1cm}+C\epsilon^{-1}h^{2}\alpha_{2p}\|f\|_{p}^{2}\left(\alpha_{2p(1+\kappa)}\mu_{0}\bar{V}^{(2p[1+\kappa])}\right)^{1/(1+\kappa)}.
\end{align*}
\end{lem}

\begin{cor}
\label{cor:C-epsilon-bound}With $h(\epsilon)=C\epsilon^{\iota}$
where $\iota\geq1$
\[
C_{\epsilon}\leq C\upsilon(\epsilon){}^{-(1+\kappa)}\epsilon^{1+\kappa-\iota}\Bigl\{\alpha_{p}\frac{\gimel\|\nabla f\|_{p}}{K}\mu_{0}\bar{V}^{(p+1/2)}\Bigr\}^{2(1+\kappa)}\cdot\alpha_{2(1+\kappa)(p+1/2)}\mu_{0}\bar{V}^{(2[1+\kappa][p+1/2])}.
\]
\end{cor}

\subsection{\label{subsec:Quantitative-CLT-constants}Quantitative bound on the
convergence of the CLT constants}

For $\epsilon>0$, and $x\in\mathbb{R}^{d}$ we define for $k\in\{0,\ldots,n-1\}$
\[
\eta_{k,\epsilon}(x):=\mathbb{E}\left[\sum_{i=0}^{n-1}f_{kh}(Y_{ih}^{s,\epsilon})\mid Y_{0}^{s}=x\right]=\sum_{i=0}^{n-1}Q_{ih\epsilon^{-1}}^{kh}f_{kh}(x)
\]
and for $s\in[0,1]$
\[
g_{s}(x):=\begin{cases}
\ell\sum_{k=0}^{\infty}Q_{k\ell}^{s}f_{s}(x) & \text{if}\;\ell=\epsilon^{-1}h>0\\
\int_{0}^{\infty}Q_{t}^{s}f_{s}(x){\rm d}t & \text{if}\;\ell=0
\end{cases}.
\]
Note that it is not difficult to show that with our assumptions, for
$\ell\geq0$ and $s\in[0,1]$,
\begin{equation}
\varsigma_{\ell}(s)=2\mathbb{E}\Big[f_{s}(Y_{0}^{s})g_{s}(Y_{0}^{s})\Big]-\ell{\rm var}\big(f_{s}(Y_{0}^{s})\big).\label{eq:homogeneousAICwithPoisson}
\end{equation}
Before presenting our results, we discuss a couple of presentational
points. The term $1/\big[1-\exp(-Kh\epsilon^{-1})\big]$ appears repeatedly
in a number of upper bounds. This term will not pose any problem whenever
$K(d)h(d)\epsilon^{-1}(d)\geq z$, for say $d\geq d_{0}$ and some
$z>0$. Our statements therefore focus on the more ``difficult''
scenario where $\limsup_{d\rightarrow\infty}K(d)h(d)\epsilon^{-1}(d)=0$,
but one should bear in mind that similar conclusions can be drawn
in the former ``easier'' scenario. We have moved the proofs of the
lemmata supporting Theorem \ref{thm:CLT-variance-convergence} to
Subsection \ref{subsec:Proofs-Quantitative-CLT-constants} in order
to focus on the main important steps of the proof.
\begin{thm}
\label{thm:CLT-variance-convergence}Assume (A\ref{hyp:basic_hyp_on_U}-\ref{hyp:U_time_cont})
and (A\ref{hyp:polynomialdependence}). Then, with the following choices
\begin{enumerate}
\item for $\gimel>1$, any $\ell>0$ and $d_{0}\in\mathbb{N}$ such that
$\gimel^{-1}\leq1-K(d)\ell/2$ for $d\geq d_{0}$ we set $h(d):=\ell\epsilon(d)$,
\item for $\ell=0$ we set $h(d)=C\epsilon^{c}(d)$ for some $c>1$,
\end{enumerate}
for any $b>0$ there exists $a_{0}>0$ such that for any $a\geq a_{0}$
and $\epsilon(d)=Cd^{-a}$ we have
\[
\limsup_{d\rightarrow\infty}d^{b}\big|\upsilon_{d}\big(\epsilon(d)\big)-\sigma_{\ell}^{2}(d)\big|<\infty.
\]

\end{thm}

\begin{cor}
\label{cor:CLT-approximation-replace-variance-with-asymp-variance}With
Lemma \ref{prop:decompositionCLT} in mind, we have
\begin{align*}
\mathbb{P}\big[\big|\upsilon^{1/2}\big(\epsilon(d)\big)/\sigma_{\ell}(d)-1\big|>\varepsilon_{2}(d)\big]= & \mathbb{I}\big\{\big|\upsilon^{1/2}(\epsilon)-\sigma_{\ell}(d)\big|>\sigma_{\ell}(d)\varepsilon_{2}(d)\big\}\\
= & \mathbb{I}\big\{\big|\upsilon\big(\epsilon(d)\big)-\sigma_{\ell}^{2}(d)\big|>\sigma_{\ell}(d)(\upsilon^{1/2}(\epsilon)+\sigma_{\ell}(d))\varepsilon_{2}(d)\big\}\\
\leq & \mathbb{I}\big\{\big|\upsilon\big(\epsilon(d)\big)-\sigma_{\ell}^{2}(d)\big|>\sigma_{\ell}^{2}(d)\varepsilon_{2}(d)\big\}.
\end{align*}
Now say that from (A\ref{hyp:polynomialdependence}) we have $\sigma_{\ell}^{2}(d)\geq Cd^{-r_{1}}$
for some $r_{1}>0$ and choose $\varepsilon_{2}(d)=Cd^{-r_{2}}$ for
some arbitrary $r_{2}>0$. Then we can choose $b$ in Theorem \ref{thm:CLT-variance-convergence}
such that $b>r_{1}+r_{2}$ and conclude that for some $d_{0}\in\mathbb{N}$,
for $d\geq d_{0}$, $\mathbb{P}\big[\big|\upsilon^{1/2}\big(\epsilon(d)\big)/\sigma_{\ell}(d)-1\big|>\varepsilon_{2}(d)\big]=0$.
\[
\]
\end{cor}

\begin{proof}
The proof relies on the decomposition in Proposition \ref{prop:CLT-variance-convergence}
and bounding of the terms $\Upsilon_{i,\epsilon}$, $i\in\{0,\ldots,7\}$.
Bounds on $\Upsilon_{1,\epsilon}$ and $\Upsilon_{2,\epsilon}$ are
given in Lemma \ref{lem:poisson-homogeneous-approximation} and Lemma
\ref{lem:sublemma-cv-poisson-homogeneous}. Bounds on $\Upsilon_{3,\epsilon}$
and $\Upsilon_{5,\epsilon}$ are given in Lemma \ref{lem:variance-cv-ergodicity}.
Bounds on $\Upsilon_{4,\epsilon}$ and $\Upsilon_{6,\epsilon}$ are
given in Lemma \ref{lem:riemannconvergence}. Bounds on $\Upsilon_{0,\epsilon}$
and $\Upsilon_{7,\epsilon}$ are given in Lemma \ref{lem:S0andS7}.
By inspection we notice that under our assumptions, with $\iota>1/3$
in Lemma \ref{lem:sublemma-cv-poisson-homogeneous} and $\zeta\in(0,1)$
in Lemma \ref{lem:riemannconvergence}, each of this term is upperbounded
by the product of a polynomial in the quantities defined in (A\ref{hyp:polynomialdependence})
only, times a positive power of $\epsilon(d)$. Consequently there
exist $C,r_{1},r_{2}>0$, such that
\[
\max_{i\in\{0,\ldots,7\}}\big|\Upsilon_{i,\epsilon(d)}\big|\leq Cd^{r_{1}}\epsilon^{r_{2}}(d).
\]
Consequently, by choosing $a_{0}$ such that $a_{0}r_{2}>(r_{1}+b)$
we conclude that for $\epsilon(d)=Cd^{-a}$ and $a\geq a_{0}$
\[
\lim\sup_{d\rightarrow\infty}\max_{i\in\{0,\ldots,7\}}d^{b}\big|\Upsilon_{i,\epsilon(d)}\big|<\infty,
\]
and we conclude.
\end{proof}
\begin{prop}
\label{prop:CLT-variance-convergence}For any $\ell\geq0$ and $\epsilon>0$
such that $n\geq2$ one has $\upsilon(\epsilon)-\sigma_{\ell}^{2}=\sum_{i=0}^{7}\Upsilon_{i,\epsilon}$
with 
\begin{align*}
\Upsilon_{0,\epsilon}:= & -\epsilon^{-1}h^{2}\sum_{k=1}^{n-2}\pi_{kh}(\bar{f}_{kh,\epsilon})\mathbb{E}\left(2\gamma_{k,\epsilon}(X_{kh}^{\epsilon})-\bar{f}_{kh,\epsilon}(X_{kh}^{\epsilon})\right),\\
\Upsilon_{1,\epsilon}:= & 2\epsilon^{-1}h^{2}\sum_{k=1}^{n-2}\mathbb{E}\left(f_{kh}(X_{kh}^{\epsilon})\left\{ \gamma_{k,\epsilon}(X_{kh}^{\epsilon})-\eta_{k,\epsilon}(X_{kh}^{\epsilon})\right\} \right),\\
\Upsilon_{2,\epsilon}:= & 2h\sum_{k=1}^{n-2}\mathbb{E}\left(f_{kh}(X_{kh}^{\epsilon})\left\{ \epsilon^{-1}h\eta_{k,\epsilon}(X_{kh}^{\epsilon})-g_{kh}(X_{kh}^{\epsilon})\right\} \right),\\
\Upsilon_{3,\epsilon}:= & 2h\sum_{k=1}^{n-2}\mathbb{E}\left(f_{kh}(X_{kh}^{\epsilon})g_{kh}(X_{kh}^{\epsilon})\right)-\pi_{kh}\big(f_{kh}g_{kh}\big),\\
\Upsilon_{4,\epsilon}:= & 2h\Bigl\{\sum_{k=1}^{n-2}\pi_{kh}\big(f_{kh}g_{kh}\big)\Bigr\}-2\int_{0}^{1}\pi_{s}(f_{s}g_{s}){\rm d}s,\\
\Upsilon_{5,\epsilon}:= & -\epsilon^{-1}h^{2}\sum_{k=1}^{n-2}\mathbb{E}\big(f_{kh}(X_{kh}^{\epsilon})\bar{f}_{kh,\epsilon}(X_{kh}^{\epsilon})\big)-{\rm var}_{\pi_{kh}}\big(f_{kh}\big),\\
\Upsilon_{6,\epsilon}:= & -\epsilon^{-1}h^{2}\Bigl\{\sum_{k=1}^{n-2}{\rm var}_{\pi_{kh}}\big(f_{kh}\big)\Bigr\}+\ell\int_{0}^{1}{\rm var}_{\pi_{s}}\big(f_{s}\big){\rm d}s,\\
\Upsilon_{7,\epsilon}:= & \epsilon^{-1}h^{2}\mathbb{E}\left(\bar{f}_{(n-1)h,\epsilon}^{2}(X_{(n-1)h}^{\epsilon})-\big[P_{0,h}\gamma_{1,\epsilon}(X_{0}^{\epsilon})\big]^{2}\right).
\end{align*}
\end{prop}

\begin{proof}
For notational simplicity we drop $\epsilon$ from $P_{s,t}^{\epsilon}$
here. For $n\geq2$, noting that $\xi_{0,\epsilon}=0$,
\begin{align*}
\upsilon(\epsilon) & =\epsilon^{-1}h^{2}\sum_{k=1}^{n-1}\mathbb{E}\left[\gamma_{k,\epsilon}^{2}(X_{kh}^{\epsilon})-\big[P_{(k-1)h,kh}\gamma_{k,\epsilon}(X_{(k-1)h}^{\epsilon})\big]^{2}\right]\\
 & =\epsilon^{-1}h^{2}\mathbb{E}\left[\gamma_{n-1,\epsilon}^{2}(X_{(n-1)h}^{\epsilon})-\big[P_{0,h}\gamma_{1,\epsilon}(X_{0}^{\epsilon})\big]^{2}+\sum_{k=1}^{n-2}\bar{f}_{kh,\epsilon}(X_{kh}^{\epsilon})\big\{\gamma_{k,\epsilon}(X_{kh}^{\epsilon})+P_{kh,(k+1)h}\gamma_{k+1,\epsilon}(X_{kh}^{\epsilon})\big\}\right]\\
 & =\epsilon^{-1}h^{2}\mathbb{E}\left[\bar{f}_{(n-1)h,\epsilon}^{2}(X_{(n-1)h}^{\epsilon})-\big[P_{0,h}\gamma_{1,\epsilon}(X_{0}^{\epsilon})\big]^{2}\right]+\epsilon^{-1}h^{2}\sum_{k=1}^{n-2}\mathbb{E}\left[\bar{f}_{kh,\epsilon}(X_{kh}^{\epsilon})\big\{2\gamma_{k,\epsilon}(X_{kh}^{\epsilon})-\bar{f}_{kh,\epsilon}(X_{kh}^{\epsilon})\big\}\right],
\end{align*}
where the second line follows from the fact that with $W_{0,\epsilon}=\gamma_{n-1,\epsilon}^{2}(X_{(n-1)h}^{\epsilon})-\big[P_{(n-2)h,(n-1)h}\gamma_{n-1,\epsilon}(X_{(n-2)\epsilon}^{\epsilon})\big]^{2}$,
$W_{1,\epsilon}=\sum_{k=1}^{n-2}\gamma_{k,\epsilon}^{2}(X_{kh}^{\epsilon})-\big[P_{kh,(k+1)h}\gamma_{k,\epsilon}(X_{kh}^{\epsilon})\big]^{2}$
and 
\begin{align*}
W_{2,\epsilon} & =\sum_{k=1}^{n-2}\big[P_{kh,(k+1)h}\gamma_{k,\epsilon}(X_{kh}^{\epsilon})\big]^{2}-\big[P_{(k-1)h,kh}\gamma_{k,\epsilon}(X_{(k-1)h}^{\epsilon})\big]^{2}\\
 & =\big[P_{(n-2)h,(n-1)h}\gamma_{n-2,\epsilon}(X_{(n-2)h}^{\epsilon})\big]^{2}-\big[P_{0,h}\gamma_{1,\epsilon}(X_{0}^{\epsilon})\big]^{2},
\end{align*}
we have
\begin{align*}
\sum_{k=1}^{n-1}\gamma_{k,\epsilon}^{2}(X_{kh}^{\epsilon})-\big[P_{(k-1)h,kh}\gamma_{k,\epsilon}(X_{(k-1)h}^{\epsilon})\big]^{2} & =W_{0,\epsilon}+W_{1,\epsilon}+W_{2,\epsilon}\\
 & =\gamma_{n-1,\epsilon}^{2}(X_{(n-1)h}^{\epsilon})-\big[P_{0,h}\gamma_{1,\epsilon}(X_{0}^{\epsilon})\big]^{2}+W_{1,\epsilon},
\end{align*}
and the fact that by definition $\bar{f}_{kh,\epsilon}(x)=\gamma_{k,\epsilon}(x)-P_{kh,(k+1)h}\gamma_{k+1,\epsilon}(x)$,
which is also used on the third line.
\end{proof}
In order to control $\Upsilon_{1,\epsilon}$ and $\Upsilon_{2,\epsilon}$
we show that $\eta_{k,\epsilon}$ approximates $\gamma_{k,\epsilon}$
in Lemma \ref{lem:sublemma-cv-poisson-homogeneous} and that $\eta_{k,\epsilon}$
can be approximated by $g_{kh}$ in Lemma \ref{lem:poisson-homogeneous-approximation}.
\begin{lem}
\label{lem:poisson-homogeneous-approximation}Let $p\geq1$. Assume
that $\mu_{0}$ satisfies (\ref{eq:poincare_assume}) for some $K_{\mu_{0}}>0$
and that $h(\epsilon)\epsilon^{-1}=O(1)$. Then there exists $C>0$
such that for any $f\in C_{1,2}^{p}\big([0,1]\times\mathbb{R}^{d}\big)$
\begin{enumerate}
\item for $\ell=0$ and any $\gimel>1$, defining
\begin{align*}
A_{1}: & =C\alpha_{2p+1/2}\left\{ L\tilde{\alpha}_{p+1/2}+\tilde{\alpha}_{p}\right\} \cdot\vvvert f\vvvert_{p}^{2}\cdot\mu_{0}\big(\bar{V}^{(2p+1/2)}\big).\\
A_{2}: & =C\alpha_{2p}K^{-1}\bigl\{1+\gimel\bigr\}\bigl\{\tilde{\alpha}_{p}\alpha_{p+1/2}\sup_{s\in[0,1]}\pi_{s}\bar{V}^{(p+1/2)}+\bigl(\tilde{\alpha}_{2p}\alpha_{2p}\big[K^{-1}+K_{\mu_{0}}^{-1}\big]\bigr)^{1/2}\bigr\}\vvvert f\vvvert_{p}^{2}\cdot\mu_{0}\big(\bar{V}^{(2p)}\big)^{2},
\end{align*}
then for any $\epsilon>0$ satisfying $1/\gimel\leq1-Kh(\epsilon)\epsilon^{-1}/2$
\[
\big|\Upsilon_{2,\epsilon}\big|\leq[A_{2}+A_{1}\big(\lceil-\log(h(\epsilon)\epsilon^{-1})\rceil/K\big)^{2}]h(\epsilon)\epsilon^{-1},
\]
\item for $\ell>0$ and $\epsilon>0$
\[
\big|\Upsilon_{2,\epsilon}\big|\leq C\ell^{2}\mu_{0}\big(\bar{V}^{(p)}\big)^{2}\vvvert f\vvvert_{p}^{2}\Bigl\{\tilde{\alpha}_{p}\alpha_{p+1/2}\sup_{s\in[0,1]}\pi_{s}\bar{V}^{(p+1/2)}+\bigl(\tilde{\alpha}_{2p}\alpha_{2p}\big[K^{-1}+K_{\mu_{0}}^{-1}\big]\bigr)^{1/2}\Bigr\}\frac{\exp\Big(-Kn(\epsilon)\ell\Big)}{1-\exp\big(-K\ell\big)}.
\]
\end{enumerate}
\end{lem}

\begin{lem}
\label{lem:sublemma-cv-poisson-homogeneous}Let $p\geq1$, $f\in C_{1,2}^{p}([0,1]\times\mathbb{R}^{d})$,
$\iota\in\big(0,1\big)$, define for any $\epsilon>0$ and $k\in\{0,\ldots,n-1\}$
$\tau_{k,\epsilon}:=(kh+\daleth h^{\iota})\wedge1$ for some $\daleth>0$,
and define for $k\in\{0,\ldots,n-1\}$ and $x\in\mathbb{R}^{d}$
\begin{align*}
T_{1,k,\epsilon} & :=\sum_{i=k}^{\lfloor\tau_{k,\epsilon}h^{-1}\rfloor-1}P_{kh,ih}^{\epsilon}\bar{f}_{ih,\epsilon}\big(x\big)-Q_{(i-k)h}^{kh,\epsilon}\bar{f}_{ih,\epsilon}(x), & T_{2,k,\epsilon}:=\sum_{i=k}^{\lfloor\tau_{k,\epsilon}h^{-1}\rfloor-1}Q_{(i-k)h}^{kh,\epsilon}f_{ih}(x)-Q_{(i-k)h}^{kh,\epsilon}f_{kh}(x),\\
T_{3,k,\epsilon} & :=-\sum_{i=k}^{\lfloor\tau_{k,\epsilon}h^{-1}\rfloor-1}\mu_{ih}^{\epsilon}f_{ih,\epsilon} & T_{4,k,\epsilon}:=\sum_{i=\lfloor\tau_{k,\epsilon}h^{-1}\rfloor}^{n-1}P_{kh,ih}^{\epsilon}\bar{f}_{ih,\epsilon}\big(x\big)-Q_{(i-k)h}^{kh,\epsilon}f_{kh}\big(x\big),
\end{align*}
with the standard conventions that that $T_{1,k,\epsilon}=T_{2,k,\epsilon}=T_{3,k,\epsilon}=0$
when $\lfloor\tau_{k,\epsilon}h^{-1}\rfloor=k$ and $T_{4,k,\epsilon}=0$
when $\lfloor\tau_{k,\epsilon}h^{-1}\rfloor=n$. Then 
\[
\big|\Upsilon_{1,\epsilon}\big|\leq2\epsilon^{-1}h^{2}\sum_{k=1}^{n-1}\Big|\mathbb{E}\left[f_{kh}(X_{kh}^{\epsilon})T_{3,k,\epsilon}\right]\Big|+2\epsilon^{-1}h\sum_{i=1,i\neq3}^{4}\max_{k\in\{0,\ldots,n-1\}}\Bigl|\mathbb{E}\left[f_{kh}(X_{kh}^{\epsilon})T_{i,k,\epsilon}\right]\Bigl|,
\]
and there exists $C>0$ such that for any $\epsilon>0$ and $\ell\geq0$,
\begin{align*}
\max_{k\in\{0,\ldots,n-1\}}\Bigl|\mathbb{E}\left[f_{kh}(X_{kh}^{\epsilon})T_{1,k,\epsilon}\right]\Bigl| & \leq C\daleth^{3}\alpha_{p}\tilde{\alpha}_{p+1/2}\alpha_{2p+1/2}M\cdot\vvvert f\vvvert_{p}^{2}\cdot\sup_{s\in[0,1]}\bar{V}(x_{s}^{\star})^{1/2}\cdot\mu_{0}\left(\bar{V}^{(2p+1/2)}\right)\cdot\epsilon^{-1}h^{3\iota},\\
\max_{k\in\{0,\ldots,n-1\}}\Bigl|\mathbb{E}\left[f_{kh}(X_{kh}^{\epsilon})T_{2,k,\epsilon}\right]\Bigl| & \leq C\daleth^{2}\tilde{\alpha}_{p}\alpha_{2p}\vvvert f\vvvert_{p}^{2}\mu_{0}\Bigl(\bar{V}^{(2p)}\Bigr)h^{2\iota},\\
2\epsilon^{-1}h^{2}\sum_{k=0}^{n-1}\Bigl|\mathbb{E}\left[f_{kh}(X_{kh}^{\epsilon})T_{3,k,\epsilon}\right]\Bigl| & \leq C\daleth\left\{ \vvvert f\vvvert_{p}\sup_{s\in[0,1]}\pi_{s}\bar{V}^{(2[p\vee p_{0}]+1/2)}\Big[\frac{\|\nabla\phi\|_{p_{0}}}{K^{2}}+\alpha_{p}\mu_{0}\bar{V}^{(p+1/2)}\Big]\right\} ^{2}\\
 & \hspace{5cm}\times\Big\{-h\ln(\epsilon)/K+\epsilon^{-1}h^{2}+\epsilon h^{\iota}\Big\}.
\end{align*}
Define
\begin{align*}
A & :=\gimel\alpha_{2p}\vvvert f\vvvert_{p}^{2}\mu_{0}V^{(2p)}\left[K^{-1}+K_{\mu_{0}}^{-1}\right]{}^{1/2}\\
 & \hspace{2cm}+\frac{\gimel}{K}\alpha_{2p}\mu_{0}\big(\bar{V}^{(2p)}\big)^{2}\vvvert f\vvvert_{p}^{2}\Bigl\{\tilde{\alpha}_{p}\alpha_{p+1/2}\sup_{s\in[0,1]}\pi_{s}\bar{V}^{(p+1/2)}+\bigl(\tilde{\alpha}_{2p}\alpha_{2p}\big[K^{-1}+K_{\mu_{0}}^{-1}\big]\bigr)^{1/2}\Bigr\}
\end{align*}
then there exists $C>0$ such that for any $\gimel>1$ and $\gimel^{-1}<1-Kh\epsilon^{-1}/2$
\[
\max_{k\in\{0,\ldots,n-1\}}\Bigl|\mathbb{E}\left[f_{kh}(X_{kh}^{\epsilon})T_{4,k,\epsilon}\right]\Bigl|\leq C\cdot A\exp\Big(-K[\daleth h^{\iota-1}-1]h\epsilon^{-1}\Big)\cdot\big[(\epsilon h^{-1})\vee1\big].
\]
\end{lem}

\begin{lem}
\label{lem:variance-cv-ergodicity}For any $\gimel>1$ and $\epsilon,h,K>0$
such that $\gimel^{-1}<1-K\epsilon^{-1}h/2$ we have for $\ell\geq0$,
\begin{align*}
\big|\Upsilon_{3,\epsilon}\big| & \leq C\frac{\gimel\tilde{\alpha}_{p}}{K}\vvvert f\vvvert_{p}^{2}\left[\sup_{s\in[0,1]}\pi_{s}\bar{V}^{(2[(2p+1/2)\vee p_{0}]+1/2)}\right]^{3}\\
 & \times\left\{ 1+K^{-2}\|\nabla\phi\|_{p_{0}}+\mu_{0}\bar{V}^{(2p+1)}\alpha_{2p+1/2}\left(1+\frac{\gimel}{K}\right)\right\} \epsilon,
\end{align*}
and
\begin{align*}
\big|\Upsilon_{5,\epsilon}\big|\leq C\vvvert f\vvvert_{2p}^{2}\sup_{s\in[0,1]}\pi_{s}\bar{V}^{(2[(2p)\vee p_{0}]+1/2)}\Big\{ K^{-2}\|\nabla\phi\|_{p_{0}}h+\frac{\gimel\alpha_{2p}}{K}\mu_{0}\bar{V}^{(2p+1/2)}\Big\} h.
\end{align*}
\end{lem}

\begin{lem}
\label{lem:riemannconvergence}For any $\gimel>1$ and $\epsilon,h,K>0$
such that $\gimel^{-1}<1-K\epsilon^{-1}h/2$ we have,
\[
\big|\Upsilon_{4,\epsilon}\big|\leq\Upsilon_{4,\epsilon}^{(1)}+\Upsilon_{4,\epsilon}^{(2)}
\]
where, with the convention $(\ell\vee1)/\ell=1$ for $\ell=0$, for
any $\zeta\in(0,1)$, with
\[
C_{fg}:=(1+\vvvert f\vvvert_{p})^{2}\tilde{\alpha}_{p}\sup_{s\in[0,1]}\pi_{s}\bar{V}^{(p+1/2)}\left\{ \frac{\gimel}{K}+\frac{C(\gimel,\zeta)(\ell\vee1)}{(1\wedge K)\ell}\Bigl(2+\tilde{\alpha}_{p}\frac{M}{K}\sup_{\tau\in[0,1]}\sqrt{\bar{V}(x_{\tau}^{\star})}\Bigr)\right\} ,
\]
\begin{align*}
\Upsilon_{4,\epsilon}^{(1)}:= & Ch^{\zeta}\tilde{\alpha}_{2p+1/2}\big[C_{fg}\vee\big(\gimel\frac{\tilde{\alpha}_{p}}{K}\vvvert f\vvvert_{p}^{2}\sup_{s\in[0,1]}\pi_{s}\bar{V}^{(p+1/2)}\big)\big]\left[1+\tilde{\alpha}_{2p+1/2}\frac{M}{K}\sup_{s\in[0,1]}\sqrt{\bar{V}(x_{s}^{\star})}\right]\\
\Upsilon_{4,\epsilon}^{(2)}:= & C\gimel\|f\|_{p}^{2}\frac{\tilde{\alpha}_{p}}{K}\sup_{s\in[0,1]}\pi_{s}\bar{V}^{(p+1/2)}\cdot\sup_{s\in[0,1]}\pi_{s}\bigl(\bar{V}^{(2p+1/2)}\bigr)\epsilon
\end{align*}
and for $\ell>0$
\begin{align*}
\big|\Upsilon_{6,\epsilon}\big| & \leq C\ell h\vvvert f\vvvert_{p}^{2}(\tilde{\alpha}_{2p}\vee\tilde{\alpha}_{p})\sup_{s\in[0,1]}\pi_{s}\bar{V}^{(p)}\left[1+(\tilde{\alpha}_{2p}\vee\tilde{\alpha}_{p})\frac{M}{K}\sup_{s\in[0,1]}\sqrt{\bar{V}(x_{s}^{\star})}\right]
\end{align*}
while for $\ell=0$ we have
\[
\big|\Upsilon_{6,\epsilon}\big|\leq\sup_{s\in[0,1]}{\rm var}_{\pi_{s}}\big(f_{s}\big)\cdot h\epsilon^{-1}.
\]
\end{lem}

\begin{lem}
\label{lem:S0andS7}There exists $C>0$ such that for any $\gimel>1$
and $\epsilon,h>0$ and $K>0$ satisfying $\gimel^{-1}<1-Kh\epsilon^{-1}/2$
$\epsilon>0$ we have
\begin{align*}
\big|\Upsilon_{0,\epsilon}\big| & \leq C\alpha_{p}\alpha_{p+1/2}\frac{\gimel\|\nabla f\|_{p}^{2}}{K}.\left\{ \mu_{0}\bar{V}^{(p+1/2)}(x)\right\} ^{2}\sup_{s\in[0,1]}\pi_{s}\bar{V}^{(2[p\vee p_{0}]+1/2)}\\
 & \hspace{5cm}\times\Big\{ K^{-2}\|\nabla\phi\|_{p_{0}}+\alpha_{p}\mu_{0}\bar{V}^{(p+1/2)}\frac{\exp\big(-K\epsilon^{-1}h\big)}{1-\exp\big(-K\epsilon^{-1}h\big)}\epsilon^{-1}h\Big\}\epsilon\\
\big|\Upsilon_{7,\epsilon}\big| & \leq C\Big\{\epsilon^{-2}h^{2}+\Big(\frac{\gimel}{K}\Big)^{2}\Big\}\|\nabla f\|^{2}\alpha_{2p}\big[K^{-1}+K_{\mu_{0}}\big]\mu_{0}\bar{V}^{(2p)}\epsilon.
\end{align*}
\end{lem}

%% file: child-intro-proofs.tex
\appendix

\section{Proofs for section \ref{sec:Introduction}\label{sec:Proofs-for-intro}}
\begin{proof}
[Proof of Theorem \ref{thm:intro_var_and_bias_bounds}] Write $P_{0,t}f(x)=\mathbb{E}[f(X_{t}^{\epsilon})|X_{0}=x]$
so that $\mu_{t}^{\epsilon}=\mu_{0}P_{0,t}$. For part 1) note that
by Lemma \ref{lem:poincare_transfer} applied with $\nu=\mu_{0}$,
\begin{align*}
\mathrm{var}_{\mu_{t}^{\epsilon}}[f] & \leq\left[(1-e^{-2Kt/\epsilon})\frac{1}{K}+e^{-2Kt/\epsilon}\frac{1}{K_{0}}\right]\mu_{t}^{\epsilon}(\|\nabla f\|^{2})\\
 & \leq\frac{1}{K_{0}\wedge K}\mu_{t}^{\epsilon}(\|\nabla f\|^{2}),\qquad\forall f\in C_{2}^{p}(\mathbb{R}^{d}),
\end{align*}
and then by Cauchy-Schwartz and Lemma \ref{lem:L_2_convergence} applied
with $\kappa_{\nu}(u)=K_{0}\wedge K$, 
\begin{align*}
|\mathbb{E}[(f_{s}(X_{s}^{\epsilon})-\mu_{s}^{\epsilon}f_{s})(f_{t}(X_{t}^{\epsilon})-\mu_{t}^{\epsilon}f_{t})]| & \leq\mathrm{var}_{\mu_{s}^{\epsilon}}[f_{s}]^{1/2}\mathrm{var}_{\mu_{s}^{\epsilon}}[P_{s,t}f_{t}]^{1/2}\\
 & \leq\mathrm{var}_{\mu_{s}^{\epsilon}}[f_{s}]^{1/2}\mathrm{var}_{\mu_{t}^{\epsilon}}[f_{t}]^{1/2}e^{-(K_{0}\wedge K)(t-s)/\epsilon}.
\end{align*}
Therefore, for part 2),
\begin{align*}
\mathrm{var}[S_{\epsilon}] & =\mathbb{E}\left[\left(\int_{0}^{1}f_{t}(X_{t}^{\epsilon})-\mu_{t}^{\epsilon}f_{t}\mathrm{d}t\right)^{2}\right]\\
 & =2\mathbb{E}\left[\int_{0}^{1}\int_{s}^{1}(f_{s}(X_{s}^{\epsilon})-\mu_{s}^{\epsilon}f_{s})(f_{t}(X_{t}^{\epsilon})-\mu_{t}^{\epsilon}f_{t})\mathrm{d}t\mathrm{d}s\right]\\
 & \leq2\sup_{t}\mathrm{var}_{\mu_{t}^{\epsilon}}[f_{t}]\int_{0}^{1}\int_{s}^{1}e^{-(K_{0}\wedge K)(t-s)/\epsilon}\mathrm{d}t\mathrm{d}s\\
 & \leq2\frac{\epsilon}{K_{0}\wedge K}\sup_{t}\mathrm{var}_{\mu_{t}^{\epsilon}}[f_{t}].
\end{align*}
Similarly,
\begin{align*}
\mathrm{var}[S_{\epsilon,h}] & =\mathbb{E}\left[\left(h\sum_{k=0}^{\left\lfloor 1/h\right\rfloor -1}f_{kh}(X_{kh}^{\epsilon})-\mu_{kh}^{\epsilon}f_{kh}\right)^{2}\right]\\
 & \leq h^{2}\sum_{k=0}^{\left\lfloor 1/h\right\rfloor -1}\mathrm{var}_{\mu_{kh}^{\epsilon}}[f_{kh}]+2h^{2}\sum_{k=0}^{\left\lfloor 1/h\right\rfloor -1}\sum_{j>k}\mathrm{var}_{\mu_{kh}^{\epsilon}}[f_{kh}]^{1/2}\mathrm{var}_{\mu_{kh}^{\epsilon}}[P_{kh,jh}f_{jh}]^{1/2}\\
 & \leq\left(h+2h^{2}\sum_{k=0}^{\left\lfloor 1/h\right\rfloor -1}\sum_{j>k}e^{-(K_{0}\wedge K)(j-k)h/\epsilon}\right)\sup_{t}\mathrm{var}_{\mu_{t}^{\epsilon}}[f_{t}]\\
 & \leq h\left(1+\frac{2}{1-e^{-(K_{0}\wedge K)h/\epsilon}}\right)\sup_{t}\mathrm{var}_{\mu_{t}^{\epsilon}}[f_{t}].
\end{align*}
For the bias bounds, we have by Lemmas \ref{lem:bias_pi_0_pi_t} and
\ref{lem:bias_nu_bar_nu},
\begin{align*}
|\mathbb{E}[f_{t}(X_{t}^{\epsilon})]| & =|\mu_{0}P_{0,t}f_{t}|\\
 & \leq|\pi_{0}P_{0,t}f_{t}-\pi_{t}f_{t}|+|(\mu_{0}-\pi_{0})P_{0,t}f_{t}|\\
 & \leq\sup_{s\in[0,t]}\mathrm{var}_{\pi_{s}}[\phi_{s}]^{1/2}\mathrm{var}_{\pi_{t}}[f_{t}]^{1/2}\frac{\epsilon}{K}(1-e^{-Kt/\epsilon})\\
 & +\alpha_{p}\|\nabla f_{t}\|_{p}W^{(p)}(\mu_{0},\pi_{0})e^{-Kt/\epsilon}.
\end{align*}
Therefore
\begin{align*}
|\mathbb{E}[S_{\epsilon}]| & \leq\sup_{t}\mathrm{var}_{\pi_{t}}[\phi_{t}]^{1/2}\sup_{t}\mathrm{var}_{\pi_{t}}[f_{t}]^{1/2}\frac{\epsilon}{K}+\alpha_{p}W^{(p)}(\mu_{0},\pi_{0})\int_{0}^{t}e^{-Kt/\epsilon}\mathrm{d}t\sup_{t}\|\nabla f_{t}\|_{p}\\
 & =\sup_{t}\mathrm{var}_{\pi_{t}}[\phi_{t}]^{1/2}\mathrm{var}_{\pi_{t}}[f_{t}]^{1/2}\frac{\epsilon}{K}+\alpha_{p}W^{(p)}(\mu_{0},\pi_{0})\frac{\epsilon}{K}\sup_{t}\|\nabla f_{t}\|_{p},
\end{align*}
 and 
\begin{align*}
|\mathbb{E}[S_{\epsilon,h}]| & \leq h\sum_{k=0}^{\left\lfloor 1/h\right\rfloor -1}|\mathbb{E}[f_{kh}(X_{kh}^{\epsilon})]|\\
 & \leq\sup_{t}\mathrm{var}_{\pi_{t}}[\phi_{t}]^{1/2}\sup_{t}\mathrm{var}_{\pi_{t}}[f_{t}]^{1/2}\frac{\epsilon}{K}h\sum_{k=0}^{\left\lfloor 1/h\right\rfloor -1}\left[1-e^{-khK/\epsilon}\right]\\
 & +\alpha_{p}\sup_{t}\|\nabla f_{t}\|_{p}hW^{(p)}(\mu_{0},\pi_{0})\sum_{k=0}^{\left\lfloor 1/h\right\rfloor -1}e^{-Kkh/\epsilon}\\
 & \leq\sup_{t}\mathrm{var}_{\pi_{t}}[\phi_{t}]^{1/2}\sup_{t}\mathrm{var}_{\pi_{t}}[f_{t}]^{1/2}\frac{\epsilon}{K}+\sup_{t}\|\nabla f_{t}\|_{p}\frac{\alpha_{p}h}{1-e^{-hK/\epsilon}}W^{(p)}(\mu_{0},\pi_{0}).
\end{align*}
\end{proof}
\begin{proof}
[Proof of Corollary \ref{cor:dimension-dependence-MSE-1}] Let us
first obtain upper bounds on:
\[
\sup_{t}\mathrm{var}_{\mu_{t}^{\epsilon}}[f_{t}],\quad\sup_{t}\mathrm{var}_{\pi_{t}}[\phi_{t}]\quad\sup_{t}\mathrm{var}_{\pi_{t}}[f_{t}],
\]

By part 1) of Theorem \ref{thm:intro_var_and_bias_bounds}, Lemma
\ref{lem:drift}, Lemma \ref{lem:alpha_dimension_dependence} and
(A\ref{hyp:dimension_dependence_mse-1}),
\begin{align*}
\sup_{t}\mathrm{var}_{\mu_{t}^{\epsilon}}[f_{t}] & \leq\sup_{t}\frac{\mu_{t}^{\epsilon}(\|\nabla f_{t}\|^{2})}{K\wedge K_{0}}\\
 & \leq\frac{3}{K\wedge K_{0}}\sup_{t}\mu_{t}^{\epsilon}(\bar{V}^{(2p)})\sup_{t\in[0,1]}\|\nabla f_{t}\|_{p}^{2}\\
 & \leq\frac{3\alpha_{2p}}{K\wedge K_{0}}[1+\mu_{0}(V^{2p})]\sup_{t}\|\nabla f_{t}\|_{p}^{2}\\
 & =O(d^{q}d^{2p(q+1)}d^{q+1}d^{2q})\\
 & =O(d^{4q+2p(q+1)+1}).
\end{align*}
By Remark \ref{rem:Poincare_pi}, (A\ref{hyp:U_time_reg}), Lemma
\ref{lem:pi_V_dim_dependence} with there $p=1$, and (A\ref{hyp:dimension_dependence_mse-1}),
\[
\sup_{t}\mathrm{var}_{\pi_{t}}[\phi_{t}]\leq\frac{1}{K}\sup_{t}\pi_{t}(\|\nabla U_{t}\|^{2})\leq\frac{3L^{2}}{K}\sup_{t}\pi_{t}(\bar{V})=O(d^{3q/2}d^{q+1})=O(d^{5q/2+1}).
\]
Lastly, $\sup_{t}\mathrm{var}_{\pi_{t}}[f_{t}]$ can be similarly
controlled using Remark \ref{rem:Poincare_pi}, (A\ref{hyp:dimension_dependence_mse-1})
and Lemma \ref{lem:pi_V_dim_dependence}, to give
\[
\sup_{t\in[0,1]}\mathrm{var}_{\pi_{t}}[f_{t}]\leq\frac{1}{K}\sup_{t\in[0,1]}\pi_{t}(\bar{V}^{2p})\sup_{t\in[0,1]}\|\nabla f_{t}\|_{p}^{2}=O(d^{q}d^{p}d^{2pq+2p})=O(d^{q+p(3+2q)}).
\]

Using the above estimates, we have from the expressions in Theorem
\ref{thm:intro_var_and_bias_bounds} and Lemma \ref{lem:alpha_dimension_dependence},
\begin{align*}
\mathrm{var}[S_{\epsilon}] & =O\left(\frac{\epsilon}{K_{0}\wedge K}d^{4q+2p(q+1)+1}\right),\\
|\mathbb{E}[S_{\epsilon}]| & =O\left(\frac{\epsilon}{K}d^{5q/4+1/2}d^{q/2+p(3+2q)/2}+d^{p(q+1)}d^{q}\frac{\epsilon}{K}d^{q}\right)\\
 & =O\left(\frac{\epsilon}{K}d^{7q/4+3pq+3p/2+1/2}+\frac{\epsilon}{K}d^{2q+pq+p}\right).
\end{align*}
Similarly,
\begin{align*}
\mathrm{var}[S_{\epsilon,h}] & =O\left(h\left(1+\frac{2}{1-e^{-(K_{0}\wedge K)h/\epsilon}}\right)d^{4q+2p(q+1)+1}\right)\\
|\mathbb{E}[S_{\epsilon,h}]| & =O\left(\frac{\epsilon}{K}d^{7q/4+3pq+3p/2+1/2}+\frac{h}{1-e^{-Kh/\epsilon}}d^{2q+pq+p}\right).
\end{align*}
\end{proof}
\begin{proof}
[Proof of Lemma \ref{lem:gamma_coupling}]The first inequality is
an immediate consequence of the definition of the total variation
distance. For the second inequality, since $E$ is Polish there exists
a maximal coupling of $X,\widetilde{X}$, \cite[Ch. I, Sec. 5, p. 18]{lindvall2002lectures},
that is a probability space $(\bar{\Omega},\bar{\mathcal{F}},\mathbf{P})$
on which are defined two $(E,\mathcal{B}(E))$-valued random elements
$Z,\widetilde{Z}$ such that 
\[
\mathbf{P}[Z\in A]=\mu(A),\quad\mathbf{P}[\widetilde{Z}\in A]=\widetilde{\mu}(A),\quad A\in\mathcal{B}(E),
\]
\[
\mathbf{P}[Z\neq\widetilde{Z}]=\|\mu-\widetilde{\mu}\|_{\mathrm{tv}}.
\]
With expectation w.r.t. $\mathbf{P}$ denoted by $\mathbf{E}$, we
then have, using Holder's inequality,
\begin{eqnarray*}
\mathbb{\mathbb{E}}[|\varphi(\widetilde{X})|^{p}]^{1/p} & = & \mathbf{E}[|\varphi(\widetilde{Z})|^{p}]^{1/p}\\
 & \leq & \mathbf{E}[|\varphi(Z)|^{p}]^{1/p}+\mathbf{E}[|\varphi(\widetilde{Z})-\varphi(Z)|^{p}]^{1/p}\\
 & = & \mathbb{E}[|\varphi(X)|^{p}]^{1/p}+\mathbf{E}[\mathbb{I}\{Z\neq\widetilde{Z}\}|\varphi(\widetilde{Z})-\varphi(Z)|^{p}]^{1/p}\\
 & \leq & \mathbb{E}[|\varphi(X)|^{p}]^{1/p}+\mathbf{P}[Z\neq\widetilde{Z}]^{1/pq}\mathbf{E}[|\varphi(\widetilde{Z})-\varphi(Z)|^{pr}]^{1/pr}\\
 & \leq & \mathbb{E}[|\varphi(X)|^{p}]^{1/p}+\|\mu-\widetilde{\mu}\|_{\mathrm{tv}}^{1/pq}\left\{ \mathbb{E}[|\varphi(X)|^{pr}]^{1/pr}+\mathbb{E}[|\varphi(\widetilde{X})|^{pr}]^{1/pr}\right\} .
\end{eqnarray*}
\end{proof}
\begin{lem}
\label{lem:logit_dim_hyp_check}If (A\ref{hyp:logit_dim_depend})
holds for some given $q$, then $f_{t}$ taken to be
\begin{equation}
f_{t}(x)=-\partial_{t}U_{t}(x)+\pi_{t}(\partial_{t}U_{t}),\label{eq:logit_f}
\end{equation}
and $K,L,M$ as in (\ref{eq:logit_KLM}) satisfy
\[
\sup_{t\in[0,1]}\|\nabla f_{t}\|_{1}\vee K^{-1}\vee L^{4}\vee M^{2}\vee\sup_{t}\|x_{t}^{\star}\|^{2}\vee\sup_{t}\|\partial_{t}x_{t}^{\star}\|^{2}=O(d^{q}),
\]
and $\pi_{0}$ as in (\ref{eq:Zs_and_pis_front}) with $U_{0}$ as
in (\ref{eq:logit_U}) satisfies
\[
\pi_{0}(V)=O(d^{q+1}),
\]
as $d\to\infty$.
\end{lem}

\begin{proof}
By Lemma \ref{lem:minimizer}, (\ref{eq:logit_KLM}) and (A\ref{hyp:logit_dim_depend}),
\[
\sup_{t}\|\partial_{t}x_{t}^{\star}\|\vee\sup_{t}\|x_{t}^{\star}\|\leq\frac{M}{K}=\xi\tilde{\sigma}^{2}=O(d^{q/2}).
\]
This fact together with $K^{-1}=\tilde{\sigma}^{2}=O(d^{q/4})$ by
(A\ref{hyp:logit_dim_depend}) validates an application of Lemma \ref{lem:pi_V_dim_dependence}
with there $p=1$ to give
\[
\pi_{0}(V)=O(d^{q+1}).
\]
Once more using (A\ref{hyp:logit_dim_depend}),
\[
\sup_{t}\|\nabla f_{t}\|_{1}\leq\|y^{T}C\|+\sum_{i=1}^{m}\|c_{i}\|=\xi=O(d^{q/4}).
\]
The proof is complete since (A\ref{hyp:logit_dim_depend}) directly
implies that $L^{4}=(0.25m\lambda_{\mathrm{max}}+\tilde{\sigma}^{-2})^{4}\vee(\xi\vee\tilde{\sigma}^{-2})^{4}=O(d^{q})$
and $M=\xi=O(d^{q/4})$.
\end{proof}
\begin{proof}
[Proof of Proposition \ref{prop:logit_order}]For part 1), using Lemma
\ref{lem:the_TI_identity} and Lemma \ref{lem:gamma_coupling}, we
have
\[
\mathbb{E}[|\Delta_{\epsilon,h}|]\leq T_{1}(\epsilon,h)+T_{2}(h),
\]
 where
\begin{align}
T_{1}(\epsilon,h) & \coloneqq\mathbb{E}\left[|S_{\epsilon,h}|\right]+\|\mu^{\epsilon}-\widetilde{\mu}^{\epsilon,h}\|_{\mathrm{tv}}^{1/2}\left\{ \mathbb{E}[|S_{\epsilon,h}|^{2}]^{1/2}+\mathbb{E}[|\widetilde{S}_{\epsilon,h}|^{2}]^{1/2}\right\} ,\label{eq:logit_T1_defn}\\
T_{2}(h) & \coloneqq\left|h\sum_{k=0}^{\left\lfloor 1/h\right\rfloor -1}\pi_{kh}(\left.\partial_{t}U_{t}\right|_{t=kh})-\int_{0}^{1}\pi_{t}(\partial_{t}U_{t})\mathrm{d}t\right|,\label{eq:logit_T2_defn}
\end{align}
$S_{\epsilon,h}$ is as in (\ref{eq:MSE_thm_S_defn}) with (\ref{eq:logit_f}),
and $\widetilde{S}_{\epsilon,h}$ is defined by replacing $X_{kh}^{\epsilon}$
in $S_{\epsilon,h}$ with $\widetilde{X}_{kh}^{\epsilon}$. 

We shall estimate $T_{1}(\epsilon,h)$ using Corollary \ref{cor:dimension-dependence-MSE-1}.
To this end, note that Lemma \ref{lem:logit_dim_hyp_check} implies
that (A\ref{hyp:dimension_dependence_mse-1}) is satisfied with there
$p=1$; $\mu_{0}=\pi_{0}$ hence $K_{0}=K$, see Remark \ref{rem:Poincare_pi};
and $f$ as in (\ref{eq:logit_f}). Also by Lemma \ref{lem:logit_dim_hyp_check},
$K^{-1}=O(d^{q})$ and $\sup_{t}\|\partial_{t}x_{t}^{\star}\|=O(d^{q/2})$,
so the hypothesis of the proposition $\epsilon d^{7q+3}=O(1)$ implies
$\epsilon\sup_{t}\|\partial_{t}x_{t}^{\star}\|/K=O(1)$. Therefore
the hypotheses of Corollary \ref{cor:dimension-dependence-MSE-1}
are satisfied, giving:
\begin{align}
\mathbb{E}\left[|S_{\epsilon,h}|\right]^{2} & \leq\mathbb{E}\left[|S_{\epsilon,h}|^{2}\right]=\mathrm{var}[S_{\epsilon,h}]+\mathbb{E}[S_{\epsilon,h}]^{2}\nonumber \\
 & =O\left(h\left[1+\frac{2}{1-e^{-Kh/\epsilon}}\right]r_{1}(d)+\left[\frac{\epsilon}{K}r_{2}(d)+\frac{h}{1-e^{-Kh/\epsilon}}r_{3}(d)\right]^{2}\right),\label{eq:logit_T1_S}
\end{align}
where 
\begin{equation}
r_{1}(d)=d^{6q+3},\quad r_{2}(d)=d^{19q/4+2},\quad r_{3}(d)=d^{3q+1}.\label{eq:logit_rd}
\end{equation}

Now (A\ref{hyp:logit_dim_depend}) implies that $K=\tilde{\sigma}^{-2}=O(d^{q/4})$,
which combined with the hypotheses of the proposition $\epsilon=o(1)$
and $\frac{h}{\epsilon^{2}}d^{3q/2+1}=O(1)$ implies $Kh/\epsilon=o(1)$.
Using this and the facts that by Lemma \ref{lem:logit_dim_hyp_check},
$K^{-1}=O(d^{q})$, and that the hypothesis of the proposition $\epsilon d^{7q+3}=O(1)$
implies $\epsilon d^{9q/2+1}=O(1)$, it follows from (\ref{eq:logit_T1_S})
and (\ref{eq:logit_rd}) that
\begin{align*}
\mathbb{E}\left[|S_{\epsilon,h}|\right]\leq\mathbb{E}[|S_{\epsilon,h}|^{2}]^{1/2} & =O\left(\sqrt{\left[h+\frac{\epsilon}{K}\right]r_{1}(d)+\left[\frac{\epsilon}{K}\{r_{2}(d)\vee r_{3}(d)\}\right]^{2}}\right)\\
 & =O\left(\sqrt{\frac{\epsilon}{K}r_{1}(d)+\left[\frac{\epsilon}{K}r_{2}(d)\right]^{2}}\right)\\
 & =O\left(\sqrt{\epsilon d^{7q+3}+\epsilon^{2}d^{23q/2+4}}\right)\\
 & =O\left(\sqrt{\epsilon d^{7q+3}(1+\epsilon d^{9q/2+1})}\right)\\
 & =O\left(\sqrt{\epsilon d^{7q+3}}\right).
\end{align*}

For the second term in $T_{1}(\epsilon,h)$, first note that by Lemma
\ref{lem:logit_dim_hyp_check}, $L^{2}/K=O(d^{3q/2})$, which combined
with the hypotheses of the proposition $\epsilon=o(1)$ and $\frac{h}{\epsilon^{2}}d^{3q/2+1}=O(1)$
implies $\frac{hL^{2}}{\epsilon K}=o(1)$ and $hd/\epsilon=O(1)$.
These facts combined with Lemma \ref{lem:logit_dim_hyp_check} validate
an application of Proposition \ref{prop:tv_dim_dependence} to give
\begin{equation}
\|\mu^{\epsilon}-\widetilde{\mu}^{\epsilon,h}\|_{\mathrm{tv}}^{1/2}=O\left(\left[\frac{h}{\epsilon^{2}}d^{4q+1}\right]^{1/4}\right).\label{eq:logit_T1_tv}
\end{equation}
Lemma \ref{lem:logit_dim_hyp_check} and (\ref{eq:logit_eps_h_hyp})
also validate an application of Lemma \ref{lem:disc_drift_p=00003D2}
to give 
\begin{equation}
\mathbb{E}[|\widetilde{S}_{\epsilon,h}|^{2}]^{1/2}\leq\sup_{t}\|f_{t}^{2}\|_{1}^{1/2}h\sum_{k=0}^{\left\lfloor 1/h\right\rfloor -1}(1+\mathbb{E}[\|\widetilde{X}_{kh}^{\epsilon,h}\|^{2}])=O\left(\sup_{t\in[0,1]}\|f_{t}^{2}\|_{1}^{1/2}\{\epsilon d^{2q+1}+hd^{q+1}+d^{q}\}\right),\label{eq:logit_T1_S_tilde}
\end{equation}
where
\begin{align}
\|f_{t}^{2}\|_{1}^{1/2} & \leq\sqrt{3\sup_{x}\frac{f_{t}(x)^{2}}{(1+\|x\|)^{2}}}=\sqrt{3}\|f_{t}\|_{1/2}\nonumber \\
 & \leq\sqrt{3}\left\{ \sup_{x}\frac{|\partial_{t}U_{t}(x)|}{1+\|x\|}+\pi_{t}(\bar{V})\|\partial_{t}U_{t}\|_{1}\right\} ,\label{eq:logit_T1_f_norm}
\end{align}
\begin{align}
\|\partial_{t}U_{t}\|_{1} & \leq3\sup_{x}\frac{|\partial_{t}U_{t}(x)|}{1+\|x\|}\nonumber \\
 & \leq3\|y^{T}C\|+3\sum_{i=1}^{d}\sup_{x}\frac{\log(1+e^{\|x\|\|c_{i}\|})}{1+\|x\|}\nonumber \\
 & \leq3\|y^{T}C\|+3\sum_{i=1}^{d}\sup_{x}\frac{\log2+\|x\|\|c_{i}\|}{1+\|x\|}\nonumber \\
 & \leq3\|y^{T}C\|+3d\log2+3\sum_{i=1}^{d}\|c_{i}\|\nonumber \\
 & =O\left(d+\xi\right),\label{eq:logit_||d_t_U_t||}
\end{align}
 and by Lemma \ref{lem:pi_V_dim_dependence},
\begin{equation}
\sup_{t\in[0,1]}\pi_{t}(\bar{V})=O(d^{q+1}).\label{eq:logit_T1_piV}
\end{equation}
Combining (\ref{eq:logit_T1_S})-(\ref{eq:logit_T1_piV}) and using
the hypotheses of the proposition $\epsilon=o(1)$ and $h=o(1)$,
we find 
\begin{align*}
\mathbb{E}[|\widetilde{S}_{\epsilon,h}|^{2}]^{1/2} & =O\left(\{(d+\xi)d^{q+1}\}\{\epsilon d^{2q+1}+hd^{q+1}+d^{q}\}\right)\\
 & =O\left(d^{q}\{d(d+\xi)+\epsilon d^{q+1}+hd+1\}\right)\\
 & =O\left(d^{q+2}+d^{q+1}\xi\right).
\end{align*}
Collecting the above estimates for $\mathbb{E}\left[|S_{\epsilon,h}|\right]$,
$\|\mu^{\epsilon}-\widetilde{\mu}^{\epsilon,h}\|_{\mathrm{tv}}^{1/2}$,
$\mathbb{E}[|S_{\epsilon,h}|^{2}]^{1/2}$, $\mathbb{E}[|\widetilde{S}_{\epsilon,h}|^{2}]^{1/2}$,
returning to (\ref{eq:logit_T1_defn}) and using that $\xi=O(d^{q/4})$
by (A\ref{hyp:logit_dim_depend}) and the hypothesis of the proposition
$\epsilon d^{7q+3}=O(1)$, we have established
\begin{align*}
T_{1}(\epsilon,h) & =O\left(\sqrt{\epsilon d^{7q+3}}+\left[\frac{h}{\epsilon^{2}}d^{4q+1}\right]^{1/4}\left[\sqrt{\epsilon d^{7q+3}}+d^{q+2}+d^{q+1}\xi\right]\right)\\
 & =O\left(\sqrt{\epsilon d^{7q+3}}+\left[\frac{h}{\epsilon^{2}}d^{4q+1}\right]^{1/4}\left[d^{q+2}+d^{5q/4+1}\right]\right)\\
 & =O\left(\sqrt{\epsilon d^{7q+3}}+\left[\frac{h}{\epsilon^{2}}\right]^{1/4}d^{9(q+1)/4}\right).
\end{align*}

To estimate $T_{2}(h)$, an application of Lemma \ref{lem:rieman_approx}
with there $p=1$, $f_{t}=-\partial_{t}U_{t}$, $\beta=1$, $R_{f}=1$,
$C_{f}=M=\xi$ as in (\ref{eq:logit_xi}) and $K=\tilde{\sigma}^{-2}$
as in (\ref{eq:logit_KLM}), followed by Lemma \ref{lem:logit_dim_hyp_check}
and Lemma \ref{lem:rieman_approx}, gives:

\begin{align}
T_{2}(h)=\left|h\sum_{k=0}^{\left\lfloor 1/h\right\rfloor -1}\pi_{kh}(\left.\partial_{t}U_{t}\right|_{t=kh})-\int_{0}^{1}\pi_{t}(\partial_{t}U_{t})\mathrm{d}t\right| & \leq h^{\beta}\tilde{\alpha}_{1}\big(M\vee\sup_{t}\|\nabla\partial_{t}U_{t}\|_{1}\big)\left[1+\tilde{\alpha}_{1}\frac{M}{K}\sup_{t\in[0,1]}\sqrt{\bar{V}(x_{t}^{\star})}\right]\nonumber \\
 & =O\left(hd^{2q+1}\left[1+d^{5q/2+1}\sqrt{1+d^{q}}\right]\right)\nonumber \\
 & =O\left(hd^{5q+2}\right).\label{eq:logit_T2_bound}
\end{align}

For part 2), first regard $\epsilon$ and $h$ as fixed. Noting
\[
\Delta_{\epsilon,h}=\widetilde{S}_{\epsilon,h}-\left[h\sum_{k=0}^{\left\lfloor 1/h\right\rfloor -1}\pi_{kh}(\left.\partial_{t}U_{t}\right|_{t=kh})-\int_{0}^{1}\pi_{t}(\partial_{t}U_{t})\mathrm{d}t\right],
\]
and using the fact, established in the proof of Proposition \ref{prop:decompositionCLT},
that for any two random variables $Z_{1}$ and $Z_{2}$ and any $\delta>0$,

\[
\sup_{w\in\mathbb{R}}\left|\mathbb{P}[Z_{1}+Z_{2}\leq w]-\Phi(w)\right|\leq\sup_{w\in\mathbb{R}}\left|\mathbb{P}[Z_{1}\leq w]-\Phi(w)\right|+\mathbb{P}[|Z_{2}|>\delta]+(2\pi)^{-1/2}\delta,
\]
we have 

\begin{align}
\sup_{w\in\mathbb{R}}\left|\mathbb{P}\left[\epsilon^{-1/2}\Delta_{\epsilon,h}/\sqrt{\sigma_{0}^{2}}\leq w\right]-\Phi(w)\right| & \leq\sup_{w\in\mathbb{R}}\left|\mathbb{P}\left[\epsilon^{-1/2}S_{\epsilon,h}/\sqrt{\sigma_{0}^{2}}\leq w\right]-\Phi(w)\right|\label{eq:logit_clt_decomp1}\\
 & +\sup_{w\in\mathbb{R}}\left|\mathbb{P}\left[\epsilon^{-1/2}S_{\epsilon,h}/\sqrt{\sigma_{0}^{2}}\leq w\right]-\mathbb{P}\left[\epsilon^{1/2}\widetilde{S}_{\epsilon,h}/\sqrt{\sigma_{0}^{2}}\leq w\right]\right|\label{eq:logit_clt_decomp2}\\
 & +\mathbb{I}[\epsilon^{-1/2}|T_{2}(h)|/\sqrt{\sigma_{0}^{2}}>\delta]+(2\pi)^{-1/2}\delta.\label{eq:logit_clt_decomp3}
\end{align}
Now let $\epsilon(d)$ and $h(d)$ be dependent on $d$ as in the
statement of part 2) of the proposition. Note that this places us
in the case $\ell=0$ in (A\ref{hyp:eps_and_h_scaling_intro}). 

To show that the term on the right of the inequality in (\ref{eq:logit_clt_decomp1})
converges to zero as $d\to\infty$, let us check the hypotheses of
Theorem \ref{thm:intro_CLT} in the case $\ell=0$. We have already
established that (A\ref{hyp:dimension_dependence_mse-1}) is satisfied
with there $p=1$, so it remains to check that $\sup_{t}\|\partial_{t}f_{t}\|_{1}$
and $\sup_{t}1/\varsigma_{0}(t)$ grow at most polynomially fast as
$d\to\infty$, where $f_{t}$ is as in (\ref{eq:d_t_U_t_logit}). 

For $\sup_{t}\|\partial_{t}f_{t}\|_{1}$, note that $f_{t}$ as in
(\ref{eq:d_t_U_t_logit}) does not depend on $t$ and it is straightforward
to check that $\partial_{t}f_{t}(x)=-\mathrm{var}_{\pi_{t}}[\partial_{t}U_{t}]$
for all $x$, so $\sup_{t}\|\partial_{t}f_{t}\|_{1}\leq\sup_{t}\pi_{t}[(\partial_{t}U_{t})^{2}]\leq\sup_{t}\pi_{t}(\bar{V}^{2})\|\partial_{t}U_{t}\|_{1}^{2}$,
which grows at most polynormially fast as $d\to\infty$ by Lemma \ref{lem:pi_V_dim_dependence}
and (\ref{eq:logit_||d_t_U_t||}). 

For $\sup_{t}1/\varsigma_{0}(t)$, let us verify the hypotheses of
Lemma \ref{lem:asymp_var_dim_depend} hold, i.e. that $\sup_{s}\|\tilde{\mathcal{L}}_{s}f_{s}\|_{p+1/2}$
and $\sup_{t\in[0,1]}1/\mathrm{var}_{\pi_{t}}[f_{t}]$ grow at most
polynomially fast as $d\to\infty$. For the former, we have $|\mathcal{\tilde{L}}_{s}f_{s}|\leq\|\nabla U_{s}\|\|\nabla f_{s}\|+|\Delta f_{s}|$,
and by (A\ref{hyp:U_time_reg}) and Lemma \ref{lem:logit_dim_hyp_check},
$\|\nabla U_{s}\|_{1/2}\leq L=O(d^{q/4})$; also by Lemma \ref{lem:logit_dim_hyp_check},
$\sup_{s}\|\nabla f_{s}\|_{1}=O(d^{q})$, and $\frac{\partial^{2}f_{t}}{\partial x_{j}^{2}}=-\sum_{i=1}^{m}c_{ij}^{2}\varrho_{i}(x)[1-\varrho_{i}(x)]$,
hence $|\Delta f_{t}|\leq\sum_{i=1}^{m}\|c_{i}\|^{2}\leq\xi=O(d^{q/4})$
by (A\ref{hyp:logit_dim_depend}). Therefore indeed $\sup_{s}\|\tilde{\mathcal{L}}_{s}f_{s}\|_{p+1/2}$
grows at most polynomially fast as $d\to\infty$. By Lemma \ref{lem:variance_lower_bound},
$\mathrm{var}_{\pi_{t}}[f_{t}]\geq L^{-1}\sum_{i=1}^{d}\pi_{t}\left(\partial_{t}U_{t}\frac{\partial U_{t}}{\partial x_{i}}\right)^{2}$,
and 
\begin{align*}
-\pi_{t}\left(\partial_{t}U_{t}\frac{\partial U_{t}}{\partial x_{j}}\right) & =t\int_{\mathbb{R}^{d}}l(y;x)\left(\sum_{i=1}^{m}c_{ij}\left(y_{i}-\varrho_{i}(x)\right)-\frac{x_{j}}{\tilde{\sigma}^{2}}\right)\mathrm{d}x,
\end{align*}
so that under the hypothesis of the proposition that (\ref{eq:logit_clt_hyp})
grows no faster than polynomially, we have by Lemma \ref{lem:asymp_var_dim_depend}
that $\sup_{t}1/\varsigma_{0}(t)$ grows no faster than polynomially.
Hence the term on the right of the inequality in (\ref{eq:logit_clt_decomp1})
indeed converges to zero as $d\to\infty$.

By Lemma \ref{lem:gamma_coupling}, Lemma \ref{lem:logit_dim_hyp_check}
and Proposition \ref{prop:tv_dim_dependence}, the term in (\ref{eq:logit_clt_decomp2})
converges to zero as $d\to\infty$ thanks to the assumed scaling $h=\epsilon^{c}$
for some $c>2$ and $\epsilon=O(d^{-a})$ for $a>0$ large enough. 

By (\ref{eq:logit_T2_bound}), $\epsilon^{-1/2}|T_{2}(h)|=O(\epsilon^{-1/2}hd^{5q+2})$
and we have already established that $\sup_{t}1/\varsigma_{0}(t)$
grows at most polynomially fast with $d$, hence the same is true
of $1/\sqrt{\sigma_{0}^{2}}$. Therefore increasing $a$ in $\epsilon=O(d^{-a})$
if necessary, and then choosing $\delta$ in (\ref{eq:logit_clt_decomp3})
to go to zero suitably slowly as $d\to\infty$, the two terms in (\ref{eq:logit_clt_decomp3})
tend to zero as $d\to\infty$. 

We have shown that all the terms on the right of the inequality in
(\ref{eq:logit_clt_decomp1})-(\ref{eq:logit_clt_decomp3}) converge
to zero as $d\to\infty$, and that completes the proof of the proposition. 

\noindent 
\end{proof}

%% file: child-poincare-proofs.tex
\section{Proofs and supporting results for section \ref{sec:Ergodicity-and-Poincare}}

\subsection{Proof of Lemma \ref{lem:drift} \label{subsec:Proofs-of-drift_lemmas}}
\begin{proof}
[Proof of Lemma \ref{lem:drift}]We have
\begin{eqnarray*}
\frac{\partial}{\partial x_{i}}\|x-x_{t}^{\star}\|^{2p} & = & \frac{\partial}{\partial x_{i}}\left(\sum_{j=1}^{d}(x_{j}-x_{t,j}^{\star})^{2}\right)^{p}=2p\|x-x_{t}^{\star}\|^{2(p-1)}(x_{i}-x_{t,i}^{\star})\\
\frac{\partial^{2}}{\partial x_{i}^{2}}\|x-x_{t}^{\star}\|^{2p} & = & 4p(p-1)\|x-x_{t}^{\star}\|^{2(p-2)}(x_{i}-x_{t,i}^{\star})^{2}+2p\|x-x_{t}^{\star}\|^{2(p-1)}\\
\partial_{t}\|x-x_{t}^{\star}\|^{2p} & = & p\|x-x_{t}^{\star}\|^{2(p-1)}2\sum_{j=1}^{d}(x_{j}-x_{t,j}^{\star})\left(-\partial_{t}x_{t,j}^{\star}\right)\\
 & = & -2p\|x-x_{t}^{\star}\|^{2(p-1)}\left\langle x-x_{t}^{\star},\partial_{t}x_{t}^{\star}\right\rangle 
\end{eqnarray*}
and via Lemma \ref{lem:strong_convex_equiv}, (A\ref{hyp:U_strong_convex})
implies
\[
\left\langle \nabla U_{t}(x)\,,\,x-x_{t}^{\star}\right\rangle \geq\frac{K}{2}\|x-x_{t}^{\star}\|^{2}.
\]
Therefore
\begin{eqnarray*}
-\left\langle \nabla U_{t}(x)\,,\,\nabla V_{t}^{p}(x)\right\rangle  & = & -2p\|x-x_{t}^{\star}\|^{2(p-1)}\left\langle \nabla U_{t}(x)\,,\,x-x_{t}^{\star}\right\rangle \\
 & \leq & -Kp\|x-x_{t}^{\star}\|^{2p},\\
\Delta V_{t}^{p}(x) & = & 4p(p-1)\|x-x_{t}^{\star}\|^{2(p-2)}\sum_{i=1}^{d}(x_{i}-x_{t,i}^{\star})^{2}+2dp\|x-x_{t}^{\star}\|^{2(p-1)}\\
 & = & 2p\left(2(p-1)+d\right)\|x-x_{t}^{\star}\|^{2(p-1)},\\
|\partial_{t}V_{t}^{p}(x)| & \leq & 2p\|x-x_{t}^{\star}\|^{2p-1}\|\partial_{t}x_{t}^{\star}\|\\
 & \leq & 2p\|x-x_{t}^{\star}\|^{2p-1}c,
\end{eqnarray*}
where in the final inequality, $c\coloneqq\sup_{t\in(0,1)}\|\partial_{t}x_{t}^{\star}\|$
is finite by Lemma \ref{lem:minimizer}. Combining the above we have
\begin{eqnarray*}
 &  & \epsilon\partial_{t}V_{t}^{p}(x)(x)+\epsilon\mathcal{L}_{t}V_{t}^{p}(x)\\
 &  & \leq-Kp\|x-x_{t}^{\star}\|^{2p}+2p\|x-x_{t}^{\star}\|^{2p-1}\left[\epsilon c+\left(2(p-1)+d\right)\|x-x_{t}^{\star}\|^{-1}\right]\\
 &  & =-(Kp-\kappa)\|x-x_{t}^{\star}\|^{2p}-\kappa\|x-x_{t}^{\star}\|^{2p}+2p\|x-x_{t}^{\star}\|^{2p-1}\left[\epsilon c+\left(2(p-1)+d\right)\|x-x_{t}^{\star}\|^{-1}\right]\\
 &  & =-(Kp-\kappa)\|x-x_{t}^{\star}\|^{2p}-\|x-x_{t}^{\star}\|^{2p}\left(\kappa-2p\left[\frac{\epsilon c}{\|x-x_{t}^{\star}\|}+\frac{2(p-1)+d}{\|x-x_{t}^{\star}\|^{2}}\right]\right).
\end{eqnarray*}
Hence
\[
\partial_{t}V_{t}^{p}(x)+\mathcal{L}_{t}V_{t}^{p}(x)\leq-\delta\|x-x_{t}^{\star}\|^{2p}+b\mathbb{I}\{\|x-x_{t}^{\star}\|\leq r\},
\]
where
\begin{eqnarray*}
\delta & \coloneqq & \epsilon^{-1}(Kp-\kappa),\\
r & \coloneqq & \sup\left\{ a>0\,:\,\frac{\epsilon c}{a}+\frac{2(p-1)+d}{a^{2}}\geq\frac{\kappa}{2p}\right\} ,\\
b & \coloneqq & 2pr^{2p-1}\left[c+\frac{2(p-1)+d}{\epsilon r}\right].
\end{eqnarray*}
Solving the quadratic inequality in the expression for $r$ completes
the proof of (\ref{eq:drift_bound_gen}).

In the remainder of the proof of the lemma, we write 
\begin{eqnarray*}
V^{p}(t,x) & \equiv & V_{t}^{p}(x)=\|x-x_{t}^{\star}\|^{2p},\\
\mathcal{L}V^{p}(t,x) & \equiv & \partial_{t}V_{t}^{p}(x)+\mathcal{L}_{t}V_{t}^{p}(x).
\end{eqnarray*}
 Fix $s\in[0,1]$ and $x\in\mathbb{R}^{d}$. Define $T_{m}\coloneqq\inf\{t\geq s:\|X_{s,t}^{x}\|>m\}$,
the dependence of $T_{m}$ on $x$ and $s$ is not shown in the notation.
By non-explosivity of the process, $T_{m}\to\infty$, a.s. 

By Dynkin's formula \cite[Lem. 3.2, p.72]{khasminskii2011stochastic}
and (\ref{eq:drift_bound_gen}), for any $m$ such that $\|x\|\leq m$,
\begin{eqnarray*}
 &  & \mathbb{E}[V^{p}(T_{m}\wedge t,X_{s,T_{m}\wedge t}^{x})]+\delta\mathbb{E}\left[\int_{s}^{T_{m}\wedge t}V^{p}(u,X_{s,u}^{x})\mathrm{d}u\right]\\
 &  & =V^{p}(s,x)+\mathbb{E}\left[\int_{s}^{T_{m}\wedge t}\mathcal{L}V^{p}(u,X_{s,u}^{x})\mathrm{d}u\right]+\delta\mathbb{E}\left[\int_{s}^{T_{m}\wedge t}V^{p}(u,X_{s,u}^{x})\mathrm{d}u\right]\\
 &  & \leq V^{p}(s,x)+b(t-s)<+\infty,
\end{eqnarray*}
hence $\mathbb{E}\left[\int_{s}^{t}V^{p}(u,X_{s,u}^{x})du\right]=\lim_{m}\mathbb{E}\left[\int_{s}^{T_{m}\wedge t}V^{p}(u,X_{s,u}^{x})du\right]<+\infty$,
where the limit exists by monotone convergence. Also, by Tonelli's
theorem $\mathbb{E}\left[\int_{s}^{t}V^{p}(u,X_{s,u}^{x})\mathrm{d}u\right]=\int_{s}^{t}P_{s,u}V_{u}^{p}(x)\mathrm{d}u$.
This completes the proof of (\ref{eq:drift_bound_int}).

Applying Fatou, (\ref{eq:drift_bound_gen}) and (\ref{eq:drift_bound_int})
we have
\begin{eqnarray*}
\mathbb{E}[V^{p}(t,X_{s,t}^{x})] & = & \mathbb{E}[\liminf_{m}V^{p}(T_{m}\wedge t,X_{s,T_{m}\wedge t}^{x})]\leq\liminf_{m}\mathbb{E}[V^{p}(T_{m}\wedge t,X_{s,T_{m}\wedge t}^{x})]\\
 & \leq & \liminf_{m}\left\{ V^{p}(s,x)-\delta\mathbb{E}\left[\int_{s}^{T_{m}\wedge t}V^{p}(u,X_{s,u}^{x})\mathrm{d}u\right]+\mathbb{E}\left[\int_{s}^{T_{m}\wedge t}b\mathbb{I}[\|X_{s,u}^{x}\|\leq r]\mathrm{d}u\right]\right\} \\
 & = & V^{p}(s,x)-\delta\mathbb{E}\left[\int_{s}^{t}V^{p}(u,X_{s,u}^{x})\mathrm{d}u\right]+\mathbb{E}\left[\int_{s}^{t}b\mathbb{I}[\|X_{s,u}^{x}\|\leq r]\mathrm{d}u\right],
\end{eqnarray*}
hence
\[
P_{s,t}V_{t}^{p}(x)\leq V_{s}^{p}(x)-\delta\int_{s}^{t}P_{s,u}V_{u}^{p}(x)\mathrm{d}u+b(t-s).
\]
This inequality is solved  to give (\ref{eq:drift_bound_semi}).

To establish (\ref{eq:drift_bound_uniform}), we have by (\ref{eq:drift_bound_semi}),
\begin{eqnarray*}
1+\mathbb{E}\left[\|X_{s,t}^{x}\|^{2p}\right] & \leq & 1+2^{2p-1}\mathbb{E}\left[V_{t}^{p}(X_{s,t}^{x})\right]+2^{2p-1}\|x_{t}^{\star}\|^{2p}\\
 & \leq & 1+2^{2p-1}V_{s}^{p}(x)+2^{2p-1}\frac{b}{\delta}+2^{2p-1}\|x_{t}^{\star}\|^{2p}\\
 & \leq & 2^{4p-2}\|x\|^{2p}+1+2^{2p-1}\frac{b}{\delta}+2^{2p-1}(1+2^{2p-1})\sup_{u\in[0,1]}\|x_{u}^{\star}\|^{2p}\\
 & \leq & \alpha_{p}(1+\|x\|^{2p}),
\end{eqnarray*}
where $\sup_{u\in[0,1]}\|x_{u}^{\star}\|^{2p}$ is finite since by
Lemma \ref{lem:minimizer} $t\mapsto x_{t}^{\star}$ is continuous
on $[0,1]$, and $\alpha_{p}$ is as in the statement of the Lemma.
The proof is complete. 
\end{proof}

\subsection{Proof and supporting results for Proposition \ref{prop:C_2^p_closed}\label{subsec:Proof-and-supporting_C_closed}}
\begin{lem}
\label{lem:continuity}For any $p\geq1$, and $\nu\in\mathcal{P}^{p}(\mathbb{R}^{d})$,
the following condition holds: 
\begin{equation}
\int_{\mathbb{R}^{d}}\mathbb{E}\left[\sup_{t\in[s,1]}\|X_{s,t}^{x}\|^{2p}\right]\nu(\mathrm{d}x)<+\infty,\label{eq:expected_sup_X}
\end{equation}
 and for any $f\in C_{0,0}^{p}([0,1]\times\mathbb{R}^{d})$, $\int_{\mathbb{R}^{d}}\mathbb{E}[f(t,X_{s,t}^{x})]\nu(\mathrm{d}x)$
is continuous in $s$ and $t$. 
\end{lem}

\begin{proof}
By assumption $\sup_{t}|f(t,x)|\leq c(1+\|x\|^{2p})$, so the assumption
$\nu\in\mathcal{P}^{p}(\mathbb{R}^{d})$ combined with equation (\ref{eq:drift_bound_int})
of Lemma \ref{lem:drift} guarantees that $\mathbb{E}[f(t,X_{s,t}^{x})]$
is integrable w.r.t. $\nu$. As noted in section \ref{subsec:Preliminaries-about-proces},
$X_{s,t}^{x}$ is continuous in $t$, a.s., and $f$ is continuous
by assumption, so to establish the continuity in $t$ of $\int_{\mathbb{R}^{d}}\mathbb{E}[f(t,X_{s,t}^{x})]\nu(\mathrm{d}x)$
by an application of dominated convergence, it suffices to show (\ref{eq:expected_sup_X}).
From (\ref{eq:sde}), 
\[
\sup_{t\in[s,1]}\|X_{s,t}^{x}\|\leq\|x\|+\epsilon^{-1}\int_{s}^{1}\|\nabla U_{u}(X_{s,u}^{x})\|\mathrm{d}u+\sqrt{2\epsilon^{-1}}\sup_{t\in[s,1]}\|B_{t}\|.
\]

Using (A\ref{hyp:U_time_reg}), the fact that $s\in[0,1]$, Jensen's
inequality, the convexity of $a\mapsto a^{2p}$, and equation (\ref{eq:drift_bound_int})
of Lemma \ref{lem:drift},
\begin{eqnarray}
\mathbb{E}\left[\left(\int_{s}^{1}\|\nabla U_{u}(X_{s,u}^{x})\|\mathrm{d}u\right)^{2p}\right] & \leq & L^{2p}2^{2p-1}\mathbb{E}\left[\int_{s}^{t}1+\|X_{s,u}^{x}\|^{2p}\mathrm{d}u\right],\nonumber \\
 & \leq & L^{2p}2^{2p-1}\alpha_{p}(1+\|x\|^{2p}).\label{eq:integrated_U_bound}
\end{eqnarray}
The integral of (\ref{eq:integrated_U_bound}) with respect to $\nu$
is finite due to the assumption $\nu\in\mathcal{P}^{p}(\mathbb{R}^{d})$.
The expected value of $\sup_{t\in[s,1]}\left\Vert \int_{s}^{t}\mathrm{d}B_{u}\right\Vert ^{2p}$
is finite by standard results for Brownian motion, e.g. \cite[Prob. 3.29 and Rem. 3.30, Ch. 3, p. 166]{karatzas1991brownian},
and does not depend on $x$. Therefore (\ref{eq:expected_sup_X})
holds as required so $\mathbb{E}[f(t,X_{s,t}^{x})]$ is continous
in $t$. The proof of continuity in $s$ is very similar so the details
are omitted.
\end{proof}
The following notations are in force throughout the remainder of section
\ref{subsec:Proof-and-supporting_C_closed}. For a matrix $A$ and
vector $b$ of appropriate sizes we write $A\circ b$ for the usual
matrix vector product. We introduce the shorthands:
\[
F_{s,t}^{x}[i]\coloneqq-\frac{1}{\epsilon}\frac{\partial U_{t}}{\partial x_{i}}(X_{s,t}^{x}),\quad\quad DF_{s,t}^{x}[i,j]\coloneqq-\frac{1}{\epsilon}\frac{\partial^{2}U_{t}}{\partial x_{i}\partial x_{j}}(X_{s,t}^{x}),\quad D^{2}F_{s,t}^{x}[i,j,k]\coloneqq-\frac{1}{\epsilon}\frac{\partial^{3}U_{t}}{\partial x_{i}\partial x_{j}\partial x_{k}}(X_{s,t}^{x}).
\]
Thus $F_{s,t}^{x}$ is a random vector of length $d$, and $DF_{s,t}^{x}$
is a random $d\times d$ matrix. 
\begin{prop}
\label{prop:Gikhman}Write (\ref{eq:sde}) component-wise as 
\begin{equation}
X_{s,t}^{x}[i]=x[i]+\int_{s}^{t}F_{s,u}^{x}[i]\mathrm{d}u+\sqrt{2\epsilon^{-1}}\int_{s}^{t}\mathrm{d}B_{u}[i],\quad t\in[s,1],\;i\in\{1,\ldots,d\}.\label{eq:gikhman_sde}
\end{equation}
Then for $(i,j,k)\in\{1,\ldots,d\}^{3}$ and $t\in[s,1]$, the solutions
of:
\begin{eqnarray}
\zeta_{s,t}^{x}[i,j] & = & \mathbb{I}[i=j]+\int_{s}^{t}\left\langle DF_{s,u}^{x}[\cdot,i]\,,\,\zeta_{s,u}^{x}[\cdot,j]\right\rangle \mathrm{d}u,\label{eq:deriv_sde_zeta}\\
\eta_{s,t}^{x}[i,j,k] & = & \int_{s}^{t}\left\langle D^{2}F_{s,u}^{x}[\cdot,\cdot,i]\circ\zeta_{s,u}^{x}[\cdot,k]\,,\,\zeta_{s,u}^{x}[\cdot,j]\right\rangle +\left\langle DF_{s,u}^{x}[\cdot,i]\,,\,\eta_{s,u}^{x}[\cdot,j,k]\right\rangle \mathrm{d}u,\label{eq:deriv_sde_eta}
\end{eqnarray}
satisfy 
\begin{equation}
\lim_{n\to\infty}\mathbb{E}\left[\left(\zeta_{s,t}^{x}[i,j]-n\{X_{s,t}^{x}[i]-X_{s,t}^{y(n)}[i]\}\right)^{2}\right]=0,\quad\mathrm{with}\quad y(n)\coloneqq x+n^{-1}e_{j}\label{eq:zeta_ms_conv}
\end{equation}
and
\begin{equation}
\lim_{n\to\infty}\mathbb{E}\left[\left(\eta_{s,t}^{x}[i,j,k]-n\{\zeta_{s,t}^{x}[i,j]-\zeta_{s,t}^{y(n)}[i,j]\}\right)^{2}\right]=0,\quad\mathrm{with}\quad y(n)\coloneqq x+n^{-1}e_{k}.\label{eq:eta_ms_conv}
\end{equation}
Moreover $\zeta_{s,t}^{x}[i,j]$ and $\eta_{s,t}^{x}[i,j,k]$ are
mean-square continuous in $x$. 
\end{prop}

\begin{proof}
Under (A\ref{hyp:U_grad_lipschitz_x}), (A\ref{hyp:U_time_reg}) and
(A\ref{hyp:third_order_deriv}), the existence of random functions
$\zeta_{s,t}^{x}[i,j]$ and $\eta_{s,t}^{x}[i,j,k]$ which satisfy
(\ref{eq:zeta_ms_conv})-(\ref{eq:eta_ms_conv}) and are mean-square
continuous in $x$ is a direct application of \cite[Thm. 2, p. 410]{gikhman1969introduction}.
The fact that $\zeta_{s,t}^{x}[i,j]$ and $\eta_{s,t}^{x}[i,j,k]$
satisfy (\ref{eq:deriv_sde_zeta})-(\ref{eq:deriv_sde_eta}), i.e.
the equations obtained by formally differentiating in (\ref{eq:grad_decomp1}),
is a classical fact noted for example by \cite[Thm. 5.10, p.166]{khasminskii2011stochastic},
see also \cite[Thm. 3.1, p. 218]{kunita1984stochastic}.
\end{proof}

\begin{lem}
\label{lem:deriv_1_id}$\;$

1) there exists a finite constant $c_{1}$ such that $\sup_{x}\sup_{0\leq s\leq t\leq1}\|\zeta_{s,t}^{x}\|_{\mathrm{H.S.}}\leq c_{1}$,
a.s.,

2) for any $s\leq t$ and $f\in C_{1}^{p}(\mathbb{R}^{d})$, $P_{s,t}f(x)$
is differentiable in $x$, the following identity holds: 
\begin{eqnarray}
 &  & \frac{\partial P_{s,t}f}{\partial x_{i}}(x)=\mathbb{E}\left[\left\langle \nabla f(X_{s,t}^{x})\,,\,\zeta_{s,t}^{x}[\cdot,i]\right\rangle \right],\label{eq:deriv_1_id}
\end{eqnarray}

and $\nabla P_{s,t}f(x)$ is continuous in $x$, $s$ and $t$.
\end{lem}

\begin{lem}
\label{lem:deriv_2_id}$\;$

1) there exists a finite constant $c_{2}$ such that $\sup_{x}\sup_{0\leq s\leq t\leq1}\|\eta_{s,t}^{x}\|_{\mathrm{H.S.}}\leq c_{2}$,
a.s.

2) for any $s\leq t$ and $f\in C_{2}^{p}(\mathbb{R}^{d})$, $P_{s,t}f(x)$
is twice differentiable in $x$, the following identity holds: 
\begin{equation}
\frac{\partial^{2}P_{s,t}f}{\partial x_{i}\partial x_{j}}(x)=\mathbb{E}\left[\left\langle \nabla^{(2)}f(X_{s,t}^{x})\circ\zeta_{s,t}^{x}[\cdot,j]\,,\,\zeta_{s,t}^{x}[\cdot,i]\right\rangle \right]+\mathbb{E}\left[\left\langle \nabla f(X_{s,t}^{x})\,,\,\eta_{s,t}^{x}[\cdot,i,j]\right\rangle \right],\label{eq:deriv_2_id}
\end{equation}

and $\nabla^{(2)}P_{s,t}f(x)$ is continuous in $x$, s and $t$. 
\end{lem}

\begin{proof}
[Proof of Lemma \ref{lem:deriv_1_id}]Throughout the proof, $c$ is
a finite constant whose value may change on each appearance. 

For part 1), it follows from (\ref{eq:deriv_sde_zeta}) that
\begin{eqnarray*}
\|\zeta_{s,t}^{x}[\cdot,j]\|^{2} & \leq & 2+2\sum_{i=1}^{d}\left(\int_{s}^{t}\left\langle DF_{s,u}^{x}[\cdot,i]\,,\,\zeta_{s,u}^{x}[\cdot,j]\right\rangle \mathrm{d}u\right)^{2}\\
 & \leq & 2+2\sum_{i=1}^{d}\left(\int_{s}^{t}\|DF_{s,u}^{x}[\cdot,i]\|\|\zeta_{s,u}^{x}[\cdot,j]\|\mathrm{d}u\right)^{2}\\
 & \leq & 2+2(t-s)\sum_{i=1}^{d}\int_{s}^{t}\|DF_{s,u}^{x}[\cdot,i]\|^{2}\|\zeta_{s,u}^{x}[\cdot,j]\|^{2}\mathrm{d}u\\
 & = & 2+2(t-s)\int_{s}^{t}\|DF_{s,u}^{x}\|_{\mathrm{H.S.}}^{2}\|\zeta_{s,u}^{x}[\cdot,j]\|^{2}\mathrm{d}u\\
 & \leq & 2+c(t-s)\int_{s}^{t}\|\zeta_{s,u}^{x}[\cdot,j]\|^{2}\mathrm{d}u.
\end{eqnarray*}
where the first inequality uses the fact that for any $a,b\in\mathbb{R}^{d}$
$\|a+b\|^{2}\leq2(\|a\|^{2}+\|b\|^{2})$;  the second inequality
uses Cauchy-Schwartz; the third inequality uses Jensen's inequality;
the final inequality uses (A\ref{hyp:U_grad_lipschitz_x}), and there
$c$ is a finite constant depending on $L$ and $\epsilon$ but independent
of $j,x$. It then follows from Gronwall's lemma that 
\[
\|\zeta_{s,t}^{x}[\cdot,j]\|^{2}\leq2\exp[c(t-s)^{2}],
\]
the r.h.s. of which is a finite constant independent of $x$ and $j$.
The claim of part 1) then holds.

Considering now that $s,t$ are fixed, we de-clutter the notation
by writing
\[
X^{x}\equiv X_{s,t}^{x},\quad\quad\zeta^{x}\equiv\zeta_{s,t}^{x}.
\]
Fix any $f\in C_{2}^{p}(\mathbb{R}^{d})$, $x\in\mathbb{R}^{d}$ and
set $y(n)\coloneqq x+n^{-1}e_{i}$. To establish the identity in part
2) we shall show that 
\[
\lim_{n\to\infty}\frac{P_{s,t}f(x)-P_{s,t}f(y(n))}{n^{-1}}=\mathbb{E}\left[\left\langle \nabla f(X^{x})\,,\,\zeta^{x}[\cdot,i]\right\rangle \right].
\]
By the mean value theorem, let us introduce a random variable $Z^{x,y(n)}$,
valued on the line segment bewteen $X^{x}$ and $X^{y(n)}$ such that:
\[
f(X^{x})-f(X^{y(n)})=\left\langle \nabla f(Z^{x,y(n)})\,,\,X^{x}-X^{y(n)}\right\rangle ,\quad a.s.
\]
Then using Cauchy-Schwartz we have 
\begin{eqnarray}
 &  & \left|\frac{P_{s,t}(x)-P_{s,t}(y(n))}{1/n}-\mathbb{E}\left[\left\langle \nabla f(X^{x})\,,\,\zeta^{x}[\cdot,i]\right\rangle \right]\right|\nonumber \\
 & = & \left|\mathbb{E}\left[\frac{f(X^{x})-f(X^{y(n)})}{1/n}-\left\langle \nabla f(X^{x})\,,\,\zeta^{x}[\cdot,i]\right\rangle \right]\right|\nonumber \\
 & = & \left|\mathbb{E}\left[\left\langle \nabla f(Z^{x,y(n)})-\nabla f(X^{x})\,,\,n(X^{x}-X^{y(n)})\right\rangle +\left\langle \nabla f(X^{x})\,,\,n(X^{x}-X^{y(n)})-\zeta^{x}[\cdot,i]\right\rangle \right]\right|\nonumber \\
 & \leq & \mathbb{E}\left[\|\nabla f(Z^{x,y(n)})-\nabla f(X^{x})\|^{2}\right]^{1/2}\mathbb{E}\left[n^{2}\|X^{x}-X^{y(n)}\|^{2}\right]^{1/2}\label{eq:grad_decomp1}\\
 &  & +\mathbb{E}\left[\|\nabla f(X^{x})\|^{2}\right]^{1/2}\mathbb{E}\left[\|n(X^{x}-X^{y(n)})-\zeta^{x}[\cdot,i]\|^{2}\right]^{1/2}.\label{eq:grad_decomp2}
\end{eqnarray}
Consider the first expectation in (\ref{eq:grad_decomp1}). We have
\begin{eqnarray}
\sup_{n}\|\nabla f(Z^{x,y(n)})-\nabla f(X^{x})\| & \leq & \sup_{n}\|\nabla f(Z^{x,y(n)})\|+\|\nabla f(X^{x})\|\nonumber \\
 & \leq & \sup_{n}c(1+\|Z^{x,y(n)}\|^{2p})+c(1+\|X^{x}\|^{2p})\nonumber \\
 & \leq & c\sup_{n}\left(1+2^{2p-1}\|X^{y(n)}-X^{x}\|^{2p}+2^{2p-1}\|X^{x}\|^{2p}\right)+c(1+\|X^{x}\|^{2p})\nonumber \\
 & \leq & c\left(1+2^{2p-1}e^{-2pK(t-s)}+2^{2p-1}\|X^{x}\|^{2p}\right)+c(1+\|X^{x}\|^{2p}),\label{eq:grad_minus_grad_bound}
\end{eqnarray}
where the second inequality uses $\|\nabla f(x)\|\leq c(1+\|x\|^{2p})$,
the third uses $\|Z^{x,y(n)}-X^{x}\|\leq\|X^{y(n)}-X^{x}\|$ and the
fourth uses Lemma \ref{lem:pathwise_sde_bound}. The quantity on the
right of the inequality in (\ref{eq:grad_minus_grad_bound}) has finite
expectation by Lemma \ref{lem:drift}. This observation combined with
the facts that $Z^{x,y(n)}\to X^{x}$ a.s. by Lemma \ref{lem:pathwise_sde_bound}
and $\nabla f$ is continuous, yield via the dominated convergence
theorem that 
\begin{equation}
\lim_{n\to\infty}\mathbb{E}\left[\|\nabla f(Z^{x,y(n)})-\nabla f(X^{x})\|^{2}\right]^{1/2}=0.\label{eq:deriv1_term1}
\end{equation}
For the second expectation in (\ref{eq:grad_decomp1}), by Lemma \ref{lem:pathwise_sde_bound},
\[
\sup_{n}n^{2}\|X^{x}-X^{y(n)}\|^{2}\leq\sup_{n}e^{-2K(t-s)}n^{2}\|x-y(n)\|^{2}=e^{-2K(t-s)},
\]
 hence
\begin{equation}
\sup_{n}\mathbb{E}\left[n^{2}\|X^{x}-X^{y(n)}\|^{2}\right]^{1/2}<+\infty.\label{eq:deriv1_term2}
\end{equation}
For the first expectation in (\ref{eq:grad_decomp2}), again using
$\|\nabla f(x)\|\leq c(1+\|x\|^{2p})$ and Lemma \ref{lem:drift}
gives
\begin{equation}
\mathbb{E}\left[\|\nabla f(X^{x})\|^{2}\right]^{1/2}<+\infty.\label{eq:deriv1_term3}
\end{equation}
For the second expectation in (\ref{eq:grad_decomp2}), Proposition
\ref{prop:Gikhman} implies
\begin{equation}
\lim_{n}\mathbb{E}\left[\|n(X^{x}-X^{y(n)})-\zeta^{x}[\cdot,i]\|^{2}\right]^{1/2}=0.\label{eq:deriv1_term4}
\end{equation}
 Combining (\ref{eq:deriv1_term1})-(\ref{eq:deriv1_term4}) and (\ref{eq:grad_decomp1})-(\ref{eq:grad_decomp2})
establishes (\ref{eq:deriv_1_id}).

To complete the proof of part 2), it remains to establish the continuity
properties. Firstly for the continuity in $x$, (\ref{eq:deriv_1_id})
and Cauchy-Schwartz give for any $x,y\in\mathbb{R}^{d}$, 
\begin{eqnarray*}
 &  & \left|\frac{\partial P_{s,t}f}{\partial x_{i}}(x)-\frac{\partial P_{s,t}f}{\partial x_{i}}(y)\right|\\
 &  & \leq\mathbb{E}\left[\|\nabla f(X^{x})-\nabla f(X^{y})\|^{2}\right]^{1/2}\mathbb{E}\left[\|\zeta^{x}\|^{2}\right]^{1/2}+\mathbb{E}\left[\|\nabla f(X^{y})\|^{2}\right]^{1/2}\mathbb{E}\left[\|\zeta^{x}-\zeta^{y}\|^{2}\right]^{1/2}.
\end{eqnarray*}
The first expectation converges to zero as $x\to y$ by very similar
arguments used above to show (\ref{eq:deriv1_term1}). The second
expectation is finite by (\ref{eq:deriv1_term2}) and (\ref{eq:zeta_ms_conv}).
The third expectation converges to $\mathbb{E}\left[\|\nabla f(X^{y})\|^{2}\right]^{1/2}$
using a dominated convergence argument similar to that above and the
limit is finite by (\ref{eq:deriv_2_part1}). The fourth expectation
converges to zero as $y\to x$ because $\zeta^{x}$ is mean-square
continuous in $x$ according to Proposition \ref{prop:Gikhman}. 

Let us next check the continuity in $t$ of $\frac{\partial P_{s,t}f}{\partial x_{i}}$.
Consider (\ref{eq:deriv_1_id}) and note that $X_{s,t}^{x}$ and $\zeta_{s,t}^{x}$
are continuous in $t$, almost surely. Then due to the almost sure
and uniform in $t$ bound on $\|\zeta_{s,t}^{x}\|_{\mathrm{H.S.}}$
from part 1), the assumption $f\in C_{2}^{p}(\mathbb{R}^{d})$ and
(\ref{eq:expected_sup_X}), the descired continuity follows by dominated
convergence. The continuity in $s$ follows very similar arguments.
This completes the proof of part 2).
\end{proof}

\begin{proof}
[Proof of Lemma \ref{lem:deriv_2_id}]Throughout the proof, $c$ is
a finite constant whose value may change on each appearance.

For part 1), 
\begin{eqnarray*}
\|\eta_{s,t}^{x}[\cdot,j,k]\|^{2} & \leq & \frac{2}{\epsilon^{2}}\sum_{i=1}^{d}\left(\int_{s}^{t}\left\langle D^{2}F_{s,u}^{x}[\cdot,\cdot,i]\circ\zeta_{s,u}^{x}[\cdot,k]\,,\,\zeta_{s,u}^{x}[\cdot,j]\right\rangle \mathrm{d}u\right)^{2}\\
 &  & +\frac{2}{\epsilon^{2}}\sum_{i=1}^{d}\left(\int_{s}^{t}\left\langle DF_{s,u}^{x}[\cdot,i]\,,\,\eta_{s,u}^{x}[\cdot,j,k]\right\rangle \mathrm{d}u\right)^{2}\\
 & \leq & \frac{2}{\epsilon^{2}}\sum_{i=1}^{d}\left(\int_{s}^{t}\|D^{2}F_{s,u}^{x}[\cdot,\cdot,i]\|_{\mathrm{H.S.}}\|\zeta_{s,u}^{x}[\cdot,k]\|\|\zeta_{s,u}^{x}[\cdot,j]\|\mathrm{d}u\right)^{2}\\
 &  & +\frac{2}{\epsilon^{2}}\sum_{i=1}^{d}\left(\int_{s}^{t}\|DF_{s,u}^{x}[\cdot,i]\|\|\eta_{s,u}^{x}[\cdot,j,k]\|\mathrm{d}u\right)^{2}\\
 & \leq & \frac{2}{\epsilon^{2}}(t-s)\int_{s}^{t}\|D^{2}F_{s,u}^{x}\|_{\mathrm{H.S.}}^{2}\|\zeta_{s,u}^{x}[\cdot,k]\|^{2}\|\zeta_{s,u}^{x}[\cdot,j]\|^{2}\mathrm{d}u\\
 & + & \frac{2}{\epsilon^{2}}(t-s)\int_{s}^{t}\|DF_{s,u}^{x}\|_{\mathrm{H.S.}}^{2}\|\eta_{s,u}^{x}[\cdot,j,k]\|^{2}\mathrm{d}u\\
 & \leq & \beta_{1}+(t-s)\beta_{2}\int_{s}^{t}\|\eta_{s,u}^{x}[\cdot,j,k]\|^{2}\mathrm{d}u,
\end{eqnarray*}
where the first inequality uses the fact that $\|a+b\|^{2}\leq2(\|a\|^{2}+\|b\|^{2})$;
the second inequality uses Cauchy-Schwartz and the fact for a matrix
$A$ and vector $b$, $\|A\circ b\|\leq\sup_{v\neq0}\frac{\|Av\|}{\|v\|}\|b\|\leq\|A\|_{\mathrm{H.S.}}\|b\|$;
the third inequality uses Jensen's inequality; the final inequality
uses (A\ref{hyp:third_order_deriv}), (A\ref{hyp:U_grad_lipschitz_x})
and part 1) of Lemma \ref{lem:deriv_1_id}, and here \textbf{$\beta_{1},\beta_{2}$
}are finite constants independent of $x,j,k,s,t$. Gronwall's lemma
then gives
\[
\|\eta_{s,t}^{x}[\cdot,j,k]\|^{2}\leq\beta_{1}\exp[\beta_{2}(t-s)^{2}],
\]
which completes the proof of part 1) of the lemma. 

For part 2), we de-clutter notation as in the proof of Lemma \ref{lem:deriv_1_id}
and write
\[
X^{x}\equiv X_{s,t}^{x},\quad\quad\zeta^{x}\equiv\zeta_{s,t}^{x},\quad\quad\eta^{x}\equiv\eta_{s,t}^{x}.
\]
Using (\ref{eq:deriv_1_id}), we have:
\begin{eqnarray*}
 &  & \frac{\partial}{\partial x_{i}}P_{s,t}f(x)-\frac{\partial}{\partial x_{i}}P_{s,t}f(y)\\
 &  & =\mathbb{E}\left[\left\langle \nabla f(X^{x})\,,\,\zeta^{x}[\cdot,i]\right\rangle \right]-\mathbb{E}\left[\left\langle \nabla f(X^{y})\,,\,\zeta^{y}[\cdot,i]\right\rangle \right]\\
 &  & =\mathbb{E}\left[\left\langle \nabla f(X^{x})-\nabla f(X^{y})\,,\,\zeta^{x}[\cdot,i]\right\rangle \right]+\mathbb{E}\left[\left\langle \nabla f(X^{y})\,,\,\zeta^{x}[\cdot,i]-\zeta^{y}[\cdot,i]\right\rangle \right].
\end{eqnarray*}
Therefore to prove the identity in part 2), with $y(n)\coloneqq x+n^{-1}e_{j}$
it is sufficient to establish
\begin{equation}
\lim_{n\to\infty}n\mathbb{E}\left[\left\langle \nabla f(X^{x})-\nabla f(X^{y(n)})\,,\,\zeta^{x}[\cdot,i]\right\rangle \right]=\mathbb{E}\left[\left\langle \nabla^{(2)}f(X^{x})\circ\zeta^{x}[\cdot,j]\,,\,\zeta^{x}[\cdot,i]\right\rangle \right]\label{eq:deriv_2_part1}
\end{equation}
and 
\begin{equation}
\lim_{n\to\infty}n\mathbb{E}\left[\left\langle \nabla f(X^{y(n)})\,,\,\zeta^{x}[\cdot,i]-\zeta^{y(n)}[\cdot,i]\right\rangle \right]=\mathbb{E}\left[\left\langle \nabla f(X^{x})\,,\,\eta^{x}[\cdot,i,j]\right\rangle \right].\label{eq:deriv_2_part2}
\end{equation}

Using the mean value theorem for vector-valued functions we have:
\[
\nabla f(X^{x})-\nabla f(X^{y(n)})=\left(\int_{0}^{1}\nabla^{(2)}f(X^{y(n)}+u(X^{x}-X^{y(n)}))du\right)\circ(X^{x}-X^{y(n)}),
\]
where the integral is element-wise. Therefore in terms of the matrices
\begin{eqnarray*}
A^{x,y(n)} & \coloneqq & \nabla^{(2)}f(X^{x})-\int_{0}^{1}\nabla^{(2)}f(X^{y(n)}+u(X^{x}-X^{y(n)}))\mathrm{d}u\\
B^{x,y(n)} & \coloneqq & \int_{0}^{1}\nabla^{(2)}f(X^{y(n)}+u(X^{x}-X^{y(n)}))\mathrm{d}u,
\end{eqnarray*}
where the integrals are elemnent-wise, we have

\begin{eqnarray*}
 &  & \left\langle \nabla^{(2)}f(X^{x})\circ\zeta^{x}[\cdot,j]\,,\,\zeta^{x}[\cdot,i]\right\rangle -n\left\langle \nabla f(X^{x})-\nabla f(X^{y(n)})\,,\,\zeta^{x}[\cdot,i]\right\rangle \\
 &  & =\left\langle \nabla^{(2)}f(X^{x})\circ\zeta^{x}[\cdot,j]-n\left(\nabla f(X^{x})-\nabla f(X^{y(n)})\right)\,,\,\zeta^{x}[\cdot,i]\right\rangle \\
 &  & =\left\langle \left\{ \nabla^{(2)}f(X^{x})-\int_{0}^{1}\nabla^{(2)}f(X^{y(n)}+u(X^{x}-X^{y(n)}))du\right\} \circ\zeta^{x}[\cdot,j]\,,\,\zeta^{x}[\cdot,i]\right\rangle \\
 &  & +\left\langle \left(\int_{0}^{1}\nabla^{(2)}f(X^{y(n)}+u(X^{x}-X^{y(n)}))du\right)\circ\left\{ \zeta^{x}[\cdot,j]-n(X^{x}-X^{y(n)})\right\} \,,\,\zeta^{x}[\cdot,i]\right\rangle \\
 &  & \equiv\left\langle A^{x,y(n)}\circ\zeta^{x}[\cdot,j]\,,\,\zeta^{x}[\cdot,i]\right\rangle +\left\langle B^{x,y(n)}\circ\left\{ \zeta^{x}[\cdot,j]-n(X^{x}-X^{y(n)})\right\} \,,\,\zeta^{x}[\cdot,i]\right\rangle .
\end{eqnarray*}
Let us apply dominated convergence to show that 
\begin{equation}
\lim_{n}\mathbb{E}\left[\left|\left\langle A^{x,y(n)}\circ\zeta^{x}[\cdot,j]\,,\,\zeta^{x}[\cdot,i]\right\rangle \right|\right]=0.\label{eq:A_term_to_zero}
\end{equation}
To this end, first note that
\begin{eqnarray*}
\left|\left\langle A^{x,y(n)}\circ\zeta_{s,t}^{x}[\cdot,j]\,,\,\zeta^{x}[\cdot,i]\right\rangle \right| & \leq & \|A^{x,y(n)}\|_{\mathrm{H.S.}}\|\zeta^{x}[\cdot,j]\|\|\zeta^{x}[\cdot,i]\|\leq c\|A^{x,y(n)}\|_{\mathrm{H.S.}},
\end{eqnarray*}
where $c$ is a finite constant given by part 1) of Lemma \ref{lem:deriv_1_id}.
Also, by Lemma \ref{lem:pathwise_sde_bound}, $X^{y(n)}\to X^{x}$
a.s., and $\nabla^{(2)}f$ is continuous, hence $\|A^{x,y(n)}\|_{\mathrm{H.S.}}\to0$
a.s. Also, again using Lemma \ref{lem:pathwise_sde_bound}, 
\begin{eqnarray*}
\left|A^{x,y(n)}[i,j]\right| & = & \left|\frac{\partial^{2}f}{\partial x_{i}\partial x_{j}}(X^{x})+\int_{0}^{1}\frac{\partial^{2}f}{\partial x_{i}\partial x_{j}}(X^{y(n)}+u(X^{x}-X^{y(n)}))\mathrm{d}u\right|\\
 & \leq & \left|\frac{\partial^{2}f}{\partial x_{i}\partial x_{j}}(X^{x})\right|+\int_{0}^{1}\left|\frac{\partial^{2}f}{\partial x_{i}\partial x_{j}}(X^{y(n)}+u(X^{x}-X^{y(n)}))\right|\mathrm{d}u\\
 & \leq & c(1+\|X^{x}\|^{2p})+c\int_{0}^{1}1+\|X^{y(n)}+u(X^{x}-X^{y(n)})\|^{2p}\mathrm{d}u\\
 & \leq & c(1+\|X^{x}\|^{2p})+c\int_{0}^{1}1+2^{2p-1}\|X^{x}\|^{2p}+2^{2p-1}\|X^{x}-X^{y(n)}\|^{2p}\mathrm{d}u\\
 & \leq & c(2+(1+2^{2p-1})\|X^{x}\|^{2p}+2^{2p-1}).
\end{eqnarray*}
Therefore using Lemma \ref{lem:drift}, $\mathbb{E}[\sup_{n\geq1}\|A^{x,y(n)}\|_{\mathrm{H.S.}}]<+\infty$,
so indeed (\ref{eq:A_term_to_zero}) holds. 

Similarly let us now show that
\begin{equation}
\lim_{n}\mathbb{E}\left[\left|\left\langle B^{x,y(n)}\circ\left\{ \zeta^{x}[\cdot,j]-n(X^{x}-X^{y(n)})\right\} \,,\,\zeta^{x}[\cdot,i]\right\rangle \right|\right]=0.\label{eq:B_term_to_zero}
\end{equation}
We have for a finite constant $c$ given by part 1) of Lemma \ref{lem:deriv_1_id},
\begin{eqnarray*}
 &  & \left|\left\langle B^{x,y(n)}\cdot\left\{ \zeta_{s,t}^{x}[\cdot,j]-n(X^{x}-X^{y(n)})\right\} \,,\,\zeta^{x}[\cdot,i]\right\rangle \right|\\
 &  & \leq\left\Vert B^{x,y(n)}\circ\left\{ \zeta_{s,t}^{x}[\cdot,j]-n(X^{x}-X^{y(n)})\right\} \right\Vert \|\zeta^{x}[\cdot,i]\|\\
 &  & \leq\|B^{x,y(n)}\|_{\mathrm{H.S.}}\|\zeta_{s,t}^{x}[\cdot,j]-n(X^{x}-X^{y(n)})\|c.
\end{eqnarray*}
By very similar arguments used to those used above in bounding $\left|A^{x,y(n)}[i,j]\right|$,
\begin{eqnarray*}
\left|B^{x,y(n)}[i,j]\right| & \leq & c(1+2^{2p-1}\|X^{x}\|^{2p}+2^{2p-1}),
\end{eqnarray*}
and therefore by Cauchy-Schwartz,
\begin{eqnarray*}
 &  & \mathbb{E}\left[\left|\left\langle B^{x,y(n)}\cdot\left\{ \zeta_{s,t}^{x}[\cdot,j]-n(X^{x}-X^{y(n)})\right\} \,,\,\zeta^{x}[\cdot,i]\right\rangle \right|\right]\\
 &  & \leq c\mathbb{E}\left[(1+2^{2p-1}\|X^{x}\|^{2p}+2^{2p-1})^{2}\right]^{1/2}\mathbb{E}\left[\|\zeta_{s,t}^{x}[\cdot,j]-n(X^{x}-X^{y(n)})\|^{2}\right]^{1/2},
\end{eqnarray*}
the first expectation is finite by Lemma \ref{lem:drift} and the
second converges to zero by Proposition \ref{prop:Gikhman}. Therefore
indeed (\ref{eq:B_term_to_zero}) holds which together with (\ref{eq:A_term_to_zero})
establishes (\ref{eq:deriv_2_part1}). 

Our next task is to prove (\ref{eq:deriv_2_part2}). Using Cauchy-Schwartz,

\begin{eqnarray*}
 &  & \left|\mathbb{E}\left[\left\langle \nabla f(X^{x})\,,\,\eta^{x}[\cdot,i,j]\right\rangle -\left\langle \nabla f(X^{y(n)})\,,\,n(\zeta^{x}[\cdot,i]-\zeta^{y(n)}[\cdot,i])\right\rangle \right]\right|\\
 &  & =\mathbb{E}\left[\left|\left\langle \nabla f(X^{x})-\nabla f(X^{y(n)})\,,\,\eta^{x}[\cdot,i,j]\right\rangle \right|\right]+\mathbb{E}\left[\left|\left\langle \nabla f(X^{y(n)})\,,\,\eta^{x}[\cdot,i,j]-n(\zeta^{x}[\cdot,i]-\zeta^{y(n)}[\cdot,i])\right\rangle \right|\right]\\
 &  & \leq\mathbb{E}\left[\|\nabla f(X^{x})-\nabla f(X^{y(n)})\|^{2}\right]^{1/2}\mathbb{E}\left[\|\eta^{x}[\cdot,i,j]\|^{2}\right]^{1/2}\\
 &  & +\mathbb{E}\left[\|\nabla f(X^{y(n)})\|^{2}\right]^{1/2}\mathbb{E}\left[\|\eta^{x}[\cdot,i,j]-n(\zeta^{x}[\cdot,i]-\zeta^{y(n)}[\cdot,i])\|^{2}\right]^{1/2}.
\end{eqnarray*}
The first expectation converges to zero as $n\to\infty$ by arguments
very similar to those used to prove (\ref{eq:deriv1_term1}). The
second expectation is finite, since we have already established that
$\|\eta^{x}[\cdot,i,j]\|$ is bounded by a finite constant, a.s. By
yet another dominated convergence argument, the third expectation
converges to $\mathbb{E}\left[\|\nabla f(X^{y(n)})\|^{2}\right]^{1/2}$,
which is finite by (\ref{eq:deriv1_term3}). The fourth expectation
converges to zero by Proposition \ref{prop:Gikhman}. The proof of
(\ref{eq:deriv_2_id}) is complete.

To complete the proof of the Lemma it remains to verify that $\nabla^{(2)}P_{s,t}f(x)$
is continuous in $x$, $s$ and $t$. From (\ref{eq:deriv_2_id})
we consider:
\begin{eqnarray*}
 &  & \mathbb{E}\left[\left\langle \nabla^{(2)}f(X^{x})\circ\zeta^{x}[\cdot,j]\,,\,\zeta^{x}[\cdot,i]\right\rangle \right]-\mathbb{E}\left[\left\langle \nabla^{(2)}f(X^{y})\circ\zeta^{y}[\cdot,j]\,,\,\zeta^{y}[\cdot,i]\right\rangle \right]\\
 &  & =\mathbb{E}\left[\left\langle \left\{ \nabla^{(2)}f(X^{x})-\nabla^{(2)}f(X^{y})\right\} \circ\zeta^{x}[\cdot,j]\,,\,\zeta^{x}[\cdot,i]\right\rangle \right]\\
 &  & +\mathbb{E}\left[\left\langle \nabla^{(2)}f(X^{y})\circ\left\{ \zeta^{x}[\cdot,j]-\zeta^{y}[\cdot,j]\right\} \,,\,\zeta^{x}[\cdot,i]\right\rangle \right]\\
 &  & +\mathbb{E}\left[\left\langle \nabla^{(2)}f(X^{y})\circ\zeta^{y}[\cdot,j]\,,\,\zeta^{x}[\cdot,i]-\zeta^{y}[\cdot,i]\right\rangle \right].
\end{eqnarray*}
All three of these expectations converge to zero as $y\to x$, by
arguments involving dominated convergence and the mean-square continuity
of $\zeta_{s,t}^{x}$ asserted in Proposition \ref{prop:Gikhman}.
The details are omitted. Similary 
\begin{eqnarray*}
 &  & \mathbb{E}\left[\left\langle \nabla f(X^{x})\,,\,\eta^{x}[\cdot,i,j]\right\rangle \right]-\mathbb{E}\left[\left\langle \nabla f(X^{y})\,,\,\eta^{y}[\cdot,i,j]\right\rangle \right]\\
 &  & =\mathbb{E}\left[\left\langle \nabla f(X^{x})-\nabla f(X^{y})\,,\,\eta^{x}[\cdot,i,j]\right\rangle \right]+\mathbb{E}\left[\left\langle \nabla f(X^{y})\,,\,\eta^{x}[\cdot,i,j]-\eta^{y}[\cdot,i,j]\right\rangle \right]
\end{eqnarray*}
converges to zero as $y\to x$ again using dominated convergence,
and the mean-square continuity in $x$ of $\eta^{x}$ asserted in
Proposition \ref{prop:Gikhman}. The continuity of $\frac{\partial^{2}P_{s,t}f}{\partial x_{i}\partial x_{j}}$
in $s$ and $t$ follows from very similar arguments to those used
to prove the continuity of $\frac{\partial P_{s,t}}{\partial x_{i}}$
in Lemma \ref{lem:deriv_1_id}.
\end{proof}

\begin{proof}
[Proof of Proposition \ref{prop:C_2^p_closed}] Lemmas \ref{lem:deriv_1_id}
and \ref{lem:deriv_2_id} together establish that for $q=1,2$, if
$f\in C_{q}^{p}(\mathbb{R}^{d})$ then $P_{s,t}f$ is $q$-times continuously
differentiable in $x$, and by (\ref{eq:drift_bound_uniform}), $P_{s,t}f\in C_{0}^{p}(\mathbb{R}^{d})$
. To complete the proof of (\ref{eq:C_closed_P}), it remains to obtain
suitable bounds on $\|\nabla P_{s,t}f\|$ and $\|\nabla^{(2)}P_{s,t}f\|_{\mathrm{H.S.}}^{2}$.
Using (\ref{eq:deriv_1_id}), (\ref{eq:deriv_2_id}), the almost sure
bounds on $\|\zeta_{s,t}^{x}\|_{\mathrm{H.S}}$, $\|\eta_{s,t}^{x}\|_{\mathrm{H.S}}$,
and Lemma \ref{lem:drift}, we have for some finite constant $c$
depending only on $f$, 
\begin{eqnarray}
\|\nabla P_{s,t}f(x)\|^{2} & = & \sum_{i=1}^{d}\mathbb{E}\left[\left\langle \nabla f(X_{s,t}^{x})\,,\,\zeta_{s,t}^{x}[\cdot,i]\right\rangle \right]^{2}\nonumber \\
 & \leq & \sum_{i=1}^{d}\mathbb{E}\left[\|\nabla f(X_{s,t}^{x})\|\|\zeta_{s,t}^{x}[\cdot,i]\|\right]^{2}\nonumber \\
 & \leq & dc^{2}c_{1}^{2}\left(1+\mathbb{E}\left[\|X_{s,t}^{x}\|^{2p}\right]\right)^{2}\nonumber \\
 & \leq & dc^{2}c_{1}^{2}\alpha_{p}^{2}(1+\|x\|^{2p})^{2}\label{eq:norm_of_1d}
\end{eqnarray}
and similarly
\begin{eqnarray}
\|\nabla^{(2)}P_{s,t}f(x)\|_{\mathrm{H.S.}}^{2} & = & \sum_{i,j=1}^{d}\left\{ \mathbb{E}\left[\left\langle \nabla^{(2)}f(X_{s,t}^{x})\circ\zeta_{s,t}^{x}[\cdot,j]\,,\,\zeta_{s,t}^{x}[\cdot,i]\right\rangle \right]+\mathbb{E}\left[\left\langle \nabla f(X_{s,t}^{x})\,,\,\eta_{s,t}^{x}[\cdot,i,j]\right\rangle \right]\right\} ^{2}\nonumber \\
 & \leq & \sum_{i,j=1}^{d}2\mathbb{E}\left[\|\nabla^{(2)}f(X_{s,t}^{x})\|_{\mathrm{H.S}}\|\zeta_{s,t}^{x}[\cdot,j]\|\|\zeta_{s,t}^{x}[\cdot,j]\|\right]^{2}+2\mathbb{E}\left[\|\nabla f(X_{s,t}^{x})\|\|\eta_{s,t}^{x}[\cdot,i,j]\|\right]^{2}\nonumber \\
 & \leq & 2d^{2}c_{1}^{4}c^{2}\left(1+\mathbb{E}\left[\|X_{s,t}^{x}\|^{2p}\right]\right)^{2}+2d^{2}c_{2}^{2}c^{2}\left(1+\mathbb{E}\left[\|X_{s,t}^{x}\|^{2p}\right]\right)^{2}\nonumber \\
 & \leq & 2d^{2}(c_{1}^{4}+c_{2}^{2})c^{2}\alpha_{p}^{2}(1+\|x\|^{2p})^{2}.\label{eq:norm_of_2d}
\end{eqnarray}
The proof of (\ref{eq:C_closed_P}) is then complete.

Now consider the first inclusion in (\ref{eq:C_transfer_L}). Observe
that since $f\in C_{1,2}^{p}([0,1]\times\mathbb{R}^{d})$ and (A\ref{hyp:U_time_cont})
holds, $|\partial_{t}f_{t}(x)|+|\mathcal{L}_{t}f_{t}(x)|$ is continuous
in $t$ and $x$, and there exists a finite constant $c$ such that
\begin{eqnarray}
|\partial_{t}f_{t}(x)|+|\mathcal{L}_{t}f_{t}(x)| & \leq & |\partial_{t}f_{t}(x)|+\epsilon^{-1}\|\nabla U_{t}(x)\|\|\nabla f_{t}(x)\|+\epsilon^{-1}|\Delta f_{t}(x)|\label{eq:inclusion_intermed}\\
 & \leq & c(1+\|x\|^{2p})\left[1+\epsilon^{-1}\|\nabla U_{t}(x)\|+d\epsilon^{-1}\right].\nonumber 
\end{eqnarray}
The proof of (\ref{eq:C_transfer_L}) is then completed by noting
(A\ref{hyp:U_time_reg}). 

For the remaining inclusion of (\ref{eq:C_transfer_L}), note that
$\mathcal{L}_{s}P_{s,t}f_{t}(x)$ is continuous in $s$ and $x$ by
(A\ref{hyp:U_time_cont}) and the second parts of Lemmas \ref{lem:deriv_1_id}
and \ref{lem:deriv_2_id}. Also 
\[
|\mathcal{L}_{s}P_{s,t}f_{t}(x)|\leq\epsilon^{-1}\|\nabla U_{s}(x)\|\|\nabla P_{s,t}f_{t}(x)\|+\epsilon^{-1}|\Delta P_{s,t}f_{t}(x)|,
\]
so the proof is complete upon again noting (A\ref{hyp:U_time_reg})
and the fact that the constants in (\ref{eq:norm_of_1d}), (\ref{eq:norm_of_2d})
are independent of $s$.
\end{proof}

\subsection{Proof and supporting results for Proposition \ref{prop:fwd_and_bck_eqs}\label{subsec:Proofs_fwd_bwd_eqs}}
\begin{proof}
[Proof of Proposition \ref{prop:fwd_and_bck_eqs}]Fix $s\in[0,1]$
and $x\in\mathbb{R}^{d}$. Define $T_{m}\coloneqq\inf\{t\geq s:\|X_{s,t}^{x}\|>m\}$,
the dependence of $T_{m}$ on $x$ and $s$ is not shown in the notation.
By non-explosivity of the process, $T_{m}\to\infty$, a.s. Write $\mathcal{L}f(t,x)\equiv\partial_{t}f(x)+\mathcal{L}_{t}f_{t}(x).$

By Dynkin's formula \cite[Lem. 3.2, p.73]{khasminskii2011stochastic},
\begin{equation}
\mathbb{E}\left[f(T_{m}\wedge t,X_{s,T_{m}\wedge t}^{x})\right]=f(s,x)+\mathbb{E}\left[\int_{s}^{T_{m}\wedge t}\mathcal{L}f(u,X_{x,u}^{x})\mathrm{d}u\right],\label{eq:dynkin_f}
\end{equation}
and therefore using equation (\ref{eq:C_transfer_L}) of Proposition
\ref{prop:C_2^p_closed},
\begin{eqnarray}
\sup_{m}|f(T_{m}\wedge t,X_{s,T_{m}\wedge t}^{x})| & \leq & |f(s,x)|+\sup_{m}\int_{s}^{T_{m}\wedge t}|\mathcal{L}f(u,X_{x,u}^{x})|\mathrm{d}u\nonumber \\
 & \leq & |f(s,x)|+\int_{s}^{t}c(1+\|X_{s,u}^{x}\|^{2p+1})\mathrm{d}u.\label{eq:f_sup}
\end{eqnarray}
The expected value of (\ref{eq:f_sup}) is finite due to equation
(\ref{eq:drift_bound_uniform}) of Lemma \ref{lem:drift} and Fubini,
so combined with the fact that $f(T_{m}\wedge t,X_{s,T_{m}\wedge t}^{x})\to f(t,X_{s,t}^{x})$,
a.s., dominated convergence may be applied to (\ref{eq:dynkin_f})
and Fubini applied once more to give:
\[
\mathbb{E}[f(t,X_{s,t}^{x})]=f(s,x)+\int_{s}^{t}\mathbb{E}\left[\mathcal{L}f(u,X_{s,u}^{x})\right]\mathrm{d}u.
\]
Integrating with respect to $\nu$ and using (\ref{eq:C_transfer_L}),
(\ref{eq:drift_bound_uniform}) and the assumption $\nu\in\mathcal{P}^{p+1/2}(\mathbb{R}^{d})$
to validate changing the order of integration we obtain
\begin{equation}
\int_{\mathbb{R}^{d}}\mathbb{E}[f(t,X_{s,t}^{x})]\nu(\mathrm{d}x)=\int_{\mathbb{R}^{d}}f(s,x)\nu(\mathrm{d}x)+\int_{s}^{t}\int_{\mathbb{R}^{d}}\mathbb{E}\left[\mathcal{L}f(u,X_{s,u}^{x})\right]\nu(\mathrm{d}x)\mathrm{d}u.\label{eq:fwd_eqn_integrated}
\end{equation}
By Lemma \ref{lem:continuity}, $\int_{\mathbb{R}^{d}}\mathbb{E}\left[\mathcal{L}f(u,X_{s,u}^{x})\right]\nu(\mathrm{d}x)$
is continuous in $u$, and so (\ref{eq:fwd_eqn_integrated}) is differentiable
in $t$ and (\ref{eq:fwd_equation}) holds.

Fix $t$ and write $g_{s}(x)\coloneqq P_{s,t}f(x)=\mathbb{E}[f(X_{s,t}^{x})]$,
and note that $g_{s}(x)=P_{s,s+\delta}P_{s+\delta,t}f(x)=\mathbb{E}[g_{s+\delta}(X_{s,s+\delta}^{x})]$.
Observe that by (\ref{eq:C_closed_P}) for any $s$, $x\mapsto g_{s}(x)\in C_{2}^{p}(\mathbb{R}^{d})$,
and also using (A\ref{hyp:U_time_reg}) and noting that the constants
in (\ref{eq:norm_of_1d}) and (\ref{eq:norm_of_2d}) do not depend
on $s$. there exists a finite constant $c$ such that
\begin{equation}
\sup_{\tau}|\Delta g_{\tau}(x)|\vee\sup_{\tau}\|\nabla g_{\tau}(x)\|\vee\sup_{\tau}\|\nabla U_{\tau}(x)\|\leq c(1+\|x\|^{2p}),\quad\forall x.\label{eq:g_bounds}
\end{equation}
Therefore by an application of Ito's formula, (\ref{eq:drift_bound_uniform})
and Fubini, for any $\delta>0$,
\begin{align}
g_{s}(x)-g_{s+\delta}(x) & =\mathbb{E}[g_{s+\delta}(X_{s,s+\delta}^{x})]-g_{s+\delta}(x)\nonumber \\
 & =\int_{s}^{s+\delta}\mathbb{E}\left[-\epsilon^{-1}\left\langle \nabla g_{s+\delta}(X_{s,u}^{x}),\nabla U_{u}(X_{s,u}^{x})\right\rangle +\epsilon^{-1}\Delta g_{s+\delta}(X_{s,u}^{x})\right]\mathrm{d}u\label{eq:g_inter}\\
 & =\mathbb{E}\left[-\epsilon^{-1}\left\langle \nabla g_{s+\delta}(X_{s,\tau}^{x}),\nabla U_{\tau}(X_{s,\tau}^{x})\right\rangle +\epsilon^{-1}\Delta g_{s+\delta}(X_{s,\tau}^{x})\right]\delta,\nonumber 
\end{align}
where the final equality is valid for some $\tau$ in the interval
$(s,s+\delta)$ since the expectation in (\ref{eq:g_inter}), which
is equal to $P_{s,u}\mathcal{L}_{u}g_{s+\delta}(x)$, depends continuously
on $u$ due to (\ref{eq:C_transfer_L}) and the continuity part of
Lemma \ref{lem:continuity}. Then using (\ref{eq:g_bounds}), (\ref{eq:drift_bound_uniform}),
Lemma \ref{lem:continuity} and dominated convergence in order to
interchange limits and expectation, 
\[
\lim_{\delta\to0}\frac{g_{s}(x)-g_{s+\delta}(x)}{\delta}=\mathcal{L}_{s}g_{s}(x).
\]
A similar argument applied to $[g_{s-\delta}(x)-g_{s}(x)]\delta^{-1}$
gives the same limit, which establishes (\ref{eq:bwd_equation}). 

It remains to check that the map $(s,x)\mapsto P_{s,t}f_{t}(x)$ is
a member of $C_{1,2}^{p+1/2}([0,1]\times\mathbb{R}^{d})$. By (\ref{eq:drift_bound_uniform}),
$\sup_{s,x}|P_{s,t}f_{t}(x)|/(1+\|x\|^{2p})<+\infty$; we have already
proved $P_{s,t}f_{t}(x)$ is differentiable in $s$ and its derivative
is $-\mathcal{L}_{s}P_{s,t}f_{t}(x)$; by Proposition \ref{prop:C_2^p_closed}
$\mathcal{L}_{s}P_{s,t}f_{t}(x)$ is continous in $s$ and $\sup_{s,x}|\mathcal{L}_{s}P_{s,t}f_{t}(x)|/(1+\|x\|^{2p+1})<+\infty$;
by (\ref{eq:C_closed_P}), for any $s$, $P_{s,t}f_{t}\in C_{2}^{p}(\mathbb{R}^{d})$,
and the proof is completed upon noting that the constants in (\ref{eq:norm_of_1d})
and (\ref{eq:norm_of_2d}) do not depend on $s$. 
\end{proof}

%% file: child-CLT-inhomogenous-proofs.tex
\section{Proofs for section \ref{sec:Quantitative-CLT-bound}}

\subsection{\label{subsec:Proof-of-Theorem-geyer}Proof of Theorem \ref{thm:variance-spectral-decomposition+Geyer}}
\begin{proof}[Proof of Theorem \ref{thm:variance-spectral-decomposition+Geyer}]
Let for any $s\in[0,1]$ and $f,g\in L_{2}(\pi_{s})$,$\bigl\langle f,g\bigr\rangle_{\pi_{s}}:=\int fg{\rm d}\pi_{s}$.
For $\ell>0$ the first statement follows from the fact that $-\mathcal{L}_{s}$
is a positive self-adjoint operator, implying that one can apply the
spectral decomposition theorem and establish that (\cite[Section 1.7.2 \& Appendix A4]{bakry2013analysis})
\begin{align*}
\bigl\langle f_{s},Q_{t}^{s}f_{s}\bigr\rangle_{\pi_{s}} & =\int_{0}^{\infty}\exp\big(-t\lambda\big)\nu_{s}({\rm d}\lambda),
\end{align*}
from which one can conclude by noting that, with ${\rm cov}[\cdot,\cdot]$
the covariance operator associated with $\mathbb{E}[\cdot]$, for
any $\epsilon>0$
\begin{align*}
{\rm var}\left[\epsilon^{-1/2}h\sum_{i=0}^{n-1}f_{s}(Y_{ih}^{s,\epsilon})\right] & =\epsilon^{-1}h^{2}\left(n{\rm var}_{\pi_{s}}[f_{s}]+2\sum_{k=1}^{n-1}(n-k){\rm cov}\big[f_{s}(Y_{0}^{s,\epsilon}),f_{s}(Y_{kh}^{s,\epsilon})\big]\right)\\
 & =\epsilon^{-1}h(nh)\left({\rm var}_{\pi_{s}}[f_{s}]+2\sum_{k=1}^{n-1}(1-k/n)\bigl\langle f_{s},Q_{kh\epsilon^{-1}}^{s}f_{s}\bigr\rangle_{\pi_{s}}\right),
\end{align*}
and using standard convergence arguments. The case $\ell=0$ is naturally
standard. For $\lambda\in(0,\infty)$ (we have a positive spectral
gap, so all cases are covered) consider the function
\[
\varphi_{\lambda}(\ell):=\ell\frac{1+\exp(-\ell\lambda)}{1-\exp(-\ell\lambda)}=\ell\left(\frac{2}{1-\exp(-\ell\lambda)}-1\right).
\]
We show that it is non-decreasing on $(0,\infty)$, as a function
of $\ell$. We have
\begin{align*}
\varphi'_{\lambda}(\ell) & =\left(\frac{2}{1-\exp(-\ell\lambda)}-1\right)-\ell\frac{2\lambda\exp(-\ell\lambda)}{(1-\exp(-\ell\lambda))^{2}}\\
 & =\frac{(1+\exp(-\ell\lambda))(1-\exp(-\ell\lambda))-2\ell\lambda\exp(-\ell\lambda)}{(1-\exp(-\ell\lambda))^{2}}.\\
 & =\frac{1-\exp(-2\ell\lambda)-2\ell\lambda\exp(-\ell\lambda)}{(1-\exp(-\ell\lambda))^{2}}.
\end{align*}
Consider the function $D(a):=1-\exp(-2a)-2a\exp(-a)$ and note that
its derivative is $D'(a)=2\exp(-2a)+2(a-1)\exp(-a)=2\exp(-a)[a-1+\exp(-a)]$.
Therefore $D'(a)\geq0$ and since $D(0)=0$ we deduce $D(a)\geq0$
for $a\geq0$. We therefore conclude that $\varphi'_{\lambda}(\ell)\geq0$
for $\ell>0$. Finally we notice that for $\lambda>0$
\[
\lim_{\ell\rightarrow0}\ell\frac{1+\exp(-\ell\lambda)}{1-\exp(-\ell\lambda)}=2/\lambda
\]
and therefore for $\ell>0$, $\varphi_{\lambda}(\ell)>2/\lambda$,
from which we conclude. 
\end{proof}

\subsection{\label{subsec:proofs-Quantitative-Martingale-approxim}Proofs for
subsection \ref{subsec:Quantitative-Martingale-approxim}}
\begin{proof}[Proof of Lemma \ref{lem:controlbiasforCLT}]
From Corollary \ref{cor:dimension-dependence-MSE-1} and Lemma \ref{lem:boundoneminusexponentialetc}
\begin{align*}
|B_{\epsilon,h}| & \leq C\frac{\epsilon}{K}\left[r_{2}(d)+\frac{Kh/\epsilon}{1-e^{-Kh/\epsilon}}r_{3}(d)\right]\\
 & \leq C\frac{\epsilon}{K}\left[r_{2}(d)+\gimel r_{3}(d)\right],
\end{align*}
and therefore
\begin{align*}
\mathbb{P}\big[|B_{\epsilon,h}|/\sqrt{\epsilon\upsilon(\epsilon)}>\varepsilon_{1}/2\big] & =\mathbb{I}\{2|B_{\epsilon,h}|\epsilon^{-1}>\sqrt{\epsilon\upsilon(\epsilon)}\varepsilon_{1}\epsilon^{-1}\},\\
 & \leq\mathbb{I}\{F>\sqrt{\upsilon(\epsilon)}\epsilon^{-1/2}\varepsilon_{1}\}.
\end{align*}
For the second part, from Corollary \ref{cor:dimension-dependence-MSE-1}
$F(d)$ grows at most polynomially in $d$, say $F(d)\leq Cd^{f}$.
Then
\[
\upsilon_{d}\big(\epsilon(d)\big)^{1/2}\epsilon(d)^{-1/2}\varepsilon_{1}(d)=\big[\sigma_{\ell}^{2}(d)+\upsilon_{d}\big(\epsilon(d)\big)-\sigma_{\ell}^{2}(d)\big]^{1/2}\epsilon(d)^{-1/2}\varepsilon_{1}(d).
\]
From (A\ref{hyp:polynomialdependence}) $\sigma_{\ell}^{2}(d)\geq Cd^{-r}$
for some $r>0$ and from Theorem \ref{thm:CLT-variance-convergence}
there exists $a_{0}>0$ such that for any $a>a_{0}$ one can make
$\upsilon_{d}\big(\epsilon(d)\big)-\sigma_{\ell}^{2}(d)$ vanish faster
than $d^{-r}$. Let $a_{1}\geq a_{0}$, then for $d$ sufficiently
large,
\[
\upsilon_{d}\big(\epsilon(d)\big)^{1/2}\epsilon(d)^{-1/2}\varepsilon_{1}(d)\geq\big[\sigma_{\ell}^{2}(d)/2\big]^{1/2}\epsilon(d)^{-1/2}\varepsilon_{1}(d).
\]
Now choose $\varepsilon_{1}(d)=\epsilon(d)^{c}$ with $c<1/2$, $a>a_{1}\vee\big[(r/2+f)/(1/2-c)\big]$
and $\epsilon(d)=Cd^{-a}$, then $\sigma_{\ell}(d)\epsilon(d)^{-1/2+c}F(d)^{-1}$
diverges and we conclude. 
\end{proof}
\begin{proof}[Proof of Lemma \ref{lem:controlremainderforCLT}]
From Markov's inequality, Lemma \ref{lem:Poisson-and-martingale}
and Lemma \ref{lem:intermediateresultsonV} 
\begin{align*}
\mathbb{P}\big[h|\gamma_{0,\epsilon}(X_{0}^{\epsilon})|/\sqrt{\epsilon\upsilon(\epsilon)}>\varepsilon_{1}/2\big] & \leq2\frac{h}{\varepsilon_{1}\sqrt{\epsilon\upsilon(\epsilon)}}\mu_{0}\big(|\gamma_{0,\epsilon}|\big)\\
\leq C\frac{\alpha_{p}\epsilon^{1/2}}{\varepsilon_{1}\sqrt{\upsilon(\epsilon)}}\frac{\|\nabla f\|_{p}h\epsilon^{-1}}{1-\exp\big(-K\epsilon^{-1}h\big)} & \left(\mu_{0}\bar{V}^{(p+1/2)}\right)^{2}\\
\leq C\frac{\alpha_{p}\epsilon^{1/2}}{\varepsilon_{1}\sqrt{\upsilon(\epsilon)}}\frac{\gimel\|\nabla f\|_{p}}{K} & \left(\mu_{0}\bar{V}^{(p+1/2)}\right)^{2}.
\end{align*}
The proof is now similar to that of the second part of Lemma \ref{lem:controlbiasforCLT}.
\end{proof}

\subsection{\label{subsec:Proofs-Quantitative-bound-CLT-variance}Proofs for
subsection \ref{subsec:Quantitative-bound-CLT-variance}}
\begin{proof}[Proof of Lemma \ref{prop:B_epsilon}]
From Lemma \ref{lem:Poisson-and-martingale} we know that for $k\in\{0,\ldots,n-1\}$
$\big(\gamma_{k,\epsilon},P_{kh,(k+1)h}^{\epsilon}\gamma_{k+1,\epsilon}\big)\in C_{2}^{p}(\mathbb{R}^{d})$,
and as a result, using Lemma \ref{lem:W(delta_x,nu)}, $\Big(\gamma_{k,\epsilon}^{2},\big(P_{kh,(k+1)h}^{\epsilon}\gamma_{k+1,\epsilon}\big)^{2}\Big)\in C_{2}^{2p}(\mathbb{R}^{d})$
and from Proposition \ref{prop:C_2^p_closed} we have that $P_{(k-1)h,kh}^{\epsilon}\Big(\gamma_{k,\epsilon}^{2}\Big),P_{(k-1)h,kh}^{\epsilon}\Big(\big(P_{kh,(k+1)h}^{\epsilon}\gamma_{k+1,\epsilon}\big)^{2}\Big)\in C_{2}^{2p}(\mathbb{R}^{d})$.
Further, from Lemma \ref{lem:Poisson-and-martingale}, we have for
$\gimel>1$ and $\gimel^{-1}<1-Kh\epsilon^{-1}/2$
\begin{align}
\big|P_{kh,(k+1)h}^{\epsilon}\gamma_{k+1,\epsilon}(x)\big|\vee\big|\gamma_{k,\epsilon}(x)\big| & \leq C\epsilon h^{-1}\alpha_{p}\frac{\gimel\|\nabla f\|_{p}}{K}\mu_{0}\bar{V}^{(p+1/2)}\cdot\bar{V}^{(p+1/2)}(x)\label{eq:term2inDepsilon}
\end{align}
and therefore from Lemma \ref{lem:W(delta_x,nu)} and Lemma \ref{lem:drift}
\begin{align*}
P_{(k-1)h,kh}^{\epsilon}\Big(\big|\bar{f}_{kh,\epsilon}\gamma_{k,\epsilon}\big|\Big)(x)\vee & P_{(k-1)h,kh}^{\epsilon}\Big(\big|\bar{f}_{kh,\epsilon}\cdot P_{kh,(k+1)h}^{\epsilon}\gamma_{k+1,\epsilon}\big|\Big)(x)\\
 & \leq C\epsilon h^{-1}\alpha_{p}\frac{\gimel[1+\alpha_{p}\mu_{0}\bar{V}^{(p)}]\|f\|_{p}\|\nabla f\|_{p}}{K}\mu_{0}\bar{V}^{(p+1/2)}\alpha_{2p+1/2}\bar{V}^{(2p+1/2)}(x),
\end{align*}
since
\begin{align*}
\big|\bar{f}_{kh,\epsilon}(x)\big|/\bar{V}^{(p)}(x) & \leq\|f\|_{p}+\|f\|_{p}\sup_{s\in[0,1]}\mu_{s}(\bar{V}^{(p)})/\bar{V}^{(p)}(x)\\
 & \leq\|f\|_{p}\big[1+\alpha_{p}\mu_{0}\bar{V}^{(p)}\big]
\end{align*}
We deduce that for $q>1$
\begin{align*}
\epsilon^{-1}h^{2}\| & \big[P_{0,h}^{\epsilon}\gamma_{1,\epsilon}(X_{0}^{\epsilon})\big]^{2}-\mathbb{E}\left(\big[P_{0,h}^{\epsilon}\gamma_{1,\epsilon}(X_{0}^{\epsilon})\big]^{2}\right)\|_{L_{q}}\\
\leq & C\epsilon\left(\alpha_{p}\frac{\gimel[1+\alpha_{p}\mu_{0}\bar{V}^{(p)}]\|f\|_{p}\|\nabla f\|_{p}}{K}\mu_{0}\bar{V}^{(p+1/2)}\right)^{2}\cdot\left(\alpha_{q(2p+1)}\mu_{0}\bar{V}^{(q[2p+1])}\right)^{1/q}
\end{align*}
Further
\begin{equation}
P_{(n-2)h,(n-1)h}^{\epsilon}\bar{f}_{n-1,\epsilon}^{2}(x)\leq[1+\alpha_{p}\mu_{0}\bar{V}^{(p)}]^{2}\|f\|_{p}^{2}\alpha_{2p}\bar{V}^{(2p)}(x)\label{eq:term1inDepsilon}
\end{equation}
and therefore, for $q>1$
\begin{align*}
\epsilon^{-1}h^{2}\|P_{(n-2)h,(n-1)h}^{\epsilon}\bar{f}_{n-1,\epsilon}^{2}(X_{(n-2)h}^{\epsilon}) & -\mathbb{E}\left(P_{(n-2)h,(n-1)h}^{\epsilon}\bar{f}_{n-1,\epsilon}^{2}(X_{(n-2)h}^{\epsilon})\right)\|_{L_{q}}\\
\leq & C\epsilon^{-1}h^{2}\alpha_{2p}[1+\alpha_{p}\mu_{0}\bar{V}^{(p)}]^{2}\|f\|_{p}^{2}\left(\alpha_{2pq}\mu_{0}\bar{V}^{(2pq)}\right)^{1/q}.
\end{align*}
Now
\begin{align*}
\|D_{\epsilon}-\upsilon(\epsilon)\|_{L_{1+\kappa}} & \leq\|\tilde{D}_{\epsilon}-\mathbb{E}\big(\tilde{D}_{\epsilon}\big)\|_{L_{1+\kappa}}+\epsilon^{-1}h^{2}\|\big[P_{0,h}^{\epsilon}\gamma_{1,\epsilon}(X_{0}^{\epsilon})\big]^{2}-\mathbb{E}\left(\big[P_{0,h}^{\epsilon}\gamma_{1,\epsilon}(X_{0}^{\epsilon})\big]^{2}\right)\|_{L_{1+\kappa}}\\
 & \hspace{1cm}+\epsilon^{-1}h^{2}\|P_{(n-2)h,(n-1)h}^{\epsilon}\bar{f}_{n-1,\epsilon}^{2}(X_{(n-2)h}^{\epsilon})-\mathbb{E}\left(P_{(n-2)h,(n-1)h}^{\epsilon}\bar{f}_{n-1,\epsilon}^{2}(X_{(n-2)h}^{\epsilon})\right)\|_{L_{1+\kappa}}.
\end{align*}
Now we apply Lemma \ref{lem:rough-Lp-norm-for-averages} for the sum
of terms $h\epsilon^{-1}\mathbb{E}\bigl[f_{kh}(X_{kh}^{\epsilon})\big(\gamma_{k,\epsilon}(X_{kh}^{\epsilon})+P_{kh,(k+1)h}^{\epsilon}\gamma_{k+1,\epsilon}(X_{kh}^{\epsilon})\big)\mid\mathcal{F}_{(k-1)h}\bigr]$,
$q=1+\kappa$, $r,m>0$ such that $r>q/2>1$ and $m=(qr-2)/(r-1)$
\begin{align*}
\|\tilde{D}_{\epsilon}-\mathbb{E}\big(\tilde{D}_{\epsilon}\big)\|_{L_{q}} & \leq C\big(\|\tilde{D}_{\epsilon}-\mathbb{E}\big(\tilde{D}_{\epsilon}\big)\|_{L_{2}}\big)^{2/(qr)}\left(\alpha_{2pm}^{1/m}\big(\mu_{0}\bar{V}^{(2pm)}\big)^{1/m}+\alpha_{2p}\mu_{0}\bar{V}^{(2p)}\right)^{1-2/(qr)}\\
 & \hspace{2cm}\times\left(\alpha_{p}\alpha_{2p+1/2}\frac{\gimel[1+\alpha_{p}\mu_{0}\bar{V}^{(p)}]\|f\|_{p}\|\nabla f\|_{p}}{K}\mu_{0}\bar{V}^{(p+1/2)}\right)^{1-2/(qr)}.
\end{align*}
We conclude.
\end{proof}
\begin{proof}[Proof of Lemma \ref{lem:C_epsilon-moment-D_epsilon}]
For $C_{\epsilon}$ we first apply Minkowski's inequality followed
with Lemma \ref{lem:Poisson-and-martingale}, Jensen's inequality
and Lemma \ref{lem:drift}
\begin{align*}
\mathbb{E}\Bigl[\xi_{k,\epsilon}^{2(1+\kappa)}\Bigr]^{1/(2+2\kappa)} & \leq\mathbb{E}\Bigl[\big|\gamma_{k,\epsilon}(X_{kh}^{\epsilon})\big|{}^{2(1+\kappa)}\Bigr]^{1/(2+2\kappa)}+\mathbb{E}\Bigl[\big|P_{(k-1)h,kh}^{\epsilon}\gamma_{k,\epsilon}(X_{(k-1)h}^{\epsilon})\big|^{2(1+\kappa)}\Bigr]^{1/(2+2\kappa)}\\
 & \leq C\alpha_{p}\frac{\|\nabla f\|_{p}}{1-\exp\big(-K\epsilon^{-1}h\big)}\mu_{0}\bar{V}^{(p+1/2)}\cdot\mathbb{E}\Bigl[\bar{V}^{(p+1/2)}(X_{(k-1)h}^{\epsilon}){}^{2(1+\kappa)}\Bigr]^{1/(2+2\kappa)}\\
 & \leq C\alpha_{p}\frac{\|\nabla f\|_{p}}{1-\exp\big(-K\epsilon^{-1}h\big)}\mu_{0}\bar{V}^{(p+1/2)}\cdot\Big[\alpha_{2(1+\kappa)(p+1/2)}\mu_{0}\bar{V}^{(2[1+\kappa][p+1/2])}\Bigr]^{1/(2+2\kappa)}.
\end{align*}
Therefore
\[
C_{\epsilon}\leq C\upsilon(\epsilon){}^{-(1+\kappa)}h^{\kappa}\Bigl\{\alpha_{p}\frac{\|\nabla f\|_{p}(\epsilon^{-1}h)^{1/2}}{1-\exp\big(-K\epsilon^{-1}h\big)}\mu_{0}\bar{V}^{(p+1/2)}\Bigr\}^{2(1+\kappa)}\cdot\alpha_{2(1+\kappa)(p+1/2)}\mu_{0}\bar{V}^{(2[1+\kappa][p+1/2])}
\]
Now from Lemma \ref{lem:boundoneminusexponentialetc}, for $1/\gimel\leq1-Kh\epsilon^{-1}/2$
\[
\frac{(\epsilon^{-1}h)^{1/2}}{1-\exp\big(-K\epsilon^{-1}h\big)}\leq\frac{\gimel}{K}(\epsilon h^{-1})^{1/2}
\]
and the term dependent on $\epsilon$ and $h$ in the upper bound
is indeed of the form $h^{\kappa}(\epsilon h^{-1})^{1+\kappa}=(\epsilon h^{-1+\kappa/(1+\kappa)})^{1+\kappa}$.
For the second statement, from Lemma \ref{lem:Poisson-and-martingale}
\begin{multline*}
\Big|\mathbb{E}\bigl[\bar{f}_{kh,\epsilon}(X_{kh}^{\epsilon})\big(\gamma_{k,\epsilon}(X_{kh}^{\epsilon})+P_{kh,(k+1)h}^{\epsilon}\gamma_{k+1,\epsilon}(X_{kh}^{\epsilon})\big)\mid\mathcal{F}_{(k-1)h}\bigr]\Big|\\
\leq C\alpha_{p}\frac{\|\nabla f\|_{p}[1+\alpha_{p}\mu_{0}\bar{V}^{(p)}]\|f\|_{p}}{1-\exp\big(-K\epsilon^{-1}h\big)}\mu_{0}\bar{V}^{(p+1/2)}\cdot P_{(k-1)h,kh}\bigl(\bar{V}^{(p+1/2)}\bar{V}^{(p)}\bigr)(X_{(k-1)h}^{\epsilon})\\
\leq C\alpha_{p}\alpha_{2p+1/2}\frac{\|\nabla f\|_{p}[1+\alpha_{p}\mu_{0}\bar{V}^{(p)}]\|f\|_{p}}{1-\exp\big(-K\epsilon^{-1}h\big)}\mu_{0}\bar{V}^{(p+1/2)}\cdot\bar{V}^{(2p+1/2)}(X_{(k-1)h}^{\epsilon}),
\end{multline*}
where we have used Lemmas \ref{lem:W(delta_x,nu)} and \ref{lem:drift}.
Consequently
\begin{align*}
\mathbb{E}\big[\big|\tilde{D}_{\epsilon}\big|{}^{1+\kappa}\big]^{1/(1+\kappa)} & \leq C\alpha_{p}\alpha_{2p+1/2}\frac{\|\nabla f\|_{p}[1+\alpha_{p}\mu_{0}\bar{V}^{(p)}]\|f\|_{p}\epsilon^{-1}h}{1-\exp\big(-K\epsilon^{-1}h\big)}\mu_{0}\bar{V}^{(p+1/2)}h^{-1}\sum_{k=1}^{n-2}\mathbb{E}\big[\big|\bar{V}^{(2p+1/2)}(X_{(k-1)h}^{\epsilon})\big|{}^{1+\kappa}\big]^{1/(1+\kappa)}\\
 & \leq C\alpha_{p}\alpha_{2p+1/2}\alpha_{(1+\kappa)(2p+1/2)}\frac{\|\nabla f\|_{p}[1+\alpha_{p}\mu_{0}\bar{V}^{(p)}]\|f\|_{p}\epsilon^{-1}h}{1-\exp\big(-K\epsilon^{-1}h\big)}\mu_{0}\bar{V}^{(p+1/2)}\cdot\bigl\{\mu_{0}\bar{V}^{([1+\kappa][2p+1/2])}\bigr\}{}^{1/(1+\kappa)}
\end{align*}
and from (\ref{eq:term2inDepsilon}) and (\ref{eq:term1inDepsilon})
in the proof of Lemma \ref{prop:B_epsilon} we can conclude.
\end{proof}
In Lemma \ref{prop:B_epsilon} it is required to control the $L_{q}$
convergence of the term $D_{\epsilon}$ defined above Proposition
\ref{prop:CLTforMartingale}, which is an ergodic average. It is possible
to get estimates of this quantity by using a Martingale approximation,
followed by the use of Burkholder's inequality. We however use here
a more direct route since no precise estimates are needed.
\begin{lem}
\label{lem:rough-Lp-norm-for-averages}Let $p\geq1$, $f\in C_{0,2}^{p}([0,1]\times\mathbb{R}^{d})$
and $q\geq1$. Then for any $r>1\vee(2/q)$ and with $m=(qr-2)/(r-1)$
\[
\Vert S_{\epsilon,h}-\mathbb{E}\big[S_{\epsilon,h}\big]\Vert_{L_{q}}\leq C\left(\Vert S_{\epsilon,h}-\mathbb{E}\big[S_{\epsilon,h}\big]\Vert_{L_{2}}\right)^{\frac{2}{qr}}\|f\|_{p}^{1-\frac{2}{qr}}\left(\alpha_{pm}^{1/m}\left(\mu_{0}\bar{V}^{(pm)}\right)^{1/m}+\alpha_{p}\mu_{0}\bar{V}^{(p)}\right)^{1-\frac{2}{qr}}.
\]
\end{lem}

\begin{proof}
Let $l:=m/(q-\frac{2}{r}),$ then $r^{-1}+l^{-1}=1$ and we apply
Hölder's inequality, 

\begin{align*}
\mathbb{E}\left[\left(S_{\epsilon,h}-\mathbb{E}\big[S_{\epsilon,h}\big]\right)^{q}\right] & =\mathbb{E}\left[\left(S_{\epsilon,h}-\mathbb{E}\big[S_{\epsilon,h}\big]\right)^{\frac{2}{r}}\left(S_{\epsilon,h}-\mathbb{E}\big[S_{\epsilon,h}\big]\right)^{q-\frac{2}{r}}\right]\\
 & \leq\mathbb{E}\left[\left(S_{\epsilon,h}-\mathbb{E}\big[S_{\epsilon,h}\big]\right)^{2}\right]^{1/r}\mathbb{E}\left[\left(S_{\epsilon,h}-\mathbb{E}\big[S_{\epsilon,h}\big]\right)^{\left(q-\frac{2}{r}\right)l}\right]^{1/l}.
\end{align*}
Using the triangle inequality we get
\begin{align*}
\Vert S_{\epsilon,h}-\mathbb{E}\big[S_{\epsilon,h}\big]\Vert_{L_{q}} & \leq\left(\Vert S_{\epsilon,h}-\mathbb{E}\big[S_{\epsilon,h}\big]\Vert_{L_{2}}\right)^{\frac{2}{qr}}\left(\left\Vert S_{\epsilon,h}\right\Vert _{L_{m}}+\|f\|_{p}\sup_{t\in[0,1]}\mu_{t}\bar{V}^{(p)}\right)^{1-\frac{2}{qr}}.
\end{align*}
Now, noting that $\mathbb{E}\big[S_{\epsilon,h}\big]=h\sum_{i=0}^{n-1}\mu_{ih}^{\epsilon}f_{ih}$,
by the triangle inequality and from Lemma \ref{lem:W(delta_x,nu)}
and Lemma \ref{lem:drift}
\begin{align*}
\left\Vert S_{\epsilon,h}\right\Vert _{L_{m}} & \leq h\sum_{i=0}^{n-1}\|f_{ih}\|_{p}\mathbb{\mathbb{E}}\left[\bar{V}^{(p)}\big(X_{ih}\big)^{m}\right]^{1/m}\\
 & \leq\|f\|_{p}2^{m-1}h\sum_{i=0}^{n-1}\mathbb{\mathbb{E}}\left[\bar{V}^{(pm)}\big(X_{ih}\big)^{m}\right]^{1/m}.\\
 & \leq\|f\|_{p}2^{m-1}\alpha_{pm}^{1/m}\left(\mu_{0}\bar{V}^{(pm)}\right)^{1/m}.
\end{align*}
\end{proof}

\subsection{\label{subsec:Proofs-Quantitative-CLT-constants}Proofs for subsection
\ref{subsec:Quantitative-CLT-constants}}
\begin{proof}[Proof of Lemma \ref{lem:poisson-homogeneous-approximation}]
Consider first the case $\ell=0$. Let $m(\cdot):\mathbb{R}_{+}\rightarrow\mathbb{N}$
be such that $\lim_{\epsilon\rightarrow0}m(\epsilon)h(\epsilon)\epsilon^{-1}=\infty$
and for $s\in[0,1]$
\[
I_{s}(\epsilon,x):=\int_{0}^{m(\epsilon)h(\epsilon)\epsilon^{-1}}Q_{t}^{s}f_{s}(x){\rm d}t,
\]
with the convention that $I_{s}(0,x):=\lim_{\epsilon\rightarrow0}I_{s}(\epsilon,x)$
(which exists, by absolute summability). Then for $k\in\{0,\ldots,n-1\}$
\begin{align*}
\mathbb{E}\left[f_{kh}(X_{kh}^{\epsilon})\big(\epsilon^{-1}h(\epsilon)\eta_{k,\epsilon}(X_{kh}^{\epsilon})-g_{kh}(X_{kh}^{\epsilon})\big)\right]=\mathbb{E}\left[f_{kh}(X_{kh}^{\epsilon})\big(R_{1}(\epsilon,X_{kh}^{\epsilon})+R_{2}(\epsilon,X_{kh}^{\epsilon})+R_{3}(\epsilon,X_{kh}^{\epsilon})\big)\right]
\end{align*}
where
\begin{align*}
R_{1}(\epsilon,x):= & h(\epsilon)\epsilon^{-1}\left(\sum_{i=0}^{m(\epsilon)-1}Q_{ih\epsilon^{-1}}^{kh}f_{kh}(x)\right)-I_{kh}\big(\epsilon,x\big),\\
R_{2}(\epsilon,x):= & h(\epsilon)\epsilon^{-1}\sum_{i=m(\epsilon)}^{n(\epsilon)-1}Q_{ih\epsilon^{-1}}^{kh}f_{kh}(x),\\
R_{3}(\epsilon,x):= & I_{kh}(\epsilon,x)-I_{kh}(0,x).
\end{align*}
For the term involving $R_{1}(\epsilon,x)$ first notice that by the
classical homogeneous equivalent of Kolmogorov's equation in Proposition
\ref{prop:fwd_and_bck_eqs}, Lemma \ref{lem:drift_Y_process}, (A\ref{hyp:U_time_reg})
and Lemma \ref{lem:intermediateresultsonV} for any $s\in[0,1]$ and
$t\in\mathbb{R}_{+}$, 
\begin{align*}
\Big|\partial_{t}Q_{t}^{s}f_{s}(x)\Big| & =\big|Q_{t}^{s}\mathcal{L}_{s}f_{s}(x)\big|\\
 & \leq Q_{t}^{s}\Big(\big|\bigl\langle\nabla U_{s},\nabla f_{s}\bigr\rangle\big|+\|\Delta f_{s}\|\big)(x),\\
 & \leq Q_{t}^{s}\Big(\|\nabla U_{s}\|\cdot\|\nabla f_{s}\|+\|\Delta f_{s}\|\big)(x),\\
 & \leq L\cdot\|\nabla f\|_{p}Q_{t}^{s}\Big(\bar{V}^{(1/2)}\bar{V}^{(p)}\big)(x)+\|\Delta f\|_{p}Q_{t}^{s}\Big(\bar{V}^{(p)}\Big)(x),\\
 & \leq C\tilde{\alpha}_{p+1/2}L\cdot\|\nabla f\|_{p}\bar{V}^{(p+1/2)}(x)+C\tilde{\alpha}_{p}\|\Delta f\|_{p}\bar{V}^{(p)}(x),\\
 & \leq C\left\{ L\tilde{\alpha}_{p+1/2}+\tilde{\alpha}_{p}\right\} \vvvert f\vvvert_{p}\bar{V}^{(p+1/2)}(x).
\end{align*}
Let $M(x):=\sup_{(s,t)\in[0,1]\times\mathbb{R}_{+}}\Big|\partial_{t}Q_{t}^{s}f_{s}(x)\Big|$
(which can be upper bounded with the above), then we know that the
difference between the Riemann sum with step-size $h(\epsilon)\epsilon^{-1}$
and its integral on the interval $[0,m(\epsilon)h(\epsilon)\epsilon^{-1}]$
yields
\[
\big|R_{1}(\epsilon,x)\big|\leq M(x)h(\epsilon)\epsilon^{-1}\big(m(\epsilon)h(\epsilon)\epsilon^{-1}\big)^{2},
\]
leading to
\begin{multline*}
\left|\mathbb{E}\left[f_{kh}(X_{kh}^{\epsilon})R_{1}(\epsilon,X_{kh}^{\epsilon})\right]\right|\\
\leq C\left\{ L\tilde{\alpha}_{p+1/2}+\tilde{\alpha}_{p}\right\} \vvvert f\vvvert_{p}\cdot\sup_{s\in[0,1]}\mu_{s}\big(|f_{s}|\bar{V}^{(p+1/2)}\big)\cdot h(\epsilon)\epsilon^{-1}\big(m(\epsilon)h(\epsilon)\epsilon^{-1}\big)^{2},\\
\leq A_{1}\cdot h(\epsilon)\epsilon^{-1}\big(m(\epsilon)h(\epsilon)\epsilon^{-1}\big)^{2}.
\end{multline*}
where
\[
A_{1}:=C\alpha_{2p+1/2}\left\{ L\tilde{\alpha}_{p+1/2}+\tilde{\alpha}_{p}\right\} \cdot\vvvert f\vvvert_{p}^{2}\cdot\mu_{0}\big(\bar{V}^{(2p+1/2)}\big).
\]
We define and upper bound the following quantities,
\begin{align*}
R_{2,1}:= & h(\epsilon)\epsilon^{-1}\sum_{i=m(\epsilon)}^{n(\epsilon)-1}{\rm var}_{\mu_{kh}^{\epsilon}}\left[Q_{ih\epsilon^{-1}}^{s}f_{kh}\right]^{1/2}\\
 & \leq\frac{1}{K}\frac{\exp\Big(-Km(\epsilon)h(\epsilon)\epsilon^{-1}\Big)}{\big[1-\exp\big(-Kh(\epsilon)\epsilon^{-1}\big)\big]/(Kh(\epsilon)\epsilon^{-1})}\sup_{(s,t)\in[0,1]\times\mathbb{R}_{+}}{\rm var}_{\mu_{s}^{\epsilon}Q_{t}^{s}}\big[f_{s}\big]^{1/2},\\
R_{2,2}:= & h(\epsilon)\epsilon^{-1}\sum_{i=m(\epsilon)}^{n(\epsilon)-1}\big|\mathbb{E}\left[Q_{ih\epsilon^{-1}}^{s}f_{kh}(X_{kh}^{\epsilon})\right]\big|\\
 & \leq\frac{\tilde{\alpha}_{p}}{K}\frac{\exp\Big(-Km(\epsilon)h(\epsilon)\epsilon^{-1}\Big)}{\big[1-\exp\big(-Kh(\epsilon)\epsilon^{-1}\big)\big]/(Kh(\epsilon)\epsilon^{-1})}\|f\|_{p}\sup_{(s,t)\in[0,1]\times\mathbb{R}_{+}}\mu_{t}^{\epsilon}\big[W^{(p)}(\delta_{\cdot},\pi_{s})\big],\\
R_{3,1}:= & \int_{m(\epsilon)h\epsilon^{-1}}^{\infty}{\rm var}_{\mu_{kh}^{\epsilon}}\big[Q_{t}^{kh}f_{kh}\big]^{1/2}{\rm d}t\leq\frac{1}{K}\exp\Big(-Km(\epsilon)h(\epsilon)\epsilon^{-1}\Big)\sup_{(s,t)\in[0,1]\times\mathbb{R}_{+}}{\rm var}_{\mu_{s}^{\epsilon}Q_{t}^{s}}\left[f_{s}\right]^{1/2},\\
R_{3,2}:= & \int_{m(\epsilon)h\epsilon^{-1}}^{\infty}\big|\mathbb{E}\big[Q_{t}^{s}f_{s}(X_{s}^{\epsilon})\big]\big|{\rm d}t\leq\frac{\tilde{\alpha}_{p}}{K}\exp\Big(-Km(\epsilon)h(\epsilon)\epsilon^{-1}\Big)\cdot\|f\|_{p}\sup_{(s,t)\in[0,1]\times\mathbb{R}_{+}}\mu_{t}^{\epsilon}\big[W^{(p)}(\delta_{\cdot},\pi_{s})\big],
\end{align*}
where the upper bounds follow from the homogeneous equivalent of Lemma
\ref{lem:L_2_convergence}, (\ref{eq:Q_minus_pi}) and Jensen's inequality.
We now apply successively the Cauchy-Schwarz and Minkowski inequalities
(the latter in its sum and integral form), and note the standard inequality
$\mathbb{E}\big[Z^{2}\big]^{1/2}\leq{\rm var}\big[Z\big]^{1/2}+\big|\mathbb{E}\big[Z\big]\big|$
for any random variable $Z$ 
\begin{multline}
\left|\mathbb{E}\left[f_{kh}(X_{kh}^{\epsilon})[R_{2}(\epsilon,X_{kh}^{\epsilon})+R_{3}(\epsilon,X_{kh}^{\epsilon})]\right]\right|\leq\mathbb{E}\left[f_{kh}(X_{kh}^{\epsilon})^{2}\right]^{1/2}\mathbb{E}\left[\mathbb{E}[R_{2}(\epsilon,X_{kh}^{\epsilon})+R_{3}(\epsilon,X_{kh}^{\epsilon})\mid\mathcal{F}_{kh}]^{2}\right]^{1/2}\\
\leq\mathbb{E}\left[f_{kh}(X_{kh}^{\epsilon})^{2}\right]^{1/2}\left\{ \mathbb{E}\left[\mathbb{E}[R_{2}(\epsilon,X_{kh}^{\epsilon})\mid\mathcal{F}_{kh}]^{2}\right]^{1/2}+\mathbb{E}\left[\mathbb{E}[R_{3}(\epsilon,X_{kh}^{\epsilon})\mid\mathcal{F}_{kh}]^{2}\right]^{1/2}\right\} ,\label{eq:CS-R1-R2-R3}
\end{multline}
\begin{align*}
\mathbb{E}\left[\mathbb{E}[R_{2}(\epsilon,X_{kh}^{\epsilon})\mid\mathcal{F}_{kh}]^{2}\right]^{1/2} & \leq h(\epsilon)\epsilon^{-1}\sum_{i=m(\epsilon)}^{n(\epsilon)-1}\mathbb{E}\left[\big(Q_{ih\epsilon^{-1}}^{kh}f_{kh}(X_{kh}^{\epsilon})\big)^{2}\right]^{1/2}\\
 & \leq R_{2,1}+R_{2,2},
\end{align*}
and similarly
\begin{align*}
\mathbb{E}\left[\mathbb{E}[R_{3}(\epsilon,X_{kh}^{\epsilon})\mid\mathcal{F}_{kh}]^{2}\right]^{1/2} & \leq\int_{m(\epsilon)h(\epsilon)\epsilon^{-1}}^{\infty}\mathbb{E}\left[\big(Q_{t}^{kh}f_{kh}(X_{kh}^{\epsilon})\big)^{2}\right]^{1/2}{\rm d}t\\
 & \leq R_{3,1}+R_{3,2}.
\end{align*}
Note that from Lemmas \ref{lem:W(delta_x,nu)}, \ref{lem:drift} and
\ref{lem:Vboundonvariances},
\begin{align*}
\mu_{t}^{\epsilon}\big[W^{(p)}(\delta_{\cdot},\pi_{s})\big] & \leq C\mu_{t}^{\epsilon}\bar{V}^{(p+1/2)}\cdot\pi_{s}\bar{V}^{(p+1/2)}\\
 & \leq C\alpha_{p+1/2}\mu_{0}\bar{V}^{(p+1/2)}\cdot\sup_{s\in[0,1]}\pi_{s}\bar{V}^{(p+1/2)},
\end{align*}
and together with Lemma \ref{lem:Vboundonvariances} we deduce that
\begin{multline*}
\sum_{i=1}^{2}\big|R_{2,i}\big|+\big|R_{3,i}\big|\leq C\left[1+\frac{Kh(\epsilon)\epsilon^{-1}}{1-\exp\big(-Kh(\epsilon)\epsilon^{-1}\big)}\right]\Bigl\{\|\nabla f\|_{p}\big[K^{-1}+K_{\mu_{0}}^{-1}\big]^{1/2}\big(\tilde{\alpha}_{2p}\alpha_{2p}\mu_{0}\bar{V}^{(2p)}\big)^{1/2}\\
+\tilde{\alpha}_{p}\alpha_{p+1/2}\|f\|_{p}\mu_{0}\bar{V}^{(p+1/2)}\cdot\sup_{s\in[0,1]}\pi_{s}\bar{V}^{(p+1/2)}\Bigr\} K^{-1}\exp\Big(-Km(\epsilon)h(\epsilon)\epsilon^{-1}\Big)\\
\leq\bar{A}_{2}\exp\Big(-Km(\epsilon)h(\epsilon)\epsilon^{-1}\Big)
\end{multline*}
where the last inequality holds for $1/\gimel<1-Kh(\epsilon)\epsilon^{-1}/2$,
thanks to Lemma \ref{lem:boundoneminusexponentialetc}, and
\[
\bar{A}_{2}:=CK^{-1}\bigl\{1+\gimel\bigr\}\bigl\{\tilde{\alpha}_{p}\alpha_{p+1/2}\sup_{s\in[0,1]}\pi_{s}\bar{V}^{(p+1/2)}+\bigl(\tilde{\alpha}_{2p}\alpha_{2p}\big[K^{-1}+K_{\mu_{0}}^{-1}\big]\bigr)^{1/2}\bigr\}\vvvert f\vvvert\mu_{0}\bar{V}^{(2p)}.
\]
Together with (\ref{eq:CS-R1-R2-R3}) we deduce that for $1/\gimel<1-Kh(\epsilon)\epsilon^{-1}/2$
\begin{multline*}
\left|\mathbb{E}\left[f_{kh}(X_{kh}^{\epsilon})\big(\epsilon^{-1}h(\epsilon)\eta_{k,\epsilon}(X_{kh}^{\epsilon})-g_{kh}(X_{kh}^{\epsilon})\big)\right]\right|\leq A_{1}h(\epsilon)\epsilon^{-1}\big(m(\epsilon)h(\epsilon)\epsilon^{-1}\big)^{2}+A_{2}\exp\Big(-Km(\epsilon)h(\epsilon)\epsilon^{-1}\Big)
\end{multline*}
with $A_{2}:=C\bar{A}_{2}\cdot\|f\|_{p}\alpha_{2p}\mu_{0}\bar{V}^{(2p)}$
and by taking $Km(\epsilon)h(\epsilon)\epsilon^{-1}=\lceil-\log(h(\epsilon)\epsilon^{-1})\rceil$
we obtain
\begin{align*}
h(\epsilon)\epsilon^{-1}\left|\mathbb{E}\left[f_{kh}(X_{kh}^{\epsilon})\big(\eta_{k,n}(X_{kh}^{\epsilon})-g_{kh}(X_{kh}^{\epsilon})\big)\right]\right| & \leq h(\epsilon)\epsilon^{-1}[A_{2}+A_{1}\big(\lceil-\log(h(\epsilon)\epsilon^{-1})\rceil/K\big)^{2}].
\end{align*}
The scenario $\ell>0$ is more direct and can be bounded in a similar
way to the term dependent on $R_{2}$ above\textendash as a result
for $k\in\{0,\ldots,n-1\}$ 
\begin{multline*}
\Bigl|\mathbb{E}\left[f_{kh}(X_{kh}^{\epsilon})\big(\ell\eta_{k,\epsilon}(X_{kh}^{\epsilon})-g_{kh}(X_{kh}^{\epsilon})\big)\right]\Bigr|=\ell\Big|\mathbb{E}\left[f_{kh}(X_{kh}^{\epsilon})\sum_{i=n(\epsilon)}^{\infty}Q_{i\ell}^{kh}f_{kh}(X_{kh}^{\epsilon})\right]\Big|\\
\leq C\ell^{2}\mu_{0}\big(\bar{V}^{(p)}\big)^{2}\vvvert f\vvvert_{p}^{2}\Bigl\{\tilde{\alpha}_{p}\alpha_{p+1/2}\sup_{s\in[0,1]}\pi_{s}\bar{V}^{(p+1/2)}+\bigl(\tilde{\alpha}_{2p}\alpha_{2p}\big[K^{-1}+K_{\mu_{0}}^{-1}\big]\bigr)^{1/2}\Bigr\}\frac{\exp\Big(-Kn(\epsilon)\ell\Big)}{1-\exp\big(-K\ell\big)}.
\end{multline*}
\end{proof}
\begin{proof}[Proof of Lemma \ref{lem:sublemma-cv-poisson-homogeneous}]
For the first statement, simply notice that for any $k\in\{0,\ldots,n-1\}$
\[
\gamma_{k,\epsilon}(x)-\eta_{k,\epsilon}(x)=T_{1,k,\epsilon}+T_{2,k,\epsilon}+T_{3,k,\epsilon}+T_{4,k,\epsilon}.
\]
From Proposition \ref{prop:fwd_and_bck_eqs} (and its time-homogeneous
version) we deduce that for $0\leq s<u\leq1$
\begin{align*}
Q_{u-s}^{s,\epsilon}f_{u}(x)-P_{s,u}^{\epsilon}f_{u}(x) & =\int_{0}^{u-s}\frac{\partial}{\partial t}Q_{t}^{s,\epsilon}P_{s+t,u}^{\epsilon}f_{u}(x){\rm d}t\\
 & =\int_{0}^{u-s}Q_{t}^{s,\epsilon}\left(\frac{\partial}{\partial t}P_{s+t,u}^{\epsilon}f_{u}+\mathcal{L}_{s}P_{s+t,u}^{\epsilon}f_{u}\right)(x){\rm d}t\\
 & =\int_{0}^{u-s}Q_{t}^{s,\epsilon}\left(\mathcal{L}_{s}-\mathcal{L}_{s+t}\right)P_{s+t,u}^{\epsilon}f_{u}(x){\rm d}t\\
 & =-\epsilon^{-1}\int_{0}^{u-s}Q_{t}^{s,\epsilon}\left(\bigl\langle\nabla U_{s}-\nabla U_{s+t},\nabla P_{s+t,u}^{\epsilon}f_{u})\bigr\rangle\right)(x){\rm d}t.
\end{align*}
Now by application of the Cauchy-Schwarz inequality, Lemma \ref{lem:com_relation}
and (A\ref{hyp:U_time_cont}), we deduce that
\begin{align*}
\big|Q_{u-s}^{s,\epsilon}f_{u}(x)-P_{s,u}^{\epsilon}f_{u}(x)\big| & \leq\epsilon^{-1}M\cdot\int_{0}^{u-s}Q_{t}^{s,\epsilon}\left(\sqrt{\bar{V}_{s}}\cdot P_{s+t,u}^{\epsilon}\|\nabla f_{u}\|\right)(x)\cdot t\exp\big(-K\epsilon^{-1}(u-s-t)\big){\rm d}t\\
 & \leq\epsilon^{-1}M\sup_{t\in[0,u-s]}Q_{t}^{s,\epsilon}\left(\sqrt{\bar{V}_{s}}\cdot P_{s+t,u}^{\epsilon}\|\nabla f_{u}\|\right)(x)\cdot\frac{1}{2}(u-s)^{2}\\
 & \leq C\epsilon^{-1}M(u-s)^{2}\sup_{t\in[0,u-s]}Q_{t}^{s,\epsilon}\left(\sqrt{\bar{V}_{s}}\cdot P_{s+t,u}^{\epsilon}\|\nabla f_{u}\|\right)(x).
\end{align*}
Further by assumption $\|\nabla f\|_{p}<\infty$ and from Lemma \ref{lem:drift}
\begin{align*}
\sup_{t\in[0,u-s]}Q_{t}^{s,\epsilon}\left(\sqrt{\bar{V}_{s}}\cdot P_{s+t,u}^{\epsilon}\|\nabla f_{u}\|\right)(x) & \leq\|\nabla f\|_{p}\cdot\sup_{t\in[0,u-s]}Q_{t}^{s,\epsilon}\left(\sqrt{\bar{V}_{s}}\cdot P_{s+t,u}^{\epsilon}\bar{V}^{(p)}\right)(x)\\
 & \leq\alpha_{p}\|\nabla f\|_{p}\cdot\sup_{t\in[0,u-s]}Q_{t}^{s,\epsilon}\left(\sqrt{\bar{V}_{s}}\cdot\bar{V}^{(p)}\right)(x).
\end{align*}
Now from Proposition \ref{lem:intermediateresultsonV} and from Lemma
\ref{lem:drift_Y_process}, for $s,t\in[0,1]$ and $\epsilon>0$
\begin{align*}
Q_{t}^{s,\epsilon}\left(\sqrt{\bar{V}_{s}}\cdot\bar{V}^{(p)}\right)(x) & \leq C\tilde{\alpha}_{p+1/2}\sqrt{\bar{V}(x_{s}^{\star})}\cdot\bar{V}^{(p+1/2)}(x).
\end{align*}
We also know that 
\begin{align*}
\mu_{s}\left(\big|f_{s}\big|\bar{V}^{(p+1/2)}\right) & \leq\|f\|_{p}\cdot\mu_{s}\left(\bar{V}^{(p)}\bar{V}^{(p+1/2)}\right)\\
 & \leq C\|f\|_{p}\cdot\mu_{s}\left(\bar{V}^{(2p+1/2)}\right)\\
 & \leq C\alpha_{2p+1/2}\|f\|_{p}\cdot\mu_{0}\left(\bar{V}^{(2p+1/2)}\right),
\end{align*}
where we have used Lemma \ref{lem:intermediateresultsonV} and Lemma
\ref{lem:drift}. Since $u\mapsto u-kh$ is non-decreasing, non-negative
for $u\geq kh$ and $\lfloor\tau_{k,\epsilon}h^{-1}\rfloor h\leq\tau_{k,\epsilon}$
\begin{align*}
\big|T_{1,k,\epsilon}\big| & \leq C\alpha_{p}\epsilon^{-1}M\|\nabla f\|_{p}\cdot\sup_{s,t\in[0,1]}Q_{t}^{s,\epsilon}\left(\sqrt{\bar{V}_{s}}\cdot\bar{V}^{(p)}\right)(x)\int_{kh}^{\tau_{k,\epsilon}}(u-kh)^{2}{\rm d}u\\
 & =C\alpha_{p}\daleth^{3}M\|\nabla f\|_{p}\cdot\sup_{s,t\in[0,1]}Q_{t}^{s,\epsilon}\left(\sqrt{\bar{V}_{s}}\cdot\bar{V}^{(p)}\right)(x)\cdot\epsilon^{-1}h^{3\iota},
\end{align*}
and with the bounds on $\sup_{s,t\in[0,1]}Q_{t}^{s,\epsilon}\left(\sqrt{\bar{V}_{s}}\cdot\bar{V}^{(p)}\right)(x)$
and $\mu_{s}\left(\big|f_{s}\big|\bar{V}^{(p+1/2)}\right)$ we obtain
\begin{multline*}
\max_{k\in\{0,\ldots,n-1\}}\mathbb{E}\left[\Bigl|f_{kh}(X_{kh}^{\epsilon})T_{1,k,\epsilon}\Bigr|\right]\leq\\
C\daleth^{3}\alpha_{p}\tilde{\alpha}_{p+1/2}\alpha_{2p+1/2}M\cdot\vvvert f\vvvert_{p}^{2}\cdot\sup_{s\in[0,1]}\bar{V}(x_{s}^{\star})^{1/2}\cdot\mu_{0}\left(\bar{V}^{(2p+1/2)}\right)\cdot\epsilon^{-1}h^{3\iota}.
\end{multline*}
\[
\]
 For the term $T_{2,k,\epsilon}$ we use the smoothness $s\mapsto f_{s}(x)$
and its derivative, the fact that $i\mapsto i-k$ is non-decreasing
and non-negative for $i\geq k$ and again the fact that $\lfloor\tau_{k,\epsilon}h^{-1}\rfloor h\leq\tau_{k,\epsilon}$
\begin{align*}
\big|T_{2,k,\epsilon}\big| & \leq\sum_{i=k}^{\lfloor\tau_{k,\epsilon}h^{-1}\rfloor-1}Q_{(i-k)h}^{kh,\epsilon}\big(\big|f_{ih}-f_{kh}\big|\big)(x)\\
 & \leq\sup_{s,t\in[0,1]}Q_{t}^{s,\epsilon}\left(\sup_{u\in[0,1]}\big|\partial_{t}f_{u}(\cdot)\big|\right)(x)\cdot\sum_{i=k}^{\lfloor\tau_{k,\epsilon}h^{-1}\rfloor-1}ih-kh\\
 & \leq\sup_{s,t\in[0,1]}Q_{t}^{s,\epsilon}\left(\sup_{u\in[0,1]}\big|\partial_{t}f_{u}(\cdot)\big|\right)(x)\cdot\int_{kh}^{\tau_{k,\epsilon}}(u-kh){\rm d}u\\
 & =\frac{\daleth^{2}}{2}\sup_{s,t\in[0,1]}Q_{t}^{s,\epsilon}\left(\sup_{u\in[0,1]}\big|\partial_{t}f_{u}(\cdot)\big|\right)(x)\cdot h^{2\iota}.
\end{align*}
Now by assumption $\|\partial_{t}f\|_{p}<\infty$ and from Lemma \ref{lem:drift_Y_process},
for $s,t\in[0,1]$ and $\epsilon>0$
\begin{align*}
Q_{t}^{s,\epsilon}\left(\sup_{u\in[0,1]}\big|\partial_{t}f_{u}(\cdot)\big|\right)(x) & \leq\|\partial_{t}f\|_{p}Q_{t}^{s,\epsilon}\bar{V}^{(p)}(x)\\
 & \leq\tilde{\alpha}_{p}\|\partial_{t}f\|_{p}\bar{V}^{(p)}(x).
\end{align*}
Therefore 
\begin{align*}
\max_{k\in\{0,\ldots,n-1\}}\mathbb{E}\left[\Bigl|f_{kh}(X_{kh}^{\epsilon})T_{2,k,\epsilon}\Bigr|\right] & \leq C\daleth^{2}\tilde{\alpha}_{p}h^{2\iota}\|\partial_{t}f\|_{p}\sup_{s\in[0,1]}\mu_{s}\left(\big|f_{s}\big|\bar{V}^{(p)}\right).\\
 & \leq C\daleth^{2}\tilde{\alpha}_{p}\alpha_{2p}\|\partial_{t}f\|_{p}\|f\|_{p}\mu_{0}\Bigl(\bar{V}^{(2p)}\Bigr)h^{2\iota}.\\
 & \leq C\daleth^{2}\tilde{\alpha}_{p}\alpha_{2p}\vvvert f\vvvert_{p}^{2}\mu_{0}\Bigl(\bar{V}^{(2p)}\Bigr)h^{2\iota}.
\end{align*}
For $T_{3,k,\epsilon}$ we note that
\[
\big|\mathbb{E}\left[f_{kh}(X_{kh}^{\epsilon})T_{3,k,\epsilon}\right]\big|=\big|\mu_{kh}^{\epsilon}f_{kh}\sum_{i=k}^{\lfloor\tau_{k,\epsilon}h^{-1}\rfloor-1}\mu_{ih}^{\epsilon}f_{ih}\big|
\]
and therefore from Lemma \ref{lem:Vboundonvariances} and the fact
that $\lfloor\tau_{k,\epsilon}h^{-1}\rfloor-k\leq\daleth h^{\iota-1}$we
deduce that for $k\in\{0,\ldots,n-1\}$
\begin{multline*}
\Big|\mathbb{E}\left[f_{kh}(X_{kh}^{\epsilon})T_{3,k,\epsilon}\right]\Big|\\
\leq C\daleth\left\{ \vvvert f\vvvert_{p}\sup_{s\in[0,1]}\pi_{s}\bar{V}^{(2[p\vee p_{0}]+1/2)}\Big[\frac{\|\nabla\phi\|_{p_{0}}}{K^{2}}\epsilon+\alpha_{p}\mu_{0}\bar{V}^{(p+1/2)}\exp\big(-K\epsilon^{-1}hk\big)\Big]\right\} ^{2}h^{\iota-1}.
\end{multline*}
and in particular for $k\geq\lceil-\ln(\epsilon)/(K\epsilon^{-1}h)\rceil$
and letting
\[
B:=\daleth\left\{ \vvvert f\vvvert_{p}\sup_{s\in[0,1]}\pi_{s}\bar{V}^{(2[p\vee p_{0}]+1/2)}\Big[\frac{\|\nabla\phi\|_{p_{0}}}{K^{2}}+\alpha_{p}\mu_{0}\bar{V}^{(p+1/2)}\Big]\right\} ^{2}
\]
we have
\[
\Big|\mathbb{E}\left[f_{kh}(X_{kh}^{\epsilon})T_{3,k,\epsilon}\right]\Big|\leq CB\epsilon^{2}h^{\iota-1}
\]
As a result
\begin{align*}
2\epsilon^{-1}h^{2}\sum_{k=1}^{n-1}\Big|\mathbb{E}\left[f_{kh}(X_{kh}^{\epsilon})T_{3,k,\epsilon}\right]\Big| & \leq C\cdot B\epsilon^{-1}h\Big\{ h\lceil-\ln(\epsilon)/(K\epsilon^{-1}h)\rceil+\epsilon^{2}h^{\iota-1}\Big\}\\
 & \leq C\cdot B\Big\{-h\ln(\epsilon)/K+\epsilon^{-1}h^{2}+\epsilon h^{\iota}\Big\}.
\end{align*}
Finally, defining
\begin{align*}
\mathcal{T}_{4,1}:= & \sum_{i=\lfloor\tau_{k,\epsilon}h^{-1}\rfloor}^{n-1}\Bigl|\mathbb{E}\Big[f_{kh}(X_{kh}^{\epsilon})P_{kh,ih}^{\epsilon}f_{ih,\epsilon}\big(X_{kh}^{\epsilon}\big)\Big]\Bigr|\\
\mathcal{T}_{4,2}:= & \sum_{i=\lfloor\tau_{k,\epsilon}h^{-1}\rfloor}^{n-1}\Bigl|\mathbb{E}\Big[f_{kh}(X_{kh}^{\epsilon})Q_{(i-k)h}^{kh,\epsilon}f_{kh}\big(X_{kh}^{\epsilon}\big)\Big]\Bigr|
\end{align*}
we have

\begin{align*}
\Big|\mathbb{E}\left[f_{kh}(X_{kh}^{\epsilon})T_{4,k,\epsilon}\right]\Big| & \leq\mathcal{T}_{4,1}+\mathcal{T}_{4,2}.
\end{align*}
The term $\mathcal{T}_{4,2}$ is bounded in the same way the $R_{2}$
dependent term in the proof of Lemma \ref{lem:poisson-homogeneous-approximation},
yielding
\begin{align*}
\mathcal{T}_{4,2} & \leq C\alpha_{2p}\mu_{0}\big(\bar{V}^{(2p)}\big)^{2}\vvvert f\vvvert_{p}^{2}\Bigl\{\tilde{\alpha}_{p}\alpha_{p+1/2}\sup_{s\in[0,1]}\pi_{s}\bar{V}^{(p+1/2)}+\bigl(\tilde{\alpha}_{2p}\alpha_{2p}\big[K^{-1}+K_{\mu_{0}}^{-1}\big]\bigr)^{1/2}\Bigr\}\tfrac{\exp\Big(-K[\lfloor\tau_{k,\epsilon}h^{-1}\rfloor-k]h\epsilon^{-1}\Big)}{[1-\exp\big(-Kh\epsilon^{-1}\big)]/h\epsilon^{-1}}\\
 & \leq C\gimel\alpha_{2p}\mu_{0}\big(\bar{V}^{(2p)}\big)^{2}\vvvert f\vvvert_{p}^{2}\Bigl\{\tilde{\alpha}_{p}\alpha_{p+1/2}\sup_{s\in[0,1]}\pi_{s}\bar{V}^{(p+1/2)}+\bigl(\tilde{\alpha}_{2p}\alpha_{2p}\big[K^{-1}+K_{\mu_{0}}^{-1}\big]\bigr)^{1/2}\Bigr\}\tfrac{\exp\Big(-K[\daleth h^{\iota-1}-1]h\epsilon^{-1}\Big)}{K}.
\end{align*}
Now we note that by the Cauchy-Schwarz inequality, Lemma \ref{lem:L_2_convergence}
and Lemma \ref{lem:Vboundonvariances},
\begin{align*}
\mathcal{T}_{4,1} & \leq\Big|\mathbb{E}\left[f_{kh}^{2}(X_{kh}^{\epsilon})\right]^{1/2}\sum_{i=\lfloor\tau_{k,\epsilon}h^{-1}\rfloor}^{n-1}{\rm var}_{\mu_{kh}^{\epsilon}}\big[P_{kh,ih}^{\epsilon}f_{ih}\big]^{1/2}\\
 & \leq C\alpha_{2p}^{1/2}\|f\|_{p}\mu_{0}\big(\bar{V}^{(2p)}\big)^{1/2}\|\nabla f\|_{p}\bigl(\alpha_{2p}\cdot\left[K^{-1}+K_{\mu_{0}}^{-1}\right]\mu_{0}\bar{V}^{(2p)}\bigr){}^{1/2}\frac{\exp\Big(-K\epsilon^{-1}\big(\lfloor\tau_{k,\epsilon}h^{-1}\rfloor-k\big)h\Big)}{1-\exp\Big(-K\epsilon^{-1}h\Big)}.\\
 & \leq C\gimel\vvvert f\vvvert_{p}^{2}\mu_{0}\bar{V}^{(2p)}\alpha_{2p}\left[K^{-1}+K_{\mu_{0}}^{-1}\right]{}^{1/2}\frac{\exp\Big(-K[\daleth h^{\iota-1}-1]h\epsilon^{-1}\Big)}{K}\epsilon h^{-1}.
\end{align*}
because $\lfloor\tau_{k,\epsilon}h^{-1}\rfloor h\geq\tau_{k,\epsilon}-h$. 
\end{proof}
\begin{proof}[Proof of Lemma \ref{lem:variance-cv-ergodicity}]
First we establish some intermediate results. Choose $r(0,\epsilon):=\lceil-\ln(\epsilon)/K\rceil$
and $r(\ell,\epsilon):=\lceil-\ln(\epsilon)/(K\ell)\rceil$ for $\ell>0$,
then
\[
\begin{cases}
e^{-K\ell r(\ell,\epsilon)}, & \ell>0\\
e^{-Kr(\ell,\epsilon)}, & \ell=0
\end{cases}\leq\epsilon.
\]
From Lemma \ref{lem:poisson_continuity_etc} this implies that for
any $\ell\geq0$ such that $\gimel^{-1}\leq1-K\ell/2$
\[
\Delta_{s,r(\ell,\epsilon)}(x)\leq C\gimel\frac{\tilde{\alpha}_{p}}{K}\|f\|_{p}\sup_{s\in[0,1]}\pi_{s}\bar{V}^{(p+1/2)}\cdot\bar{V}^{(p+1/2)}(x)\epsilon,
\]
and
\[
\sup_{s,r\in[0,1]\times\mathbb{N}\cup\{\infty\}}\|g_{s,r}\|_{p+1/2}\leq C\gimel\frac{\tilde{\alpha}_{p}}{K}\|f\|_{p}\sup_{s\in[0,1]}\pi_{s}\bar{V}^{(p+1/2)}.
\]
From the homogeneous version of Lemma \ref{lem:com_relation} and
Lemma \ref{lem:drift} we have that for any $x\in\mathbb{R}^{d}$
\[
\|\nabla g_{s,r}(x)\|\leq\frac{\tilde{\alpha}_{p}}{K}\|\nabla f\|_{p}\bar{V}^{(p)}(x)\begin{cases}
\frac{K\ell}{1-e^{-K\ell}}, & \ell>0,\\
1, & \ell=0
\end{cases}.
\]
and since $\bar{V}^{(p)}(x)\leq C\bar{V}^{(p+1/2)}(x)$ we deduce
that for any $\ell\geq0$ such that $\gimel^{-1}\leq1-K\ell/2$
\[
\sup_{(s,r)\in[0,1]\times\mathbb{N}}\|g_{s,r}\|_{p}\vee\|\nabla g_{s,r}\|_{p+1/2}\leq C\gimel\frac{\tilde{\alpha}_{p}}{K}\|\nabla f\|_{p}\sup_{s\in[0,1]}\pi_{s}\bar{V}^{(p+1/2)}.
\]
Now for $r\in\mathbb{N}$
\[
\big|\Upsilon_{3,\epsilon}\big|\leq\Upsilon_{3,\epsilon,r}^{(1)}+\Upsilon_{3,\epsilon,r}^{(2)},
\]
with
\begin{align*}
\Upsilon_{3,\epsilon,r}^{(1)}:= & 2h\Big|\sum_{k=1}^{n-2}\mathbb{E}\left(f_{kh}(X_{kh}^{\epsilon})g_{kh,r}(X_{kh}^{\epsilon})\right)-\pi_{kh,r}\big(f_{kh}g_{kh,r}\big)\Big|,\\
\Upsilon_{3,\epsilon,r}^{(2)}:= & 2h\sum_{k=1}^{n-2}\mathbb{E}\left(\big|f_{kh}(X_{kh}^{\epsilon})\big|\Delta_{kh,r}(X_{kh}^{\epsilon})\right)+\pi_{kh}\big(\big|f_{kh}\big|\Delta_{kh,r}\big).
\end{align*}
Note that from above and Lemma \ref{lem:intermediateresultsonV} we
have $\|\nabla(fg_{s,r})\|_{2p+1/2}\leq4\|\nabla f\|_{p}\|g_{s,r}\|_{p+1/2}+4\|f\|_{p}\|\nabla g_{s,r}\|_{p+1/2}$,
and we deduce that for any $\ell\geq0$ and $\gimel^{-1}\leq1-Kh\epsilon^{-1}/2$
\[
\sup_{r\in\mathbb{N}}\|\nabla(fg_{r})\|_{2p+1/2}\leq C\gimel\frac{\tilde{\alpha}_{p}}{K}\vvvert f\vvvert_{p}^{2}\sup_{s\in[0,1]}\pi_{s}\bar{V}^{(p+1/2)}.
\]
From Lemma \ref{lem:Vboundonvariances} 
\begin{multline*}
\Upsilon_{3,\epsilon,r}^{(1)}\leq C\sup_{(s,r)\in[0,1]\times\mathbb{N}}\|\nabla(fg_{s,r})\|_{2p+1/2}\times\\
\times\sup_{s\in[0,1]}\pi_{s}\bar{V}^{(2[(2p+1/2)\vee p_{0}]+1/2)}\Big\{ K^{-2}\|\nabla\phi\|_{p_{0}}+\frac{\alpha_{2p+1/2}}{K}\mu_{0}\bar{V}^{(2p+1)}\frac{1}{\big[1-\exp\big(-K\epsilon^{-1}h\big)\big]/(Kh\epsilon^{-1})}\Big\}\epsilon\\
\leq C\sup_{(s,r)\in[0,1]\times\mathbb{N}}\|\nabla(fg_{s,r})\|_{2p+1/2}\sup_{s\in[0,1]}\pi_{s}\bar{V}^{(2[(2p+1/2)\vee p_{0}]+1/2)}\Big\{ K^{-2}\|\nabla\phi\|_{p_{0}}+\frac{\alpha_{2p+1/2}}{K}\mu_{0}\bar{V}^{(2p+1)}\gimel\Big\}\epsilon.
\end{multline*}
Further from the bound on $\Delta_{s,r}(\cdot)$ above
\[
\Upsilon_{3,\epsilon,r}^{(2)}\leq C\gimel\frac{\tilde{\alpha}_{p}}{K}\|f\|_{p}^{2}\sup_{s\in[0,1]}\pi_{s}\bar{V}^{(p+1/2)}\cdot\big(\alpha_{2p+1/2}\mu_{0}\bar{V}^{(2p+1/2)}+\sup_{s\in[0,1]}\pi_{s}\bar{V}^{(2p+1/2)}\big)\epsilon.
\]
As a result
\begin{align*}
\big|\Upsilon_{3,\epsilon}\big| & \leq C\gimel\frac{\tilde{\alpha}_{p}}{K}\vvvert f\vvvert_{p}^{2}\sup_{s\in[0,1]}\pi_{s}\bar{V}^{(p+1/2)}\sup_{s\in[0,1]}\pi_{s}\bar{V}^{(2[(2p+1/2)\vee p_{0}]+1/2)}\\
 & \left\{ \sup_{s\in[0,1]}\pi_{s}\bar{V}^{(2p+1/2)}+K^{-2}\|\nabla\phi\|_{p_{0}}+\mu_{0}\bar{V}^{(2p+1)}\alpha_{2p+1/2}\left(1+\frac{\gimel}{K}\right)\right\} \epsilon
\end{align*}
from which we conclude. We turn to the second statement. Note that
we have the slight simplification $\Upsilon_{5,\epsilon}=-\epsilon^{-1}h^{2}\sum_{k=1}^{n-2}\mathbb{E}\big[f_{kh}^{2}(X_{kh}^{\epsilon})\big]-{\rm var}_{\pi_{kh}}\big[f_{kh}\big],$
that
\begin{align*}
\big|\mathbb{E}\Big[f_{kh}^{2}(X_{kh}^{\epsilon})\Big]-{\rm var}_{\pi_{kh}}\big[f_{kh}\big]\big| & =\big|\mathbb{E}\Big[f_{kh}^{2}(X_{kh}^{\epsilon})\Big]-\pi_{kh}f_{kh}^{2}\big|
\end{align*}
and $f\in C_{0,2}^{p}\big(\mathbb{R}^{d}\big)$ implies that $f^{2}\in C_{0,2}^{2p}\big([0,1]\times\mathbb{R}^{d}\big)$
from Lemma \ref{lem:intermediateresultsonV}. Now from Lemma \ref{lem:Vboundonvariances},
\begin{multline*}
\big|\Upsilon_{5,\epsilon}\big|\leq C\|f\|_{2p}\|\nabla f\|_{2p}\sup_{s\in[0,1]}\pi_{s}\bar{V}^{(2[(2p)\vee p_{0}]+1/2)}\Big\{ K^{-2}\|\nabla\phi\|_{p_{0}}h+\alpha_{2p}\mu_{0}\bar{V}^{(2p+1/2)}\frac{\exp\big(-K\epsilon^{-1}h\big)}{1-\exp\big(-K\epsilon^{-1}h\big)}\epsilon^{-1}h\Big\} h.
\end{multline*}
Therefore
\[
\big|\Upsilon_{5,\epsilon}\big|\leq C\vvvert f\vvvert_{2p}^{2}\sup_{s\in[0,1]}\pi_{s}\bar{V}^{(2[(2p)\vee p_{0}]+1/2)}\Big\{ K^{-2}\|\nabla\phi\|_{p_{0}}h+\frac{\gimel\alpha_{2p}}{K}\mu_{0}\bar{V}^{(2p+1/2)}\Big\} h.
\]
\[
\]

\[
\]
\end{proof}
\begin{proof}[Proof of Lemma \ref{lem:riemannconvergence}]
For $\Upsilon_{4,\epsilon}$, with the notation of Lemma \ref{lem:poisson_continuity_etc},
we introduce for $r\in\mathbb{N}$
\begin{align*}
\Upsilon_{4,\epsilon,r}^{(1)}:= & 2\left|h\Bigl\{\sum_{k=1}^{n-2}\pi_{kh}\big(f_{kh}g_{kh,r}\big)\Bigr\}-\int_{0}^{1}\pi_{s}(f_{s}g_{s,r}){\rm d}s\right|,\\
\Upsilon_{4,\epsilon,r}^{(2)}:= & 2\left|h\Bigl\{\sum_{k=1}^{n-2}\pi_{kh}\big(f_{kh}\Delta_{kh,r}\big)\Bigr\}-\int_{0}^{1}\pi_{s}(f_{s}\Delta_{kh,r}){\rm d}s\right|.
\end{align*}
From the rough upper bound on $\Delta_{s,r}$ in Lemma \ref{lem:poisson_continuity_etc}
and with $r(0,\epsilon):=\lfloor-\ln(\epsilon)/K\rfloor$ or $r(\ell,\epsilon):=\lfloor-\ln(\epsilon)/(\ell K)\rfloor$
for $\ell>0$, we have
\begin{align*}
\Upsilon_{4,\epsilon,r(\ell,\epsilon)}^{(2)} & \leq C\gimel\|f\|_{p}\frac{\tilde{\alpha}_{p}}{K}\sup_{s\in[0,1]}\pi_{s}\bar{V}^{(p+1/2)}\cdot\sup_{s\in[0,1]}\pi_{s}\bigl(f_{s}\bar{V}^{(p+1/2)}\bigr)\epsilon\\
 & \leq C\gimel\|f\|_{p}^{2}\frac{\tilde{\alpha}_{p}}{K}\sup_{s\in[0,1]}\pi_{s}\bar{V}^{(p+1/2)}\cdot\sup_{s\in[0,1]}\pi_{s}\bigl(\bar{V}^{(2p+1/2)}\bigr)\epsilon.
\end{align*}
For the other terms we note that from Lemma \ref{lem:poisson_continuity_etc}
for $|s-t|\leq R_{f}=1$ and any $\zeta\in(0,1)$
\begin{align*}
\bigl|f_{s}(x)g_{s,r}(x) & -f_{t}(x)g_{t,r}(x)\bigr|\leq\bigl|f_{s}(x)-f_{t}(x)\bigr|\cdot\bigr|g_{s,r}(x)\bigr|+\bigl|f_{t}(x)\bigr|\bigl|g_{s,r}(x)-g_{t,r}(x)\bigr|\\
 & \leq C\frac{\gimel\tilde{\alpha}_{p}}{K}\|f\|_{p}\bar{V}^{(p+1/2)}(x)\sup_{s\in[0,1]}\pi_{s}\bar{V}^{(p+1/2)}\bigl|f_{s}(x)-f_{t}(x)\bigr|+\|f\|_{p}\bar{V}^{(p)}(x)\bigl|g_{s,r}(x)-g_{t,r}(x)\bigr|
\end{align*}
Now, since $f\in C_{1,2}^{p}\big([0,1]\times\mathbb{R}^{d}\big)$,
\[
\bigl|f_{s}(x)-f_{t}(x)\bigr|\leq\|\partial f\|_{p}\bar{V}^{(p)}(x)|s-t|
\]
and from Lemma \ref{lem:poisson_continuity_etc}
\begin{multline*}
|g_{s,r}(x)-g_{u,r}(x)|\leq C(\gimel,\zeta)|s-u|^{\zeta}\frac{\tilde{\alpha}_{p}(\ell\vee1)}{(1\wedge K)\ell}\big(1\vee\vvvert f\vvvert_{p}\big)\\
\times\Bigl(1+\tilde{\alpha}_{p}\frac{M}{K}\sup_{\tau\in[0,1]}\sqrt{\bar{V}(x_{\tau}^{\star})}+\sup_{s\in[0,1]}\pi_{s}\bar{V}^{p+1/2}\Bigr).
\end{multline*}
Hence, using Lemma \ref{lem:W(delta_x,nu)}
\begin{align*}
\bigl|f_{s}(x) & g_{s,r}(x)-f_{t}(x)g_{t,r}(x)\bigr|\\
 & \leq C(1+\vvvert f\vvvert_{p})^{2}\bar{V}^{(2p+1/2)}(x)\tilde{\alpha}_{p}\sup_{s\in[0,1]}\pi_{s}\bar{V}^{(p+1/2)}\left\{ \frac{\gimel}{K}+\frac{C(\gimel,\zeta)(\ell\vee1)}{(1\wedge K)\ell}\Bigl(2+\tilde{\alpha}_{p}\frac{M}{K}\sup_{\tau\in[0,1]}\sqrt{\bar{V}(x_{\tau}^{\star})}\Bigr)\right\} |s-t|^{\zeta}
\end{align*}
where $C(\gimel,\xi)$ depends on the arguments shown only. Now, defining
\[
C_{fg}:=(1+\vvvert f\vvvert_{p})^{2}\tilde{\alpha}_{p}\sup_{s\in[0,1]}\pi_{s}\bar{V}^{(p+1/2)}\left\{ \frac{\gimel}{K}+\frac{C(\gimel,\zeta)(\ell\vee1)}{(1\wedge K)\ell}\Bigl(2+\tilde{\alpha}_{p}\frac{M}{K}\sup_{\tau\in[0,1]}\sqrt{\bar{V}(x_{\tau}^{\star})}\Bigr)\right\} 
\]
from Lemma \ref{lem:rieman_approx}
\[
\Upsilon_{4,\epsilon,r}^{(1)}\leq Ch^{\zeta}\tilde{\alpha}_{2p+1/2}\big(C_{fg}\vee\|\nabla fg\|_{2p+1/2}\big)\left[1+\tilde{\alpha}_{2p+1/2}\frac{M}{K}\sup_{s\in[0,1]}\sqrt{\bar{V}(x_{s}^{\star})}\right]
\]
and we have found in the proof of Lemma \ref{lem:variance-cv-ergodicity}
that
\[
\sup_{r\in\mathbb{N}}\|\nabla(fg_{r})\|_{2p+1/2}\leq C\gimel\frac{\tilde{\alpha}_{p}}{K}\vvvert f\vvvert_{p}^{2}\sup_{s\in[0,1]}\pi_{s}\bar{V}^{(p+1/2)},
\]
from which the first bound follows. For $\Upsilon_{6,\epsilon}$,
first consider $\ell=0$. In this case
\[
\big|\Upsilon_{6,\epsilon}\big|\leq h\epsilon^{-1}\sup_{s\in[0,1]}{\rm var}_{\pi_{s}}\big(f_{s}\big)
\]
 Now consider $\ell>0$, we apply Lemma \ref{lem:rieman_approx} with
the function $f^{2}$ to obtain the result. By assumptions $f\in C_{1,2}^{p}\big([0,1]\times\mathbb{R}^{d}\big)$
implies that
\[
\bigl|f_{s}(x)-f_{t}(x)\bigr|\leq\vvvert f\vvvert_{p}\bar{V}_{p}(x)|s-t|
\]
and consequently
\begin{align*}
\bigl|f_{s}^{2}(x)-f_{t}^{2}(x)\bigr| & \leq2\vvvert f\vvvert_{p}^{2}\big[\bar{V}^{(p)}(x)\big]^{2}|s-t|\\
 & \leq C\vvvert f\vvvert_{p}^{2}\bar{V}^{(2p)}(x)|s-t|
\end{align*}
and by application of Lemma \ref{lem:rieman_approx} we deduce
\[
\big|\Upsilon_{6,\epsilon}\big|\leq C\ell h\vvvert f\vvvert_{p}^{2}\tilde{\alpha}_{2p}\sup_{s\in[0,1]}\pi_{s}\bar{V}^{(p)}\left[1+(\tilde{\alpha}_{2p}\frac{M}{K}\sup_{s\in[0,1]}\sqrt{\bar{V}(x_{s}^{\star})}\right].
\]
\end{proof}
\begin{proof}[Proof of Lemma \ref{lem:S0andS7}]
First we have the simplification
\begin{align*}
\Upsilon_{0,\epsilon} & =\epsilon^{-1}h^{2}\sum_{k=1}^{n-2}\pi_{kh}\bar{f}_{kh,\epsilon}\mathbb{E}\big[2\gamma_{k,\epsilon}(X_{kh}^{\epsilon})-\bar{f}_{kh,\epsilon}(X_{kh}^{\epsilon})\big]\\
 & =2\epsilon^{-1}h^{2}\sum_{k=1}^{n-2}\pi_{kh}\bar{f}_{kh,\epsilon}\mathbb{E}\big[\gamma_{k,\epsilon}(X_{kh}^{\epsilon})\big],
\end{align*}
and from Lemma \ref{lem:Poisson-and-martingale}, 
\begin{align*}
\big|\mathbb{E}\left(\gamma_{k,\epsilon}(X_{kh}^{\epsilon})\right)\big| & \leq C\alpha_{p}\frac{\|\nabla f\|_{p}}{1-\exp\big(-K\epsilon^{-1}h\big)}\mu_{0}\bar{V}^{(p+1/2)}\cdot\sup_{s\in[0,1]}\mu_{s}\bar{V}^{(p+1/2)}(x)\\
 & \leq C\alpha_{p}\alpha_{p+1/2}\frac{\|\nabla f\|_{p}}{1-\exp\big(-K\epsilon^{-1}h\big)}\mu_{0}\bar{V}^{(p+1/2)}\cdot\mu_{0}\bar{V}^{(p+1/2)}
\end{align*}
where we have used Lemma \ref{lem:drift} on the last line. Further
from Lemma \ref{lem:Vboundonvariances} 
\[
\big|\pi_{t}\bar{f}_{t,\epsilon}\big|=|\mathbb{E}\big[f_{t}(X_{t}^{\epsilon})\big]|\leq C\|\nabla f\|_{p}\cdot\sup_{s\in[0,1]}\pi_{s}\bar{V}^{(2[p\vee p_{0}]+1/2)}\Big\{ K^{-2}\|\nabla\phi\|_{p_{0}}\epsilon+\alpha_{p}\mu_{0}\bar{V}^{(p+1/2)}\exp\big(-K\epsilon^{-1}t\big)\Big\}
\]
and therefore
\[
2\epsilon^{-1}h^{2}\sum_{k=1}^{n-2}\big|\pi_{kh}\bar{f}_{kh,\epsilon}\big|\leq C\|\nabla f\|_{p}\cdot\sup_{s\in[0,1]}\pi_{s}\bar{V}^{(2[p\vee p_{0}]+1/2)}\Big\{ K^{-2}\|\nabla\phi\|_{p_{0}}h+\alpha_{p}\mu_{0}\bar{V}^{(p+1/2)}\frac{\exp\big(-K\epsilon^{-1}h\big)}{1-\exp\big(-K\epsilon^{-1}h\big)}\epsilon^{-1}h^{2}\Big\}
\]
We have
\[
\Upsilon_{7,\epsilon}:=\epsilon^{-1}h^{2}\mathbb{E}\left[\bar{f}_{(n-1)h,\epsilon}^{2}(X_{(n-1)h}^{\epsilon})-\big[P_{0,h}\gamma_{1,\epsilon}(X_{0}^{\epsilon})\big]^{2}\right].
\]
Notice that
\begin{align*}
\epsilon^{-1}h\mathbb{E}\left[\big[P_{0,h}\gamma_{1,\epsilon}(X_{0}^{\epsilon})\big]^{2}\right]^{1/2} & =\epsilon^{-1}h\mathbb{E}\left[\big(\sum_{i=1}^{n-1}P_{0,ih}f_{ih}(X_{0}^{\epsilon})\big)^{2}\right]^{1/2}\\
 & \leq\epsilon^{-1}h\sum_{i=1}^{n-1}{\rm var}_{\mu_{0}}\big(P_{0,ih}f_{ih}\big)^{1/2}\\
 & \leq\frac{\epsilon^{-1}h}{1-\exp(-Kh\epsilon^{-1})}\sup_{s\in[0,1]}{\rm var}_{\mu_{s}}\big(f_{s}\big)^{1/2}\\
 & \leq C\frac{\gimel}{K}\|\nabla f\|\big\{\alpha_{2p}\big[K^{-1}+K_{\mu_{0}}\big]\mu_{0}\bar{V}^{(2p)}\big\}^{1/2}.
\end{align*}
We conclude by using that 
\[
\mathbb{E}\left[\bar{f}_{(n-1)h,\epsilon}^{2}(X_{(n-1)h}^{\epsilon})\right]\leq\sup_{s\in[0,1]}{\rm var}_{\mu_{s}^{\epsilon}}\big(f_{s}\big)\leq C\|\nabla f\|^{2}\big\{\alpha_{2p}\big[K^{-1}+K_{\mu_{0}}\big]\mu_{0}\bar{V}^{(2p)}\big\}.
\]
\end{proof}

\subsection{Some tractable bounds}

We gather here intermediate technical results which lead to tractable
bounds and allow us to conclude about the complexity of the procedure.
For the reader's convenience we recall that for $q>0$ and $x\in\mathbb{R}^{d}$,
$V(x):=\|x\|^{2}$, $V^{(q)}:=V^{q}$, $\bar{V}^{(q)}:=1+V^{(q)}$,
with $t\in[0,1]$ $V_{t}(x):=\|x-x_{t}^{\star}\|^{2}$, $V_{t}^{(q)}:=V_{t}^{q}$,
$\bar{V}_{t}^{(q)}:=1+V_{t}^{(q)}$ (with notational simplifications
$\bar{V}_{t}:=\bar{V}_{t}^{(1)}$ and $V_{t}:=V_{t}^{(1)}$ etc.)
and for $\nu\in\mathcal{P}^{q+1/2}(\mathbb{R}^{d})$
\begin{align*}
W^{(q)}(\delta_{x},\nu):= & \int_{\mathbb{R}^{d}}\left(1+\|x\|^{2q}\vee\|y\|^{2q}\right)\|x-y\|\nu(\mathrm{d}y).
\end{align*}
\begin{lem}
\label{lem:W(delta_x,nu)}For any $p\geq1$ and $\nu\in\mathcal{P}^{p+1/2}(\mathbb{R}^{d})$,
\[
W^{(p)}(\delta_{x},\nu)\leq V^{p+1/2}(x)+V^{p}(x)\nu(V^{1/2})+V^{1/2}(x)[1+\nu(V^{p})]+\nu(V^{p+1/2}),\quad x\in\mathbb{R}^{d},
\]
and as a result
\[
\sup_{x\in\mathbb{R}^{d}}\frac{W^{(p)}(\delta_{x},\nu)}{1+\|x\|^{2p+1}}<+\infty.
\]
Further there exists $C>0$ such that for any $x\in\mathbb{R}^{d}$
and $\nu\in\mathcal{P}^{p+1/2}(\mathbb{R}^{d})$ 
\begin{equation}
W^{(p)}(\delta_{x},\nu)\leq C\nu\bar{V}^{(p+1/2)}\cdot\bar{V}^{(p+1/2)}(x).\label{eq:bounds_on_W_in_terms_of_V}
\end{equation}
\end{lem}

\begin{proof}
By considering the scenarios $\|x\|\leq\|y\|$ and $\|x\|>\|y\|$
separately we have
\begin{align*}
W^{(p)}(\delta_{x},\nu) & \leq\|x\|+\nu(V^{1/2})+\|x\|^{2p+1}+\|x\|^{2p}\nu(V^{1/2})+\|x\|\nu(V^{p})+\nu(V^{p+1/2}),\\
 & =\|x\|^{2p+1}+\|x\|^{2p}\nu(V^{1/2})+\|x\|[1+\nu(V^{p})]+\nu(V^{p+1/2}),
\end{align*}
and the first statement follows from the assumption on $\nu$. Finally
by considering the scenarios $V(x)\geq1$ and $V(x)<1$ separately
twice one shows that 
\begin{align*}
W^{(p)}(\delta_{x},\nu) & \leq2\big[1+V^{p+1/2}(x)\big]\big[1+\nu\big(V^{1/2}+V^{p}+V^{p+1/2}\big)\big],\\
 & \leq8\nu\bar{V}^{(p+1/2)}\cdot\bar{V}^{(p+1/2)}(x).
\end{align*}
\end{proof}
\begin{lem}
\label{lem:intermediateresultsonV}For any $p\geq0$, 
\begin{enumerate}
\item for any $q\geq0$ and $x\in\mathbb{R}^{d}$
\begin{align*}
\bar{V}^{(p)}(x)\bar{V}^{(q)}(x) & \leq4\cdot\bar{V}^{(p+q)}(x),
\end{align*}
\[
V^{(p)}(x)\vee V^{(q)}(x)\leq2\cdot V^{(p\vee q)}(x),
\]
for any $q\geq1$
\[
\big[\bar{V}^{(p)}(x)\big]^{q}\leq2^{q-1}\bar{V}^{(qp)}(x),
\]
and for $\varphi,\psi\in C^{p}\big(\mathbb{R}^{d}\big)\times C^{q}\big(\mathbb{R}^{d}\big)$
for $p,q\geq1$
\[
\|\varphi\psi\|_{p+q}\leq4\|\varphi\|_{p}\|\psi\|_{q}
\]
\item for any $s\in[0,1]$ and $x\in\mathbb{R}^{d}$,
\[
\sqrt{\bar{V}_{s}(x)}\bar{V}^{(p)}(x)\leq\sqrt{12}\bar{V}(x_{s}^{\star})^{1/2}\bar{V}^{(p+1/2)}(x)
\]
\end{enumerate}
\end{lem}

\begin{proof}
First we have $\bar{V}^{(p)}(x)\bar{V}^{(q)}(x)\leq4\bar{V}^{(p+q)}(x)$
because $\bar{V}^{(p)}(x)\bar{V}^{(q)}(x)=1+\|x\|^{2(p+q)}+\|x\|^{2q}+\|x\|^{2p}$
and one can consider the scenarios $\|x\|\geq1$ and $\|x\|<1$ separately.
For the second statement one can again consider the scenarios $\|x\|\geq1$
and $\|x\|<1$. For the third statement, the result follows from Jensen's
inequality,
\[
\big[1+\|x\|^{2p}\big]^{q}\leq2^{q}\frac{1+\|x\|^{2pq}}{2}.
\]
The next statement follows from
\[
\frac{\varphi(x)\psi(x)}{\bar{V}^{(p+q)}(x)}=\frac{\varphi(x)\psi(x)}{\bar{V}^{(p)}(x)\bar{V}^{(q)}(x)}\frac{\bar{V}^{(p)}(x)\bar{V}^{(q)}(x)}{\bar{V}^{(p+q)}(x)}
\]
and our first result above. Now we note that for $z\geq0$ and $C>0$
\begin{align*}
A(z):=(C+z)(1+z^{p})^{2} & =z^{2p+1}+Cz^{2p}+2[z^{p+1}+Cz^{p}]+z+C\\
B(z):=(1+z^{p+1/2})^{2} & =z^{2p+1}+2z^{p+1/2}+1
\end{align*}
are such that for $z\geq1$ $A(z)\leq z^{2p+1}[1+C+2(1+C)+1+C]$ and
for $z\leq1$ $A(z)\leq[1+C+2(1+C)+1+C]$, and therefore for $z\geq0$
\begin{align*}
A(z) & \leq4(1+z^{2p+1})[1+C]\\
 & \leq4(1+C)B(z)
\end{align*}
as a consequence with $C=1/2+\|x_{s}^{\star}\|^{2}$ and $z=\|x\|^{2}$
we deduce that (with $\|x-x_{s}^{\star}\|^{2}\leq2[\|x\|^{2}+\|x_{t}^{\star}\|^{2})]$)
\[
\sqrt{2}\sqrt{1/2+1/2\|x-x_{s}^{\star}\|^{2}}(1+\|x\|^{2p})\leq\sqrt{8(1+1/2+\|x_{s}^{\star}\|^{2})}\big(1+\|x\|^{2p+1}\big)
\]
that is
\[
\sqrt{\bar{V}_{s}(x)}\cdot\bar{V}^{(p+1/2)}(x)\leq\sqrt{12\bar{V}(x_{s}^{\star})}\cdot\bar{V}^{(p+1/2)}(x)
\]
\end{proof}
\begin{lem}
\label{lem:Vboundonvariances}$\;$
\begin{enumerate}
\item There exists $C>0$ such that for any $p\geq1$, $\nu\in\mathcal{P}^{2p}(\mathbb{R}^{d})$
such that there exists a constant $K_{\nu}>0$ such that for all $f\in C_{2}^{p}(\mathbb{R}^{d})$
\[
{\rm var}_{\nu}[f]\leq K_{\nu}^{-1}\nu\big(\|\nabla f\|^{2}\big),
\]
then for any $f\in C_{2}^{p}(\mathbb{R}^{d})$ and $\epsilon>0$
\begin{align*}
\sup_{0\leq s\leq t\leq1}{\rm var}_{\nu P_{s,t}^{\epsilon}}\big[f\big] & \leq C\alpha_{2p}\cdot\|\nabla f\|_{p}^{2}\cdot\left[K^{-1}+K_{\nu}^{-1}\right]\nu(\bar{V}^{(2p)})\\
\sup_{(s,t)\in[0,1]\times\mathbb{R}_{+}}{\rm var}_{\nu Q_{t}^{s,\epsilon}}\big[f\big] & \leq C\tilde{\alpha}_{2p}\cdot\|\nabla f\|_{p}^{2}\cdot\left[K^{-1}+K_{\nu}^{-1}\right]\nu(\bar{V}^{(2p)})
\end{align*}
where $\alpha_{2p}$ and $\tilde{\alpha}_{2p}$ are given in Lemma
\ref{lem:drift} and \ref{lem:drift_Y_process} respectively.
\item There exists $C>0$ such that for any $\phi_{t}$ as in (\ref{eq:phi_defn}),
\[
\sup_{t\in[0,1]}{\rm var}_{\pi_{t}}[\phi_{t}]\leq CK^{-1}\|\nabla\phi\|_{p_{0}}^{2}\cdot\sup_{t\in[0,1]}\pi_{t}\big(\bar{V}^{(2p_{0})}\big).
\]
\item Let $p\geq1$, then for any $f\in C_{2}^{p}(\mathbb{R}^{d})$ 
\begin{multline*}
|\mathbb{E}\big[f_{t}(X_{t}^{\epsilon})\big]|\leq\sup_{s\in[0,1]}\mathrm{var}_{\pi_{s}}[\phi_{s}]^{1/2}\sup_{s\in[0,1]}\mathrm{var}_{\pi_{s}}[f_{s}]^{1/2}\frac{\epsilon}{K}\left[1-\exp(-K\epsilon^{-1}t)\right]\\
+\alpha_{p}\|\nabla f_{t}\|_{p}W^{(p)}(\mu_{0},\pi_{0})\exp\big(-K\epsilon^{-1}t\big)
\end{multline*}
and a rough bound is
\[
|\mathbb{E}\big[f_{t}(X_{t}^{\epsilon})\big]|\leq C\|\nabla f\|_{p}\cdot\sup_{s\in[0,1]}\pi_{s}\bar{V}^{(2[p\vee p_{0}]+1/2)}\Big\{ K^{-2}\|\nabla\phi\|_{p_{0}}\epsilon+\alpha_{p}\mu_{0}\bar{V}^{(p+1/2)}\exp\big(-K\epsilon^{-1}t\big)\Big\}.
\]
\end{enumerate}
\end{lem}

\begin{cor}
As a consequence for $t\in[0,1]$
\begin{align*}
{\rm var}_{\mu_{t}^{\epsilon}}\big[f\big] & \leq C\alpha_{2p}\cdot\|\nabla f\|_{p}^{2}\left[K^{-1}+K_{\mu_{0}}^{-1}\right]\mu_{0}(\bar{V}^{(2p)})
\end{align*}
and using Lemmas \ref{lem:poincare_transfer} and \ref{lem:drift}
for any $(s,t)\in[0,1]\times\mathbb{R}_{+}$ 
\begin{align*}
{\rm var}_{\mu_{s}Q_{t}^{s,\epsilon}}\big[f\big] & \leq C\tilde{\alpha}_{2p}\alpha_{2p}\cdot\|\nabla f\|_{p}^{2}\left[K^{-1}+K_{\mu_{0}}^{-1}\right]\mu_{s}(\bar{V}^{(2p)})
\end{align*}
and
\[
\]
\end{cor}

\begin{proof}
We first apply Lemma \ref{lem:poincare_transfer}, yielding for $0\leq s\leq t\leq1$
\begin{align*}
{\rm var}_{\nu P_{s,t}^{\epsilon}}\big[f\big] & \leq\left[K^{-1}+K_{\nu}^{-1}\right]\cdot\nu P_{s,t}\big(\|\nabla f\|^{2}\big)\\
 & \leq\left[K^{-1}+K_{\nu}^{-1}\right]\|\nabla f\|_{p}^{2}\cdot\nu P_{s,t}\big([\bar{V}^{(p)}]^{2}\big).
\end{align*}
Now we apply (\ref{eq:drift_bound_uniform}) in Lemma \ref{lem:drift}
and Lemma \ref{lem:intermediateresultsonV} to conclude. We proceed
similarly for the time homogeneous scenario and Lemma \ref{lem:drift_Y_process}.
We use Remark \ref{rem:Poincare_pi} noting the fact, established
in the proof of Lemma \ref{lem:deriv_int_interchange}, that $\phi_{t}\in C_{0,2}^{p_{0}}([0,1]\times\mathbb{R}^{d})$.
As a result for $t\in[0,1]$ we have
\begin{align*}
{\rm var}_{\pi_{t}}[\phi_{t}] & \leq K^{-1}\pi_{t}\big(\|\nabla\phi_{t}\|^{2}\big)\\
 & \leq K^{-1}\|\nabla\phi_{t}\|_{p_{0}}^{2}\pi_{t}\big([\bar{V}^{(p_{0})}]^{2}\big),
\end{align*}
and we conclude with Lemma \ref{lem:intermediateresultsonV}. For
the bias, we note that for $t\in[0,1]$
\[
\mathbb{E}[f_{t}(X_{t})]=\mu_{0}P_{0,t}f_{t}=\pi_{0}P_{0,t}f_{t}-\pi_{t}f_{t}+(\mu_{0}-\pi_{0})P_{0,t}f_{t},
\]
and by Lemmas \ref{lem:bias_nu_bar_nu} and \ref{lem:bias_pi_0_pi_t},
we deduce
\begin{align*}
| & \mathbb{E}\big[f_{t}(X_{t}^{\epsilon})\big]|\leq\sup_{s\in[0,1]}\mathrm{var}_{\pi_{s}}[\phi_{s}]^{1/2}\sup_{s\in[0,1]}\mathrm{var}_{\pi_{s}}[f_{s}]^{1/2}\frac{\epsilon}{K}\left[1-\exp(-K\epsilon^{-1}t)\right]+\alpha_{p}\|\nabla f_{t}\|_{p}W^{(p)}(\mu_{0},\pi_{0})\exp\big(-K\epsilon^{-1}t\big).
\end{align*}
We can now apply our earlier result and Remark \ref{rem:Poincare_pi}
to show, 
\begin{multline*}
\sup_{s\in[0,1]}\mathrm{var}_{\pi_{s}}[\phi_{s}]^{1/2}\sup_{s\in[0,1]}\mathrm{var}_{\pi_{s}}[f_{s}]^{1/2}\frac{\epsilon}{K}\left[1-\exp(-K\epsilon^{-1}t)\right]\\
\leq CK^{-2}\|\nabla\phi\|_{p_{0}}\|\nabla f\|_{p}\cdot\Big\{\sup_{s\in[0,1]}\pi_{s}\bar{V}^{(2p_{0})}\cdot\sup_{s\in[0,1]}\pi_{s}\bar{V}^{(2p)}\Big\}^{1/2}\epsilon\\
\leq CK^{-2}\|\nabla\phi\|_{p_{0}}\|\nabla f\|_{p}\cdot\sup_{s\in[0,1]}\pi_{s}\bar{V}^{(2[p_{0}\vee p])}\epsilon
\end{multline*}
and from Lemma \ref{lem:W(delta_x,nu)}
\begin{align*}
\alpha_{p}\|\nabla f_{t}\|_{p}W^{(p)}(\mu_{0},\pi_{0})\exp\big(-K\epsilon^{-1}t\big) & \leq C\alpha_{p}\|\nabla f\|_{p}\mu_{0}\bar{V}^{(p+1/2)}\cdot\pi_{0}\bar{V}^{(p+1/2)}\exp\big(-K\epsilon^{-1}t\big)
\end{align*}
from which we deduce 
\[
|\mathbb{E}\big[f_{t}(X_{t}^{\epsilon})\big]|\leq C\|\nabla f\|_{p}\sup_{s\in[0,1]}\pi_{s}\bar{V}^{(2[p\vee p_{0}]+1/2)}\Big\{ K^{-2}\|\nabla\phi\|_{p_{0}}\epsilon+\alpha_{p}\mu_{0}\bar{V}^{(p+1/2)}\exp\big(-K\epsilon^{-1}t\big)\Big\}.
\]
\end{proof}
\begin{lem}
\label{lem:boundoneminusexponentialetc} For $0\leq z<2$
\[
\frac{z}{1-\exp(-z)}\leq\frac{1}{1-z/2}
\]
\end{lem}

\begin{proof}
We have that for $z\geq0$ $\exp(-z)\leq1-z+\frac{z^{2}}{2}$, which
implies $[1-\exp(-z)]/z\geq1-z/2$ and therefore the result.
\end{proof}

%% file: child-poisson.tex
\section{Drift and solution of Poisson's equation for the time-homogeneous
diffusions\label{sec:Solution-of-Poisson's}}

Throughout section \ref{sec:Solution-of-Poisson's} the notational
conventions of section \ref{sec:Quantitative-CLT-bound} are in force,
except that $f_{t}$ is not assumed centred with respect to $\pi_{t}$,
and we write $\bar{f}_{t}:=f_{t}-\pi_{t}f_{t}$ (which should not
be confused with $\bar{f}_{t,\epsilon}$).
\begin{lem}
\label{lem:drift_Y_process}For any $\epsilon>0$, $p\geq1$ and $\kappa\in(0,Kp)$,
define

\begin{align*}
\delta & \coloneqq\epsilon^{-1}(Kp-\kappa),\\
\tilde{r} & \coloneqq\sqrt{\frac{4p(p-1)+2pd}{\kappa}}\\
\tilde{b} & \coloneqq2p\tilde{r}^{2(p-1)}\frac{2(p-1)+d}{\epsilon}\\
\tilde{\alpha}_{p} & \coloneqq2^{4p-2}\vee\left[1+2^{2p-1}\left(\frac{2p\tilde{r}^{2(p-1)}}{(Kp-\kappa)}\left[2(p-1)+d\right]+(1+2^{2p-1})\sup_{t\in[0,1]}\|x_{t}^{\star}\|^{2p}\right)\right]
\end{align*}
Then

\begin{align}
Q_{t}^{s,\epsilon}(V_{s}^{p})(x) & \leq e^{-\delta t}V_{s}^{p}(x)+\frac{\tilde{b}}{\delta}(1-e^{-\delta t}),\quad\forall(s,t)\in[0,1]\times\mathbb{R}_{+},\label{eq:drift_homogeneous}\\
\sup_{(s,t)\in[0,1]\times\mathbb{R}_{+}}Q_{t}^{s,\epsilon}\bar{V}^{(p)}(x) & \leq\tilde{\alpha}_{p}\bar{V}^{(p)}(x).\nonumber 
\end{align}
\end{lem}

\begin{proof}
The result follows by almost identical arguments to those in the proof
of Lemma \ref{lem:drift}, with some elementary simplifications afforded
by the time-homogeneity of the process $Y_{t}^{s,\epsilon}$.
\end{proof}
\begin{lem}
\label{lem:poisson_continuity_etc}Let $p\geq1$ and $f\in C_{0,2}^{p}([0,1]\times\mathbb{R}^{d})$
such that for constants $C_{f}<+\infty$, $R_{f}\in(0,1]$ and $\beta\in(0,1]$
\begin{equation}
|s-u|\leq R_{f}\quad\Rightarrow\quad|f_{s}(x)-f_{u}(x)|\leq C_{f}|s-u|^{\beta}\bar{V}^{(p)}(x),\quad\forall x\in\mathbb{R}^{d},\label{eq:f_lipschitz_hyp}
\end{equation}
and define for any $s\in[0,1]$ and $r\in\mathbb{N}\cup\{\infty\}$,
\[
g_{s,r}(x)\coloneqq\begin{cases}
\sum_{k=0}^{r}\ell Q_{k\ell}^{s}\bar{f}_{s}(x), & \quad\mathrm{if}\quad\ell>0,\\
\int_{0}^{r}Q_{t}^{s}\bar{f}_{s}(x)\mathrm{d}t, & \quad\mathrm{if}\quad\ell=0.
\end{cases}
\]
Then, with $\tilde{\alpha}_{p,1}$ as in Lemma \ref{lem:drift_Y_process}
with there $\epsilon=1$, 
\begin{enumerate}
\item we have
\begin{equation}
|s-u|\leq R_{f}\quad\Rightarrow\quad|\pi_{s}f_{s}-\pi_{u}f_{u}|\leq C|s-u|^{\beta}\tilde{\alpha}_{p}\big(C_{f}\vee\|\nabla f\|_{p}\big)\left[1+\tilde{\alpha}_{p,1}\frac{M}{K}\sup_{\tau\in[0,1]}\sqrt{\bar{V}(x_{\tau}^{\star})}\right]\quad\forall x\in\mathbb{R}^{d}.\label{eq:pi_s_f_pi_t_f}
\end{equation}
\item $g_{s,r}(\cdot)$ has the following properties:
\begin{enumerate}
\item \label{enu:enu:poisson_continuity_etc:classgsr}for any $\ell\geq0$,
$s\in[0,1]$ and $r<\infty$, the map $x\mapsto g_{s,r}(x)$ is a
member of $C_{2}^{p}(\mathbb{R}^{d})$,
\item \label{enu:poisson_continuity_etc:fboundgsr}for any $\ell\geq0$,
$s\in[0,1]$ and $r\in\mathbb{N}\cup\{\infty\}$,
\[
|g_{s,r}(x)|\leq\begin{cases}
h\epsilon^{-1}\|f\|_{p}\tilde{\alpha}_{p}W^{(p)}(\delta_{x},\pi_{s})\frac{1}{1-e^{-Kh\epsilon^{-1}}}, & \ell>0,\\
\|f\|_{p}\tilde{\alpha}_{p}W^{(p)}(\delta_{x},\pi_{s})\frac{1}{K}, & \ell=0
\end{cases}
\]
and further for any $\gimel>1$ and $\gimel^{-1}\leq1-Kh\epsilon^{-1}/2$
we have the simplified upper bound
\[
\sup_{(r,s)\in\mathbb{N}\cup\{\infty\}\times[0,1]}\|g_{s,r}\|_{p+1/2}\leq C\gimel\frac{\tilde{\alpha}_{p,1}}{K}\|f\|_{p}\sup_{s\in[0,1]}\pi_{s}\bar{V}^{(p+1/2)}.
\]
\item \label{enu:poisson_continuity_etc:boundgsrminusgsinfty}for any $s\in[0,1]$
, $r\in\mathbb{N}\cup\{\infty\}$ and $x\in\mathbb{R}^{d}$, 
\[
\Delta_{s,r}(x):=|g_{s,\infty}(x)-g_{s,r}(x)|\leq\begin{cases}
\ell\|f\|_{p}\tilde{\alpha}_{p}W^{(p)}(\delta_{x},\pi_{s})\frac{e^{-Kh\epsilon^{-1}r}}{1-e^{-Kh\epsilon^{-1}}}, & \ell>0,\\
\|f\|_{p}\tilde{\alpha}_{p}W^{(p)}(\delta_{x},\pi_{s})\frac{e^{-Kr}}{K}, & \ell=0.
\end{cases}
\]
and further for any $\gimel>1$ and $\gimel^{-1}\leq1-Kh\epsilon^{-1}/2$
we have the simplified upper bound
\[
\sup_{(r,s)\in\mathbb{N}\cup\{\infty\}\times[0,1]}\|\Delta_{s,r}\|_{p+1/2}\leq C\gimel\frac{\tilde{\alpha}_{p,1}}{K}\|f\|_{p}\sup_{s\in[0,1]}\pi_{s}\bar{V}^{(p+1/2)}\begin{cases}
e^{-Kh\epsilon^{-1}r}, & \ell>0,\\
e^{-Kr}, & \ell=0.
\end{cases}
\]
\item \label{enu:poisson_continuity_etc:gsrlipschitzcty}for any $\zeta\in(0,\beta)$
there exists $C>0$  such that for any $\gimel>1$, $\gimel^{-1}\leq1-K\ell/2$
if $\ell>0$, $r\in\mathbb{N}\cup\{\infty\}$ and $x\in\mathbb{R}^{d}$,$|s-u|\leq R_{f}$
\begin{multline*}
|g_{s,r}(x)-g_{u,r}(x)|\leq C(\beta,\gimel,R_{f},\zeta)|s-u|^{\zeta}\frac{\tilde{\alpha}_{p,1}(\ell\vee1)}{(1\wedge K)\ell}\big(C_{f}\vee\vvvert f\vvvert_{p}\big)\\
\times\Bigl(1+\tilde{\alpha}_{p,1}\frac{M}{K}\sup_{\tau\in[0,1]}\sqrt{\bar{V}(x_{\tau}^{\star})}+\sup_{s\in[0,1]}\pi_{s}\bar{V}^{p+1/2}\Bigr).
\end{multline*}
where $C(\beta,\gimel,R_{f},\zeta)$ depends only on the arguments
shown and the convention that $(\ell\vee1)/\ell=1$ for $\ell=0$.
\end{enumerate}
\end{enumerate}
\end{lem}

\begin{proof}
Consider for arbitrary $s,u\in[0,1]$, $x\in\mathbb{R}^{d}$, and
$t>0$, the decomposition $\pi_{s}f_{s}-\pi_{u}f_{u}=R_{1}(t,x)+R_{2}(t,x)+R_{3}(t,x)$,
where
\begin{align*}
R_{1}(t,x) & \coloneqq\pi_{s}f_{s}-Q_{t}^{s}f_{s}(x)+Q_{t}^{u}f_{u}(x)-\pi_{u}f_{u},\\
R_{2}(t,x) & \coloneqq Q_{t}^{s}(f_{s}-f_{u})(x),\\
R_{3}(t,x) & \coloneqq(Q_{t}^{s}-Q_{t}^{u})(f_{u})(x).
\end{align*}
For $R_{1}$, it can be shown by arguments which are almost identical
to those used to prove Lemma \ref{lem:bias_nu_bar_nu} that
\begin{equation}
|Q_{t}^{s}f_{s}(x)-\pi_{s}f_{s}|\leq\|f_{s}\|_{p}\tilde{\alpha}_{p,1}e^{-Kt}W^{(p)}(\delta_{x},\pi_{s}).\label{eq:Q_minus_pi}
\end{equation}
Hence
\begin{align*}
|R_{1}(t,x)| & \leq\|f\|_{p}\tilde{\alpha}_{p,1}e^{-Kt}\left[W^{(p)}(\delta_{x},\pi_{s})+W^{(p)}(\delta_{x},\pi_{u})\right],\\
 & \leq C\|f\|_{p}\tilde{\alpha}_{p,1}\sup_{s\in[0,1]}\pi_{s}\bar{V}^{(p+1/2)}\bar{V}^{(p+1/2)}(x)e^{-Kt},
\end{align*}
where we have used the estimates of Lemma \ref{lem:W(delta_x,nu)}.
For $R_{2}$, using (\ref{eq:f_lipschitz_hyp}) and Lemma \ref{lem:drift_Y_process},
\begin{align*}
\sup_{t\in\in\mathbb{R}_{+}}|R_{2}(t,x)| & \leq C_{f}|s-u|^{\beta}\sup_{t\in\mathbb{R}_{+}}Q_{t}^{s}\bar{V}^{(p)}(x)\\
 & \leq C_{f}\tilde{\alpha}_{p,1}|s-u|^{\beta}\bar{V}^{(p)}(x).
\end{align*}
For $R_{3}$, assuming w.l.o.g. that $u\leq s$,
\begin{align}
|Q_{t}^{s}f_{u}-Q_{t}^{u}f_{u}| & =\left|\int_{0}^{t}\partial_{\tau}Q_{\tau}^{u}Q_{t-\tau}^{s}f_{u}\mathrm{d}\tau\right|\nonumber \\
 & =\left|\int_{0}^{t}Q_{\tau}^{u}\left\langle \nabla U_{s}-\nabla U_{u},\nabla Q_{t-\tau}^{s}f_{u}\right\rangle \mathrm{d}\tau\right|\nonumber \\
 & \leq\int_{0}^{t}Q_{\tau}^{u}(\|\nabla U_{s}-\nabla U_{u}\|\|\nabla Q_{t-\tau}^{s}f_{u}\|)\mathrm{d}\tau\nonumber \\
 & \leq M|s-u|\int_{0}^{t}Q_{\tau}^{u}\left(\sqrt{\bar{V}_{u}}\cdot Q_{t-\tau}^{s}\|\nabla f_{u}\|\right)e^{-K(t-\tau)}\mathrm{d}\tau\nonumber \\
 & \leq\|\nabla f_{u}\|_{p}\tilde{\alpha}_{p,1}M|s-u|\int_{0}^{t}Q_{\tau}^{u}\left(\sqrt{\bar{V}_{u}}\cdot\bar{V}^{(p)}\right)e^{-K(t-\tau)}\mathrm{d}\tau.\label{eq:Q_Lipschitz}
\end{align}
We now use Lemma \ref{lem:W(delta_x,nu)} and Lemma \ref{lem:drift_Y_process},
\begin{align*}
\sup_{\tau\in[0,1]}Q_{\tau}^{u}\left(\sqrt{\bar{V}_{u}}\cdot\bar{V}^{(p)}\right)(x) & \leq C\tilde{\alpha}_{p,1}\sqrt{\bar{V}(x_{u}^{\star})}\cdot\bar{V}^{(p+1/2)}(x)
\end{align*}
and combining this observation with (\ref{eq:Q_Lipschitz}) gives
\begin{align*}
\sup_{t\in\mathbb{R}_{+}}|R_{3}(t,x)| & \leq C\tilde{\alpha}_{p,1}^{2}\frac{M}{K}|s-u|\cdot\|\nabla f\|_{p}\sqrt{\bar{V}(x_{u}^{\star})}\cdot\bar{V}^{(p+1/2)}(x).
\end{align*}
Since $x$ was arbitrary we may now choose $x=0$, and noting also
that $t$ was arbitrary and $|s-u|\leq1$, combining the above bounds
on $|R_{1}|,|R_{2}|,|R_{3}|$ then gives
\begin{align*}
|\pi_{s}f_{s}-\pi_{u}f_{u}| & \leq\|f\|_{p}\tilde{\alpha}_{p,1}\left[W^{(p)}(\delta_{0},\pi_{s})+W^{(p)}(\delta_{0},\pi_{u})\right]\inf_{t\in\mathbb{R}_{+}}e^{-Kt}+C_{f}\tilde{\alpha}_{p,1}|s-u|^{\beta}\\
 & \hspace{2cm}+C\tilde{\alpha}_{p,1}^{2}\frac{M}{K}|s-u|\cdot\|\nabla f\|_{p}\sup_{\tau\in[0,1]}\sqrt{\bar{V}(x_{\tau}^{\star})}\\
 & \leq C|s-u|^{\beta}\tilde{\alpha}_{p}\left[C_{f}+\tilde{\alpha}_{p,1}\frac{M}{K}\|\nabla f\|_{p}\sup_{\tau\in[0,1]}\sqrt{\bar{V}(x_{\tau}^{\star})}\right].
\end{align*}
This completes the proof of (\ref{eq:pi_s_f_pi_t_f}). For property
\ref{enu:enu:poisson_continuity_etc:classgsr} in the statement, by
the Proposition \ref{prop:C_2^p_closed} in the time-homogeneous case,
for any given $s$, $f_{s}\in C_{2}^{p}(\mathbb{R}^{d})\Rightarrow Q_{k\ell}^{s}f\in C_{2}^{p}(\mathbb{R}^{d}),$
hence for any $r<+\infty$ and any $\ell\geq0$, $x\mapsto g_{s,r}(x)$
is a member of $C_{2}^{p}(\mathbb{R}^{d})$. For property \ref{enu:poisson_continuity_etc:fboundgsr}
in the statement, using (\ref{eq:Q_minus_pi}),
\[
|g_{s,\infty}(x)|\leq\begin{cases}
\ell\|f\|_{p}\tilde{\alpha}_{p}W^{(p)}(\delta_{x},\pi_{s})\frac{1}{1-e^{-K\ell}}, & \ell>0,\\
\|f\|_{p}\tilde{\alpha}_{p}W^{(p)}(\delta_{x},\pi_{s})\frac{1}{K}, & \ell=0,
\end{cases}
\]
which together with Lemma \ref{lem:W(delta_x,nu)} and (\ref{eq:pi_t_finite_polynomial_moments})
imply that for any $\ell\geq0$ and $r\in\mathbb{N}_{0}\cup\{\infty\}$,
$\sup_{s,x}|g_{s,\infty}(x)|/(1+\|x\|^{2p+1})<+\infty$. For property
\ref{enu:poisson_continuity_etc:boundgsrminusgsinfty}, by similar
manipulations, 
\[
|g_{s,\infty}(x)-g_{s,r}(x)|\leq\begin{cases}
\ell\|f\|_{p}\tilde{\alpha}_{p}W^{(p)}(\delta_{x},\pi_{s})\frac{e^{-K\ell r}}{1-e^{-K\ell}}, & \ell>0,\\
\|f\|_{p}\tilde{\alpha}_{p}W^{(p)}(\delta_{x},\pi_{s})\frac{e^{-Kr}}{K}, & \ell=0.
\end{cases}
\]
For property \ref{enu:poisson_continuity_etc:gsrlipschitzcty}, in
the setting $\ell>0$, with $R_{1},R_{2}$ and $R_{3}$ as above we
have
\[
g_{u,r}(x)-g_{s,r}(x)=\ell\sum_{k=0}^{r}R_{1}(k\ell,x)=(r+1)\ell\big(\pi_{s}f_{s}-\pi_{u}f_{u}\big)-\ell\sum_{k=0}^{r}R_{2}(k\ell,x)+R_{3}(k\ell,x)
\]
and therefore for any $N-1\geq r$ for $r\in\mathbb{N}$ and any $N\in\mathbb{N}$
for $r=\infty$
\begin{align*}
|g_{s,r}(x)-g_{u,r}(x)|\leq & N\ell|\pi_{s}f_{s}-\pi_{u}f_{u}|+\ell\sum_{k=0}^{N-1}|R_{2}(k\ell,x)|+|R_{3}(k\ell,x)|+\ell\sum_{k=N}^{\infty}|R_{1}(k\ell,x)|\\
\leq & C\ell(C_{1}\vee C_{2})\left(N|s-u|^{\beta}+\frac{e^{-KN\ell}}{1-e^{-K\ell}}\right)\bar{V}^{(p+1/2)}(x),
\end{align*}
with
\begin{align*}
C_{1} & =\tilde{\alpha}_{p,1}\left[C_{f}+\tilde{\alpha}_{p,1}\frac{M}{K}\|\nabla f\|_{p}\sup_{\tau\in[0,1]}\sqrt{\bar{V}(x_{\tau}^{\star})}\right],\\
C_{2} & =\|f\|_{p}\tilde{\alpha}_{p,1}.\sup_{s\in[0,1]}\pi_{s}\bar{V}^{p+1/2}
\end{align*}
Clearly
\begin{align*}
C_{1}\vee C_{2} & \leq C\tilde{\alpha}_{p,1}(\ell\vee1)\big(C_{f}\vee\vvvert f\vvvert_{p}\big)\Bigl(1+\tilde{\alpha}_{p,1}\frac{M}{K}\sup_{\tau\in[0,1]}\sqrt{\bar{V}(x_{\tau}^{\star})}+\sup_{s\in[0,1]}\pi_{s}\bar{V}^{p+1/2}\Bigr).
\end{align*}
Now when $|s-u|^{\beta}\geq\frac{e^{-K\ell}}{1-e^{-K\ell}},$ one
can choose $N=1$ and conclude. Otherwise we take $N=\lceil-(K\ell)^{-1}\log\big(|s-u|^{\beta}\big)\rceil$
which with $\gimel^{-1}\leq1-K\ell/2$ leads, on the one hand, to
\[
\frac{e^{-KN\ell}}{1-e^{-K\ell}}\leq\frac{\gimel}{K\ell}|s-u|^{\beta}
\]
and on the other hand to
\[
N|s-u|^{\beta}\leq\big[1-(K\ell)^{-1}\log\big(|s-u|^{\beta}\big)\big]|s-u|^{\beta}
\]
So we study $\varphi(x)=x^{a}\log x$ for $x\geq0.$ $\varphi'(x)=x^{a-1}\big[a\log(x)+1\big]$
so $\varphi(x)$ reaches its minimum at $\exp(-a^{-1})$, and therefore
since $\varphi(x)\leq0$ for $0\leq x\leq1$, for any $b\geq0$ 
\[
\sup_{x\in[0,b]}\big|\varphi(x)\big|\leq\big|\varphi(a)\big|\vee\big|\varphi(b)\big|.
\]
Therefore for $|s-u|\leq R_{f}$ and $\zeta\in(0,\beta)$ we have
\begin{align*}
N|s-u|^{\beta-\zeta} & \leq R_{f}{}^{\beta-\zeta}+\frac{\beta}{K\ell}\big[e^{-1}/(\beta-\zeta)\big]\vee\big(R_{f}^{\beta-\zeta}\big|\log R_{f}\big|\big)
\end{align*}
and in total we have the bound
\[
N|s-u|^{\beta}+\frac{e^{-KN\ell}}{1-e^{-K\ell}}\leq\frac{1}{K\ell}\left[(2\vee\gimel)R_{f}{}^{\beta-\zeta}+\beta\big[e^{-1}/(\beta-\zeta)\big]\vee\big(R_{f}^{\beta-\zeta}\big|\log R_{f}\big|\big)\right]|s-u|^{\zeta}.
\]
For the case $\ell=0$ a reasoning similar as that above leads to
\[
|g_{s,r}(x)-g_{u,r}(x)|\leq C\ell(C_{1}\vee C_{2})\left(N|s-u|^{\beta}+\frac{e^{-KN}}{K}\right)\bar{V}^{(p+1/2)}(x),
\]
\end{proof}
and for $|s-u|^{\beta}\geq e^{-K}/K,$ set $N=1$, and otherwise set
$N=\left\lceil -K^{-1}\log\big(|s-u|^{\beta}\big)\right\rceil $ and
deduce from above that
\begin{align*}
N|s-u|^{\beta}+\frac{e^{-KN}}{K} & \leq\big[1-K^{-1}\log\big(|s-u|^{\beta}\big)+K^{-1}\big]|s-u|^{\beta}\\
 & \leq\frac{1}{K}\left[KR_{f}{}^{\beta-\zeta}+\beta\big[e^{-1}/(\beta-\zeta)\big]\vee\big(R_{f}^{\beta-\zeta}\big|\log R_{f}\big|\big)\right]|s-u|^{\zeta}
\end{align*}
and we conclude by combining all the cases.
\begin{lem}
\label{lem:rieman_approx} Assume that for some $p\geq1$ and $f\in C_{0,2}^{p}([0,1]\times\mathbb{R}^{d})$
there exist constants $C_{f}<+\infty$, $R_{f}>0$ and $\beta\in(0,1]$
such that
\[
|s-t|\leq R_{f}\quad\Rightarrow\quad|f_{s}(x)-f_{t}(x)|\leq C_{f}|s-t|^{\beta}\bar{V}^{(p)}(x),\quad\forall x\in\mathbb{R}^{d}.
\]
Then for any $h\in(0,R_{f}]$
\[
\left|h\sum_{k=0}^{\left\lfloor 1/h\right\rfloor -1}\pi_{kh}f_{kh}-\int_{0}^{1}\pi_{t}f_{t}\mathrm{d}t\right|\leq h^{\beta}\tilde{\alpha}_{p}\big(C_{f}\vee\|\nabla f\|_{p}\big)\left[1+\tilde{\alpha}_{p}\frac{M}{K}\sup_{t\in[0,1]}\sqrt{\bar{V}(x_{t}^{\star})}\right].
\]
\end{lem}

\begin{proof}
Using Lemma \ref{lem:poisson_continuity_etc},
\begin{align*}
\Bigl|h\sum_{k=0}^{\left\lfloor 1/h\right\rfloor -1}\pi_{kh}f_{kh} & -\int_{0}^{1}\pi_{t}f_{t}\mathrm{d}t\Bigr|\\
 & \leq\sum_{k=0}^{\left\lfloor 1/h\right\rfloor -1}\int_{kh}^{(k+1)h}|\pi_{kh}f_{kh}-\pi_{t}f_{t}|\mathrm{d}t\\
 & \leq h^{\beta}\tilde{\alpha}_{p}\left[C_{f}+\tilde{\alpha}_{p}\frac{M}{K}\|\nabla f\|_{p}\cdot\sup_{t\in[0,1]}\sqrt{\bar{V}(x_{t}^{\star})}\right]\sum_{k=0}^{\left\lfloor 1/h\right\rfloor -1}\int_{kh}^{(k+1)h}\mathrm{d}t\\
 & \leq h^{\beta}\tilde{\alpha}_{p}\left[C_{f}+\tilde{\alpha}_{p}\frac{M}{K}\|\nabla f\|_{p}\cdot\sup_{t\in[0,1]}\sqrt{\bar{V}(x_{t}^{\star})}\right].
\end{align*}
\end{proof}

%% file: child-discretization.tex
\section{Controlling the discretization error\label{sec:Controlling-the-discretization}}

Throughout section \ref{sec:Controlling-the-discretization}, $(\widetilde{X}_{t}^{\epsilon,h})_{t\in[0,1]}$,
$\mu^{\epsilon}$ , and $\widetilde{\mu}^{\epsilon,h}$ are as defined
in section \ref{subsec:intro_Discretization}. 

\subsection{Bounding the total variation distance}
\begin{prop}
\label{prop:tv_bound}If $h/\epsilon\in(0,2K/L^{2})$, then for any
$\delta\in(0,1)$
\[
\|\mu^{\epsilon}-\widetilde{\mu}^{\epsilon,h}\|_{\mathrm{tv}}\leq\frac{1}{2}\left[L^{2}d\frac{h}{\epsilon^{2}}+\frac{h^{3}}{3\epsilon}\left(M^{2}+\frac{L^{4}}{\epsilon^{2}}\right)\left(\frac{1}{h}+\frac{1}{1-\lambda}\left[\mu_{0}(V_{0})+\frac{b}{h}\right]\right)\right]^{1/2},
\]
where
\begin{eqnarray*}
\lambda & \coloneqq & 1-\left(\frac{2hK}{\epsilon}-\left(\frac{h}{\epsilon}\right)^{2}L^{2}\right)(1-\delta),\\
b & \coloneqq & \sup_{t\in(0,1)}\|\partial_{t}x_{t}^{\star}\|^{2}\left[\frac{4h^{2}}{\delta\left(\frac{2hK}{\epsilon}-\left(\frac{h}{\epsilon}\right)^{2}L^{2}\right)}+h^{2}\right]+2d\frac{h}{\epsilon}.
\end{eqnarray*}
\end{prop}

\begin{proof}
The proof is quite similar to \cite[Proof of Lemma 2]{dalalyan2016theoretical},
except that here we need to account for the dependence of $U_{t}$
on $t$. Consider
\begin{eqnarray*}
\Xi_{t} & \coloneqq & \frac{1}{\sqrt{2\epsilon}}\left\{ \widetilde{\nabla U}_{t}(\widetilde{X}_{t}^{\epsilon,h})-\nabla U_{t}(\widetilde{X}_{t}^{\epsilon,h})\right\} \\
Z_{t} & \coloneqq & \exp\left(\sum_{i=1}^{d}\int_{0}^{t}\Xi_{s}^{i}\mathrm{d}B_{s}^{i}-\frac{1}{2}\int_{0}^{t}\|\Xi_{s}\|^{2}\mathrm{d}s\right).
\end{eqnarray*}
By Girsanov's theorem,  under the probability measure $\widetilde{\mathbb{P}}_{\mathcal{F}_{1}}[A]\coloneqq\mathbb{E}[\mathbb{I}_{A}Z_{1}]$,
$A\in\mathcal{F}_{1}$, the process $\int_{0}^{t}\mathrm{d}B_{s}-\Xi_{s}\mathrm{d}s$
is a $d$-dimensional $(\mathcal{F}_{t})_{t\in[0,1]}$-Brownian motion
and the law of $(\widetilde{X}_{t}^{\epsilon,h})_{t\in[0,1]}$ is
$\mu$. Denoting by $\mathbb{P}_{\mathcal{F}_{1}}$ the restriction
of $\mathbb{P}$ to $\mathcal{F}_{1}$, we therefore have by Pinsker's
inequality
\begin{equation}
\|\mu^{\epsilon}-\widetilde{\mu}^{\epsilon,h}\|_{\mathrm{tv}}\leq\|\widetilde{\mathbb{P}}_{\mathcal{F}_{1}}-\mathbb{P}_{\mathcal{F}_{1}}\|_{\mathrm{tv}}\leq\sqrt{-\frac{1}{2}\mathbb{E}[\log Z_{1}]}=\frac{1}{2}\sqrt{\mathbb{E}\left[\int_{0}^{t}\|\Xi_{s}\|^{2}\mathrm{d}s\right]}.\label{eq:pinsker}
\end{equation}
For $s\in[kh,(k+1)h)$, we have from (\ref{eq:sde_disc-1}) and (A\ref{hyp:U_grad_lipschitz_x}),
\begin{eqnarray}
\mathbb{E}[\|\widetilde{X}_{kh}^{\epsilon,h}-\widetilde{X}_{s}^{\epsilon,h}\|^{2}] & = & \frac{1}{\epsilon^{2}}(s-kh)^{2}\mathbb{E}[\|\nabla U_{kh}(\widetilde{X}_{kh}^{\epsilon,h})\|^{2}]+\frac{2d}{\epsilon}(s-kh)\nonumber \\
 & \leq & \frac{1}{\epsilon^{2}}(s-kh)^{2}L^{2}\mathbb{E}[1+\|\widetilde{X}_{kh}^{\epsilon,h}-x_{kh}^{\star}\|^{2}]+\frac{2d}{\epsilon}(s-kh).\label{eq:disc_mean_square}
\end{eqnarray}
The considering the expectation in (\ref{eq:pinsker}), we find from
(\ref{eq:tilde_grad_U}), (A\ref{hyp:U_time_cont}) , (A\ref{hyp:U_grad_lipschitz_x}),
(\ref{eq:disc_mean_square}), and Lemma \ref{lem:disc_drift_p=00003D2},
 
\begin{eqnarray*}
 &  & \mathbb{E}\left[\int_{0}^{t}\|\Xi_{s}\|^{2}\mathrm{d}s\right]\\
 &  & =\frac{1}{2\epsilon}\sum_{k=0}^{\left\lfloor 1/h\right\rfloor -1}\int_{kh}^{(k+1)h}\mathbb{E}[\|\nabla U_{kh}(\widetilde{X}_{kh}^{\epsilon,h})-\nabla U_{s}(\widetilde{X}_{s}^{\epsilon,h})\|^{2}]\mathrm{d}s\\
 &  & \leq\frac{1}{\epsilon}\sum_{k=0}^{\left\lfloor 1/h\right\rfloor -1}\int_{kh}^{(k+1)h}\mathbb{E}[\|\nabla U_{kh}(\widetilde{X}_{kh}^{\epsilon,h})-\nabla U_{s}(\widetilde{X}_{kh}^{\epsilon,h})\|^{2}]+\mathbb{E}[\|\nabla U_{s}(\widetilde{X}_{kh}^{\epsilon,h})-\nabla U_{s}(\widetilde{X}_{s}^{\epsilon,h})\|^{2}]\mathrm{d}s\\
 &  & \leq\frac{1}{\epsilon}\sum_{k=0}^{\left\lfloor 1/h\right\rfloor -1}\int_{kh}^{(k+1)h}M^{2}(s-kh)^{2}\mathbb{E}[1+\|\widetilde{X}_{kh}^{\epsilon,h}-x_{kh}^{\star}\|^{2}]+L^{2}\mathbb{E}[\|\widetilde{X}_{kh}^{\epsilon,h}-\widetilde{X}_{s}^{\epsilon,h}\|^{2}]\mathrm{d}s\\
 &  & \leq\frac{1}{\epsilon}\sum_{k=0}^{\left\lfloor 1/h\right\rfloor -1}\int_{kh}^{(k+1)h}M^{2}(s-kh)^{2}\mathbb{E}[1+\|\widetilde{X}_{kh}^{\epsilon,h}-x_{kh}^{\star}\|^{2}]+L^{2}\left(\frac{1}{\epsilon^{2}}(s-kh)^{2}L^{2}\mathbb{E}[1+\|\widetilde{X}_{kh}^{\epsilon,h}-x_{kh}^{\star}\|^{2}]+\frac{2d}{\epsilon}(s-kh)\right)\mathrm{d}s\\
 &  & =\frac{1}{\epsilon}\left(M^{2}+\frac{L^{4}}{\epsilon^{2}}\right)\sum_{k=0}^{\left\lfloor 1/h\right\rfloor -1}\mathbb{E}[1+\|\widetilde{X}_{kh}^{\epsilon,h}-x_{kh}^{\star}\|^{2}]\int_{kh}^{(k+1)h}(s-kh)^{2}\mathrm{d}s\\
 &  & \quad+\frac{1}{\epsilon}L^{2}\frac{2d}{\epsilon}\sum_{k=0}^{\left\lfloor 1/h\right\rfloor -1}\int_{kh}^{(k+1)h}(s-kh)\mathrm{d}s\\
 &  & =L^{2}d\frac{h}{\epsilon^{2}}+\frac{h^{3}}{3\epsilon}\left(M^{2}+\frac{L^{4}}{\epsilon^{2}}\right)\sum_{k=0}^{\left\lfloor 1/h\right\rfloor -1}\mathbb{E}[1+\|\widetilde{X}_{kh}^{\epsilon,h}-x_{kh}^{\star}\|^{2}]\\
 &  & \leq L^{2}d\frac{h}{\epsilon^{2}}+\frac{h^{3}}{3\epsilon}\left(M^{2}+\frac{L^{4}}{\epsilon^{2}}\right)\left(\frac{1}{h}+\frac{1}{1-\lambda}\left[\mu_{0}(V)+\frac{b}{h}\right]\right).
\end{eqnarray*}
Substituting in to (\ref{eq:pinsker}) completes the proof.
\end{proof}

\subsection{Drift condition for the discretized process}

Define
\[
\widetilde{P}_{k}(x,A)\coloneqq\int_{A}\frac{1}{\sqrt{4\pi h/\epsilon}}\exp\left(-\frac{1}{4h/\epsilon}\|x-h/\epsilon\nabla U_{kh}(x)-y\|^{2}\right)\mathrm{d}y,
\]
where the dependence of $\widetilde{P}_{k}$ on $\epsilon$ and $h$
is not shown in the notation. 

\begin{lem}
\label{lem:disc_drift_p=00003D2}If $h/\epsilon\in(0,2K/L^{2})$,
then for any $\delta\in(0,1)$, 
\begin{eqnarray}
\widetilde{P}_{k}V_{kh}(x) & \leq & \lambda V_{(k-1)h}(x)+b,\label{eq:disc_drift_p=00003D2}\\
\sum_{k=0}^{\left\lfloor 1/h\right\rfloor -1}\mathbb{E}[1+\|\widetilde{X}_{kh}^{\epsilon,h}-x_{kh}^{\star}\|^{2}] & \leq & \frac{1}{h}+\frac{1}{1-\lambda}\left[\mu_{0}(V_{0})+\frac{b}{h}\right],\label{eq:disc_mom_p=00003D2}
\end{eqnarray}
where 
\begin{eqnarray*}
\lambda & \coloneqq & 1-\left(\frac{2hK}{\epsilon}-\left(\frac{h}{\epsilon}\right)^{2}L^{2}\right)(1-\delta),\\
b & \coloneqq & \sup_{t}\|\partial_{t}x_{t}^{\star}\|^{2}\left[\frac{4h^{2}}{\delta\left(\frac{2hK}{\epsilon}-\left(\frac{h}{\epsilon}\right)^{2}L^{2}\right)}+h^{2}\right]+2d\frac{h}{\epsilon}.
\end{eqnarray*}
\end{lem}

\begin{proof}
To simplify presentation in the proof we write $\widetilde{X}_{k}\coloneqq\widetilde{X}_{kh}^{\epsilon}$,
$x_{k-1}\coloneqq x_{(k-1)h}$, $x_{k}^{\star}\coloneqq x_{kh}^{\star}$
, $\nabla U_{k-1}(x)\coloneqq\nabla U_{(k-1)h}(x)$ etc. With $\xi\sim\mathcal{N}(0_{d},2h/\epsilon I_{d})$,
we have
\begin{eqnarray*}
\widetilde{P}_{k}V_{kh}(x)=\mathbb{E}\left[\left.\|\widetilde{X}_{k}-x_{k}^{\star}\|^{2}\right|\widetilde{X}_{k-1}=x\right] & = & \mathbb{E}\left[\left\Vert x-\frac{h}{\epsilon}\nabla U_{k-1}(x)+\xi-x_{k}^{\star}\right\Vert ^{2}\right]\\
 & \leq & \left(\left\Vert x-x_{k-1}^{\star}-\frac{h}{\epsilon}\nabla U_{k-1}(x)\right\Vert +\|x_{k}^{\star}-x_{k-1}^{\star}\|\right)^{2}+\mathbb{E}[\|\xi\|^{2}],
\end{eqnarray*}
where in view of Lemma \ref{lem:minimizer},
\[
\|x_{k}^{\star}-x_{k-1}^{\star}\|\leq ch,\quad\quad c\coloneqq\sup_{t\in(0,1)}\|\partial_{t}x_{t}^{\star}\|<+\infty,
\]
and
\[
\mathbb{E}[\|\xi\|^{2}]=2d\frac{h}{\epsilon}.
\]
Now writing $\beta\coloneqq\frac{2hK}{\epsilon}-\left(\frac{h}{\epsilon}\right)^{2}L^{2}$,
noting the assumption $h/\epsilon\in(0,2K/L^{2})$, using (A\ref{hyp:U_strong_convex})
and (A\ref{hyp:U_grad_lipschitz_x}) we have for any $\delta\in(0,1)$
\begin{eqnarray*}
 &  & \left\Vert x-\frac{h}{\epsilon}\nabla U_{k-1}(x)-x_{k-1}^{\star}\right\Vert ^{2}\\
 &  & \leq\|x-x_{k-1}^{\star}\|^{2}-\frac{2h}{\epsilon}\left\langle x-x_{k-1}^{\star},\nabla U_{k-1}(x)\right\rangle +\left(\frac{h}{\epsilon}\right)^{2}\|\nabla U_{k-1}(x)\|^{2}\\
 &  & \leq(1-\beta)\|x-x_{k-1}^{\star}\|^{2}\\
 &  & =\lambda\|x-x_{k-1}^{\star}\|^{2}-\delta\beta\|x-x_{k-1}^{\star}\|^{2},
\end{eqnarray*}
where $\lambda\coloneqq1-\beta(1-\delta)<1$. Combining the above
gives: 
\begin{eqnarray*}
 & \widetilde{P}_{k}V_{kh}(x) & \leq\lambda\|x-x_{k-1}^{\star}\|^{2}-\delta\beta\|x-x_{k-1}^{\star}\|^{2}+2ch\|x-x_{k-1}^{\star}\|+c^{2}h^{2}+2d\frac{h}{\epsilon}\\
 &  & \leq\lambda\|x-x_{k-1}^{\star}\|^{2}+\frac{4c^{2}h^{2}}{\delta\beta}+c^{2}h^{2}+2d\frac{h}{\epsilon},
\end{eqnarray*}
where the final inequality follows by considering whether or not $2ch\leq\delta\beta\|x-x_{k-1}^{\star}\|$.
Thus (\ref{eq:disc_drift_p=00003D2}) holds and iterating gives

\[
\mathbb{E}\left[\|\widetilde{X}_{k}-x_{k}^{\star}\|^{2}|X_{0}=x\right]\leq\lambda^{k}V_{0}(x)+b\sum_{j=0}^{k-1}\lambda^{j},
\]
from which (\ref{eq:disc_mom_p=00003D2}) follows.
\end{proof}

%% file: child-auxiliary-results.tex
\section{Auxiliary results and proofs\label{sec:Auxiliary-results-and}}

\subsection{Preliminaries}
\begin{lem}
\label{lem:the_TI_identity}
\begin{align*}
\partial_{t}\log Z_{t} & =-\int_{\mathbb{R}^{d}}\partial_{t}U_{t}(x)\pi_{t}(\mathrm{d}x).
\end{align*}
\end{lem}

\begin{proof}
Using (A\ref{hyp:U_strong_convex}), Lemma \ref{lem:strong_convex_equiv},
the reverse triangle inequality and the convexity of $a\mapsto a^{2}$,
\begin{align*}
\sup_{t}\exp\left[-U_{t}(x)\right] & \leq\sup_{t}\exp\left[-U_{t}(x_{t}^{\star})-\frac{K}{2}\|x-x_{t}^{\star}\|^{2}\right]\\
 & \leq\exp\left[-\inf_{t}U_{t}(x_{t}^{\star})-\frac{K}{4}\|x\|^{2}+\frac{K}{2}\sup_{t}\|x_{t}^{\star}\|^{2}\right],
\end{align*}
where $\sup_{t\in[0,1]}\|x_{t}^{\star}\|$ and $-\inf_{t}U_{t}(x_{t}^{\star})$
are finite, since by Lemma \ref{lem:minimizer}, $t\mapsto\|x_{t}^{\star}\|$
is continuous on $[0,1]$, and $U_{t}(x)$ is continous in $(t,x)$
by (A\ref{hyp:basic_hyp_on_U}). Also by (A\ref{hyp:basic_hyp_on_U}),
there exists some $p\geq1$ and $c<+\infty$ such that 
\[
\sup_{t}\left|\partial_{t}U_{t}(x)\right|\leq c(1+\|x\|^{2p}),\quad\forall x.
\]
 Hence the following interchange of differentiation and integration
is permitted:
\begin{align*}
\partial_{t}\log Z_{t} & =\frac{1}{Z_{t}}\partial_{t}\int_{\mathbb{R}^{d}}\exp\left[-U_{t}(x)\right]\mathrm{d}x\\
 & =-\frac{1}{Z_{t}}\int_{\mathbb{R}^{d}}\exp\left[-U_{t}(x)\right]\partial_{t}U_{t}(x)\mathrm{d}x\\
 & =-\int_{\mathbb{R}^{d}}\partial_{t}U_{t}(x)\pi_{t}(\mathrm{d}x).
\end{align*}
\end{proof}
\begin{lem}
\label{lem:strong_convex_equiv}For any given $f\in C_{2}(\mathbb{R}^{d})$
and $c>0$, the following conditions are equivalent:
\begin{eqnarray*}
f(y)-f(x) & \geq & \left\langle \nabla f(x),y-x\right\rangle +\frac{1}{2}c\|y-x\|^{2},\quad\forall x,y\in\mathbb{R}^{d},\\
\left\langle \nabla f(x)-\nabla f(y),x-y\right\rangle  & \geq & c\|x-y\|^{2},\quad\forall x,y\in\mathbb{R}^{d},\\
\inf_{x\in\mathbb{R}^{d}}\,\sum_{i,j}v_{i}\frac{\partial^{2}f(x)}{\partial x_{i}\partial x_{j}}v_{j} & \geq & c\|v\|^{2},\quad\forall v\in\mathbb{R}^{d}.
\end{eqnarray*}
\end{lem}

\begin{proof}
See \cite{nesterov2013introductory}.
\end{proof}

\begin{lem}
\label{lem:minimizer}Let $x_{t}^{\star}$ be the unique minimizer
of $U_{t}$. Then the map $t\mapsto x_{t}^{\star}$ is continuous
on $[0,1]$, continously differentiable on $(0,1)$ and 
\[
\sup_{t\in(0,1)}\|\partial_{t}x_{t}^{\star}\|\vee\sup_{t\in[0,1]}\|x_{t}^{\star}\|\leq\frac{M}{K}.
\]
\end{lem}

\begin{proof}
Fix any $t\in(0,1)$. The strong convexity assumption (A\ref{hyp:U_strong_convex})
implies $\nabla^{(2)}U_{t}(x)$ is invertible for all $x$. Therefore
by the implicit function theorem there exist open neighborhoods $\mathcal{T}$
of $t$ and $\mathcal{X}$ of $x_{t}^{\star}$ and a unique continuously
differentiable function $\zeta:\mathcal{T}\to\mathcal{X}$ such that
$\{(s,\zeta(s))\,;\,s\in\mathcal{T}\}=\{(s,x)\,;\,\nabla U_{s}(x)=0,(s,x)\in\mathcal{T}\times\mathcal{X}\}$.
Since $t\in(0,1)$ was arbitrary, the interval $(0,1)$ can be covered
with such neighborhoods $\mathcal{T}$, and the uniqueness under (A\ref{hyp:U_strong_convex})
of the minimizer $U_{t}(\cdot)$ for each $t$ implies that the continuously
differentiable functions must agree on the non-empty intersections
between the $\mathcal{T}$'s, yielding a continuously differentiable
function $\zeta:(0,1)\to\mathbb{R}^{d}$ such that $\zeta(t)=x_{t}^{\star}$.
Let us now prove that $\lim_{t\searrow0}\zeta(t)=x_{0}^{\star}$.
First note that $\nabla U_{t}$ is continuous in $t$ on $[0,1]$
by assumption, so $\lim_{n\to+\infty}\|\nabla U_{n^{-1}}(x_{0}^{\star})\|=\|\nabla U_{0}(x_{0}^{\star})\|=0$.
By way of a contradiction, suppose that there exists $\delta>0$ such
that for all $n_{0}>0$ there exists $n\geq n_{0}$ such that
\[
\|x_{0}^{\star}-\zeta(n^{-1})\|\geq\delta,
\]
which together with (A\ref{hyp:U_strong_convex}), Lemma \ref{lem:strong_convex_equiv}
and Cauchy-Schwartz implies 
\begin{eqnarray*}
\|\nabla U_{n^{-1}}(x_{0}^{\star})\| & = & \|\nabla U_{n^{-1}}(x_{0}^{\star})-\nabla U_{n^{-1}}(\zeta(n^{-1}))\|\\
 & \geq & K\|x_{0}^{\star}-\zeta(n^{-1})\|\geq K\delta,
\end{eqnarray*}
giving a contradiction as required. By a similar argument $\lim_{t\nearrow1}\zeta(t)=x_{1}^{\star}$,
and therefore $t\mapsto x_{t}^{\star}$ is continuous on $[0,1]$. 

We also have:
\begin{equation}
\left\Vert \partial_{t}x_{t}^{\star}\right\Vert =\left\Vert [\nabla^{(2)}U_{t}]^{-1}(x_{t}^{\star})\cdot\left.\partial_{t}\nabla U_{t}(x)\right|_{x=x_{t}^{\star}}\right\Vert \leq\frac{1}{K}\left\Vert \left.\partial_{t}\nabla U_{t}(x)\right|_{x=x_{t}^{\star}}\right\Vert ,\label{eq:velocity_bounded_derive}
\end{equation}
where the equality is due to the implicit function theorem and the
inequality uses the facts that: for a symmetric matrix $H$, the operator
norm $\|H\|_{\mathrm{op}}$ induced by the Euclidean distance on $\mathbb{R}^{d}$
is equal to the largest eigenvalue of $H$; $\|H^{-1}x\|\leq\|H^{-1}\|_{\mathrm{op}}\|x\|$;
and (A\ref{hyp:U_strong_convex}) implies all the eigenvalues of $\nabla^{(2)}U_{t}(x)$
are lower bounded by $K$. The term on the right of (\ref{eq:velocity_bounded_derive})
is uniformly bounded over $t\in(0,1)$ by $M/K$ because (A\ref{hyp:U_time_cont})
implies
\[
\|\nabla U_{t}(x_{t}^{\star})-\nabla U_{t+\delta}(x_{t}^{\star})\|\leq M\delta.
\]
 Integrating this bound and noting that $x_{0}^{\star}=0$ by (A\ref{hyp:U_strong_convex}),
\[
\sup_{t\in[0,1]}\|x_{t}^{\star}\|\leq\|x_{0}^{\star}\|+\sup_{t\in[0,1]}\int_{0}^{t}\|\partial_{s}x_{s}^{\star}\|\mathrm{d}s\leq\frac{M}{K}.
\]
 
\end{proof}
\begin{lem}
\label{lem:variance_lower_bound}For any $p\geq1$, $t\in[0,1]$ and
$f\in C_{0}^{p}(\mathbb{R}^{d})$,
\[
\mathrm{var}_{\pi_{t}}[f]\geq L^{-1}\sum_{i=1}^{d}\pi_{t}\left(f\frac{\partial U_{t}}{\partial x_{i}}\right)^{2}.
\]
\end{lem}

\begin{proof}
Fix any $t\in[0,1]$ and $f\in C_{0}^{p}(\mathbb{R}^{d})$. The first
part of the proof follows arguments used to derive Cramer-Rao inequalties,
see \cite{cacoullos1982} for perspective on this kind of technique.
Let $\Theta$ be any compact subset of $\mathbb{R}^{d}$ containing
$0$, and then introduce an artificial location parameter $\theta\in\Theta$.
Suppressing $t$ to simplify notation, consider the probability measure
$\pi^{\theta}$ defined by

\[
\pi^{\theta}(\mathrm{d}x)\coloneqq\pi^{\theta}(x)\mathrm{d}x,\quad\pi^{\theta}(x)\coloneqq Z_{t}^{-1}\exp\{-U^{\theta}(x)\}\mathrm{d}x,\quad U^{\theta}(x)\coloneqq U_{t}(x-\theta).
\]
 Then with expectation and variance with respect to $\pi^{\theta}$
denoted respectively by $\mathbb{E}^{\theta}[\cdot]$ and $\mathrm{var}^{\theta}[\cdot]$,
and gradient with respect to $\theta$ denoted by $\nabla_{\theta}$,
define the vector $g_{\theta}\coloneqq\nabla_{\theta}\mathbb{E}^{\theta}[f(X)]$
and the matrix $J_{\theta}\coloneqq-\mathbb{E}^{\theta}[\nabla_{\theta}^{(2)}\log\pi^{\theta}(X)]$,
where in the latter and similar expressions below, the expectation
is element-wise. Using (A\ref{hyp:U_strong_convex}), (A\ref{hyp:U_grad_lipschitz_x}),
(A\ref{hyp:U_time_reg}) and Lemma \ref{lem:minimizer}, it can be
checked using manipulations similar to those in the proof of Lemma
\ref{lem:the_TI_identity} that the following identities hold by differentiation
under the integral sign:
\begin{align*}
g_{\theta} & =\mathbb{E}^{\theta}[f(X)\nabla_{\theta}\log\pi^{\theta}(X)],\\
0 & =\mathbb{E}^{\theta}[\nabla_{\theta}\log\pi^{\theta}(X)],\\
J_{\theta} & =\mathbb{E}^{\theta}[\nabla_{\theta}\log\pi^{\theta}(X)\cdot\{\nabla_{\theta}\log\pi^{\theta}(X)\}^{T}],
\end{align*}
and $J_{\theta}$ is invertible. Using these identities and Cauchy-Schwartz,

\begin{align*}
g_{\theta}^{T}J_{\theta}^{-1}g_{\theta} & =g_{\theta}^{T}J_{\theta}^{-1}\mathbb{E}^{\theta}[f(X)\nabla_{\theta}\log\pi^{\theta}(X)]\\
 & =g_{\theta}^{T}J_{\theta}^{-1}\mathbb{E}^{\theta}[\{f(X)-\mathbb{E}^{\theta}[f(X)]\}\nabla_{\theta}\log\pi^{\theta}(X)]\\
 & =\mathbb{E}^{\theta}[\{f(X)-\mathbb{E}^{\theta}[f(X)]\}g_{\theta}^{T}J_{\theta}^{-1}\nabla_{\theta}\log\pi^{\theta}(X)]\\
 & \leq\mathrm{var}^{\theta}[f(X)]^{1/2}\mathbb{E}^{\theta}[(g_{\theta}^{T}J_{\theta}^{-1}\nabla_{\theta}\log\pi^{\theta}(X))^{2}]^{1/2}\\
 & =\mathrm{var}^{\theta}[f(X)]^{1/2}(g_{\theta}^{T}J_{\theta}^{-1}g_{\theta})^{1/2},
\end{align*}
hence
\begin{equation}
\mathrm{var}^{\theta}[f(X)]\geq g_{\theta}^{T}J_{\theta}^{-1}g_{\theta}.\label{eq:cramer_rao}
\end{equation}

Noting that $\nabla_{\theta}\log\pi^{\theta}(x)=\nabla U(x-\theta)$
and $\nabla_{\theta}^{(2)}\log\pi^{\theta}(x)=-\nabla^{(2)}U(x-\theta)$,
the lower bound (\ref{eq:cramer_rao}) with $\theta=0$ reads:

\begin{equation}
\mathrm{var}_{\pi}[f]\geq\mathbb{E}_{\pi}[f\nabla U]^{T}\mathbb{E}_{\pi}[\nabla^{(2)}U]^{-1}\mathbb{E}_{\pi}[f\nabla U].\label{eq:var_lower_bound_inter}
\end{equation}
Using Cauchy-Schwartz and the Lipschitz assumption (A\ref{hyp:U_grad_lipschitz_x}),
we have for any $\tau>0$ and $v\in\mathbb{R}^{d}$ 
\begin{align*}
\frac{1}{\tau}\int_{0}^{\tau}\left\langle \nabla^{(2)}U(x+\lambda v)\cdot v,v\right\rangle \mathrm{d}\lambda & =\frac{1}{\tau}\left\langle \nabla U(x+\tau v)-\nabla U(x),v\right\rangle \\
 & \leq\frac{1}{\tau}\|\nabla U(x+\tau v)-\nabla U(x)\|\|v\|\\
 & \leq L\|v\|^{2}.
\end{align*}
Taking $\tau\to0$ we find $v^{T}\mathbb{E}_{\pi}[\nabla_{x}^{(2)}U]v\leq L\|v\|^{2}$,
so $v^{T}\mathbb{E}_{\pi}[\nabla_{x}^{(2)}U]^{-1}v\geq L^{-1}\|v\|^{2}$,
which applied to (\ref{eq:var_lower_bound_inter}) completes the proof.
\end{proof}

\subsection{Intermediate results concerning dimension dependence\label{subsec:Intermediate-results-concerning}}
\begin{lem}
\label{lem:alpha_dimension_dependence}Fix $p\geq1$ and consider
the quantities $\alpha_{p}$ and $\tilde{\alpha}_{p}$ defined in
Lemmas \ref{lem:drift} and \ref{lem:drift_Y_process}, choosing there
$\kappa=Kp/2$.

\noindent 1) $\tilde{\alpha}_{p}$ does not depend on $\epsilon$.
For any $q\geq0$, if $K^{-1}\vee\sup_{t}\|x_{t}^{\star}\|^{2}=O(d^{q})$
as $d\to\infty$, then $\tilde{\alpha}_{p}=O(d^{p(q+1)})$.

\noindent 2) For any $q\geq0$, if $K^{-1}\vee\sup_{t}\|x_{t}^{\star}\|^{2}=O(d^{q})$
and $\frac{\epsilon}{K}\sup_{t}\|\partial_{t}x_{t}^{\star}\|=O(1)$
as $d\to\infty$, then $\alpha_{p}=O(d^{p(q+1)})$.
\end{lem}

\begin{proof}
For part 1) the expression for $\tilde{\alpha}_{p}$ in Lemma \ref{lem:drift_Y_process}
with $\kappa$ chosen to be $Kp/2$ is: 
\begin{align*}
\tilde{\alpha}_{p} & =2^{4p-2}\vee\left[1+2^{2p-1}\left(\frac{4}{K^{p}}\left(8(p-1)+4d\right)^{p-1}\left[2(p-1)+d\right]+(1+2^{2p-1})\sup_{t}\|x_{t}^{\star}\|^{2p}\right)\right]\\
 & =O\left(1+\frac{d^{p}}{K^{p}}+\sup_{t}\|x_{t}^{\star}\|^{2p}\right),
\end{align*}
from which the second claim of part 1) follows. 

For part 2), writing out the expression for $\alpha_{p}$ from Lemma
\ref{lem:drift} with $\kappa=Kp/2$ and the shorthand $v\coloneqq\sup_{t}\|\partial_{t}x_{t}^{\star}\|$,
\begin{align*}
\alpha_{p} & =2^{4p-2}\vee\left[1+2^{2p-1}\left(\frac{4}{K}r^{2p-2}\left[r\epsilon v+[2(p-1)+d]\right]+(1+2^{2p-1})\sup_{t}\|x_{t}^{\star}\|^{2p}\right)\right]
\end{align*}
where
\[
r=\frac{\epsilon v}{K}+2\sqrt{\frac{\epsilon^{2}v^{2}}{K^{2}}+\frac{1}{K}[2(p-1)+d]}.
\]
Using the hypotheses of part 2), we find $r=O(1+\sqrt{1+d/K})=O(\sqrt{d/K})$,
and so
\begin{align*}
\alpha_{p} & =O\left(r^{2p-2}\left(r\frac{\epsilon v}{K}+\frac{d}{K}\right)+\sup_{t}\|x_{t}^{\star}\|^{2p}\right)\\
 & =O\left(\left(\frac{d}{K}\right)^{(p-1)}\left(r+\frac{d}{K}\right)+\sup_{t}\|x_{t}^{\star}\|^{2p}\right)\\
 & =O\left(\frac{d^{p}}{K^{p}}+\sup_{t}\|x_{t}^{\star}\|^{2p}\right)\\
 & =O(d^{p(q+1)}).
\end{align*}
\end{proof}
\begin{lem}
\label{lem:pi_V_dim_dependence}Fix $p\geq1$. For any $q\geq0$,
if $K^{-1}\vee\sup_{t\in[0,1]}\|x_{t}^{\star}\|^{2}=O(d^{q})$ as
$d\to\infty$, then $\sup_{t\in[0,1]}\pi_{t}(\bar{V}^{(p)})=O(d^{p(q+1)})$.
\end{lem}

\begin{proof}
We have
\begin{align}
\pi_{t}(\bar{V}^{(p)}) & \leq1+2^{2p-1}\pi_{t}(V_{t}^{p})+2^{2p-1}\|x_{t}^{\star}\|^{2p}.\label{eq:piV_t_bound}
\end{align}
By an application of (\ref{eq:drift_homogeneous}) with there $\epsilon=1$
and $\kappa=Kp/2$, we have for any $s>0$,
\[
\pi_{t}(V_{t}^{p})=\pi_{t}Q_{s}^{t,1}V_{t}^{p}\leq e^{-\delta s}\pi_{t}(V_{t}^{p})+\frac{\tilde{b}}{\delta},
\]
where
\[
\tilde{r}=2\sqrt{\frac{2(p-1)+d}{K}},\quad\tilde{b}=2p\tilde{r}^{2(p-1)}(2(p-1)+d),\quad\delta=Kp/2.
\]
hence taking $s\to\infty$, we obtain under the hypothesis $K^{-1}=O(d^{q})$,
\begin{align*}
\sup_{t}\pi_{t}(V_{t}^{p}) & \leq\frac{\tilde{b}}{\delta}=\frac{4}{K}2^{2(p-1)}\left(\frac{2(p-1)+d}{K}\right)^{(p-1)}(2(p-1)+d)\\
 & =O\left(\frac{1}{K}\left(\frac{d}{K}\right)^{p-1}d\right)=O\left(\frac{d^{p}}{K^{p}}\right)=O(d^{p+pq}),
\end{align*}
and combining this with (\ref{eq:piV_t_bound}) and the hypothesis
$\sup_{t}\|x_{t}^{\star}\|^{2}=O(d^{q})$ completes the proof.
\end{proof}
\begin{lem}
\label{lem:disc_drift_dim_depend}For any $q\geq0$, if
\[
K^{-1}\vee\sup_{t}\|\partial_{t}x_{t}^{\star}\|^{2}\vee\sup_{t}\|x_{t}^{\star}\|^{2}=O(d^{q}),\quad\mu_{0}(V)=O(d^{q+1}),
\]
\[
h\vee\epsilon\vee\frac{h}{\epsilon}\frac{L^{2}}{K}=o(1),\quad\frac{h}{\epsilon^{2}}d^{3q}=O(1),
\]
as $d\to\infty$, then
\[
h\sum_{k=0}^{\left\lfloor 1/h\right\rfloor -1}1+\mathbb{E}[\|\widetilde{X}_{kh}^{\epsilon,h}\|^{2}]=O(\epsilon d^{2q+1}+hd^{q+1}+d^{q}).
\]
\end{lem}

\begin{proof}
We have
\begin{equation}
h\sum_{k=0}^{\left\lfloor 1/h\right\rfloor -1}1+\mathbb{E}[\|\widetilde{X}_{kh}^{\epsilon,h}\|^{2}]\leq2h\sum_{k=0}^{\left\lfloor 1/h\right\rfloor -1}1+\mathbb{E}[\|\widetilde{X}_{kh}^{\epsilon,h}-x_{kh}^{\star}\|^{2}]+2h\sum_{k=0}^{\left\lfloor 1/h\right\rfloor -1}\|x_{kh}^{\star}\|^{2}.\label{eq:drift_disc_dim_bound}
\end{equation}

To estimate the first term on the r.h.s. of (\ref{eq:drift_disc_dim_bound}),
consider Lemma \ref{lem:disc_drift_p=00003D2} with $\delta$ there
chosen to be $1/2$ and note that under the hypothesis $\frac{h}{\epsilon}\frac{L^{2}}{K}=o(1)$,
we have $h/\epsilon\in(0,2K/L^{2})$ for all $d$ large enough. For
any such $d$, the bound of (\ref{eq:disc_mom_p=00003D2}) written
out explicitly together with the hypotheses $K^{-1}\vee\sup_{t}\|\partial_{t}x_{t}^{\star}\|^{2}=O(d^{q})$,
$\mu_{0}(V)=O(d^{q+1})$ and $\frac{h}{\epsilon^{2}}d^{3q}=O(1)$,
$h\vee\epsilon=o(1)$ then gives
\begin{align*}
 & h\sum_{k=0}^{\left\lfloor 1/h\right\rfloor -1}1+\mathbb{E}[\|\widetilde{X}_{kh}^{\epsilon,h}-x_{kh}^{\star}\|^{2}]\\
 & \leq1+\frac{h}{\frac{hK}{\epsilon}\left(1-\frac{1}{2}\frac{h}{\epsilon}\frac{L^{2}}{K}\right)}\left[\mu_{0}(V)+\sup_{t}\|\partial_{t}x_{t}^{\star}\|^{2}h^{2}\left\{ \frac{4}{\frac{hK}{\epsilon}\left(1-\frac{1}{2}\frac{h}{\epsilon}\frac{L^{2}}{K}\right)}+1\right\} +2d\frac{h}{\epsilon}\right]\\
 & =O\left(1+\frac{\epsilon}{K}\left[d^{q+1}+d^{q}h^{2}\left\{ \frac{\epsilon}{hK}+1\right\} +d\frac{h}{\epsilon}\right]\right)\\
 & =O\left(1+\epsilon d^{q}\left[d^{q+1}+d^{2q}h\epsilon+d^{q}h^{2}+d\frac{h}{\epsilon}\right]\right)\\
 & =O\left(1+\epsilon d^{2q+1}+d^{3q}h\epsilon^{2}+d^{2q}h^{2}\epsilon+d^{q+1}h\right)\\
 & =O\left(1+\epsilon d^{2q+1}+d^{q+1}h\right).
\end{align*}
The proof is completed by combining this estimate with the fact that
the second term on the r.h.s. of (\ref{eq:drift_disc_dim_bound})
is in $O(d^{q})$ due to the hypothesis $\sup_{t}\|x_{t}^{\star}\|^{2}=O(d^{q})$. 
\end{proof}
\begin{proof}
[Proof of Proposition \ref{prop:tv_dim_dependence}]First note that
the hypothesis $\frac{h}{\epsilon}\frac{L^{2}}{K}\in o(1)$ implies
that for $d$ large enough, $h/\epsilon\in(0,2K/L^{2})$. Then for
such $d$ and choosing $\delta=1/2$ in Proposition \ref{prop:tv_bound},
we have
\begin{align*}
 & \|\mu^{\epsilon}-\widetilde{\mu}^{\epsilon,h}\|_{\mathrm{tv}}^{2}\\
 & \leq L^{2}d\frac{h}{\epsilon^{2}}+\frac{h^{3}}{3\epsilon}\left(M^{2}+\frac{L^{4}}{\epsilon^{2}}\right)\\
 & \quad\cdot\left(\frac{1}{h}+\frac{1}{1-\lambda}\left[\mu_{0}(V_{0})+\frac{1}{h}\left(\sup_{t\in(0,1)}\|\partial_{t}x_{t}^{\star}\|^{2}\left[\frac{4h^{2}}{\delta\left(\frac{2hK}{\epsilon}-\left(\frac{h}{\epsilon}\right)^{2}L^{2}\right)}+h^{2}\right]+2d\frac{h}{\epsilon}\right)\right]\right)\\
 & =L^{2}d\frac{h}{\epsilon^{2}}+\frac{1}{3}\left(hM^{2}+\frac{h}{\epsilon^{2}}L^{4}\right)\\
 & \quad\cdot\left(\frac{h}{\epsilon}+\frac{\frac{h}{\epsilon}}{\frac{hK}{\epsilon}-\left(\frac{h}{\epsilon}\right)^{2}\frac{L^{2}}{2}}\left[h\mu_{0}(V_{0})+\epsilon h\sup_{t\in(0,1)}\|\partial_{t}x_{t}^{\star}\|^{2}\left[\frac{4\frac{h}{\epsilon}}{\left(\frac{2hK}{\epsilon}-\left(\frac{h}{\epsilon}\right)^{2}\frac{L^{2}}{2}\right)}+\frac{h}{\epsilon}\right]+2d\frac{h}{\epsilon}\right]\right)\\
 & =L^{2}d\frac{h}{\epsilon^{2}}+\frac{1}{3}\left(hM^{2}+\frac{h}{\epsilon^{2}}L^{4}\right)\\
 & \quad\cdot\left(\frac{h}{\epsilon}+\frac{1}{K-\frac{h}{\epsilon}\frac{L^{2}}{2}}\left[h\mu_{0}(V_{0})+\epsilon h\sup_{t\in(0,1)}\|\partial_{t}x_{t}^{\star}\|^{2}\left[\frac{4}{K-\frac{h}{\epsilon}\frac{L^{2}}{2}}+\frac{h}{\epsilon}\right]+2d\frac{h}{\epsilon}\right]\right).
\end{align*}
Using the hypotheses (\ref{eq:tv_bound_constants_hyp}), $\frac{h}{\epsilon}L^{2}/K=o(1)$,
, $dh/\epsilon=O(1)$, $h=o(1)$, and $\epsilon=o(1)$, we obtain
\begin{align*}
\|\mu^{\epsilon}-\widetilde{\mu}^{\epsilon,h}\|_{\mathrm{tv}}^{2} & =O\left(d^{q/2+1}\frac{h}{\epsilon^{2}}+\left(hd^{q}+\frac{h}{\epsilon^{2}}d^{q}\right)\left(\frac{h}{\epsilon}+d^{q}\left[hd^{q+1}+\epsilon hd^{q}[d^{q}+\frac{h}{\epsilon}]+d\frac{h}{\epsilon}\right]\right)\right)\\
 & =O\left(d^{q/2+1}\frac{h}{\epsilon^{2}}+\left(hd^{q}+\frac{h}{\epsilon^{2}}d^{q}\right)\left(\frac{h}{\epsilon}+hd^{2q+1}+\epsilon hd^{3q}+d^{q+1}\frac{h}{\epsilon}\right)\right)\\
 & =O\left(d^{q/2+1}\frac{h}{\epsilon^{2}}+\left(\frac{h^{2}}{\epsilon}d^{q}+h^{2}d^{3q+1}+\epsilon h^{2}d^{4q}+\frac{h^{2}}{\epsilon}d^{2q+1}\right)+\left(\frac{h^{2}}{\epsilon^{3}}d^{q}+\frac{h^{2}}{\epsilon^{2}}d^{3q+1}+\frac{h^{2}}{\epsilon}d^{4q}+\frac{h^{2}}{\epsilon^{3}}d^{2q+1}\right)\right)\\
 & =O\left(\left[\epsilon h^{2}+\frac{h^{2}}{\epsilon}\right]d^{4q}+\left[h^{2}+\frac{h^{2}}{\epsilon^{2}}\right]d^{3q+1}+\left[\frac{h^{2}}{\epsilon}+\frac{h^{2}}{\epsilon^{3}}\right]d^{2q+1}+\left[\frac{h^{2}}{\epsilon}+\frac{h^{2}}{\epsilon^{3}}\right]d^{q}+\frac{h}{\epsilon^{2}}d^{q/2+1}\right)\\
 & =O\left(\frac{h^{2}}{\epsilon}d^{4q}+\frac{h^{2}}{\epsilon^{2}}d^{3q+1}+\frac{h^{2}}{\epsilon^{3}}d^{2q+1}+\frac{h}{\epsilon^{2}}d^{q/2+1}\right)\\
 & =O\left(\frac{h}{\epsilon^{2}}d^{4q}\left[\epsilon h+hd^{1-q}+\frac{h}{\epsilon}d^{1-2q}+d^{1-7q/8}\right]\right)\\
 & =O\left(\frac{h}{\epsilon^{2}}d^{4q+1}\right).
\end{align*}
Taking the square root completes the proof.
\end{proof}
\begin{lem}
\label{lem:asymp_var_dim_depend}Fix $p\geq1$ and for each $d\in\mathbb{N}$,
$f\in C_{1,2}^{p}([0,1]\times\mathbb{R}^{d})$. Assume that (A\ref{hyp:dimension_dependence_mse-1})
holds and that $\sup_{s}\|\tilde{\mathcal{L}}_{s}f_{s}\|_{p+1/2}$,
grows at most polynomially fast as $d\to\infty$, where $\mathcal{\tilde{L}}_{s}f_{s}=-\left\langle \nabla U_{s},\nabla f_{s}\right\rangle +\Delta f_{s}$.
If $\sup_{t\in[0,1]}1/\mathrm{var}_{\pi_{t}}[f_{t}]$ grows at most
polynomially fast as $d\to\infty$, then for any $\ell\geq0$ so does
$\sup_{t\in[0,1]}1/\varsigma_{\ell}(t)$. 
\end{lem}

\begin{proof}
We first address the case $\ell=0.$ Using the formula (\ref{eq:homogeneousAICwithPoisson}),
we have
\[
\varsigma_{0}(s)=\int_{0}^{\infty}\rho_{s}(t)\mathrm{dt},
\]
where assuming w.l.o.g. that $f_{t}$ is centrered with respect to
$\pi_{t}$, $\rho_{s}(t)\coloneqq\pi_{s}(f_{s}Q_{t}^{s}f_{s})$. Due
to the reversibility of $Q_{t}^{s}$ with respect to $\pi_{s}$, $\rho_{s}(t)$
is a nonnegative, therefore for any $r\geq0$
\begin{equation}
\varsigma_{0}(s)\geq\int_{0}^{r}\rho_{s}(t)\mathrm{d}t.\label{eq:var_Sig_lowerbound}
\end{equation}
 We shall now show that 
\begin{equation}
\sup_{s}|\rho_{s}(0)-\rho_{s}(t)|\leq tC(d),\label{eq:rho0-rhot}
\end{equation}
where $C(d)$, to be identified below, grows at most polynomially
fast with $d$. To this end, note that
\[
|\rho_{s}(0)-\rho_{s}(t)|\leq\pi_{s}(|f_{s}||(Id-Q_{t}^{s})(f_{s})|)
\]
 and by the time-homogeneous counterpart of Proposition \ref{prop:fwd_and_bck_eqs},
\[
|(Q_{t}^{s}-Id)(f_{s})|(x)=\left|\int_{0}^{t}\partial_{u}Q_{u}^{s}f_{s}(x)\mathrm{du}\right|=\left|\int_{0}^{t}Q_{u}\mathcal{\tilde{L}}_{s}f_{s}(x)\mathrm{du}\right|\leq t\|\mathcal{\tilde{L}}f_{s}\|_{p+1/2}\tilde{\alpha}_{p+1/2}\bar{V}^{(p+1/2)}(x),
\]
where $\tilde{\alpha}_{p+1/2}$ is as in Proposition \ref{lem:drift_Y_process}
with $\kappa$ there chosen to be $Kp/2$, and we note that $\|\mathcal{L}_{s}f_{s}\|_{p+1/2}$
is finite by Proposition \ref{prop:C_2^p_closed}. We therefore have
\[
|\rho_{s}(0)-\rho_{s}(t)|\leq t\|\mathcal{L}_{s}f_{s}\|_{p+1/2}\tilde{\alpha}_{p+1/2}\pi_{s}(\bar{V}^{(p)}\bar{V}^{(p+1/2)}),
\]
and (\ref{eq:rho0-rhot}) holds as claimed with $C(d)\coloneqq\tilde{\alpha}_{p+1/2}\sup_{s}\|\mathcal{L}_{s}f_{s}\|_{p+1/2}\sup_{s}\pi_{s}(\bar{V}^{(p)}\bar{V}^{(p+1/2)})$,
which indeed grows at most polynomially with $d$ by the hypotheses
of the lemma, Lemma \ref{lem:alpha_dimension_dependence} and Lemma
\ref{lem:pi_V_dim_dependence}.

Returning then to (\ref{eq:var_Sig_lowerbound}) and applying (\ref{eq:rho0-rhot}),
we otbain
\[
\frac{1}{\varsigma_{0}(t)}\leq\frac{1}{r\rho_{t}(0)}\frac{1}{\left(1-\frac{rC(d)}{2\rho_{t}(0)}\right)}.
\]
Noting the hypothesis of the lemma on $\sup_{t}1/\mathrm{var}_{\pi_{t}}[f_{t}]$,
and that $\rho_{t}(0)=\mathrm{var}_{\pi_{t}}[f_{t}]$, the proof is
completed by choosing $r=d^{-a}$ for $a>0$ large enough.

The case $\ell>0$ is more straightforward, since in that situation
by (\ref{eq:homogeneousAICwithPoisson}) and the reversibility of
$Q_{t}^{s}$, $\varsigma_{\ell}(s)\geq\ell\mathrm{var}_{\pi_{s}}[f_{s}]$. 
\end{proof}